\pdfoutput=1
\RequirePackage{ifpdf}
\ifpdf 
\documentclass[pdftex]{sigma}
\else
\documentclass{sigma}
\fi

\usepackage[all]{xy}

\numberwithin{equation}{section}

\newtheorem{Theorem}{Theorem}[section]
\newtheorem*{Theorem*}{Theorem}
\newtheorem{Corollary}[Theorem]{Corollary}
\newtheorem{Lemma}[Theorem]{Lemma}
\newtheorem{Proposition}[Theorem]{Proposition}
\newtheorem{Conjecture}[Theorem]{Conjecture}
 { \theoremstyle{definition}

\newtheorem{Example}[Theorem]{Example}
\newtheorem{Remark}[Theorem]{Remark} }

\newcommand{\BZ}{\mathbb{Z}}
\newcommand{\BR}{\mathbb{R}}
\newcommand{\BC}{\mathbb{C}}
\newcommand{\BH}{\mathbb{H}}
\newcommand{\BO}{\mathbb{O}}
\newcommand{\BF}{\mathbb{F}}
\newcommand{\bk}{\mathbf{k}}
\newcommand{\bl}{\mathbf{l}}
\newcommand{\bm}{\mathbf{m}}
\newcommand{\bn}{\mathbf{n}}
\newcommand{\cB}{\mathcal{B}}
\newcommand{\cF}{\mathcal{F}}
\newcommand{\cH}{\mathcal{H}}
\newcommand{\cO}{\mathcal{O}}
\newcommand{\cP}{\mathcal{P}}
\newcommand{\cU}{\mathcal{U}}
\newcommand{\fa}{\mathfrak{a}}
\newcommand{\fg}{\mathfrak{g}}
\newcommand{\fh}{\mathfrak{h}}
\newcommand{\fk}{\mathfrak{k}}
\newcommand{\fl}{\mathfrak{l}}
\newcommand{\fm}{\mathfrak{m}}
\newcommand{\fn}{\mathfrak{n}}
\newcommand{\fp}{\mathfrak{p}}
\newcommand{\rK}{\mathrm{K}}

\newcommand{\eqspace}{\mathrel{\phantom{=}}}
\newcommand{\ds}{\displaystyle}
\newcommand{\itinv}{{\mathit{-1}}}
\newcommand{\Hom}{\operatorname{Hom}}
\newcommand{\End}{\operatorname{End}}
\renewcommand{\Re}{\operatorname{Re}}

\renewcommand{\det}{\operatorname{det}}
\newcommand{\Det}{\operatorname{Det}}
\newcommand{\Tr}{\operatorname{Tr}}
\newcommand{\Proj}{\operatorname{Proj}}
\newcommand{\rank}{\operatorname{rank}}

\newcommand{\Sym}{\operatorname{Sym}}
\newcommand{\Alt}{\operatorname{Alt}}
\newcommand{\Herm}{\operatorname{Herm}}
\newcommand{\hsum}{\sideset{}{^\oplus}\sum}
\newcommand{\hboxtimes}{\mathbin{\hat{\boxtimes}}}
\newcommand{\hotimes}{\mathbin{\hat{\otimes}}}
\def\labelenumi{({\arabic{enumi}})}

\begin{document}
\allowdisplaybreaks

\newcommand{\arXivNumber}{2207.11663}

\renewcommand{\PaperNumber}{049}

\FirstPageHeading

\ShortArticleName{Computation of Weighted Bergman Inner Products}

\ArticleName{Computation of Weighted Bergman Inner Products\\ on Bounded Symmetric Domains and\\ Parseval--Plancherel-Type Formulas under Subgroups}

\Author{Ryosuke NAKAHAMA~$^{\rm ab}$}

\AuthorNameForHeading{R.~Nakahama}

\Address{$^{\rm a)}$~Institute of Mathematics for Industry, Kyushu University,\\
\hphantom{$^{\rm a)}$}~744 Motooka, Nishi-ku Fukuoka 819-0395, Japan}
\Address{$^{\rm b)}$~NTT Institute for Fundamental Mathematics, NTT Communication Science Laboratories,\\
\hphantom{$^{\rm b)}$}~Nippon Telegraph and Telephone Corporation,\\
\hphantom{$^{\rm b)}$}~3-9-11 Midori-cho, Musashino-shi, Tokyo 180-8585, Japan}
\EmailD{\href{mailto:ryosuke.nakahama@ntt.com}{ryosuke.nakahama@ntt.com}}

\ArticleDates{Received September 21, 2022, in final form June 26, 2023; Published online July 21, 2023}

\Abstract{Let $(G,G_1)=(G,(G^\sigma)_0)$ be a~symmetric pair of holomorphic type, and we consider a~pair of Hermitian symmetric spaces $D_1=G_1/K_1\subset D=G/K$, realized as bounded symmetric domains in complex vector spaces $\fp^+_1:=(\fp^+)^\sigma\subset\fp^+$ respectively. Then the universal covering group $\widetilde{G}$ of $G$ acts unitarily on the weighted Bergman space $\cH_\lambda(D)\subset\cO(D)=\cO_\lambda(D)$ on $D$ for sufficiently large $\lambda$. Its restriction to the subgroup~$\widetilde{G}_1$ decomposes discretely and multiplicity-freely, and its branching law is given explicitly by Hua--Kostant--Schmid--Kobayashi's formula in terms of the $\widetilde{K}_1$-decomposition of the space~\smash{$\cP\bigl(\fp^+_2\bigr)$} of polynomials on $\fp^+_2:=(\fp^+)^{-\sigma}\subset\fp^+$. The object of this article is to understand~the~decomposition of the restriction $\cH_\lambda(D)|_{\widetilde{G}_1}$ by studying the weighted Bergman inner product on each \smash{$\widetilde{K}_1$}-type in \smash{$\cP\bigl(\fp^+_2\bigr)\subset\cH_\lambda(D)$}. For example, by computing explicitly the norm~$\Vert f\Vert_\lambda$ for \smash{$f=f(x_2)\in\cP\bigl(\fp^+_2\bigr)$}, we can determine the Parseval--Plancherel-type formula for the decomposition of $\cH_\lambda(D)|_{\widetilde{G}_1}$. Also, by computing the poles of \smash{$\bigl\langle f(x_2),{\rm e}^{(x|\overline{z})_{\fp^+}}\bigr\rangle_{\lambda,x}$} for \smash{$f(x_2)\in\cP\bigl(\fp^+_2\bigr)$},$x=(x_1,x_2)$, \smash{$z\in\fp^+=\fp^+_1\oplus\fp^+_2$}, we can get some information on branching of $\cO_\lambda(D)|_{\widetilde{G}_1}$ also for $\lambda$ in non-unitary range. In this article we consider these problems for all $\widetilde{K}_1$-types in \smash{$\cP\bigl(\fp^+_2\bigr)$}.}


\Classification{22E45; 43A85; 17C30}

\tableofcontents

\section{Introduction}

The purpose of this article is to study the restriction of holomorphic discrete series representations to some subgroups,
e.g., to determine the Parseval--Plancherel-type formulas,
by computing the weighted Bergman inner products on bounded symmetric domains.
This article is a~continuation of the author's previous articles~\cite{N, N2}.

We consider a~Hermitian symmetric space $D\simeq G/K$, realized as a~bounded symmetric domain $D\subset\fp^+$ in a~complex vector space $\fp^+$ centered at the origin.
We assume that $G$ is connected, and $K$ is the isotropy subgroup at $0\in\fp^+$.
Let $\widetilde{G}$, $\widetilde{K}$ denote the universal covering groups of $G$, $K$, let $(\tau,V)$ be a~finite-dimensional representation of $\widetilde{K}$,
and we consider the homogeneous vector bundle $\widetilde{G}\times_{\widetilde{K}} V\to \widetilde{G}/\widetilde{K}\simeq D$.
Then we can trivialize this bundle, and the space of holomorphic sections is identified with the space $\cO(D,V)=\cO_\tau(D,V)$ of $V$-valued holomorphic functions on $D$,
on which $\widetilde{G}$ acts by
\[ (\hat{\tau}(g)f)(x):=\tau\bigl(\kappa\bigl(g^{-1},x\bigr)\bigr)^{-1}f\bigl(g^{-1}x\bigr), \qquad g\in\widetilde{G},\quad x\in D,\quad f\in\cO(D,V), \]
by using some function $\tau\colon \widetilde{G}\times D\to\widetilde{K}^\BC$.
According to the choice of the representation $(\tau,V)$ of $\widetilde{K}$, there may or may not exist a~Hilbert subspace $\cH_\tau(D,V)\subset\cO_\tau(D,V)$
on which $\widetilde{G}$ acts unitarily. The classification of $(\tau,V)$ such that $\cH_\tau(D,V)$ exists is given by~\cite{EHW, J}.
Especially, if the $\widetilde{G}$-invariant inner product of $\cH_\tau(D,V)$ is given by an integral on $D$,
then this is called a~\textit{weighted Bergman inner product} and $(\hat{\tau},\cH_\tau(D,V))$ is called a~\textit{holomorphic discrete series representation}.
For example, suppose $G$ is simple and $(\tau,V)=\bigl(\chi^{-\lambda},\BC\bigr)$ is one-dimensional, where $\chi$ is a~suitably normalized character of $K$.
Then for sufficiently large $\lambda\in\BR$, the weighted Bergman inner product is given by the integral of the form
\[ \langle f,g\rangle_\lambda:=C_\lambda\int_D f(x)\overline{g(x)}h(x,\overline{x})^{\lambda-p}\,{\rm d}x, \]
where $p\in\BZ_{>0}$, $h(x,\overline{y})$ is a~suitable polynomial on $\fp^+\times\overline{\fp^+}$,
and the corresponding reproducing kernel (\textit{weighted Bergman kernel}) is given by $h(x,\overline{y})^{-\lambda}$, if we choose the normalizing constant $C_\lambda$ suitably.
In this case we write $(\hat{\tau},\cH_\tau(D,V))=:(\tau_\lambda,\cH_\lambda(D))$, and call it a~holomorphic discrete series representation of \textit{scalar type}.

Let us give a~concrete example. We consider
\[ G={\rm U}(q,s):=\left\{g=\begin{pmatrix}a&b\\c&d\end{pmatrix}\in {\rm GL}(q+s,\BC)\ \middle|\ g^*\begin{pmatrix}I_q&\hphantom{-}0\\0&-I_s\end{pmatrix}g=\begin{pmatrix}I_q& \hphantom{-}0\\0&-I_s\end{pmatrix}\right\}. \]
Then $G$ acts transitively on
\[ D:=\{x\in {\rm M}(q,s;\BC)\mid I-xx^*\text{ is positive definite}\}\subset\fp^+:={\rm M}(q,s;\BC) \]
by the linear fractional transform
\[ \begin{pmatrix}a&b\\c&d\end{pmatrix}.x:=(ax+b)(cx+d)^{-1}=(a^*+xb^*)^{-1}(c^*+xd^*). \]
Next let $\lambda_1,\lambda_2\in\BC$. Then the universal covering group $\widetilde{G}$ acts on $\cO(D)$ by
\[ \left(\tau_{\lambda_1,\lambda_2}\left(\begin{pmatrix}a&b\\c&d\end{pmatrix}^{-1}\right)f\right)(x)
:=\det(a^*+xb^*)^{-\lambda_1}\det(cx+d)^{-\lambda_2}f\bigl((ax+b)(cx+d)^{-1}\bigr). \]
We note that $\det(a^*+xb^*)^{-\lambda_1}\det(cx+d)^{-\lambda_2}$ is not well-defined on $G\times D$ unless $\lambda_1,\lambda_2\in\BZ$,
but is well-defined on the universal covering space $\widetilde{G}\times D$.
If $\lambda_1,\lambda_2\in\BR$ satisfies $\lambda_1+\lambda_2>q+s-1$, then this preserves the weighted Bergman inner product
\[ \langle f,g\rangle_{\lambda_1+\lambda_2}:=C_\lambda\int_D f(x)\overline{g(x)}\det(I-xx^*)^{\lambda_1+\lambda_2-(q+s)}\,{\rm d}x, \]
and the corresponding Hilbert space $\cH_{\lambda_1+\lambda_2}(D)\subset\cO(D)$ gives a~holomorphic discrete series representation of scalar type.
The restriction of $\tau_{\lambda_1,\lambda_2}$ to the subgroup ${\rm SU}(q,s)$ depends only on the sum $\lambda_1+\lambda_2$.

Next we consider an involution $\sigma$ on $G$, and let $G_1:=(G^\sigma)_0$ be the identity component of the group of fixed points by $\sigma$.
Without loss of generality we may assume $\sigma$ stabilizes $K$, and let $K_1:=G_1\cap K$.
Then $D_1:=G_1.0\simeq G_1/K_1$ is either a~complex submanifold or a~totally real submanifold of $D\simeq G/K$.
The pair $(G,G_1)$ is called a~symmetric pair of \textit{holomorphic type} for the former case, and of \textit{anti-holomorphic type} for the latter case (see~\cite[Section 3.4]{Kmf1}).
In the following we assume $(G,G_1)$ is of holomorphic type. Then $\sigma$ induces a~holomorphic action on $\fp^+\simeq T_0(G/K)$,
and let $(\fp^+)^\sigma=:\fp^+_1$, $(\fp^+)^{-\sigma}=:\fp^+_2$, so that $\fp^+=\fp^+_1\oplus\fp^+_2$ holds.
Now we consider the restriction of a~holomorphic discrete series representation $\cH_\tau(D,V)$ of $\widetilde{G}$ to the subgroup $\widetilde{G}_1$.
Then the restriction $\cH_\tau(D,V)|_{\widetilde{G}_1}$ decomposes into a~Hilbert direct sum of irreducible representations of $\widetilde{G}_1$,
and this decomposition is given in terms of the decomposition of $\cP\bigl(\fp^+_2\bigr)\otimes (V|_{\widetilde{K}_1})$ under $\widetilde{K}_1$,
where $\cP\bigl(\fp^+_2\bigr)$ denotes the space of polynomials on $\fp^+_2$. That is, if $\cP\bigl(\fp^+_2\bigr)\otimes (V|_{\widetilde{K}_1})$ is decomposed under $\widetilde{K}_1$ as
\[
\cP\bigl(\fp^+_2\bigr)\otimes (V|_{\widetilde{K}_1})\simeq\bigoplus_{j} m(\tau,\rho_j)(\rho_j,W_j), \qquad m(\tau,\rho_j)\in\BZ_{\ge 0},
\]
then $\cH_\tau(D,V)|_{\widetilde{G}_1}$ is decomposed abstractly into the Hilbert direct sum
\[ \cH_\tau(D,V)|_{\widetilde{G}_1}\simeq\hsum_j m(\tau,\rho_j)\cH_{\rho_j}(D_1,W_j) \]
(see Kobayashi~\cite[Lemma 8.8]{Kmf1},~\cite[Section 8]{Kmf0-1}, for earlier results, see also~\cite{JV,Ma}).
We note that if the unitary subrepresentation $\cH_\tau(D,V)$ is not holomorphic discrete,
then its $\widetilde{K}$-finite part $\cH_\tau(D,V)_{\widetilde{K}}$ may be strictly smaller than $\cO_\tau(D,V)_{\widetilde{K}}=\cP(\fp^+)\otimes V$,
and the above decomposition may not hold in general.

Let us observe this decomposition when $G$ is simple and $\cH_\tau(D,V)=\cH_\lambda(D)$ is a~holomorphic discrete series representation of scalar type.
In general, $\fp^+$ has a~\textit{Jordan triple system} structure, and if $\fp^+$ is simple, then for an involution $\sigma$ on $\fp^+$,
$\fp^+_2=(\fp^+)^{-\sigma}$ is a~direct sum of at most~2 simple Jordan triple subsystems.
Let $\chi$, $\chi_1$ be suitable characters of $\widetilde{K}$, $\widetilde{K}_1$, respectively, and let $\chi|_{\widetilde{K}_1}=\chi_1^{\varepsilon_1}$, where $\varepsilon_1\in\{1,2\}$.
Then the decomposition of $\chi^{-\lambda}\otimes\cP\bigl(\fp^+_2\bigr)$ under $\widetilde{K}_1$,
\[ \bigl(\chi^{-\lambda}|_{\widetilde{K}_1}\bigr)\otimes\cP\bigl(\fp^+_2\bigr)\simeq \chi^{-\varepsilon_1\lambda}\otimes\bigoplus_{\tilde{\bk}}\cP_{\tilde{\bk}}\bigl(\fp^+_2\bigr) \]
is given by using the parameter set
\begin{alignat*}{3}
&\tilde{\bk}=\bk \in\BZ_{++}^{r_2}:=\{(k_1,\dots,k_{r_2})\in\BZ^{r_2}\mid k_1\ge\cdots\ge k_{r_2}\ge 0\}, \qquad && \fp^+_2\colon\text{simple} ,& \\
&\tilde{\bk}=(\bk,\bl) \in\BZ_{++}^{r'}\times\BZ_{++}^{r''}, && \fp^+_2\colon\text{non-simple} &
\end{alignat*}
for some $r_2,r',r''\in\BZ_{>0}$. According to this decomposition, $\cH_\lambda(D)$ is decomposed under $\widetilde{G}_1$ as
\begin{equation}\label{intro_decomp}
\cH_\lambda(D)|_{\widetilde{G}_1}\simeq \hsum_{\tilde{\bk}}\cH_{\varepsilon_1\lambda}\bigl(D_1,\cP_{\tilde{\bk}}\bigl(\fp^+_2\bigr)\bigr)
\end{equation}
(see Kobayashi~\cite[Theorem 8.3]{Kmf1}), where $\cH_{\varepsilon_1\lambda}\bigl(D_1,\cP_{\tilde{\bk}}\bigl(\fp^+_2\bigr)\bigr)$ denotes the holomorphic discrete series representation of $\widetilde{G}_1$
with the fiber $\chi_1^{-\varepsilon_1\lambda}\otimes\cP_{\tilde{\bk}}\bigl(\fp^+_2\bigr)$. Each $\cH_{\varepsilon_1\lambda}\bigl(D_1,\cP_{\tilde{\bk}}\bigl(\fp^+_2\bigr)\bigr)$-isotypic component in $\cH_\lambda(D)$
is generated by $\cP_{\tilde{\bk}}\bigl(\fp^+_2\bigr)\subset\cP(\fp^+)=\cH_\lambda(D)_{\widetilde{K}}$ as a~$\widetilde{G}_1$-module.
Our aim is to understand this decomposition concretely by studying the weighted Bergman inner product
\begin{equation}\label{intro_aim}
\bigl\langle f(x_2),{\rm e}^{(x|\overline{z})_{\fp^+}}\bigr\rangle_{\lambda,x}, \qquad x=(x_1,x_2), \quad z\in\fp^+=\fp^+_1\oplus\fp^+_2,\quad f(x_2)\in\cP_{\tilde{\bk}}\bigl(\fp^+_2\bigr) .
\end{equation}
Here the subscript $x$ stands for the variable of integration.

One of main problems on the restriction of representations is to determine the $\widetilde{G}_1$-intertwining operators
\begin{alignat*}{4}
&\cF_{\tau\rho_j}^\downarrow\colon \ &\cH_\tau(D,V)|_{\widetilde{G}_1}&\longrightarrow \cH_{\rho_j}(D_1,W_j) \quad &&\text{or} \quad
& \cO_\tau(D,V)|_{\widetilde{G}_1}&\longrightarrow \cO_{\rho_j}(D_1,W_j), \\
&\cF_{\tau\rho_j}^\uparrow\colon \ &\cH_{\rho_j}(D_1,W_j)&\longrightarrow \cH_\rho(D,V)|_{\widetilde{G}_1} \quad &&\text{or} \quad
& \cO_{\rho_j}(D_1,W_j)&\longrightarrow \cO_\rho(D,V)|_{\widetilde{G}_1}.
\end{alignat*}
$\cF_{\tau\rho_j}^\downarrow$ is called a~\textit{symmetry breaking operator} and $\cF_{\tau\rho_j}^\uparrow$ is called a~\textit{holographic operator},
according to the terminology introduced in~\cite{KP1, KP2} and~\cite{KP3} respectively.
Such a~problem is proposed by Kobayashi from the viewpoint of the representation theory (see~\cite{K1}),
and studied from various viewpoints, e.g., by~\cite{BCK, Cl, C, IKO, Ju, Kfmeth, KKP, KOSS, KP1, KP2, KP3, KS1, KS2, MO, N, N2, OR, P, PZ, R}
for holomorphic discrete series, principal series, and complementary series representations.
Especially, it is proved by~\cite{Kfmeth, KP1} that in the holomorphic setting, symmetry breaking operators $\cF_{\tau\rho_j}^\downarrow\colon\cO_\tau(D,V)|_{\widetilde{G}_1}\to\cO_{\rho_j}(D_1,W_j)$
are always given by differential operators, and their symbols are characterized as polynomial solutions of certain systems of differential equations (F-method).
Also, by~\cite{N} when $\cH_\tau(D,V)$ is holomorphic discrete, the differential operator
\[
\tilde{\cF}_{\tau\rho_j}^\downarrow\colon \ \cH_\tau(D,V)\longrightarrow\cH_{\rho_j}(D_1,W_j), \qquad
f(x)=f(x_1,x_2)\longmapsto \tilde{F}_{\tau\rho_j}^\downarrow\left(\frac{\partial}{\partial x}\right)f(x)\biggr|_{x_2=0}
\]
defined by using the operator-valued polynomial $\tilde{F}_{\tau\rho_j}^\downarrow(z)\in\cP(\fp^-)\otimes\Hom_\BC(V,W_j)$ on the dual space $\fp^-$ of $\fp^+$,
\[
\tilde{F}_{\tau\rho_j}^\downarrow(z)=\bigl\langle {\rm e}^{(x|z)_{\fp^+}},\rK(x_2)\bigr\rangle_{\cH_{\tau}(D,V),x},\qquad
\rK(x_2)\in \bigl(\cP\bigl(\fp^+_2\bigr)\otimes\Hom_\BC\bigl(\overline{V},\overline{W_j}\bigr)\bigr)^{\widetilde{K}_1}
\]
becomes a~symmetry breaking operator. Here $(\cdot|\cdot)_{\fp^+}$ is a~suitable non-degenerate pairing on $\fp^+\times\fp^-$.
Especially when $\cH_{\tau}(D,V)=\cH_\lambda(D)$ is of scalar type, the construction of symmetry breaking operators
\[
\cF_{\lambda,\tilde{\bk}}^\downarrow\colon \ \cH_\lambda(D)|_{\widetilde{G}_1}\longrightarrow \cH_{\varepsilon_1\lambda}\bigl(D_1,\cP_{\tilde{\bk}}\bigl(\fp^+_2\bigr)\bigr)
\]
is reduced to the computation of~(\ref{intro_aim}).
Since the restriction of a~unitary highest weight representation of scalar type $\cH_\lambda(D)|_{\widetilde{G}_1}$ decomposes multiplicity-freely by~\cite[Theorem A]{Kmf1},
symmetry breaking operators from $\cH_\lambda(D)|_{\widetilde{G}_1}$ are unique up to constant multiple.
On the other hand, if~$\lambda$ is not in the holomorphic discrete range, we do not know a~priori whether symmetry breaking operators from $\cO_\lambda(D)|_{\widetilde{G}_1}$ are unique or not.
Indeed, for tensor product case, the space of symmetry breaking operators from $\cO_\lambda(D)\hotimes\cO_\mu(D)$ may become greater than 1-dimensional for some $(\lambda,\mu)\in\BC$
(see~\cite[Section 9]{KP2}, or Section~\ref{section_tensor} of this paper).

In~\cite{N2}, when $\fp^+$ is simple and when both $\fp^+$, $\fp^+_2$ are of tube type, we explicitly computed the inner product~(\ref{intro_aim}) for $\tilde{\bk}$ of the form one of
$\tilde{\bk}=\bk=(k+l,k,\dots,k)$, $\tilde{\bk}=\bk=(k+1,\allowbreak\dots,k+1,k,\dots,k)$ or $\tilde{\bk}=(\bk,\bl)=((k,\dots,k),\bl)$ according to the choice of $\bigl(\fp^+,\fp^+_1,\fp^+_2\bigr)$,
and determined the symmetry breaking operators for these $\tilde{\bk}$.
Especially if $\tilde{\bk}=(k,\dots,k)$ or $\tilde{\bk}=((k,\dots,k),(l,\dots,l))$, then $\cH_{\varepsilon_1\lambda}\bigl(D_1,\cP_{\tilde{\bk}}\bigl(\fp^+_2\bigr)\bigr)$ becomes of scalar type,
and the inner product~(\ref{intro_aim}) and the symmetry breaking operator are given by using Heckman--Opdam's hypergeometric polynomials of type $BC$.
Also in~\cite{N}, we have constructed the holographic operators for some~$\tilde{\bk}$ containing the above cases, as infinite-order differential operators.

We continue to assume $\cH_\tau(D,V)=\cH_\lambda(D)$ is of scalar type. Then since the decomposition~(\ref{intro_decomp}) is a~Hilbert direct sum and is multiplicity-free,
under a~suitable normalization of $\cF_{\lambda,\tilde{\bk}}^\downarrow$, there exist constants $C\bigl(\lambda,\tilde{\bk}\bigr)>0$ such that
\[ \Vert f\Vert_\lambda^2=\sum_{\tilde{\bk}} C\bigl(\lambda,\tilde{\bk}\bigr)\big\Vert \cF_{\lambda,\tilde{\bk}}^\downarrow f\big\Vert_{\varepsilon_1\lambda,\tilde{\bk}}^2, \qquad f\in\cH_\lambda(D) \]
holds, where $\Vert\cdot\Vert_{\varepsilon_1\lambda,\tilde{\bk}}$ is a~$\widetilde{G}_1$-invariant norm on $\cH_{\varepsilon_1\lambda}\bigl(D_1,\cP_{\tilde{\bk}}\bigl(\fp^+_2\bigr)\bigr)$.
This formula is regarded as an analogue of the Parseval or the Plancherel theorems for the Fourier analysis.
Such Parseval--Plancherel-type formulas for symmetric pairs of holomorphic type are studied, e.g., by~\cite{BS1, BS2, BS3, HK1, HK2, KP3}
under different realizations of $\cH_\lambda(D)$, $\cH_{\varepsilon_1\lambda}\bigl(D_1,\cP_{\tilde{\bk}}\bigl(\fp^+_2\bigr)\bigr)$,
and those for anti-holomorphic type cases are studied, e.g., by~\cite{MS, Ne1, Ne2, Ne3, OZ, Sep1, Sep2, Sep3,vDP1, vDP2, Z1, Z2}.
In our setting, for~$(G,G_1)$ of holomorphic type, each $\cH_{\varepsilon_1\lambda}\bigl(D_1,\cP_{\tilde{\bk}}\bigl(\fp^+_2\bigr)\bigr)$-isotypic component in $\cH_\lambda(D)$ is generated by the minimal $\widetilde{K}_1$-type
$\cP_{\tilde{\bk}}\bigl(\fp^+_2\bigr)\subset\cP(\fp^+)=\cH_\lambda(D)_{\widetilde{K}}$, and we assume that $\cF_{\lambda,\tilde{\bk}}^\downarrow$ is normalized such that
$\bigl\Vert \cF_{\lambda,\tilde{\bk}}^\downarrow f\bigr\Vert_{\varepsilon_1\lambda,\tilde{\bk}}$ is independent of the parameter $\lambda$ for $f=f(x_2)\in\cP_{\tilde{\bk}}\bigl(\fp^+_2\bigr)$.
Under this setting we want to determine the $\lambda$-dependence of $C\bigl(\lambda,\tilde{\bk} \bigr)$ explicitly.
To do this, it is enough to compute $\Vert f\Vert_\lambda$ for $f=f(x_2)\in\cP_{\tilde{\bk}}\bigl(\fp^+_2\bigr)$ for each $\tilde{\bk}$,
and this is deduced from the top term, i.e., the value at $z_1=0$, of the inner product~(\ref{intro_aim}) (see Proposition~\ref{prop_Plancherel}).

Next we consider $\cO_\lambda(D)$ for $\lambda$ not in the holomorphic discrete range.
Following Faraut--Kor\'anyi~\cite{FK0, FK}, the weighted Bergman inner product $\langle\cdot,\cdot\rangle_\lambda$ is explicitly computed
on each $\widetilde{K}$-type in $\cP(\fp^+)=\cH_\lambda(D)_{\widetilde{K}}$ for sufficiently large $\lambda$,
and this is meromorphically continued for all $\lambda\in\BC$.
For smaller~$\lambda$, $\cO_\lambda(D)_{\widetilde{K}}$ becomes reducible as a~$\bigl(\fg,\widetilde{K}\bigr)$-module if $\lambda$ is a~pole of~$\langle\cdot,\cdot\rangle_\lambda$,
and for any $f\in\cP(\fp^+)$, by computing the poles and their orders of the inner product $\bigl\langle f,{\rm e}^{(\cdot|\overline{z})_{\fp^+}}\bigr\rangle_{\lambda}$ with respect to $\lambda$,
we can determine which submodule $f$ sits in. Similarly, for a~symmetric subgroup $G_1\subset G$ of holomorphic type, the decomposition~(\ref{intro_decomp}) is restated as
\[
\cO_\lambda(D)_{\widetilde{K}}\bigr|_{(\fg_1,\widetilde{K}_1)}=\cP(\fp^+)\bigr|_{(\fg_1,\widetilde{K}_1)}=\bigoplus_{\tilde{\bk}}{\rm d}\tau_\lambda(\cU(\fg_1))\cP_{\tilde{\bk}}\bigl(\fp^+_2\bigr)
\]
for sufficiently large $\lambda$, where $\cU(\fg_1)$ denotes the universal enveloping algebra of $\fg_1$, but this does not hold for smaller $\lambda$ in general.
By computing the poles and their orders of the inner product~(\ref{intro_aim}) with respect to $\lambda$, we can determine which $\bigl(\fg,\widetilde{K}\bigr)$-submodule of $\cO_\lambda(D)_{\widetilde{K}}$
contains the $\bigl(\fg_1,\widetilde{K}_1\bigr)$-module ${\rm d}\tau_\lambda(\cU(\fg_1))\cP_{\tilde{\bk}}\bigl(\fp^+_2\bigr)$ (see Proposition~\ref{prop_poles}).

In this article we treat irreducible symmetric pairs $(G,G_1)=(G,(G^\sigma)_0)$ of holomorphic type, and treat holomorphic discrete series representations $\cH_\lambda(D)$ of scalar type.
According to the decomposition $\cP\bigl(\fp^+_2\bigr)=\bigoplus_{\tilde{\bk}}\cP_{\tilde{\bk}}\bigl(\fp^+_2\bigr)$ of the space of polynomials on $\fp^+_2=(\fp^+)^{-\sigma}$,
we compute the top terms (Theorems~\ref{thm_topterm_nonsimple},~\ref{thm_topterm_simple}, and~\ref{thm_topterm_tensor})
and the poles (Theorems~\ref{thm_poles_nonsimple},~\ref{thm_poles_simple}, and~\ref{thm_poles_tensor})
of the weighted Bergman inner product~(\ref{intro_aim}) for all $\tilde{\bk}=\bk\in\BZ_{++}^{r_2}$ or ${\tilde{\bk}=(\bk,\bl)\in\BZ_{++}^{r'}\times\BZ_{++}^{r''}}$.
Also, we apply these results for the determination of Parseval--Plancherel-type formulas (Corollaries~\ref{cor_Plancherel_easycase},~\ref{cor_Plancherel_nonsimple},~\ref{cor_Plancherel_simple}, and~\ref{cor_Plancherel_tensor}) and $\bigl(\fg,\widetilde{K}\bigr)$-modules (Corollaries~\ref{cor_submodule_easycase},~\ref{cor_submodule_nonsimple},~\ref{cor_submodule_simple}, and~\ref{cor_submodule_tensor}).

This paper is organized as follows. In Section~\ref{section_preliminaries}, we review Jordan triple systems, Jordan algebras and holomorphic discrete series representations.
In Section~\ref{section_keylemmas}, we introduce some lemmas needed in later sections.
In Section~\ref{section_easycase}, we treat $(G,G_1)$ such that~(\ref{intro_aim}) is easily computable.
For these $(G,G_1)$, the symmetry breaking operators are given by normal derivatives along $\fp^+_2$ independent of the parameter $\lambda$,
and~(\ref{intro_aim}) is computed directly by using Faraut--Kor\'anyi's result~\cite{FK0,FK}.
In Section~\ref{section_nonsimple}, we treat the cases such that $\fp^+_2$ is a~direct sum of two simple Jordan triple systems,
and in Section~\ref{section_simple}, we treat the cases such that $\fp^+_2$ is simple.
However, for some exceptional $\bigl(\fp^+,\fp^+_1,\fp^+_2\bigr)$, we need some extra computation to determine the poles of~(\ref{intro_aim}).
In Section~\ref{section_rank3}, we compute the poles of~(\ref{intro_aim}) when $\fp^+$, $\fp^+_2$ are simple of rank 3. The above exceptional case is contained in these rank 3 cases.
Also, in this section we give a~conjecture on the full computation of~(\ref{intro_aim}) for rank 3 cases.
In Section~\ref{section_tensor}, we treat the cases $\bigl(\fp^+,\fp^+_1\bigr)=\bigl(\fp^+_0\oplus\fp^+_0,\Delta\bigl(\fp^+_0\bigr)\bigr)$, and the tensor product representation $\cH_\lambda(D_0)\hotimes\cH_\mu(D_0)$.
In this section we also discuss the higher-multiplicity phenomena, which is a~generalization of the argument in~\cite[Section 9]{KP2}.

\section{Preliminaries}\label{section_preliminaries}

In this section we review Jordan triple systems, Jordan algebras and holomorphic discrete series representations.
We use almost the same notations as in~\cite{N2}, and readers who read~\cite{N2} can skip Sections~\ref{subsection_HPJTS}--\ref{subsection_classification}.
For detail see also, e.g.,~\cite[Parts III and V]{FKKLR} and~\cite{FK,L,Sat}.

\subsection{Hermitian positive Jordan triple systems}\label{subsection_HPJTS}

We consider a~Hermitian positive Jordan triple system $\bigl(\fp^+,\fp^-,\{\cdot,\cdot,\cdot\},\overline{\cdot}\bigr)$,
where $\fp^\pm$ are finite-dimensional vector spaces over $\BC$, with a~non-degenerate bilinear form $(\cdot|\cdot)_{\fp^\pm}\colon \fp^\pm\times\fp^\mp\to\BC$,
$\{\cdot,\cdot,\cdot\}\colon\fp^\pm\times\fp^\mp\times\fp^\pm\to\fp^\pm$ is a~$\BC$-trilinear map satisfying
\begin{gather*}
\begin{split}
& \{x,y,z\}=\{z,y,x\}, \\
& \{u,v,\{x,y,z\}\}=\{\{u,v,x\},y,z\}-\{x,\{v,u,y\},z\}+\{x,y,\{u,v,z\}\}, \\
& (\{u,v,x\}|y)_{\fp^\pm}=(x|\{v,u,y\})_{\fp^\pm}
\end{split}
\end{gather*}
for any $u,x,z\in\fp^\pm$, $v,y\in\fp^\mp$, and $\overline{\cdot}\colon\fp^\pm\to\fp^\mp$ is a~$\BC$-antilinear involutive isomorphism
satisfying $(x|\overline{x})_{\fp^\pm}\ge 0$ for any $x\in\fp^\pm$.
Let $D,B\colon\fp^\pm\times\fp^\mp\to\End_\BC\bigl(\fp^\pm\bigr)$, $Q\colon\fp^\pm\times\fp^\pm\to\Hom_\BC(\fp^\mp,\fp^\pm)$, $Q\colon\fp^\pm\to\Hom_\BC(\fp^\mp,\fp^\pm)$ be the maps given by
\begin{gather*}
D(x,y)z=Q(x,z)y:=\{x,y,z\}, \\
Q(x):=\frac{1}{2}Q(x,x), \\
B(x,y):=I_{\fp^\pm}-D(x,y)+Q(x)Q(y)
\end{gather*}
for $x,z\in\fp^\pm$, $y\in\fp^\mp$, let $h=h_{\fp^\pm}\colon\fp^\pm\times\fp^\mp\to\BC$ be the \textit{generic norm}, which is a~polynomial on $\fp^+\times\fp^-$
irreducible on each simple Jordan triple subsystem, and write $B(x):=B(x,\overline{x})$, $h(x):=h(x,\overline{x})$.
If $\fp^+$ is simple, then $B(x,y)$ and $h(x,y)$ are related as
\begin{equation}\label{formula_hB}
h(x,y)^p=\Det_{\fp^+}B(x,y), \qquad x\in\fp^+,\quad y\in\fp^-
\end{equation}
for some $p\in\BZ_{>0}$.

If $e\in\fp^+$ is a~tripotent, i.e., $\{e,\overline{e},e\}=2e$, then $D(e,\overline{e})\in\End_\BC(\fp^+)$ has the eigenvalues~$0$,~$1$,~$2$. For $j=0,1,2$, let
\begin{align}
\fp^+(e)_j=\fp^+(\overline{e})_j:=\{x\in\fp^+\mid D(e,\overline{e})x=jx\}\subset\fp^+,\nonumber\\
\fp^-(\overline{e})_j=\fp^-(e)_j:=\{x\in\fp^-\mid D(\overline{e},e)x=jx\}\subset\fp^-, \label{Peirce}
\end{align}
so that $\fp^\pm=\fp^\pm(e)_2\oplus\fp^\pm(e)_1\oplus\fp^\pm(e)_0$ holds (\textit{Peirce decomposition}).
A non-zero tripotent~$e\in\fp^+$ is called \textit{primitive} if $\fp^+(e)_2=\BC e$, and \textit{maximal} if $\fp^+(e)_0=\{0\}$.
If $\fp^+(e)_2=\fp^+$ holds for some (or equivalently any) maximal tripotent $e\in\fp^+$, then we say that $\fp^+$ is of \textit{tube type}.
Throughout the paper, we assume that the bilinear form $(\cdot|\cdot)_{\fp^\pm}\colon\fp^\pm\times\fp^\mp\to\BC$ is normalized such that
$(e|\overline{e})_{\fp^+}=(\overline{e}|e)_{\fp^-}=1$ holds for any primitive tripotent $e\in\fp^+$.
If $\fp^+$ is simple, then under this normalization $(x|y)_{\fp^+}$ and $D(x,y)$ are related as
\begin{equation}\label{formula_pairingD}
p(x|y)_{\fp^+}=\Tr_{\fp^+}D(x,y), \qquad x\in\fp^+,\quad y\in\fp^-
\end{equation}
with the same $p\in\BZ_{>0}$ as in~(\ref{formula_hB}).

Next we fix a~tripotent $e\in\fp^+$, and consider $\fp^+(e)_2\subset\fp^+$ as in~(\ref{Peirce}). Then
\[
Q(\overline{e})\colon \ \fp^+(e)_2\longrightarrow\fp^-(e)_2, \qquad Q(e)\colon\ \fp^-(e)_2\longrightarrow\fp^+(e)_2
\]
are mutually inverse. For $x\in\fp^+(e)_2$ let $P(x):=Q(x)Q(\overline{e})\in\End_\BC(\fp^+(e)_2)$. Then $\fp^+(e)_2$ becomes a~Jordan algebra with the product
\[
 x\cdot y:=\frac{1}{2}\{x,\overline{e},y\}, \qquad x,y\in\fp^+(e)_2,
 \]
with the unit element $e$, and with the inverse
\[
 x^\itinv:=P(x)^{-1}x, \qquad x\in\fp^+(e)_2.
\]
The Euclidean real form $\fn^+\subset\fp^+(e)_2$ is given by
\[
 \fn^+:=\big\{x\in\fp^+(e)_2 \mid Q(e)\overline{x}=x\big\}.
\]
Let $(\cdot|\cdot)_{\fn^+}\colon\fp^+(e)_2\times\fp^+(e)_2\to\BC$ be the symmetric bilinear form on $\fp^+(e)_2=\fn^{+\BC}$ given by
\[
 (x|y)_{\fn^+}:=(x|Q(\overline{e})y)_{\fp^+}=(y|Q(\overline{e})x)_{\fp^+}, \qquad x,y\in\fp^+(e)_2.
\]
This is positive definite on $\fn^+$. Also, let $\det_{\fn^+}(x)$ be the \textit{determinant polynomial} on ${\fp^+(e)_2\!=\!\fn^{+\BC}}$,
and let $\det_{\fn^-}(y):=\det_{\fn^+}(Q(e)y)$ for $y\in\fp^-(e)_2$.
Then the generic norm $h(x,y)$ and $\det_{\fn^+}(x)$ are related as
\begin{gather*}
t^r h\bigl(t^{-1}x,y\bigr)\bigr|_{t=0}=(-1)^r\det_{\fn^+}(x)\det_{\fn^-}(y), \\
h(x,\overline{e})=\det_{\fn^+}(e-x), \qquad x\in\fp^+(e)_2,\quad y\in\fp^-(e)_2,
\end{gather*}
where $r:=\rank\fp^+(e)_2=\deg\det_{\fn^+}(x)$. In addition, let
\begin{align*}
\Omega:\hspace{-3pt}&=(\text{connected component of }\{x\in\fn^+\mid P(x)\text{ is positive definite}\}\text{ which contains }e) \\
&=(\text{connected component of }\{x\in\fn^+\mid \det_{\fn^+}(x)>0\}\text{ which contains }e)
\end{align*}
be the \textit{symmetric cone}.

\subsection{Structure groups and the Kantor--Koecher--Tits construction}\label{subsection_KKT}

In this subsection we consider some Lie algebras corresponding to Jordan triple systems $\fp^\pm$.
For $l\in\End_\BC(\fp^+)$, let $\overline{l},{}^t\hspace{-1pt}l\in\End_\BC(\fp^-)$, $l^*\in\End_\BC(\fp^+)$ be the elements given by
$\overline{l}\overline{x}=\overline{lx}$, $(lx|\overline{y})_{\fp^+}=(x|{}^t\hspace{-1pt}l\overline{y})_{\fp^+}=(x|\overline{l^*y})_{\fp^+}$ for $x,y\in\fp^+$.
Then the \textit{Structure group} $\operatorname{Str}(\fp^+)$ and the \textit{automorphism group} $\operatorname{Aut}(\fp^+)$ are defined as
\begin{gather*}
\operatorname{Str}(\fp^+):=\big\{l\in {\rm GL}_\BC(\fp^+) \, \big| \, \big\{lx,{}^t\hspace{-1pt}l^{-1}y,lz\big\}=l\{x,y,z\}, \, x,z\in\fp^+,\, y\in\fp^-\big\}, \\
\operatorname{Aut}(\fp^+):=\big\{k\in \operatorname{Str}(\fp^+) \, \big| \, k^{-1}=k^*\big\}.
\end{gather*}
Let $\mathfrak{str}(\fp^+)=\fk^\BC$ and $\mathfrak{der}(\fp^+)=\fk$ denote the Lie algebras of $\operatorname{Str}(\fp^+)$ and $\operatorname{Aut}(\fp^+)$ respectively.
Then $D\bigl(\fp^+,\fp^-\bigr)=\mathfrak{str}(\fp^+)$ holds.

Next we recall the \textit{Kantor--Koecher--Tits construction}. As vector spaces let
\begin{gather*}
\fg^\BC:=\fp^+\oplus\fk^\BC\oplus\fp^-, \\
\fg:=\big\{(x,k,\overline{x})\mid x\in\fp^+,\, k\in\fk\big\}\subset\fg^\BC,
\end{gather*}
and give the Lie algebra structure on $\fg^\BC$ by
\[
[(x,k,y),(z,l,w)]:=\bigl(kz-lx,[k,l]+D(x,w)-D(z,y),-{}^t\hspace{-1pt}kw+{}^t\hspace{-1pt}ly\bigr).
\]
Then $\fg^\BC$ becomes a~Lie algebra and $\fg$ becomes a~real form of $\fg^\BC$.
Let ${(\cdot|\cdot)_{\fg^\BC}\colon \fg^\BC\times\fg^\BC\to\BC}$ be the $\fg^\BC$-invariant bilinear form normalized such that $(x|y)_{\fg^\BC}=(x|y)_{\fp^+}$ holds for any ${x\in\fp^+}$, ${y\in\fp^-}$.
We fix a~connected complex Lie group $G^\BC$ with the Lie algebra $\fg^\BC$,
and let ${G,K^\BC,K,P^+,P^-\!\subset\! G^\BC}$ be the connected closed subgroups corresponding to the Lie algebras ${\fg,\fk^\BC,\fk,\fp^+,\fp^-\subset\fg^\BC}$, respectively.
Then $\operatorname{Ad}|_{\fp^+}\colon K^\BC\to\operatorname{Str}(\fp^+)_0$ gives a~covering map, and for $l\in K^\BC$, $x\in\fp^+$, we abbreviate $\operatorname{Ad}(l)x=:lx$.

Next we fix a~tripotent $e\in\fp^+$, regard $\fp^+(e)_2\subset\fp^+$ as a~Jordan algebra, and let $\fn^+\subset\fp^+_2(e)$ be the Euclidean real form. Also, let
\begin{gather*}
\begin{split}
&\fn^- :=\overline{\fn^+}=Q(\overline{e})\fn^+ \subset\fp^-(e)_2, \\
&\fk^\BC(e)_2 :=D\bigl(\fp^+(e)_2,\fp^-(e)_2\bigr)=\big[\fp^+(e)_2,\fp^-(e)_2\big] \subset\fk^\BC, \\
&\fk(e)_2 :=\fk^\BC(e)_2\cap\fk \subset\fk^\BC(e)_2, \\
&\fl :=D\bigl(\fn^+,\fn^-\bigr)=\big[\fn^+,\fn^-\big] \subset\fk^\BC(e)_2, \\
&\fg^\BC(e)_2 :=\fp^+(e)_2\oplus\fk^\BC(e)_2\oplus\fp^-(e)_2 \subset\fg^\BC, \\
&\fg(e)_2 :=\fg^\BC(e)_2\cap\fg \subset\fg^\BC(e)_2, \\
&{}^c\hspace{-1pt}\fg :=\fn^+\oplus\fl\oplus\fn^- \subset\fg^\BC(e)_2.
\end{split}
\end{gather*}
These become Lie subalgebras of the right-hand sides. Let $K^\BC(e)_2,K(e)_2,G(e)_2,{}^c\hspace{-1pt}G\subset G^\BC$ be the connected closed subgroups corresponding to the Lie algebras
$\fk^\BC(e)_2,\fk(e)_2,\fg(e)_2,{}^c\hspace{-1pt}\fg\subset\fg^\BC$ respectively,
and let $L:=K^\BC\cap {}^c\hspace{-1pt}G$, $K_L:=L\cap K$, $\fk_\fl:=\fl\cap\fk$.
Then $G(e)_2$ and ${}^c\hspace{-1pt}G$ are isomorphic via the Cayley transform in $G^\BC(e)_2$,
both $K(e)_2$ and $L$ are real forms of $K^\BC(e)_2$,
$L$ acts transitively on the symmetric cone $\Omega\subset\fn^+$, and $K_L$ acts on $\fn^+$ as Jordan algebra automorphisms.
Also, for $l\in\fk^\BC(e)_2=\fl^\BC\subset\End_\BC(\fp^+(e)_2)=\End_\BC\bigl(\fn^{+\BC}\bigr)$,
let $l^\top=Q(e){}^t\hspace{-1pt}lQ(\overline{e})\in \mathfrak{k}^{\mathbb{C}}(e)_2 =\fl^\BC$, and extend to the anti-automorphism on $K^\BC(e)_2=L^\BC$, so that $(lx|y)_{\fn^+}=\bigl(x|l^\top y\bigr)_{\fn^+}$ holds.

\subsection{Simultaneous Peirce decomposition}\label{subsection_simul_Peirce}

In this subsection we assume $\fp^+$ is simple, or equivalently, the corresponding Lie algebra $\fg$ is simple.
We fix a~Jordan frame $\{e_1,\dots,e_r\}\subset\fp^+$, i.e., a~maximal set of primitive tripotents in~$\fp^+$ satisfying $D(e_i,\overline{e_j})=0$ for $i\ne j$,
where $r=\rank\fp^+=\rank_\BR\fg$. Then $e:=e_1+\cdots+e_r\in\fp^+$ becomes a~maximal tripotent,
and we take $\fn^+\subset\fp^+(e)_2\subset\fp^+$, $\fl\subset\fk^\BC(e)_2\subset\fk^\BC$ as in the previous subsection.
Next let $h_j:=D(e_j,\overline{e_j})=[e_j,\overline{e_j}]\in\fl\subset\fk^\BC$, and let
\begin{gather*}
\fa_\fl :=\bigoplus_{j=1}^r\BR h_j, \qquad \fa_\fl^\BC:=\bigoplus_{j=1}^r\BC h_j, \\
\fp^\pm_{ij} :=\big\{x\in\fp^\pm\mid [h_l,x]=\pm(\delta_{il}+\delta_{jl})x, \, l=1,\dots,r\big\}, \qquad 1\le i\le j\le r, \\
\fp^\pm_{0j} :=\big\{x\in\fp^\pm\mid [h_l,x]=\pm\delta_{jl}x, \, l=1,\dots,r\big\}, \qquad 1\le j\le r, \\
\fk^\BC_{ij} :=\big\{x\in\fk^\BC\mid [h_l,x]=(\delta_{il}-\delta_{jl})x, \, l=1,\dots,r \big\}, \qquad 1\le i,j\le r,\quad i\ne j, \\
\fk^\BC_{i0} :=\big\{x\in\fk^\BC\mid [h_l,x]=\delta_{il}x, \, l=1,\dots,r\big\}, \qquad 1\le i\le r, \\
\fk^\BC_{0j} :=\big\{x\in\fk^\BC\mid [h_l,x]=-\delta_{jl}x, \, l=1,\dots,r\big\}, \qquad 1\le j\le r, \\
\fm^\BC :=\big\{x\in\fk^\BC\mid [h_l,x]=0, \, (h_l|x)_{\fg^\BC}=0, \, l=1,\dots,r\big\}, \\
\fl_{ij} :=\fk^\BC_{ij}\cap \fl, \qquad 1\le i,j\le r,\quad i\ne j, \\
\fm_\fl :=\fm^\BC\cap\fl,
\end{gather*}
so that
\begin{gather*}
\fp^\pm=\bigoplus_{\substack{0\le i\le j\le r \\ (i,j)\ne(0,0)}}\fp^\pm_{ij}, \qquad
\fk^\BC=\fa_\fl^\BC\oplus\fm^\BC\oplus\bigoplus_{\substack{0\le i,j\le r \\ i\ne j}}\fk^\BC_{ij}, \qquad
\fl=\fa_\fl\oplus\fm_\fl\oplus\bigoplus_{\substack{1\le i,j\le r \\ i\ne j}}\fl_{ij}
\end{gather*}
hold. The decomposition of $\fp^\pm$ is called the \textit{simultaneous Peirce decomposition}. Using this decomposition, we define integers $(d,b,p,n)$ by
\begin{alignat}{3}
&d:=\dim\fp^+_{ij}, \quad 1\le i<j\le r, \qquad && b:=\dim\fp^+_{0j}, \quad 1\le j\le r, &\nonumber \\
&p:=2+d(r-1)+b, \qquad && n:=\dim\fp^+=r+\frac{d}{2}r(r-1)+br.& \label{str_const}
\end{alignat}
We note that if $r=1$, then $d$ is not determined uniquely, and any number is allowed.
Then $p$ coincides with the one in~(\ref{formula_hB}) and~(\ref{formula_pairingD}).
Also we set
\begin{gather*}
\fn_\fl:=\bigoplus_{1\le i<j\le r}\fl_{ij}, \qquad \fn_\fl^\top:=\bigoplus_{1\le i<j\le r}\fl_{ji}, \qquad
M_L:=\{k\in K_L\mid \operatorname{Ad}(k)h_l=0, \, l=1,\dots,r\},
\end{gather*}
and let $A_L,N_L,N_L^\top\subset L$ be the connected closed subgroups corresponding to the Lie subalgebras~$\fa_\fl$,~$\fn_\fl$,~$\fn_\fl^\top$ respectively,
so that $M_LA_LN_L,M_LA_LN_L^\top\subset L$ are minimal parabolic subgroups.

\subsection{Space of polynomials on Jordan triple systems}\label{subsection_polynomials}

In this subsection we consider the space $\cP\bigl(\fp^\pm\bigr)$ of polynomials on the Jordan triple system $\fp^\pm$, on which $K^\BC$ acts by
\begin{alignat*}{3}
&(\operatorname{Ad}|_{\fp^+}(l))^\vee f(x)=f\bigl(l^{-1}x\bigr), \qquad && l\in K^\BC,\quad f\in\cP(\fp^+),\quad x\in\fp^+,& \\
&(\operatorname{Ad}|_{\fp^-}(l))^\vee f(y)=f\bigl({}^tly\bigr), \qquad && l\in K^\BC,\quad f\in\cP(\fp^-),\quad y\in\fp^-.&
\end{alignat*}
We assume $\fp^+$ is simple, fix a~Jordan frame $\{e_1,\dots,e_r\}\subset\fp^+$, and consider the tripotents
$e^k:=\sum_{j=1}^k e_j$, $\check{e}^k:=\sum_{j=r-k+1}^r e_j$ for $k=1,2,\dots,r$. We also write $e^r=\check{e}^r=:e$. Then the subalgebras
\[
\fp^+\bigl(e^k\bigr)_2=\bigoplus_{1\le i\le j\le k}\fp^+_{ij}, \qquad \fp^+\bigl(\check{e}^k\bigr)_2=\bigoplus_{r-k+1\le i\le j\le r}\fp^+_{ij}
\]
have Jordan algebra structures. Let $\fn^+\bigl(e^k\bigr)\subset\fp^+\bigl(e^k\bigr)_2$, $\fn^+\bigl(\check{e}^k\bigr)\subset\fp^+\bigl(\check{e}^k\bigr)_2$ be the Euclidean real forms,
and we extend the determinant polynomials $\det_{\fn^+(e^k)}$, $\det_{\fn^+(\check{e}^k)}$ on $\fp^+\bigl(e^k\bigr)_2$, $\fp^+\bigl(\check{e}^k\bigr)_2$ to polynomials on $\fp^+$.
Using these, for $\bm\in\BC^r$ we define the functions $\Delta^{\fn^+}_\bm(x)$, $\check{\Delta}^{\fn^+}_\bm(x)$ on the symmetric cone
$\Omega\subset\fn^+:=\fn^+(e^r)=\fn^+(\check{e}^r)$ by
\begin{gather}\label{principal_minor}
\Delta^{\fn^+}_\bm(x):=\prod_{k=1}^r \det_{\fn^+(e^k)}(x)^{m_k-m_{k+1}}, \!\qquad \check{\Delta}^{\fn^+}_\bm(x):=\prod_{k=1}^r \det_{\fn^+(\check{e}^k)}(x)^{m_k-m_{k+1}},\! \quad x\in\Omega,\!\!
\end{gather}
where we set $m_{r+1}:=0$. Then these satisfy
\begin{align*}
\Delta^{\fn^+}_\bm(man.x)={}&\Delta^{\fn^+}_\bm(man.e)\Delta^{\fn^+}_\bm(x)=\Delta^{\fn^+}_\bm(a.e)\Delta^{\fn^+}_\bm(x)={\rm e}^{2t_1m_1+2t_2m_2+\cdots+2t_rm_r}\Delta^{\fn^+}_\bm(x),\\
&
m\in M_L,\quad a={\rm e}^{t_1h_1+\cdots+t_rh_r}\in A_L,\quad n\in N^\top_L, \\
\check{\Delta}^{\fn^+}_\bm(man.x)={}&\check{\Delta}^{\fn^+}_\bm(man.e)\check{\Delta}^{\fn^+}_\bm(x)=\check{\Delta}^{\fn^+}_\bm(a.e)\check{\Delta}^{\fn^+}_\bm(x)
={\rm e}^{2t_1m_r+t_2m_{r-1}+\cdots+2t_rm_1}\check{\Delta}^{\fn^+}_\bm(x),\\
&
m\in M_L,\quad a={\rm e}^{t_1h_1+\cdots+t_rh_r}\in A_L,\quad n\in N_L.
\end{align*}
Especially if
\[ \bm\in\BZ_{++}^r:=\{\bm=(m_1,\dots,m_r)\in\BZ^r \mid m_1\ge m_2\ge \cdots\ge m_r\ge 0\}, \]
then $\Delta^{\fn^+}_\bm(x)$, $\check{\Delta}^{\fn^+}_\bm(x)$ are extended to polynomials on $\fp^+$.
For $\bm=(m_1,\dots,m_r)\in\BC^r$, we write $\bm^\vee:=(m_r,\dots,m_1)\in\BC^r$. Then the following holds.

\begin{Lemma}[{\cite[Proposition VII.1.6]{FK}}]\label{lem_diff}
For $\bk\in\BZ_{++}^r$, $\bm\in\BC^r$, $x\in\Omega$, we have
\begin{align*}
\check{\Delta}^{\fn^+}_\bk\left(\frac{\partial}{\partial x}\right)\Delta^{\fn^+}_\bm(x)
=\prod_{j=1}^r\left(m_j-k_{r-j+1}+\frac{d}{2}(r-j)+1\right)_{k_{r-j+1}}\Delta^{\fn^+}_{\bm-\bk^\vee}(x).
\end{align*}
\end{Lemma}

Here we normalize the differential operator $\frac{\partial}{\partial x}$ on $\fn^{+\BC}=\fp^+(e)_2$ with respect to the bilinear form $(\cdot|\cdot)_{\fn^+}=(\cdot|Q(\overline{e})\cdot)_{\fp^+}$.
Especially, if $k_1=\cdots=k_r=k\in\BZ_{\ge 0}$ and $\bm$ is of the form $\bm=(\mu+l_1,\mu+l_2,\dots,\mu+l_r)$ with $\mu,l_j\in\BC$, then we have
\begin{gather}
\det_{\fn^+}\bigg(\frac{\partial}{\partial x}\bigg)^k\det_{\fn^+}(x)^\mu\Delta^{\fn^+}_\bl(x)
=\prod_{j=1}^r\bigg(\mu+l_j-k+\frac{d}{2}(r-j)+1\bigg)_k\det_{\fn^+}(x)^{\mu-k}\Delta^{\fn^+}_\bl(x) \notag \\
\qquad\qquad=(-1)^{kr}\prod_{j=1}^r\biggl(-\mu-l_{r-j+1}-\frac{d}{2}(j-1)\biggr)_k\det_{\fn^+}(x)^{\mu-k}\Delta^{\fn^+}_\bl(x). \label{formula_diff}
\end{gather}

Next, for $\bm\in\BZ_{++}^r$, let
\begin{gather*}
\cP_\bm(\fp^+) :=\operatorname{span}_\BC \big\{\Delta^{\fn^+}_\bm\bigl(l^{-1}x\bigr)\, \big| \, l\in K^\BC\big\} \subset\cP(\fp^+), \\
\cP_\bm(\fp^-) :=\big\{\overline{f(\overline{y})}\, \big| \, f(x)\in\cP_\bm(\fp^+)\big\} \subset\cP(\fp^-).
\end{gather*}
Then the following holds.

\begin{Theorem}[{Hua--Kostant--Schmid,~\cite[Part III, Theorem V.2.1]{FKKLR}}]\label{thm_HKS}
Under the $K^\BC$-action, $\cP\bigl(\fp^\pm\bigr)$ is decomposed into the sum of irreducible submodules as
\[
\cP\bigl(\fp^\pm\bigr)=\bigoplus_{\bm\in\BZ_{++}^r}\cP_\bm\bigl(\fp^\pm\bigr).
\]
\end{Theorem}

In addition, if $\fp^+$ is of tube type, then for
\[
\bm\in\BZ_+^r:=\{\bm=(m_1,\dots,m_r)\in\BZ^r\mid m_1\ge m_2\ge\cdots\ge m_r\},
\]
let
\[
\cP_\bm\bigl(\fp^\pm\bigr):=\cP_{(m_1-m_r,m_2-m_r,\dots,m_{r-1}-m_r,0)}\bigl(\fp^\pm\bigr)\det_{\fn^\pm}(x)^{m_r}\subset\cP\bigl(\fp^\pm\bigr)\big[\det_{\fn^\pm}(x)^{-1}\big],
\]
so that $\cP_\bm\bigl(\fp^\pm\bigr)\simeq\cP_{-\bm^\vee}(\fp^\mp)$ holds as a~$K^\BC$-module.

\subsection{Holomorphic discrete series representations}\label{subsection_HDS}

In this subsection we review holomorphic discrete series representations.
First we recall the bounded symmetric domain realization (Harish-Chandra realization) of the Hermitian symmetric space $G/K$ via the Borel embedding,
\[ \xymatrix{ G/K \ar[r] \ar@{-->}[d]^{\mbox{\rotatebox{90}{$\sim$}}} & G^\BC/K^\BC P^- \\
 D \ar@{^{(}->}[r] & \fp^+ \ar[u]_{\exp}, } \]
where
\begin{align*}
D&=(\text{connected component of }\big\{x\in\fp^+\mid B(x)\text{ is positive definite}\big\}\text{ which contains }0) \\
&=(\text{connected component of }\big\{x\in\fp^+\mid h(x)>0\big\}\text{ which contains }0).
\end{align*}
Here we write $B(x):=B(x,\overline{x})$, $h(x):=h(x,\overline{x})$.
For $g\in G^\BC$, $x\in D$, if $g\exp(x)\in P^+K^\BC P^-$ holds, then we write
\[ g\exp(x)=\exp\bigl(\pi^+(g,x)\bigr)\kappa(g,x)\exp(\pi^-(g,x)), \]
where $\pi^\pm(g,x)\in \fp^\pm$ and $\kappa(g,x)\in K^\BC$. Then the map $\pi^+\colon G\times D\to D$ gives an action of $G$ on~$D$, and we abbreviate $\pi^+(g,x)=:gx$.

Next let $(\tau,V)$ be an irreducible holomorphic representation of the universal covering group $\widetilde{K}^\BC$ of~$K^\BC$,
with the $\widetilde{K}$-invariant inner product $(\cdot,\cdot)_\tau$.
Then the universal covering group~$\widetilde{G}$ of~$G$ acts on the space $\cO(D,V)=\cO_\tau(D,V)$ of $V$-valued holomorphic functions on $D$ by
\[
(\hat{\tau}(g)f)(x):=\tau\bigl(\kappa\bigl(g^{-1},x\bigr)\bigr)^{-1}f\bigl(g^{-1}\bigr), \qquad g\in\widetilde{G},\quad x\in D,\quad f\in\cO(D,V),
\]
where we lift the map $\kappa\colon G\times D\to K^\BC$ to the universal covering spaces, and represent by the same symbol $\kappa\colon \widetilde{G}\times D\to \widetilde{K}^\BC$.
Its differential action is given by
\begin{equation}\label{HDS_diff_action}
({\rm d}\hat{\tau}(z,k,w)f)(x)={\rm d}\tau(k-D(x,w))f(x)+\frac{{\rm d}}{{\rm d}t}\biggr|_{t=0}f(x-t(z+kx-Q(x)w))
\end{equation}
for $z\in\fp^+$, $k\in\fk^\BC$, $w\in\fp^-$. This becomes a~highest weight representation with the minimal $\widetilde{K}$-type $(\tau,V)$.
If this contains a~unitary subrepresentation $\cH_\tau(D,V)\subset\cO_\tau(D,V)$,
then its reproducing kernel is proportional to $\tau(B(x,\overline{y}))$, and such unitary subrepresentation is unique.
Here we lift the map $B\colon D\times \overline{D}\to \operatorname{Str}(\fp^+)_0\subset\End_\BC(\fp^+)$ to the universal covering space,
and represent by the same symbol $B\colon D\times \overline{D}\to\widetilde{K}^\BC$. Especially, if the $\widetilde{G}$-invariant inner product is given by the converging integral
\[ \langle f,g\rangle_{\hat{\tau}}:=C_\tau\int_D \bigl(\tau\bigl(B(x)^{-1}\bigr)f(x),g(x)\bigr)_\tau \Det_{\fp^+}(B(x))^{-1}\,{\rm d}x \]
(a \textit{weighted Bergman inner product}), then $(\hat{\tau},\cH_\tau(D,V))$ is called a~\textit{holomorphic discrete series representation}.
Here we normalize the Lebesgue measure ${\rm d}x$ on $D\subset\fp^+$ with respect to the inner product $(\cdot|\overline{\cdot})_{\fp^+}$,
and determine the constant $C_\tau$ such that $\Vert v\Vert_{\hat{\tau}}=|v|_\tau$ holds for all constant functions $v\in V$.

Next, let $\chi\colon \widetilde{K}^\BC\to\BC^\times$ be the character of $\widetilde{K}^\BC$ normalized such that
\begin{equation}\label{chi_normalize}
{\rm d}\chi([x,y])=(x|y)_{\fp^+}, \qquad x\in\fp^+,\quad y\in\fp^-,
\end{equation}
so that $h(x,y)=\chi(B(x,y))$ holds, and we fix a~representation $(\tau_0,V)$ of $K^\BC$.
When $(\tau,V)$ is of the form $(\tau,V)=\bigl(\chi^{-\lambda}\otimes\tau_0,V\bigr)$, we write $\cH_\tau(D,V)=\cH_\lambda(D,V)\subset\cO_\tau(D,V)=\cO_\lambda(D,V)$.
In addition, if $(\tau,V)=\bigl(\chi^{-\lambda},\BC\bigr)$, then we write $\cH_\tau(D,V)=\cH_\lambda(D)\subset\cO_\tau(D,V)=\cO_\lambda(D)$ and~$\hat{\tau}=\tau_\lambda$.

In the rest of this subsection we assume $\fp^+$ is simple and $(\tau,V)=\bigl(\chi^{-\lambda},\BC\bigr)$.
Then $\cH_\lambda(D)$ is holomorphic discrete if $\lambda>p-1$, and then the inner product is given by
\begin{gather}
\langle f,g\rangle_\lambda=\langle f,g\rangle_{\lambda,\fp^+}:=C_\lambda\int_D f(x)\overline{g(x)}h(x)^{\lambda-p}\,{\rm d}x, \nonumber\\
C_\lambda:=\frac{\prod_{j=1}^r\Gamma\bigl(\lambda-\frac{d}{2}(j-1)\bigr)}{\pi^n\prod_{j=1}^r\Gamma\bigl(\lambda-\frac{n}{r}-\frac{d}{2}(j-1)\bigr)}. \label{Bergman_inner_prod}
\end{gather}
We consider another inner product on $\cP(\fp^+)$, called the \textit{Fischer inner product} (see, e.g.,~\cite[Section XI.1]{FK}), defined by
\begin{equation}\label{Fischer_inner_prod}
\langle f,g\rangle_{F}=\langle f,g\rangle_{F,\fp^+}:=\frac{1}{\pi^n}\int_{\fp^+}f(x)\overline{g(x)}{\rm e}^{-(x|\overline{x})_{\fp^+}}\,{\rm d}x
=\overline{g\left(\overline{\frac{\partial}{\partial x}}\right)}f(x)\biggr|_{x=0}.
\end{equation}
Here $\overline{g(\overline{\cdot})}$ is a~holomorphic polynomial on $\fp^-$, and we normalize $\frac{\partial}{\partial x}$ with respect to the bilinear form
$(\cdot|\cdot)_{\fp^+}\colon\fp^+\times\fp^-\to\BC$. Then the following holds.

\begin{Theorem}[{Faraut--Kor\'anyi,~\cite{FK0} and~\cite[Part III, Corollary V.3.9]{FKKLR}}]\label{thm_FK}
For $\lambda>p-1$, $\bm\in\BZ_{++}^r$, $f\in\cP_\bm(\fp^+)$, $g\in\cP(\fp^+)$, we have
\[ \langle f,g\rangle_\lambda=\frac{1}{(\lambda)_{\bm,d}}\langle f,g\rangle_F, \]
where
\begin{equation}\label{Pochhammer}
(\lambda)_{\bm,d}:=\prod_{j=1}^r\left(\lambda-\frac{d}{2}(j-1)\right)_{m_j},
\end{equation}
and $(\lambda)_m:=\lambda(\lambda+1)(\lambda+2)\cdots(\lambda+m-1)$.
\end{Theorem}

Since the reproducing kernel on $\cP(\fp^+)$ with respect to $\langle\cdot,\cdot\rangle_F$ is given by ${\rm e}^{(x|\overline{y})_{\fp^+}}$, the following holds.

\begin{Corollary}\label{cor_FK}
For $\lambda>p-1$, $\bm\in\BZ_{++}^r$, $f(x)\in\cP_\bm(\fp^+)$, we have
\[ \bigl\langle f(x),{\rm e}^{(x|\overline{y})_{\fp^+}}\bigr\rangle_{\lambda,x}=\frac{1}{(\lambda)_{\bm,d}}f(y). \]
\end{Corollary}

Here the subscript $x$ stands for the variable of integration.
By this theorem, $\langle\cdot,\cdot\rangle_{\lambda}$ is meromorphically continued for all $\lambda\in\BC$,
and the $\widetilde{K}$-finite part $\cO_\lambda(D)_{\widetilde{K}}=\chi^{-\lambda}\otimes\cP(\fp^+)$ is reducible as a~$\bigl(\fg,\widetilde{K}\bigr)$-module
if and only if $\lambda$ is a~pole of $\langle\cdot,\cdot\rangle_{\lambda}$.
Especially, for $j=1,2,\dots,r$, $\lambda\in\frac{d}{2}(j-1)-\BZ_{\ge 0}$,
\begin{equation}\label{submodule}
M_j(\lambda)=M_j^\fg(\lambda):=\bigoplus_{\substack{\bm\in\BZ_{++}^r \\ m_j\le \frac{d}{2}(j-1)-\lambda}}\chi^{-\lambda}\otimes\cP_\bm(\fp^+)
\subset\chi^{-\lambda}\otimes\cP(\fp^+)=\cO_\lambda(D)_{\widetilde{K}}
\end{equation}
gives a~$\bigl(\fg,\widetilde{K}\bigr)$-submodule. Moreover, $\langle\cdot,\cdot\rangle_\lambda$ is positive definite on $\cP(\fp^+)$ for $\lambda>\frac{d}{2}(r-1)$,
and on $M_j^\fg\bigl(\frac{d}{2}(j-1)\bigr)$ for $\lambda=\frac{d}{2}(j-1)$, $j=1,\dots,r$, that is, $\cO_\lambda(D)_{\widetilde{K}}$ contains a~unitary submodule~$\cH_\lambda(D)$
if $\lambda$ sits in the \textit{Wallach set}
\begin{equation}\label{Wallach_set}
\lambda\in\left\{0,\frac{d}{2},d,\dots,\frac{d}{2}(r-1)\right\}\cup\left(\frac{d}{2}(r-1),\infty\right),
\end{equation}
and its $\widetilde{K}$-finite part is given by
\[ \cH_\lambda(D)_{\widetilde{K}}=\begin{cases} \chi^{-\lambda}\otimes\cP(\fp^+), & \lambda>\frac{d}{2}(r-1), \\
M_j^\fg\bigl(\frac{d}{2}(j-1)\bigr), & \lambda=\frac{d}{2}(j-1),\ j=1,2,\dots,r. \end{cases} \]
In addition, the quotient module $\cO_\lambda(D)_{\widetilde{K}}/M_r(\lambda)$, $\lambda\in\frac{d}{2}(r-1)-\BZ_{\ge 0}$, also gives an infinitesimally unitary module (see~\cite{FK0}).
Also by the corollary, for $\bk\in\BZ_{++}^r$, $f(x)\in\cP(\fp^+)$, $(\lambda)_{\bk,d}\bigl\langle f(x),{\rm e}^{(x|\overline{y})_{\fp^+}}\bigr\rangle_{\lambda,x}$ is
holomorphically continued for all $\lambda\in\BC$ if and only if
\[ f(x)\in\bigoplus_{\substack{\bm\in\BZ_{++}^r \\ m_j\le k_j, \, j=1,\dots,r }}\cP_\bm(\fp^+) \]
holds, and if this is satisfied, then for $j=1,\dots,r$,
\[ f(x)\in M_j(\lambda)\qquad \text{holds if}\quad \lambda\in\frac{d}{2}(j-1)-k_j-\BZ_{\ge 0}. \]

\subsection{Classification}\label{subsection_classification}

A simple Hermitian positive Jordan triple system $\fp^\pm$ is isomorphic to one of the following:
\begin{alignat*}{4}
&\fp^\pm=\BC^n, \quad n\ne 2, \qquad &&\Sym(r,\BC), \qquad&&{\rm M}(q,s;\BC),& \\
&\Alt(s,\BC), \qquad&& \Herm(3,\BO)^\BC, \qquad&& {\rm M}(1,2;\BO)^\BC.&
\end{alignat*}
Here $\Sym(r,\BC)$ and $\Alt(s,\BC)$ denote the spaces of symmetric and alternating matrices over~$\BC$, respectively,
and $\Herm(3,\BO)^\BC$ denotes the complexification of the space of $3\times 3$ Hermitian matrices over the octonions $\BO$.
Then the corresponding Lie groups $G$ and their maximal compact subgroups $K$ are given by
\[ (G,K)=\begin{cases} ({\rm SO}_0(2,n),{\rm SO}(2)\times {\rm SO}(n)), & \fp^\pm=\BC^n, \\ ({\rm Sp}(r,\BR),{\rm U}(r)), & \fp^\pm=\Sym(r,\BC), \\
({\rm SU}(q,s),{\rm S}({\rm U}(q)\times {\rm U}(s))), & \fp^\pm={\rm M}(q,s;\BC), \\ ({\rm SO}^*(2s),{\rm U}(s)), & \fp^\pm=\Alt(s,\BC), \\
(E_{7(-25)},{\rm U}(1)\times E_6), & \fp^\pm=\Herm(3,\BO)^\BC, \\ (E_{6(-14)},{\rm U}(1)\times {\rm Spin}(10)), & \fp^\pm={\rm M}(1,2;\BO)^\BC \end{cases} \]
(up to covering), and the numbers $(n,r,d,b,p)$ (see~(\ref{str_const})) are given by
\[ (n,r,d,b,p)=\begin{cases} (n,2,n-2,0,n), & \fp^\pm=\BC^n,\ n\ge 3, \\
(1,1,-,0,2), & \fp^\pm=\BC, \\
\left(\frac{1}{2}r(r+1), r, 1, 0, r+1\right), & \fp^\pm=\Sym(r,\BC), \\
(qs, \min\{q,s\}, 2, |q-s|, q+s), & \fp^\pm={\rm M}(q,s;\BC), \\
\left(\frac{1}{2}s(s-1), \frac{s}{2}, 4, 0, 2(s-1)\right), & \fp^\pm=\Alt(s,\BC),\ s\colon\text{even}, \\
\left(\frac{1}{2}s(s-1), \left\lfloor\frac{s}{2}\right\rfloor, 4, 2, 2(s-1)\right), & \fp^\pm=\Alt(s,\BC),\ s\colon\text{odd}, \\
(27,3,8,0,18), & \fp^\pm=\Herm(3,\BO)^\BC, \\ (16,2,6,4,12), & \fp^\pm={\rm M}(1,2;\BO)^\BC. \end{cases} \]
Here, if $r=1$, then $d$ is not determined uniquely, and any number is allowed.
When $b=0$, $\fp^\pm$ is of tube type, and has a~Jordan algebra structure. That is, $\fp^\pm=\BC^n$, $\Sym(r,\BC)$, ${\rm M}(r,\BC)$, $\Alt(2r,\BC)$ and $\Herm(3,\BO)^\BC$ are of tube type.
For these cases $p=\frac{2n}{r}=d(r-1)+2$ holds. When $\fp^\pm={\rm M}(q,s;\BC)$ we also consider $(G,K)=({\rm U}(q,s),{\rm U}(q)\times {\rm U}(s))$ instead of $({\rm SU}(q,s),{\rm S}({\rm U}(q)\times {\rm U}(s)))$.

Next we fix the parametrization of finite-dimensional irreducible representations of $\widetilde{{\rm GL}}(s,\BC)$ and ${\rm Spin}(n,\BC)$, $n\ge 3$.
We take a~basis $\{\epsilon_j\}_{j=1}^s\subset\fh^{\BC\vee}$ of the dual space of a~Cartan subalgebra $\fh^\BC\subset\mathfrak{gl}(s,\BC)$
such that the positive root system is given by $\{\epsilon_i-\epsilon_j\mid 1\le i<j\le s\}$.
For $\bm\in\BC^s$ with $m_j-m_{j+1}\in\BZ_{\ge 0}$, let $V_\bm^{(s)}$ be the irreducible representation of $\widetilde{{\rm GL}}(s,\BC)$ with the highest weight $\sum_j m_j\epsilon_j$,
and let $V_\bm^{(s)\vee}$ be that with the lowest weight $-\sum_j m_j\epsilon_j$.
If $\bm\in\BZ_+^r$, then~$V_\bm^{(s)}$,~$V_\bm^{(s)\vee}$ are reduced to the representations of ${\rm GL}(s,\BC)$.
Similarly, we take a~basis $\{\epsilon_j\}_{j=1}^{\lfloor n/2\rfloor}\subset\fh^{\BC\vee}$ of the dual space of a~Cartan subalgebra $\fh^\BC\subset\mathfrak{so}(n,\BC)$
such that the positive root system is given by
\begin{alignat*}{3}
&\{\epsilon_i\pm\epsilon_j\mid 1\le i<j\le n/2\}, && n\colon\text{even},& \\
&\{\epsilon_i\pm\epsilon_j\mid 1\le i<j\le \lfloor n/2\rfloor\}\cup\{\epsilon_j\mid 1\le j\le \lfloor n/2\rfloor\}, \qquad&& n\colon\text{odd}.&
\end{alignat*}
For $\bm\in\BZ^{\lfloor n/2\rfloor}\cup\left(\BZ+\frac{1}{2}\right)^{\lfloor n/2\rfloor}$ with $m_1\ge\cdots\ge m_{n/2-1}\ge |m_{n/2}|$ for even $n$,
$m_1\ge\cdots\ge m_{\lfloor n/2\rfloor}\ge 0$ for odd $n$, let $V_\bm^{[n]}$ be the irreducible representation of ${\rm Spin}(n,\BC)$ with the highest weight $\sum_j m_j\epsilon_j$,
and let $V_\bm^{[n]\vee}$ be that with the lowest weight $-\sum_j m_j\epsilon_j$.
If $\bm\in\BZ^{\lfloor n/2\rfloor}$, then~$V_\bm^{[n]}$,~$V_\bm^{[n]\vee}$ are reduced to the representations of ${\rm SO}(n,\BC)$.
Under this notation, the $\widetilde{K}$-type decompositions of the holomorphic discrete series representations of scalar type
\[
\cO_\lambda(D)_{\widetilde{K}}=\chi^{-\lambda}\otimes\cP(\fp^+)=\chi^{-\lambda}\otimes\bigoplus_{\bm\in\BZ_{++}^r}\cP_\bm(\fp^+)
\]
are given as
\begin{gather}
\chi^{-\lambda}\simeq\begin{cases}
\BC_{-\lambda}\boxtimes V_{(0,\dots,0)}^{[n]\vee}, & \fp^+=\BC^n,\ n\ge 3, \vspace{1mm}\\
V_{(\lambda,\dots,\lambda)}^{(r)\vee}, & \fp^+=\Sym(r,\BC), \vspace{1mm}\\
V_{\left(\frac{s\lambda}{q+s},\dots,\frac{s\lambda}{q+s}\right)}^{(q)\vee}\boxtimes V_{\left(\frac{q\lambda}{q+s},\dots,\frac{q\lambda}{q+s}\right)}^{(s)}, & \fp^+={\rm M}(q,s;\BC), \vspace{1mm}\\
V_{\left(\frac{\lambda}{2},\dots,\frac{\lambda}{2}\right)}^{(s)\vee}, & \fp^+=\Alt(s,\BC), \vspace{1mm}\\
\BC_{-\lambda}\boxtimes V_{(0,\dots,0)}^{[10]\vee}, & \fp^+={\rm M}(1,2;\BO)^\BO, \end{cases}\label{explicit_Ktype}
\\
\cP_\bm(\fp^+)\simeq\begin{cases}
\BC_{-m_1-m_2}\boxtimes V_{(m_1-m_2,0,\dots,0)}^{[n]\vee}, & \fp^+=\BC^n,\ n\ge 3, \vspace{1mm}\\
V_{(2m_1,\dots,2m_r)}^{(r)\vee}=:V_{2\bm}^{(r)\vee}, & \fp^+=\Sym(r,\BC), \vspace{1mm}\\
V_{\bm}^{(q)\vee}\boxtimes V_{(m_1,\dots,m_q,0,\dots,0)}^{(s)}=:V_\bm^{(q)\vee}\boxtimes V_\bm^{(s)}, & \fp^+={\rm M}(q,s;\BC),\ q\le s, \vspace{1mm}\\
V_{(m_1,\dots,m_s,0,\dots,0)}^{(q)\vee}\boxtimes V_{\bm}^{(s)}=:V_\bm^{(q)\vee}\boxtimes V_\bm^{(s)}, & \fp^+={\rm M}(q,s;\BC),\ q\ge s, \vspace{1mm}\\
V_{(m_1,m_1,m_2,m_2,\dots,m_{\lfloor s/2\rfloor},m_{\lfloor s/2\rfloor}(,0))}^{(s)\vee}=:V_{\bm^2}^{(s)\vee}, & \fp^+=\Alt(s,\BC), \vspace{1mm}\\
\BC_{-\frac{3}{4}|\bm|}\boxtimes V_{\left(\frac{m_1+m_2}{2},\frac{m_1-m_2}{2},\frac{m_1-m_2}{2},\frac{m_1-m_2}{2},\frac{m_1-m_2}{2}\right)}^{[10]\vee}, & \fp^+={\rm M}(1,2;\BO)^\BC,
\end{cases}\nonumber
\end{gather}
when $K^\BC$ is classical, if we normalize the representations $\BC_{-\lambda}$ of $\widetilde{{\rm SO}}(2)\simeq \widetilde{{\rm U}}(1)$ for the first and the last cases suitably.
When $\fp^+=\BC^n$ with $n=1,2$, we have isomorphisms
\begin{gather*}
\widetilde{{\rm SO}}_0(2,1)\simeq \widetilde{{\rm SL}}(2,\BR)=\widetilde{{\rm Sp}}(1,\BR), \\
\widetilde{{\rm SO}}_0(2,2)\simeq \widetilde{{\rm SL}}(2,\BR)\times\widetilde{{\rm SL}}(2,\BR)=\widetilde{{\rm Sp}}(1,\BR)\times\widetilde{{\rm Sp}}(1,\BR),
\end{gather*}
and we write
\[
\cH_{\lambda}(D_{{\rm SO}_0(2,1)}):=\cH_{2\lambda}(D_{{\rm SL}(2,\BR)}), \qquad \cH_{\lambda}(D_{{\rm SO}_0(2,2)}):=\cH_\lambda(D_{{\rm SL}(2,\BR)})\hboxtimes\cH_\lambda(D_{{\rm SL}(2,\BR)}),
\]
and similar for $\cO_\lambda(D)$. When we consider $\fp^+={\rm M}(q,s;\BC)$, $G={\rm U}(q,s)$, for $\lambda_1,\lambda_2\in\BC$,
let $\chi^{-\lambda_1,-\lambda_2}$ be the character of $\widetilde{K}^\BC=\widetilde{{\rm GL}}(q,\BC)\times\widetilde{{\rm GL}}(s,\BC)$ given by
\[ \chi^{-\lambda_1,-\lambda_2}\simeq V_{(\lambda_1,\dots,\lambda_1)}^{(q)\vee}\boxtimes V_{(\lambda_2,\dots,\lambda_2)}^{(s)}, \]
and for a~fixed representation $V'\boxtimes V''$ of $K^\BC$, let $\cH_{\lambda_1+\lambda_2}(D,V'\boxtimes V'')\subset\cO_{\lambda_1+\lambda_2}(D,V'\boxtimes V'')$ be the
representations of $\widetilde{G}$ with the minimal $\widetilde{K}$-type $\chi^{-\lambda_1,-\lambda_2}\otimes(V'\boxtimes V'')$.

The inner product $\langle\cdot,\cdot\rangle_\lambda$ of the holomorphic discrete series representation $\cH_\lambda(D)$ of scalar type originally converges for $\lambda>p-1$,
and by Theorem~\ref{thm_FK}, this is meromorphically continued for all $\lambda\in\BC$. This has poles at $\lambda\in\frac{d}{2}(j-1)-\BZ_{\ge 0}$ for $j=1,2,\dots,r$,
and then $M_j(\lambda)=M_j^\fg(\lambda)\subset\cO_\lambda(D)_{\widetilde{K}}$ defined in~(\ref{submodule}) becomes a~$\bigl(\fg,\widetilde{K}\bigr)$-submodule.
Especially, when $\fp^+=\BC^n$ with $n\ge 3$, $\cO_\lambda(D)_{\widetilde{K}}$ is reducible if and only if $\lambda\in -\BZ_{\ge 0}\cap\left(\frac{n-2}{2}-\BZ_{\ge 0}\right)$,
and then we have the $\bigl(\mathfrak{so}(2,n),\widetilde{{\rm SO}}(2)\times {\rm SO}(n)\bigr)$-submodules
\begin{alignat*}{3}
&\cO_\lambda(D)_{\widetilde{K}}\supset M_2(\lambda)\supset M_1(\lambda)\supset\{0\}, \qquad&& n\colon\text{even},\ \lambda\in\BZ,\ \lambda\le 0, &\\
&\cO_\lambda(D)_{\widetilde{K}}\supset M_2(\lambda)\supset \{0\}, && \textstyle n\colon\text{even},\ \lambda\in\BZ,\ 1\le\lambda\le\frac{n-2}{2}, &\\	
&\cO_\lambda(D)_{\widetilde{K}}\supset M_1(\lambda)\supset \{0\}, && n\colon\text{odd},\ \lambda\in\BZ,\ \lambda\le 0, &\\	
&\cO_\lambda(D)_{\widetilde{K}}\supset M_2(\lambda)\supset \{0\}, && \textstyle n\colon\text{odd},\ \lambda\in\BZ+\frac{1}{2},\ \lambda\le\frac{n-2}{2}.&
\end{alignat*}
When $\fp^+=\Sym(r,\BC)$ with $r\ge 2$, $\cO_\lambda(D)_{\widetilde{K}}$ is reducible if and only if $\lambda\in \frac{1}{2}(r-1)-\frac{1}{2}\BZ_{\ge 0}$,
and then we have the $(\mathfrak{sp}(r,\BR),\widetilde{{\rm U}}(r))$-submodules
\begin{alignat*}{3}
&\cO_\lambda(D)_{\widetilde{K}}\supset M_{2\left\lceil\frac{r}{2}\right\rceil-1}(\lambda)\supset M_{2\left\lceil\frac{r}{2}\right\rceil-3}(\lambda)\supset
\cdots\supset M_{\max\{2\lambda,0\}+1}(\lambda)\supset\{0\}, \qquad && \lambda\in\BZ,& \\
&\cO_\lambda(D)_{\widetilde{K}}\supset M_{2\left\lfloor\frac{r}{2}\right\rfloor}(\lambda)\supset M_{2\left\lfloor\frac{r}{2}\right\rfloor-2}(\lambda)\supset
\cdots\supset M_{\max\{2\lambda,1\}+1}(\lambda)\supset\{0\},\qquad && \textstyle \lambda\in\BZ+\frac{1}{2}.&
\end{alignat*}
When $\fp^+={\rm M}(q,s;\BC)$, $\cO_\lambda(D)_{\widetilde{K}}$ is reducible if and only if $\lambda\in \min\{q,s\}-1-\BZ_{\ge 0}$,
and then we have the $(\mathfrak{su}(q,s),{\rm S}({\rm U}(q)\times {\rm U}(s))^\sim)$-submodules
\[ \cO_\lambda(D)_{\widetilde{K}}\supset M_{\min\{q,s\}}(\lambda)\supset M_{\min\{q,s\}-1}(\lambda)\supset\cdots\supset M_{\max\{\lambda,0\}+1}(\lambda)\supset\{0\}. \]
When $\fp^+=\Alt(s,\BC)$, $\cO_\lambda(D)_{\widetilde{K}}$ is reducible if and only if $\lambda\in 2\left(\left\lfloor\frac{s}{2}\right\rfloor-1\right)-\BZ_{\ge 0}$,
and then we have the $\bigl(\mathfrak{so}^*(2s),\widetilde{{\rm U}}(s)\bigr)$-submodules
\[ \cO_\lambda(D)_{\widetilde{K}}\supset M_{\left\lfloor\frac{s}{2}\right\rfloor}(\lambda)\supset M_{\left\lfloor\frac{s}{2}\right\rfloor -1}(\lambda)\supset
\cdots\supset M_{\max\left\{\left\lceil\frac{\lambda}{2}\right\rceil,0\right\}+1}(\lambda)\supset\{0\}. \]
When $\fp^+=\Herm(3,\BO)^\BC$, $\cO_\lambda(D)_{\widetilde{K}}$ is reducible if and only if $\lambda\in 8-\BZ_{\ge 0}$, and then we have the
$\bigl(\mathfrak{e}_{7(-25)},\widetilde{{\rm U}}(1)\times E_6\bigr)$-submodules
\begin{alignat*}{3}
&\cO_\lambda(D)_{\widetilde{K}}\supset M_3(\lambda)\supset M_2(\lambda)\supset M_1(\lambda)\supset\{0\}, \qquad&& \lambda\in\BZ,\ \lambda\le 0,& \\
&\cO_\lambda(D)_{\widetilde{K}}\supset M_3(\lambda)\supset M_2(\lambda)\supset\{0\}, && \lambda\in\BZ,\ 1\le \lambda\le 4, &\\
&\cO_\lambda(D)_{\widetilde{K}}\supset M_3(\lambda)\supset\{0\}, && \lambda\in\BZ,\ 5\le \lambda\le 8.&
\end{alignat*}
When $\fp^+={\rm M}(1,2;\BO)^\BC$, $\cO_\lambda(D)_{\widetilde{K}}$ is reducible if and only if $\lambda\in 3-\BZ_{\ge 0}$, and then we have the
$\bigl(\mathfrak{e}_{6(-14)},\widetilde{{\rm U}}(1)\times {\rm Spin}(10)\bigr)$-submodules
\begin{alignat*}{3}
&\cO_\lambda(D)_{\widetilde{K}}\supset M_2(\lambda)\supset M_1(\lambda)\supset\{0\}, \qquad&& \lambda\in\BZ,\ \lambda\le 0, &\\
&\cO_\lambda(D)_{\widetilde{K}}\supset M_2(\lambda)\supset\{0\}, && \lambda\in\BZ,\ 1\le \lambda\le 3.&
\end{alignat*}

\subsection{Restriction to symmetric subgroups}\label{subsection_sym_subalg}

In this subsection we consider a~$\BC$-linear involution $\sigma$ on a~Hermitian positive Jordan triple system $\fp^\pm$, i.e.,
a Jordan triple system automorphism $\sigma\colon\fp^\pm\to\fp^\pm$ of order 2 which commutes with the $\BC$-antilinear map $\overline{\cdot}\colon \fp^\pm\to\fp^\mp$,
and extend to the involution of the Lie algebra $\fg^\BC=\fp^+\oplus\fk^\BC\oplus\fp^-$ by letting $\sigma$ act on $\fk^\BC=\mathfrak{str}(\fp^+)\subset\End_\BC(\fp^+)$ by
$\sigma(l):=\sigma l\sigma$. Also let ${\vartheta:=-I_{\fp^+}+I_{\fk^\BC}-I_{\fp^-}}$. Using these, we set
\begin{gather*}
\fp^\pm_1 :=\bigl(\fp^\pm\bigr)^\sigma=\big\{x\in\fp^\pm\mid \sigma(x)=x\big\}, \\
\fp^\pm_2 :=\bigl(\fp^\pm\bigr)^{-\sigma}=\big\{x\in\fp^\pm\mid \sigma(x)=-x\big\}, \\
\fk^\BC_1 :=\bigl(\fk^\BC\bigr)^\sigma=\big\{l\in\fk^\BC\mid \sigma(l)=l\big\}, \\
\fk_1 :=\fk^\sigma=\fk^\BC_1\cap\fk, \\
\fg_1^\BC :=\bigl(\fg^\BC\bigr)^\sigma=\fp^+_1\oplus\fk^\BC_1\oplus\fp^-_1, \\
\fg_2^\BC :=\bigl(\fg^\BC\bigr)^{\sigma\vartheta}=\fp^+_2\oplus\fk^\BC_1\oplus\fp^-_2, \\
\fg_1 :=\fg^\sigma=\fg^\BC_1\cap\fg, \\
\fg_2 :=\fg^{\sigma\vartheta}=\fg^\BC_2\cap\fg,
\end{gather*}
and let $G_1,G_1^\BC,G_2,G_2^\BC,K_1,K_1^\BC\subset G^\BC$ be the connected closed subgroups corresponding to~$\fg_1$, $\fg_1^\BC$, $\fg_2$, $\fg_2^\BC$, $\fk_1$, $\fk_1^\BC$, respectively.
Such $(G,G_1)$ is called a~symmetric pair of holomorphic type (see~\cite[Section 3.4]{Kmf1}).
Also let $\widetilde{G}_1\subset\widetilde{G}$ and $\widetilde{K}_1^\BC\subset\widetilde{K}^\BC$ be the connected closed subgroups of the universal covering groups of
$G$ and $K^\BC$ corresponding to $\fg_1$ and $\fk_1^\BC$ respectively, and let $D\subset\fp^+$, $D_1\subset\fp^+_1$ be the corresponding bounded symmetric domains,
so that $D\simeq G/K$, $D_1\simeq G_1/K_1$ hold.
Let $\chi$ be the character of $K^\BC$ given in~(\ref{chi_normalize}). Similarly, we normalize the bilinear form
$(\cdot|\cdot)_{\fp^+_j}\colon \fp^+_j\times\fp^-_j\to\BC$ $(j=1,2)$ such that $(e|\overline{e})_{\fp^+_j}=1$ holds for any primitive tripotent $e\in\fp^+_j$,
and define the characters $\chi_j$ $(j=1,2)$ of $K_1^\BC$ by
\begin{alignat*}{3}
&{\rm d}\chi_j([x,y])=(x|y)_{\fp^+_j}, \qquad&& x\in\fp^+_j,\ y\in\fp^-_j,& \\
&{\rm d}\chi_j(l)=0, && l\in \big[\fp^+_j,\fp^-_j\big]^\bot\subset \fk^\BC_1.&
\end{alignat*}
As in Section~\ref{subsection_HDS}, for $\lambda\in\BR$ and for an irreducible representation $V$ of $K_1^\BC$,
let $\cH_\lambda(D)\subset\cO_\lambda(D)$ and $\cH_{\lambda}(D_1,V)\subset\cO_\lambda(D_1,V)$ be the unitary representations of $\widetilde{G}$ and $\widetilde{G}_1$
with the minimal $\widetilde{K}$-type $\chi^{-\lambda}$ and with the minimal $\widetilde{K}_1$-type $\chi_1^{-\lambda}\otimes V$ respectively, if they exist.

In the following we assume $(G,G_1)$ is an irreducible symmetric pair. Then $\fp^+_2$ is a~direct sum of at most two simple Jordan triple systems.
Let $\rank\fp^+=:r$, $\rank\fp^+_2=:r_2$, and define $\varepsilon_1,\varepsilon_2\in\{1,2\}$ by ${\rm d}\chi|_{\fk_1^\BC}=\varepsilon_j {\rm d}\chi_j$, or equivalently, by
\begin{equation}\label{epsilon_j}
(x|y)_{\fp^+}=\varepsilon_j(x|y)_{\fp^+_j}, \qquad j=1,2,\quad x\in\fp^+_j,\quad y\in\fp^-_j.
\end{equation}
When $\fp^+_2$ is not simple, we write $\fp^+_2=:\fp^+_{11}\oplus\fp^+_{22}$, $\fp^+_1=:\fp_{12}^+$,
and let $\rank\fp^+_{11}=:r'$, $\rank\fp^+_{22}=r''$. Now we consider the restriction of the representation $\cH_\lambda(D)$ of $\widetilde{G}$ to the subgroup $\widetilde{G}_1$.
If $\lambda$ satisfies~(\ref{Wallach_set}), then the unitary representation $\cH_\lambda(D)$ exists, $\cH_\lambda(D)|_{\widetilde{G}_1}$ is discretely decomposable,
and every $\widetilde{G}_1$-submodule in $\cH_\lambda(D)|_{\widetilde{G}_1}$ contains a~$\fp^+_1$-null vector,
that is, has an intersection with $(\cH_\lambda(D)_{\widetilde{K}})^{\fp^+_1}=\bigl(\chi_1^{-\varepsilon_1\lambda}\otimes\cP\bigl(\fp^+_2\bigr)\bigr)\cap\cH_\lambda(D)_{\widetilde{K}}$.
Especially, $\cH_\lambda(D)_{\widetilde{K}}=\chi^{-\lambda}\otimes\cP(\fp^+)$ holds if $\lambda>\frac{d}{2}(r-1)$,
and $\cH_\lambda(D)$ is holomorphic discrete if $\lambda>p-1$. For such $\lambda$, according to the decomposition (Theorem~\ref{thm_HKS})
\begin{alignat*}{3}
&\cP\bigl(\fp^+_2\bigr)=\bigoplus_{\bk\in\BZ_{++}^{r_2}}\cP_\bk\bigl(\fp^+_2\bigr), && \fp^+_2\colon \ \text{simple},& \\
&\cP\bigl(\fp^+_2\bigr)=\bigoplus_{\bk\in\BZ_{++}^{r'}}\bigoplus_{\bl\in\BZ_{++}^{r''}}\cP_\bk\bigl(\fp^+_{11}\bigr)\boxtimes\cP_\bl\bigl(\fp^+_{22}\bigr), \qquad&& \fp^+_2\colon\ \text{non-simple},&
\end{alignat*}
the following holds.

\begin{Theorem}[{Kobayashi,~\cite[Theorems 8.3 and 8.4]{Kmf1}}]\label{thm_HKKS}\quad
\begin{enumerate}\itemsep=0pt
\item[$(1)$] Suppose $\fp^+$ is simple. For $\lambda>\frac{d}{2}(r-1)$, the restriction of $\cH_\lambda(D)$ to the subgroup $\widetilde{G}_1$ is
decomposed into the Hilbert direct sum of irreducible representations as
\begin{alignat*}{3}
&\cH_\lambda(D)|_{\widetilde{G}_1}\simeq\hsum_{\bk\in\BZ_{++}^{r_2}}\cH_{\varepsilon_1\lambda}\bigl(D_1,\cP_\bk\bigl(\fp^+_2\bigr)\bigr), && \fp^+_2\colon \ \text{simple}, \\
&\cH_\lambda(D)|_{\widetilde{G}_1}\simeq\hsum_{\bk\in\BZ_{++}^{r'}}\,\hsum_{\bl\in\BZ_{++}^{r''}}\cH_{\varepsilon_1\lambda}\bigl(D_1,\cP_\bk\bigl(\fp^+_{11}\bigr)\boxtimes\cP_\bl\bigl(\fp^+_{22}\bigr)\bigr),
\qquad&& \fp^+_2\colon \ \text{non-simple}.
\end{alignat*}
\item[$(2)$] Suppose $\fp^+=\fp^+_0\oplus\fp^+_0$ with $\fp^+_0$ simple of rank $r_0$, and $\sigma\colon (x,y)\mapsto (y,x)$, so that $(G,G_1)$ is of the form $(G_0\times G_0,\Delta(G_0))$.
For $\lambda,\mu>\frac{d}{2}(r_0-1)$, the tensor product representation $\cH_\lambda(D_0)\hotimes\cH_\mu(D_0)$ is decomposed under the diagonal subgroup $\Delta\bigl(\widetilde{G}_0\bigr)$
into the Hilbert direct sum of irreducible representations as
\[ \cH_\lambda(D_0)\hotimes\cH_\mu(D_0)\simeq\hsum_{\bk\in\BZ_{++}^{r_0}}\cH_{\lambda+\mu}\bigl(D_0,\cP_\bk\bigl(\fp^+_0\bigr)\bigr). \]
\end{enumerate}
\end{Theorem}

In the following, suppose $\fp^+$ is simple. Similar results also hold for the tensor product case.
Let $\langle\cdot,\cdot\rangle_{\lambda}=\langle\cdot,\cdot\rangle_{\lambda,\fp^+}$ be the $\widetilde{G}$-invariant inner product on $\cH_\lambda(D)$,
which is originally defined for $\lambda>p-1$ by~(\ref{Bergman_inner_prod}).
The purpose of this article is to study the above decomposition by observing the inner product
\begin{alignat}{3}
&\bigl\langle f(x_2),{\rm e}^{(x|\overline{z})_{\fp^+}}\bigr\rangle_{\lambda,x}, && f(x_2)\in\cP_\bk\bigl(\fp^+_2\bigr),& \nonumber\\
&\bigl\langle f(x_{11})g(x_{22}),{\rm e}^{(x|\overline{z})_{\fp^+}}\bigr\rangle_{\lambda,x}, \qquad&& f(x_{11})\in\cP_\bk\bigl(\fp^+_{11}\bigr),\ g(x_{22})\in\cP_\bl\bigl(\fp^+_{22}\bigr),& \label{purpose_inner_prod}
\end{alignat}
where $z\in\fp^+$, $x=x_1+x_2\in\fp^+=\fp^+_1\oplus\fp^+_2$ or $x=x_{11}+x_{12}+x_{22}\in\fp^+=\fp^+_{11}\oplus\fp^+_{12}\oplus\fp^+_{22}$,
and the subscript $x$ stands for the variable of integration.

In the author's previous article~\cite{N2}, when $\fp^+$ and $\fp^+_2$ are of tube type, we computed the above inner products explicitly
for some special $\bk\in\BZ_{++}^{r_2}$ or $(\bk,\bl)\in\BZ_{++}^{r'}\times\BZ_{++}^{r''}$, that is, for
\begin{alignat*}{3}
&\bk=(k+l,k,\dots,k), && \fp^+_2\colon \ \text{simple}, \ \varepsilon_2=1,& \\
&\bk=(k+1,\dots,k+1,k\dots,k), \qquad&& \fp^+_2\colon\ \text{simple},\ \varepsilon_2=2,& \\
&(\bk,\bl)=((k,\dots,k),\bl), && \fp^+_2\colon \ \text{non-simple},&
\end{alignat*}
and applied these results for determination of the $\widetilde{G}_1$-intertwining operators (symmetry breaking operators)
\begin{alignat*}{3}
&\cF_{\lambda,\bk}^\downarrow\colon \ \cH_\lambda(D)|_{\widetilde{G}_1}\longrightarrow\cH_{\varepsilon_1\lambda}\bigl(D_1,\cP_\bk\bigl(\fp^+_2\bigr)\bigr) && \fp^+_2\colon \ \text{simple}, &\\
&\cF_{\lambda,\bk,\bl}^\downarrow\colon \ \cH_\lambda(D)|_{\widetilde{G}_1}\longrightarrow\cH_{\varepsilon_1\lambda}\bigl(D_1,\cP_\bk\bigl(\fp^+_{11}\bigr)\boxtimes\cP_\bl\bigl(\fp^+_{22}\bigr)\bigr) \qquad
&& \fp^+_2\colon\ \text{non-simple}&
\end{alignat*}
for the above special $\bk$ or $(\bk,\bl)$. In this article we treat general $\bk\in\BZ_{++}^{r_2}$ or $(\bk,\bl)\in\BZ_{++}^{r'}\times\BZ_{++}^{r''}$,
compute the top term of~(\ref{purpose_inner_prod}), i.e., the value of~(\ref{purpose_inner_prod}) at $z_1=0$ or $z_{12}=0$,
and also compute the poles of~(\ref{purpose_inner_prod}) as a~function of $\lambda\in\BC$.

In the following, when $\fp^+_2$ is not simple, we write $\tilde{\bk}=(\bk,\bl)\in\BZ_{++}^{r'}\times\BZ_{++}^{r''}$,
and write $\cP_{\tilde{\bk}}\bigl(\fp^+_2\bigr):=\cP_\bk\bigl(\fp^+_{11}\bigr)\boxtimes\cP_\bl\bigl(\fp^+_{22}\bigr)$. Then since the map
\[ \cP\bigl(\fp^+_2\bigr)\longrightarrow\cP\bigl(\fp^+_2\bigr), \qquad f(x_2)\longmapsto \bigl\langle f(x_2),{\rm e}^{(x|\overline{z})_{\fp^+}}\bigr\rangle_{\lambda,x}\big|_{z_1=0} \]
is $K_1$-equivariant and since $\cP\bigl(\fp^+_2\bigr)$ decomposes multiplicity-freely under $K_1$ by Theorem~\ref{thm_HKS},
there exist constants $C_{\fp^+,\fp^+_2}\bigl(\lambda,\tilde{\bk}\bigr)\in\BC$ such that
\begin{equation}\label{const_Plancherel}
\bigl\langle f(x_2),{\rm e}^{(x|\overline{z})_{\fp^+}}\bigr\rangle_{\lambda,x}\big|_{z_1=0}=C_{\fp^+,\fp^+_2}\bigl(\lambda,\tilde{\bk}\bigr)f(z_2), \qquad f(x_2)\in\cP_{\tilde{\bk}}\bigl(\fp^+_2\bigr)
\end{equation}
holds for every $\tilde{\bk}\in\BZ_{++}^{r_2}$ or $\tilde{\bk}\in\BZ_{++}^{r'}\times\BZ_{++}^{r''}$.
Then the Parseval--Plancherel-type formula is given by using these constants. Let $V_{\tilde{\bk}}$ be an abstract $K_1$-module isomorphic to $\cP_{\tilde{\bk}}\bigl(\fp^+_2\bigr)$,
let $\Vert\cdot\Vert_{\varepsilon_1\lambda,\tilde{\bk},\fp^+_1}$ be the $\widetilde{G}_1$-invariant norm on $\cH_{\varepsilon_1\lambda}\bigl(D_1,V_{\tilde{\bk}}\bigr)$
normalized such that $\Vert v\Vert_{\varepsilon_1\lambda,\tilde{\bk},\fp^+_1}=|v|_{V_{\tilde{\bk}}}$ holds for all constant functions $v\in V_{\tilde{\bk}}$,
and let $\Vert\cdot\Vert_{F,\fp^+}$ be the Fischer norm on $\fp^+$ given in~(\ref{Fischer_inner_prod}).
\begin{Proposition}\label{prop_Plancherel}
For $\lambda>p-1$ and for $\tilde{\bk}\in\BZ_{++}^{r_2}$ \big(when $\fp^+_2$ is simple\big) or $\tilde{\bk}\in\BZ_{++}^{r'}\times\BZ_{++}^{r''}$ \big(when $\fp^+_2$ is not simple\big),
let $C_{\fp^+,\fp^+_2}\bigl(\lambda,\tilde{\bk}\bigr)\in\BC$ be as in~\eqref{const_Plancherel}.
\begin{enumerate}\itemsep=0pt
\item[$(1)$] For $f(x_2)\in\cP_{\tilde{\bk}}\bigl(\fp^+_2\bigr)$, we have $\Vert f\Vert_{\lambda,\fp^+}^2=C_{\fp^+,\fp^+_2}\bigl(\lambda,\tilde{\bk}\bigr)\Vert f\Vert_{F,\fp^+}^2$.
\item[$(2)$] We take a~vector-valued polynomial $\rK_{\tilde{\bk}}(x_2)\in\cP\bigl(\fp^+_2,\overline{V_{\tilde{\bk}}}\bigr)^{K_1}$ normalized such that
\[
\big| \langle f(x_2),\rK_{\tilde{\bk}}(x_2)\rangle_{F,\fp^+}\big|_{V_{\tilde{\bk}}}^2
=\Vert f\Vert_{F,\fp^+}^2, \qquad f(x_2)\in\cP_{\tilde{\bk}}\bigl(\fp^+_2\bigr),
\]
and define the vector-valued polynomial $F_{\lambda,\tilde{\bk}}^\downarrow(z)\in \cP\bigl(\fp^-,V_{\tilde{\bk}}\bigr)$ by
\begin{equation}\label{SBO1}
F_{\lambda,\tilde{\bk}}^\downarrow(z):=\frac{1}{C_{\fp^+,\fp^+_2}\bigl(\lambda,\tilde{\bk}\bigr)}\bigl\langle {\rm e}^{(x|z)_{\fp^+}},\rK_{\tilde{\bk}}(x_2)\bigr\rangle_{\lambda,x}.
\end{equation}
Then the differential operator
\begin{equation}\label{SBO2}
\cF_{\lambda,\tilde{\bk}}^\downarrow\colon \ \cH_\lambda(D)|_{\widetilde{G}_1}\longrightarrow\cH_{\varepsilon_1\lambda}\bigl(D_1,V_{\tilde{\bk}}\bigr), \qquad
\bigl(\cF_{\lambda,\tilde{\bk}}^\downarrow f\bigr)(x_1):=F_{\lambda,\tilde{\bk}}^\downarrow \left(\frac{\partial}{\partial x}\right)f(x)\biggr|_{x_2=0}
\end{equation}
becomes a~symmetry breaking operator satisfying
\[
\big\Vert \cF_{\lambda,\tilde{\bk}}^\downarrow f\big\Vert_{\varepsilon_1\lambda,\tilde{\bk},\fp^+_1}^2=\Vert f\Vert_{F,\fp^+}^2, \qquad f(x_2)\in\cP_{\tilde{\bk}}\bigl(\fp^+_2\bigr).
\]
\item[$(3)$] For $f\in\cH_\lambda(D)$, we have
$\ds \Vert f\Vert_{\lambda,\fp^+}^2=\sum_{\tilde{\bk}} C_{\fp^+,\fp^+_2}\bigl(\lambda,\tilde{\bk}\bigr)\big\Vert \cF_{\lambda,\tilde{\bk}}^\downarrow f\big\Vert_{\varepsilon_1\lambda,\tilde{\bk},\fp^+_1}^2$.
\end{enumerate}
\end{Proposition}

\begin{proof}
(1) Since the reproducing kernel of $\langle\cdot,\cdot\rangle_{F,\fp^+}$ is given by ${\rm e}^{(z|\overline{x})_{\fp^+}}$ (see~\cite[Proposition~XI.1.1]{FK}),
for $f,g\in\cP(\fp^+)$ we have
\[
\langle f(x),g(x)\rangle_{\lambda,x}=\bigl\langle f(x),\bigl\langle g(z),{\rm e}^{(z|\overline{x})_{\fp^+}}\bigr\rangle_{F,z}\bigr\rangle_{\lambda,x}
=\bigl\langle\bigl\langle f(x),{\rm e}^{(x|\overline{z})_{\fp^+}}\bigr\rangle_{\lambda,x},g(z)\bigr\rangle_{F,z}.
\]
Then by putting $f=g\in\cP_{\tilde{\bk}}\bigl(\fp^+_2\bigr)\subset\cP(\fp^+)$, we get the desired formula.

(2) The symmetry breaking property follows from~\cite[Theorem~3.10\,(1)]{N}. Let $f(x_2)\in\cP_{\tilde{\bk}}\bigl(\fp^+_2\bigr)$.
Then since $\exp\bigl(x\big|\frac{\partial}{\partial y}\bigr)_{\fp^+}f(y)\big|_{y=0}=f(x)$ holds, by (1), we have
\begin{align*}
\bigl(\cF_{\lambda,\tilde{\bk}}^\downarrow f\bigr)(y_1)
&=\frac{1}{C_{\fp^+,\fp^+_2}\bigl(\lambda,\tilde{\bk}\bigr)}\,\biggl\langle \exp\left(x\middle|\frac{\partial}{\partial y}\right)_{\fp^+},\rK_{\tilde{\bk}}(x_2)\biggr\rangle_{\lambda,x}f(y_2)
\biggr|_{y_2=0} \\
&=\frac{1}{C_{\fp^+,\fp^+_2}\bigl(\lambda,\tilde{\bk}\bigr)}\left\langle f(x_2),\rK_{\tilde{\bk}}(x_2)\right\rangle_{\lambda,x}\biggr|_{y_2=0}
=\left\langle f(x_2),\rK_{\tilde{\bk}}(x_2)\right\rangle_{F,x}\in V_{\tilde{\bk}}.
\end{align*}
Then by the normalization of $\Vert\cdot\Vert_{\varepsilon_1\lambda,\tilde{\bk},\fp^+_1}$ and $\rK_{\tilde{\bk}}(x_2)$, we get
\[
\big\Vert \cF_{\lambda,\tilde{\bk}}^\downarrow f\big\Vert_{\varepsilon_1\lambda,\tilde{\bk},\fp^+_1}^2
=\big| \langle f(x_2),\rK_{\tilde{\bk}}(x_2)\rangle_{F,\fp^+} \big|_{V_{\tilde{\bk}}}^2
=\Vert f\Vert_{F,\fp^+}^2, \qquad f(x_2)\in\cP_{\tilde{\bk}}\bigl(\fp^+_2\bigr).
\]

(3) Let
\[
\cF_{\lambda,\tilde{\bk}}^\uparrow\colon \ \cH_{\varepsilon_1\lambda}\bigl(D_1,V_{\tilde{\bk}}\bigr)\longrightarrow\cH_\lambda(D)|_{\widetilde{G}_1}
\]
be the $\widetilde{G}_1$-intertwining operator (holographic operator) normalized such that
\[
\cF_{\lambda,\tilde{\bk}}^\downarrow \circ \cF_{\lambda,\tilde{\bk}}^\uparrow=I_{\cH_{\varepsilon_1\lambda}\bigl(D_1,V_{\tilde{\bk}}\bigr)}
\]
holds. Then since $\cH_\lambda(D)$ decomposes multiplicity-freely under $\widetilde{G}_1$, for $f\in\cH_\lambda(D)$, we have
\[
f=\sum_{\tilde{\bk}}\cF_{\lambda,\tilde{\bk}}^\uparrow\cF_{\lambda,\tilde{\bk}}^\downarrow f, \qquad
\Vert f\Vert_{\lambda,\fp^+}^2=\sum_{\tilde{\bk}}\big\Vert \cF_{\lambda,\tilde{\bk}}^\uparrow\cF_{\lambda,\tilde{\bk}}^\downarrow f\big\Vert_{\lambda,\fp^+}^2.
\]
If $f(x_2)\in\cP_{\tilde{\bk}}\bigl(\fp^+_2\bigr)$, then since $\cF_{\lambda,\tilde{\bk}}^\uparrow\cF_{\lambda,\tilde{\bk}}^\downarrow f=f$ holds, by~(1) and~(2), we have
\[
\big\Vert \cF_{\lambda,\tilde{\bk}}^\uparrow\cF_{\lambda,\tilde{\bk}}^\downarrow f\big\Vert_{\lambda,\fp^+}^2
=\Vert f\Vert_{\lambda,\fp^+}^2=C_{\fp^+,\fp^+_2}\bigl(\lambda,\tilde{\bk}\bigr)\Vert f\Vert_{F,\fp^+}^2
=C_{\fp^+,\fp^+_2}\bigl(\lambda,\tilde{\bk}\bigr)\big\Vert \cF_{\lambda,\tilde{\bk}}^\downarrow f\big\Vert_{\varepsilon_1\lambda,\tilde{\bk},\fp^+_1}^2,
\]
and since $\cF_{\lambda,\tilde{\bk}}^\uparrow$ is an isometry up to scalar multiple,
\[
\big\Vert \cF_{\lambda,\tilde{\bk}}^\uparrow\cF_{\lambda,\tilde{\bk}}^\downarrow f\big\Vert_{\lambda,\fp^+}^2
=C_{\fp^+,\fp^+_2}\bigl(\lambda,\tilde{\bk}\bigr)\big\Vert \cF_{\lambda,\tilde{\bk}}^\downarrow f\big\Vert_{\varepsilon_1\lambda,\tilde{\bk},\fp^+_1}^2
\]
holds for all $f\in\cH_\lambda(D)$. Hence we get the desired formula.
\end{proof}

Next we consider the meromorphic continuation of~(\ref{purpose_inner_prod}),~(\ref{SBO1}), and~(\ref{SBO2}).
The differential operator $\cF_{\lambda,\tilde{\bk}}^\downarrow$ in~(\ref{SBO2}) is extended to the map
\[
\cF_{\lambda,\tilde{\bk}}^\downarrow\colon \ \cO_\lambda(D)|_{\widetilde{G}_1}\longrightarrow\cO_{\varepsilon_1\lambda}\bigl(D_1,V_{\tilde{\bk}}\bigr),
\]
and this is meromorphically continued for all $\lambda\in\BC$.
For generic $\lambda$, we have an abstract $(\fg_1,\widetilde{K}_1)$-isomorphism
\[
\cO_{\varepsilon_1\lambda}\bigl(D_1,V_{\tilde{\bk}}\bigr)_{\widetilde{K}_1}\simeq {\rm d}\tau_\lambda(\cU(\fg_1))\cP_{\tilde{\bk}}\bigl(\fp^+_2\bigr)
\subset\cP(\fp^+)=\cO_\lambda(D)_{\widetilde{K}},
\]
where $\cU(\fg_1)$ is the universal enveloping algebra of $\fg_1^\BC$, and ${\rm d}\tau_\lambda$ is the differential of $(\tau_\lambda,\cO_\lambda(D))$.
The restriction of $\cF_{\lambda,\tilde{\bk}}^\downarrow$ gives the explicit isomorphism from the right to the left.
On the other hand, this does not hold for all $\lambda\in\BC$.
By computing the poles of~(\ref{purpose_inner_prod}) with respect to $\lambda$, we can get some information on ${\rm d}\tau_\lambda(\cU(\fg_1))\cP_{\tilde{\bk}}\bigl(\fp^+_2\bigr)$ for such singular $\lambda$.

\begin{Proposition}\label{prop_poles}
Let $\bl=(l_1,\dots,l_{r})\in\BZ_{++}^r$. Suppose
\[ (\lambda)_{\bl,d}\bigl\langle f(x_2),{\rm e}^{(x|\overline{z})_{\fp^+}}\bigr\rangle_{\lambda,x} \]
is holomorphically continued for all $\lambda\in\BC$ for some non-zero $f(x_2)\in\cP_{\tilde{\bk}}\bigl(\fp^+_2\bigr)$.
\begin{enumerate}\itemsep=0pt
\item[$(1)$] We have
\[ \cP_{\tilde{\bk}}\bigl(\fp^+_2\bigr)\subset\bigoplus_{\substack{\bm\in\BZ_{++}^r \\ m_j\le l_j, \, j=1,\dots,r}}\cP_\bm(\fp^+), \]
and especially, for $j=1,\dots,r$,
\[ {\rm d}\tau_\lambda(\cU(\fg_1))\cP_{\tilde{\bk}}\bigl(\fp^+_2\bigr)\subset M_j^\fg(\lambda)\qquad \text{holds if}\quad \lambda\in\frac{d}{2}(j-1)-l_j-\BZ_{\ge 0}, \]
where $M_j^\fg(\lambda)\subset\cO_\lambda(D)_{\widetilde{K}}$ is as in~\eqref{submodule}.
\item[$(2)$] For $a=0,1,\dots,r-1$, if $l_{a+1}=0$ and $C_{\fp^+,\fp^+_2}\bigl(\frac{d}{2}a,\tilde{\bk}\bigr)\ne 0$,
then $\cF_{\lambda,\tilde{\bk}}^\downarrow$ is holomorphic at $\lambda=\frac{d}{2}a$, and its restriction gives the symmetry breaking operator
\[ \cF_{\frac{d}{2}a,\tilde{\bk}}^\downarrow\colon \
\cH_{\frac{d}{2}a}(D)|_{\widetilde{G}_1}\longrightarrow\cH_{\varepsilon_1\frac{d}{2}a}\bigl(D_1,V_{\tilde{\bk}}\bigr). \]
\end{enumerate}
\end{Proposition}

\begin{proof}
(1) Follows from the last paragraph of Section~\ref{subsection_HDS}.

(2) The holomorphy of $\cF_{\lambda,\tilde{\bk}}^\downarrow$ at $\lambda=\frac{d}{2}a$ is clear from the assumption.
For the latter claim, $\cP_{\tilde{\bk}}\bigl(\fp^+_2\bigr)\subset M_{a+1}^\fg\bigl(\frac{d}{2}a\bigr)=\cH_{\frac{d}{2}a}(D)_{\widetilde{K}}$ holds by (1),
and the image of $\cP_{\tilde{\bk}}\bigl(\fp^+_2\bigr)$ by $\cF_{\frac{d}{2}a,\tilde{\bk}}^\downarrow$ is non-zero by the normalization assumption of $\cF_{\lambda,\tilde{\bk}}^\downarrow$.
Therefore, the image of $\cH_{\frac{d}{2}a}(D)$ by $\cF_{\frac{d}{2}a,\tilde{\bk}}^\downarrow$ becomes a~non-zero unitary submodule of $\cO_{\varepsilon_1\frac{d}{2}a}\bigl(D_1,V_{\tilde{\bk}}\bigr)$,
which is unique and denoted by $\cH_{\varepsilon_1\frac{d}{2}a}\bigl(D_1,V_{\tilde{\bk}}\bigr)$.
\end{proof}

Since $\cH_\lambda(D)|_{\widetilde{G}_1}$ decomposes multiplicity-freely~\cite[Theorem A]{Kmf1}, symmetry breaking operators on $\cH_\lambda(D)$ are unique up to constant multiple
for all $\lambda$ in~(\ref{Wallach_set}). On the other hand, we do not know a~priori whether symmetry breaking operators on $\cO_\lambda(D)$ are unique or not for $\lambda\le \frac{d}{2}(r-1)$.
For non-unique case see~\cite[Section 9]{KP2} or Section~\ref{section_tensor} of this paper.
It is known that symmetry breaking operators on $\cO_\lambda(D)$ are always given by differential operators (localness theorem,~\cite[Theorem 5.3]{KP1}),
and their symbols are characterized as polynomial solutions of certain differential equations (F-method,~\cite[Theorem 3.1]{KP2}).
That is, for $\lambda\in\BC$ let $\cB_\lambda^{\fp^\pm}$ be the vector-valued differential operator
\begin{gather*}
\cB_\lambda^{\fp^\pm}\colon \ \cP\bigl(\fp^\pm\bigr)\longrightarrow \cP\bigl(\fp^\pm\bigr)\otimes\fp^\mp, \\
\cB_\lambda^{\fp^\pm}f(z):=\sum_{\alpha,\beta}\frac{1}{2}Q\bigl(e^\mp_\alpha,e^\mp_\beta\bigr)z\frac{\partial^2f}{\partial z_\alpha^\pm\partial z_\beta^\pm}(z)
+\lambda\sum_{\alpha}e_\alpha^\mp \frac{\partial f}{\partial z_\alpha^\pm}(z),
\end{gather*}
where $\big\{e_\alpha^\mp\big\}\subset\fp^\mp$ is a~basis of $\fp^\mp$, and $\big\{z_\alpha^\pm\big\}$ is the coordinate of $\fp^\pm$ dual to $\big\{e_\alpha^\mp\big\}$
(\textit{Bessel operator}~\cite{D} and~\cite[Section XV.2]{FK}), let
\[
\bigl(\cB_\lambda^{\fp^\pm}\bigr)_1\colon \ \cP\bigl(\fp^\pm\bigr)\longrightarrow\cP\bigl(\fp^\pm\bigr)\otimes\fp^\mp_1
\]
be the orthogonal projection of $\cB_\lambda^{\fp^\pm}$ onto $\fp^\mp_1$, and let
\[
\operatorname{Sol}_{\cP\bigl(\fp^\pm\bigr)}\bigl(\bigl(\cB_\lambda^{\fp^\pm}\bigr)_1\bigr):=\big\{ f(z)\in\cP\bigl(\fp^\pm\bigr)\ \big|\ \bigl(\cB_\lambda^{\fp^\pm}\bigr)_1 f=0 \big\}.
\]
Then we have
\begin{align}
\Hom_{\widetilde{G}_1}(\cO_\lambda(D),\cO_{\varepsilon_1\lambda}(D_1,V))
&\simeq \Hom_{\fk_1^\BC\oplus\fp_1^-}\bigl(\chi_1^{\varepsilon_1\lambda}\otimes V^\vee,\operatorname{ind}_{\fk^\BC\oplus\fp^-}^{\fg^\BC}\bigl(\chi^{\lambda}\bigr)\bigr) \nonumber\\
&\simeq\bigl(\operatorname{Sol}_{\cP(\fp^-)}\bigl(\bigl(\cB_\lambda^{\fp^-}\bigr)_1\bigr)\otimes V\bigr)^{K_1}. \label{Fmethod}
\end{align}
By this F-method, we can prove the following, which gives an analogue of~\cite[Corollary 12.8\,(3)]{KS1}.

\begin{Proposition}\label{prop_det_factorize}
Suppose $\fp^+=\fn^{+\BC}$, $\fp^+_2=\fn^{+\BC}_2$ are of tube type.
\begin{enumerate}\itemsep=0pt
\item[$(1)$]
For $a\in\BZ_{>0}$, the following differential operator intertwines the $\widetilde{G}$-action
\[ \det_{\fn^-}\left(\frac{\partial}{\partial x}\right)^a\colon \ \cO_{\frac{n}{r}-a}(D)\longrightarrow\cO_{\frac{n}{r}+a}(D). \]
{\rm (See \cite[Theorem 6.4]{A},~\cite[Propositions 1.2 and 2.2]{J0},~\cite[Lemma 7.1, formulas (8.6)--(8.8)]{Sh}.)}
\item[$(2)$] We fix $\tilde{\bk}\in\BZ_{++}^{r_2}$ $\bigl($when $\fp^+_2$ is simple$\bigr)$ or $\tilde{\bk}\in\BZ_{++}^{r'}\times\BZ_{++}^{r''}$ \big(when $\fp^+_2$ is not simple\big),
and let $\bl\in\BZ_{++}^r$, $a\in\BZ_{>0}$.
Suppose that $(\lambda)_{\bl,d}\bigl\langle \det_{\fn^+_2}(x_2)^{\varepsilon_2a}f(x_2),{\rm e}^{(x|\overline{z})_{\fp^+}}\bigr\rangle_{\lambda,x}$ is holomorphic at $\lambda=\frac{n}{r}-a$, and
\[
\det_{\fn^+}(z)^{-a}(\lambda)_{\bl,d}\bigl\langle \det_{\fn^+_2}(x_2)^{\varepsilon_2a}f(x_2),{\rm e}^{(x|\overline{z})_{\fp^+}}\bigr\rangle_{\lambda,x}\big|_{\lambda=\frac{n}{r}-a}\in\cP(\fp^+)
\]
holds for some non-zero $f(x_2)\in\cP_{\tilde{\bk}}\bigl(\fp^+_2\bigr)$. Then there exists $C\in\BC$ such that
\[
(\lambda)_{\bl,d}\bigl\langle \det_{\fn^+_2}(x_2)^{\varepsilon_2a}f(x_2),{\rm e}^{(x|\overline{z})_{\fp^+}}\bigr\rangle_{\lambda,x}\big|_{\lambda=\frac{n}{r}-a}
=C\det_{\fn^+}(z)^a\bigl\langle f(x_2),{\rm e}^{(x|\overline{z})_{\fp^+}}\bigr\rangle_{\frac{n}{r}+a,x}
\]
holds for all $f(x_2)\in\cP_{\tilde{\bk}}\bigl(\fp^+_2\bigr)$.
\end{enumerate}
\end{Proposition}

We note that $\det_{\fn^+_2}(x_2)=\det_{\fn^+_{11}}(x_{11})\det_{\fn^+_{22}}(x_{22})$ holds when $\fp^+_2$ is not simple.
\begin{proof}
(1) By the proof of~\cite[Proposition XV.2.4]{FK}, as a~differential operator we have
\begin{equation}\label{Bessel_det}
\cB_{\frac{n}{r}-a}^{\fp^\pm}\det_{\fn^\pm}(z)^a=\det_{\fn^\pm}(z)^a\cB_{\frac{n}{r}+a}^{\fp^\pm}, \qquad a\in\BC,
\end{equation}
and hence for $a\in\BZ_{>0}$, as a~polynomial we have $\det_{\fn^\pm}(z)^a\in\operatorname{Sol}_{\cP(\fp^\pm)}\bigl(\cB_{\frac{n}{r}-a}^{\fp^\pm}\bigr)$.
Now since $\BC\det_{\fn^\pm}(z)^a\simeq \chi^{\mp 2a}$ holds as a~$K$-module, by applying the F-method for $\sigma=I_{\fg^\BC}$, $\fp^+_1=\fp^+$ case, we get the intertwining property.

(2) As in~\cite[formula (2.21)]{N2}, we have $\bigl(\cB_\lambda^{\fp^+}\bigr)_1\bigl\langle g(x_2),{\rm e}^{(x|\overline{z})_{\fp^+}}\bigr\rangle_{\lambda,x}=0$ for all $g(x_2)\in\cP\bigl(\fp^+_2\bigr)$,
and by~(\ref{Bessel_det}), we have
\[
\bigl(\cB_{\frac{n}{r}+a}^{\fp^+}\bigr)_1\det_{\fn^+}(z)^{-a}(\lambda)_{\bl,d}
\bigl\langle \det_{\fn^+_2}(x_2)^{\varepsilon_2a}f(x_2),{\rm e}^{(x|\overline{z})_{\fp^+}}\bigr\rangle_{\lambda,x}\big|_{\lambda=\frac{n}{r}-a}=0.
\]
Hence under the polynomiality assumption, we have
\begin{gather*}
\bigl(f(x_2)\mapsto \det_{\fn^+}(z)^{-a}(\lambda)_{\bl,d}\bigl\langle \det_{\fn^+_2}(x_2)^{\varepsilon_2a}f(x_2),{\rm e}^{(x|\overline{z})_{\fp^+}}\bigr\rangle_{\lambda,x}
\big|_{\lambda=\frac{n}{r}-a}\bigr) \\
\hphantom{\bigl(f(x_2)\mapsto}{}\in\Hom_{K_1}\bigl(\cP_{\tilde{\bk}}\bigl(\fp^+_2\bigr),\operatorname{Sol}_{\cP(\fp^+)}\bigl(\bigl(\cB_{\frac{n}{r}+a}^{\fp^+}\bigr)_1\bigr)\bigr)
=\BC\bigl( f(x_2)\mapsto \bigl\langle f(x_2),{\rm e}^{(x|\overline{z})_{\fp^+}}\bigr\rangle_{\frac{n}{r}+a,x}\bigr),
\end{gather*}
where the last equality holds since the space~(\ref{Fmethod}) is 1-dimensional for $\lambda>\frac{n}{r}-1=\frac{d}{2}(r-1)$, and this completes the proof.
\end{proof}

\section{Key lemmas}\label{section_keylemmas}

In this section, we introduce some equations on inner products and results on finite-dimensional representations of ${\rm U}(s)$ used later.

\subsection{Reduction to smaller algebras}\label{subsection_reduction}

We consider a~simple Jordan triple system $\fp^+$. We fix a~tripotent $e\in\fp^+$,
and let $\fp^{+\prime}:=\fp^+(e)_2\subset\fp^+$ denote the eigenspace of $D(e,\overline{e})$ with the eigenvalue 2, as in~(\ref{Peirce}).
This $\fp^{+\prime}$ becomes of tube type.
For $x\in\fp^+$, let $x'\in\fp^{+\prime}$ be the orthogonal projection. Also, let $D\subset\fp^+$, $D'\subset\fp^{+\prime}$ be the corresponding bounded symmetric domains,
and let $\langle\cdot,\cdot\rangle_{\lambda,\fp^+}$, $\langle\cdot,\cdot\rangle_{\lambda,\fp^{+\prime}}$ be the weighted Bergman inner products on $D$, $D'$ respectively,
as in~(\ref{Bergman_inner_prod}).
\begin{Proposition}\label{prop_reduction}
We fix a~tripotent $e\in\fp^+$, and let $\fp^{+\prime}:=\fp^+(e)_2$. Then for $\Re\lambda>p-1$, for $f(x')\in\cP(\fp^{+\prime})\subset\cP(\fp^+)$, we have
\[ \bigl\langle f(x'),{\rm e}^{(x|\overline{z})_{\fp^+}}\bigr\rangle_{\lambda,x,\fp^+}=\bigl\langle f(x'),{\rm e}^{(x'|\overline{z'})_{\fp^{+\prime}}}\bigr\rangle_{\lambda,x',\fp^{+\prime}}. \]
\end{Proposition}
Here the subscripts $x$, $x'$ denote the variables of integration. Note that \smash{$\bigl(x'|\overline{z'}\bigr)_{\fp^+} = \bigl(x'|\overline{z'}\bigr)_{\fp^{+\prime}}$} holds for any $x',z'\in\fp^{+\prime}$,
since a~primitive tripotent $\tilde{e}\in\fp^{+\prime}$ is also primitive in $\fp^+$, and both inner products are normalized such that
\smash{$\bigl(\tilde{e}|\overline{\tilde{e}}\bigr)_{\fp^+}=\bigl(\tilde{e}|\overline{\tilde{e}}\bigr)_{\fp^{+\prime}}=1$} holds.
This lemma is proved easily by using Corollary~\ref{cor_FK} and $\cP_\bm\bigl(\fp^{+\prime}\bigr)\subset\cP_\bm(\fp^+)$.
Here we give a~proof that does not rely on Corollary~\ref{cor_FK}.
\begin{proof}
We write $x=x_2+x_1+x_0\in\fp^+=\fp^+(e)_2\oplus\fp^+(e)_1\oplus\fp^+(e)_0$, so that $x'=x_2$. First we prove
\begin{equation}\label{lem_reduction_proof1}
\bigl\langle f(x_2),{\rm e}^{(x|\overline{z})_{\fp^+}}\bigr\rangle_{\lambda,x,\fp^+}=\bigl\langle f(x_2),{\rm e}^{(x_2|\overline{z_2})_{\fp^+}}\bigr\rangle_{\lambda,x,\fp^+}
\end{equation}
for $f(x_2)\in\cP\bigl(\fp^+(e)_2\bigr)$. Since
\[
{\rm e}^{(x|\overline{z})_{\fp^+}}={\rm e}^{(x_2|\overline{z_2})_{\fp^+}}\sum_{k_1,k_0=0}^\infty \frac{1}{k_1!k_0!}(x_1|\overline{z_1})_{\fp^+}^{k_1}(x_0|\overline{z_0})_{\fp^+}^{k_0}
\]
holds, it is enough to prove $\langle f(x_2),g(x_2,x_1,x_0)\rangle_{\lambda,x,\fp^+}=0$ for $g(x_2,x_1,x_0)\in\cO\bigl(\fp^+(e)_2\bigr)\otimes\cP\bigl(\fp^+(e)_1\bigr)\otimes\cP\bigl(\fp^+(e)_0\bigr)$ satisfying
\[
g(x_2,sx_1,tx_0)=s^{k_1}t^{k_0}g(x_2,x_1,x_0),\qquad s,t\in\BC
\]
for some $k_1,k_0\in\BZ_{\ge 0}$ with $(k_1,k_0)\ne (0,0)$. For $t\in\sqrt{-1}\BR$, let
\[
l:=\exp(t(2I_{\fp^+}-D(e,\overline{e})))=I_{\fp^+(e)_2}+e^tI_{\fp^+(e)_1}+{\rm e}^{2t}I_{\fp^+(e)_0}\in K,
\]
and let $\Proj_2\colon\fp^+\to\fp^+(e)_2$ be the orthogonal projection. Then we have
\begin{align*}
\langle f(x_2),g(x_2,x_1,x_0)\rangle_{\lambda,x,\fp^+}&=\bigl\langle f\bigl(\Proj_2\bigl(l^{-1}(x_2,x_1,x_0)\bigr)\bigr),g(x_2,x_1,x_0)\bigr\rangle_{\lambda,x,\fp^+} \\
&=\langle f(\Proj_2(x_2,x_1,x_0)),g(l(x_2,x_1,x_0))\rangle_{\lambda,x,\fp^+}\\
&=\bigl\langle f(x_2),g\bigl(x_2,e^tx_1,{\rm e}^{2t}x_0\bigr)\bigr\rangle_{\lambda,x,\fp^+} \\
&=\overline{{\rm e}^{(k_1+2k_0)t}}\langle f(x_2),g(x_2,x_1,x_0)\rangle_{\lambda,x,\fp^+}
\end{align*}
for all $t\in\sqrt{-1}\BR$, and hence this vanishes if $(k_1,k_0)\ne (0,0)$. Therefore,~(\ref{lem_reduction_proof1}) holds.
Next, to prove
\begin{equation}\label{lem_reduction_proof2}
\bigl\langle f(x_2),{\rm e}^{(x_2|\overline{z_2})_{\fp^+}}\bigr\rangle_{\lambda,x,\fp^+}=\bigl\langle f(x_2),{\rm e}^{(x_2|\overline{z_2})_{\fp^+}}\bigr\rangle_{\lambda,x_2,\fp^+(e)_2},
\end{equation}
we check that the natural inclusion $\cH_\lambda(D')_{\widetilde{K}(e)_2}=\cP\bigl(\fp^+(e)_2\bigr)\hookrightarrow\cH_\lambda(D)_{\widetilde{K}}=\cP(\fp^+)$ intertwines the $\bigl(\fg(e)_2,\widetilde{K}(e)_2\bigr)$-action,
where $\fg(e)_2$, $K(e)_2$ are as in the last paragraph of Section~\ref{subsection_KKT}.
For example, by~(\ref{HDS_diff_action}), the action of $\fp^-$ on $\cH_\lambda(D)_{\widetilde{K}}=\cP(\fp^+)$ is given by
\[
({\rm d}\tau_\lambda(0,0,w)f)(x)=\lambda(x|w)_{\fp^+}f(x)+\frac{{\rm d}}{{\rm d}t}\biggr|_{t=0}f(x+tQ(x)w), \qquad w\in\fp^-.
\]
If $w=w_2\in\fp^-(e)_2$, then we have
\begin{gather*}
(x|w_2)_{\fp^+}=(x_2|w_2)_{\fp^+}, \\
x+tQ(x)w_2=(x_2+tQ(x_2)w_2)+(x_1+tQ(x_2,x_1)w_2)+(x_0+tQ(x_1)w_2) \\
\hphantom{x+tQ(x)w_2=}{}\in\fp^+(e)_2\oplus\fp^+(e)_1\oplus\fp^+(e)_0,
\end{gather*}
and hence if $f(x)=f(x_2)\in\cP\bigl(\fp^+(e)_2\bigr)$, we have
\[
({\rm d}\tau_\lambda(0,0,w_2)f)(x)=\lambda(x_2|w_2)_{\fp^+}f(x_2)+\frac{{\rm d}}{{\rm d}t}\biggr|_{t=0}f(x_2+tQ(x_2)w_2).
\]
Thus, the natural inclusion is $\fp^-(e)_2$-equivariant. The $\widetilde{K}(e)_2$- and $\fp^+(e)_2$-equivariance are also proved similarly.
Hence, this is $\bigl(\fg(e)_2,\widetilde{K}(e)_2\bigr)$-equivariant, and therefore is an isometry up to scalar multiple.
Moreover, by the normalization assumption $\Vert 1\Vert_{\lambda,\fp^+}=\Vert 1\Vert_{\lambda,\fp^+(e)_2}=1$, this is exactly an isometry.
Hence~(\ref{lem_reduction_proof2}) holds, and this proves the proposition.
\end{proof}

Now we consider an involution $\sigma$ on $\fp^+$, and let $\fp^+_1:=(\fp^+)^\sigma$, $\fp^+_2:=(\fp^+)^{-\sigma}$, $K_1:=K^\sigma$ as in Section~\ref{subsection_sym_subalg}.
Then by Proposition~\ref{prop_reduction}, the computation of
\[
\bigl\langle f(x_2),{\rm e}^{(x|\overline{z})_{\fp^+}}\bigr\rangle_{\lambda,x}, \qquad f(x_2)\in\cP_{\tilde{\bk}}\bigl(\fp^+_2\bigr)
\]
is reduced to the cases that both $\fp^+$ and $\fp^+_2$ are of tube type.
We take a~maximal tripotent $e\in\fp^+_2\subset\fp^+$ of $\fp^+_2$, and let $\fp^{+\prime}:=\fp^+(e)_2$, $\fp^{+\prime}_2:=\fp^+_2(e)_2=\fp^+_2\cap\fp^{+\prime}$,
$\fp^{+\prime}_1:=\fp^+_1\cap\fp^{+\prime}$. If
\[
\bigl\langle f\bigl(x_2'\bigr),{\rm e}^{(x'|\overline{z'})_{\fp^{+\prime}}}\bigr\rangle_{\lambda,x',\fp^{+\prime}}\big|_{z'_1=0}
=C_{\fp^{+\prime},\fp^{+\prime}_2}\bigl(\lambda,\tilde{\bk}\bigr)f\bigl(z_2'\bigr), \qquad f\bigl(x_2'\bigr)\in\cP_{\tilde{\bk}}\bigl(\fp^{+\prime}_2\bigr)
\]
holds for some $C_{\fp^{+\prime},\fp^{+\prime}_2}\bigl(\lambda,\tilde{\bk}\bigr)\in\BC$, then by Proposition~\ref{prop_reduction},
\[
\bigl\langle f(x_2),{\rm e}^{(x|\overline{z})_{\fp^+}}\bigr\rangle_{\lambda,x,\fp^+}\big|_{z_1=0}
=C_{\fp^{+\prime},\fp^{+\prime}_2}\bigl(\lambda,\tilde{\bk}\bigr)f(z_2)
\]
holds for all $f(x_2)=f\bigl(x_2'\bigr)\in\cP_{\tilde{\bk}}\bigl(\fp^{+\prime}_2\bigr)\subset\cP_{\tilde{\bk}}\bigl(\fp^+_2\bigr)$, and by the $K_1$-equivariance,
this also holds for all $f(x_2)\in\cP_{\tilde{\bk}}\bigl(\fp^+_2\bigr)$.
Similarly, let $b(\lambda)\in\BC[\lambda]$. If for all $f\bigl(x_2'\bigr)\in\cP_{\tilde{\bk}}\bigl(\fp^{+\prime}_2\bigr)$,
\[
b(\lambda)\bigl\langle f\bigl(x_2'\bigr),{\rm e}^{(x'|\overline{z'})_{\fp^{+\prime}}}\bigr\rangle_{\lambda,x',\fp^{+\prime}}
\]
is holomorphically continued for all $\lambda\in\BC$, then for all $f(x_2)\in\cP_{\tilde{\bk}}\bigl(\fp^+_2\bigr)$,
\[
b(\lambda)\bigl\langle f(x_2),{\rm e}^{(x|\overline{z})_{\fp^+}}\bigr\rangle_{\lambda,x,\fp^+}
\]
is also holomorphically continued for all $\lambda\in\BC$. That is, these computation for $\bigl(\fp^+,\fp^+_2\bigr)$ is reduced to those for $\bigl(\fp^{+\prime},\fp^{+\prime}_2\bigr)$,
and both $\fp^{+\prime}$, $\fp^{+\prime}_2$ are of tube type, with a~common maximal tripotent $e$.

\subsection{Rodrigues-type formulas}

In this subsection let $\fp^+$ be simple and of tube type. Let $r:=\rank\fp^+$, $n:=\dim\fp^+=r+\frac{d}{2}r(r-1)$ as in~(\ref{str_const}).
In this case $p=\frac{2n}{r}$ holds. We fix a~maximal tripotent $e\in\fp^+$, let $\Omega\subset\fn^+\subset\fp^+$ be the corresponding symmetric cone and the Euclidean real form,
and let $\det_{\fn^+}(x)$ be the determinant polynomial. We normalize the differential operator $\frac{\partial}{\partial z}$ with respect to the bilinear form
$(\cdot|\cdot)_{\fn^+}=(\cdot|Q(\overline{e})\cdot)_{\fp^+}$ on $\fp^+=\fn^{+\BC}$. Also, for $k\in\BZ$, we write
\[ \underline{k}_r:=(\underbrace{k,\dots,k}_r)\in \BZ^r, \]
and for $\lambda\in\BC$, $\bm\in\BZ_{++}^r$, let $(\lambda)_{\bm,d}$ be as in~(\ref{Pochhammer}).
First we recall the following theorem from the author's previous article~\cite{N2}.

\begin{Proposition}\label{prop_Rodrigues}
Let $\Re\lambda>\frac{2n}{r}-1$, $z,a\in\Omega\subset\fn^+\subset\fp^+$.
\begin{enumerate}\itemsep=0pt
\item[$(1)$] For $f\in\cP(\fp^+)$, we have
\begin{gather*}
\bigl\langle f(x),{\rm e}^{(x|\overline{z})_{\fp^+}}\bigr\rangle_{\lambda,x} \\
\qquad\qquad =\det_{\fn^+}(z)^{-\lambda+\frac{n}{r}}\frac{\Gamma_r^d(\lambda)}{(2\pi\sqrt{-1})^n}
\int_{a+\sqrt{-1}\fn^+}{\rm e}^{(z|w)_{\fn^+}}f\bigl(w^\itinv\bigr)\det_{\fn^+}(w)^{-\lambda}\,{\rm d}w,
\end{gather*}
where $\Gamma_r^d(\lambda):=(2\pi)^{dr(r-1)/4}\prod_{j=1}^r\Gamma\bigl(\lambda-\frac{d}{2}(j-1)\bigr)$.
{\rm (See~\cite[Corollary 4.3]{N2}.)}
\item[$(2)$] Let $f_1,f_2\in\cP(\fp^+)$. We take $k\in\BZ_{\ge 0}$ such that
\[
f_1^{\sharp(k)}(x):=\det_{\fn^+}(x)^kf_1\bigl(x^\itinv\bigr)
\]
is a~polynomial. Then we have
\begin{gather*}
\bigl\langle f_1(x)f_2(x),{\rm e}^{(x|\overline{z})_{\fp^+}}\bigr\rangle_{\lambda,x} \\
\qquad\qquad=\frac{1}{(\lambda)_{\underline{k}_r,d}}\det_{\fn^+}(z)^{-\lambda+\frac{n}{r}}f_1^{\sharp(k)}\left(\frac{\partial}{\partial z}\right)
\det_{\fn^+}(z)^{\lambda+k-\frac{n}{r}}
\bigl\langle f_2(x),{\rm e}^{(x|\overline{z})_{\fp^+}}\bigr\rangle_{\lambda+k,x}.
\end{gather*}
{\rm (See~\cite[Theorem 4.1]{N2}.)}
\end{enumerate}
\end{Proposition}

Next we consider an involution $\sigma$ on $\fp^+$, let $\fp^+_1:=(\fp^+)^\sigma$, $\fp^+_2:=(\fp^+)^{-\sigma}$, and assume $\fp^+_2$ is also of tube type.
We take a~common maximal tripotent $e\in\fp^+_2\subset\fp^+$ of $\fp^+_2$ and $\fp^+$, and let $\fn^+\subset\fp^+$, $\fn^+_2\subset\fp^+_2$ be the corresponding Euclidean real forms.
When $\fp^+_2$ is not simple, then we also write $\fp^+_1=:\fp^+_{12}$, $\fp^+_2=:\fp^+_{11}\oplus\fp^+_{22}$, $\fn^+_2=:\fn^+_{11}\oplus\fn^+_{22}$.
Then as in the proof of~\cite[Theorem~4.4]{N2}, when $f_1(x)=\det_{\fn^+_2}(x_2)^k$, we have
\[
\bigl(\det_{\fn^+_2}^k\bigr)^{\sharp(2k/\varepsilon_2)}(x)=\det_{\fn^+}(x)^{2k/\varepsilon_2}\det_{\fn^+_2}\bigl(\Proj_2\bigl(x^\itinv\bigr)\bigr)^k=\det_{\fn^+_2}(x_2)^k,
\]
where $\varepsilon_2$ is as in~(\ref{epsilon_j}), and $\Proj_2\colon\fp^+\to\fp^+_2$ denotes the orthogonal projection.
Similarly, when $\fp^+_2=\fp^+_{11}\oplus\fp^+_{22}$ and $f_1(x)=\det_{\fn^+_{11}}(x_{11})^k$ or $f_1(x)=\det_{\fn^+_{22}}(x_{22})^k$, we have
\begin{align*}
\bigl(\det_{\fn^+_{11}}^k\bigr)^{\sharp(k)}(x)=\det_{\fn^+}(x)^{k}\det_{\fn^+_{11}}\bigl(\Proj_{11}\bigl(x^\itinv\bigr)\bigr)^k=\det_{\fn^+_{22}}(x_{22})^k, \\
\bigl(\det_{\fn^+_{22}}^k\bigr)^{\sharp(k)}(x)=\det_{\fn^+}(x)^{k}\det_{\fn^+_{22}}\bigl(\Proj_{22}\bigl(x^\itinv\bigr)\bigr)^k=\det_{\fn^+_{11}}(x_{11})^k,
\end{align*}
where $\Proj_{jj}\colon\fp^+\to\fp^+_{jj}$ denotes the orthogonal projection.
Thus the following holds. Here we normalize the differential operator $\frac{\partial}{\partial z_2}$ on $\fp^+_2$ with respect to the bilinear form
$(\cdot|\cdot)_{\fn^+_2}=\varepsilon_2^{-1}(\cdot|\cdot)_{\fn^+}$ on $\fp^+_2=\fn^{+\BC}_2$.

\begin{Theorem}\label{thm_key_identity}\quad
\begin{enumerate}\itemsep=0pt
\item[$(1)$] Let $\Re\lambda>\frac{2n}{r}-1$ and let $f\in\cP(\fp^+)$. Then for $z=z_1+z_2\in\Omega\subset\fn^+\subset\fp^+$, we have
\begin{align*}
&\bigl\langle \det_{\fn^+_2}(x_2)^kf(x),{\rm e}^{(x|\overline{z})_{\fp^+}}\bigr\rangle_{\lambda,x} \\
&=\frac{1}{(\lambda)_{\underline{2k/\varepsilon_2}_r,d}}\det_{\fn^+}(z)^{-\lambda+\frac{n}{r}}\det_{\fn^+_2}\!\left(\frac{1}{\varepsilon_2}\frac{\partial}{\partial z_2}\right)^k
\det_{\fn^+}(z)^{\lambda+\frac{2k}{\varepsilon_2}-\frac{n}{r}}\bigl\langle f(x),{\rm e}^{(x|\overline{z})_{\fp^+}}\bigr\rangle_{\lambda+\frac{2k}{\varepsilon_2},x} .
\end{align*}
\item[$(2)$] Suppose $\fp^+_2=\fp^+_{11}\oplus\fp^+_{22}$, let $\Re\lambda>\frac{2n}{r}-1$, $k\in\BZ_{\ge 0}$
and let $f\in\cP(\fp^+)$. Then for $z=z_{11}+z_{12}+z_{22}\in\Omega\subset\fn^+\subset\fp^+$, we have
\begin{align*}
&\bigl\langle \det_{\fn^+_{11}}(x_{11})^kf(x),{\rm e}^{(x|\overline{z})_{\fp^+}}\bigr\rangle_{\lambda,x} \\
&\quad\qquad=\frac{1}{(\lambda)_{\underline{k}_r,d}}\det_{\fn^+}(z)^{-\lambda+\frac{n}{r}}\det_{\fn^+_{22}}\left(\frac{\partial}{\partial z_{22}}\right)^k
\det_{\fn^+}(z)^{\lambda+k-\frac{n}{r}}\bigl\langle f(x),{\rm e}^{(x|\overline{z})_{\fp^+}}\bigr\rangle_{\lambda+k,x}, \\
&\bigl\langle \det_{\fn^+_{22}}(x_{22})^kf(x),{\rm e}^{(x|\overline{z})_{\fp^+}}\bigr\rangle_{\lambda,x} \\
&\quad\qquad=\frac{1}{(\lambda)_{\underline{k}_r,d}}\det_{\fn^+}(z)^{-\lambda+\frac{n}{r}}\det_{\fn^+_{11}}\left(\frac{\partial}{\partial z_{11}}\right)^k
\det_{\fn^+}(z)^{\lambda+k-\frac{n}{r}}\bigl\langle f(x),{\rm e}^{(x|\overline{z})_{\fp^+}}\bigr\rangle_{\lambda+k,x}.
\end{align*}
\end{enumerate}
\end{Theorem}

Similarly, we consider the outer tensor product $\cH_\lambda(D)\hboxtimes\cH_\mu(D)$ of Hilbert spaces, and let~$\langle\cdot,\cdot\rangle_{\lambda\otimes\mu}$ denote its inner product.
Then the following holds.

\begin{Theorem}\label{thm_key_tensor}
Let $\Re\lambda,\Re\mu>\frac{2n}{r}-1$ and let $f(x,y)\in\cP\bigl(\fp^+\oplus\fp^+\bigr)$. Then for $z,-w\in\Omega\subset\fn^+\subset\fp^+$, we have
\begin{align*}
&\bigl\langle f(x,y)\det_{\fn^+}(x-y)^k, {\rm e}^{(x|\overline{z})_{\fp^+}+(y|\overline{w})_{\fp^+}}\bigr\rangle_{\lambda\otimes\mu,(x,y)} \\
&=\frac{1}{(\lambda)_{\underline{k}_r,d}(\mu)_{\underline{k}_r,d}}\det_{\fn^+}(z)^{-\lambda+\frac{n}{r}}\det_{\fn^+}(-w)^{-\mu+\frac{n}{r}}
\det_{\fn^+}\left(\frac{\partial}{\partial z}-\frac{\partial}{\partial w}\right)^k \\
&\eqspace{}\times\det_{\fn^+}(z)^{\lambda+k-\frac{n}{r}}\det_{\fn^+}(-w)^{\mu+k-\frac{n}{r}}
\bigl\langle f(x,y), {\rm e}^{(x|\overline{z})_{\fp^+}+(y|\overline{z})_{\fp^+}}\bigr\rangle_{(\lambda+k)\otimes(\mu+k),(x,y)}, \\
&\bigl\langle f(x,y)\det_{\fn^+}(x-y)^k, {\rm e}^{(x-y|\overline{z})_{\fp^+}}\bigr\rangle_{\lambda\otimes\mu,(x,y)} \\
&=\frac{\det_{\fn^+}(z)^{-\lambda-\mu+\frac{2n}{r}}}{(\lambda)_{\underline{k}_r,d}(\mu)_{\underline{k}_r,d}}\det_{\fn^+}\!\left(\frac{\partial}{\partial z}\right)\!\!{\vphantom{\biggr)}}^k
\det_{\fn^+}(z)^{\lambda+\mu+2k-\frac{2n}{r}}\bigl\langle f(x,y), {\rm e}^{(x-y|\overline{z})_{\fp^+}}\bigr\rangle_{(\lambda+k)\otimes(\mu+k),(x,y)}.
\end{align*}
\end{Theorem}

\begin{proof}
By Proposition~\ref{prop_Rodrigues}\,(1), for $a,z,-w\in\Omega$, we have
\begin{gather*}
\bigl\langle f(x,y)\det_{\fn^+}(x-y)^k, {\rm e}^{(x|\overline{z})_{\fp^+}+(y|\overline{w})_{\fp^+}}\bigr\rangle_{\lambda\otimes\mu,(x,y)} \\
\qquad{}=\bigl\langle f(x,-y)\det_{\fn^+}(x+y)^k, {\rm e}^{(x|\overline{z})_{\fp^+}-(y|\overline{w})_{\fp^+}}\bigr\rangle_{\lambda\otimes\mu,(x,y)} \\
\qquad{}=\det_{\fn^+}(z)^{-\lambda+\frac{n}{r}}\det_{\fn^+}(-w)^{-\mu+\frac{n}{r}}\frac{\Gamma_r^d(\lambda)\Gamma_r^d(\mu)}{(2\pi\sqrt{-1})^{2n}}\iint_{(a+\sqrt{-1}\fn^+)^2} {\rm e}^{(x|z)_{\fn^+}-(y|w)_{\fn^+}} \\
\qquad\quad {}\times f\bigl(x^\itinv,-y^\itinv\bigr)\det_{\fn^+}\bigl(x^\itinv+y^\itinv\bigr)^k\det_{\fn^+}(x)^{-\lambda}\det_{\fn^+}(y)^{-\mu}\,{\rm d}x{\rm d}y,
\end{gather*}
and since
\[
\det_{\fn^+}\bigl(x^\itinv+y^\itinv\bigr)=\det_{\fn^+}(x+y)\det_{\fn^+}(x)^{-1}\det_{\fn^+}(y)^{-1}
\]
holds (see~\cite[Lemma X.4.4\,(i)]{FK}), by putting $w=-z$, we get
\begin{gather*}
 \bigl\langle f(x,y)\det_{\fn^+}(x-y)^k, {\rm e}^{(x-y|\overline{z})_{\fp^+}}\bigr\rangle_{\lambda\otimes\mu,(x,y)} \\
 =\det_{\fn^+}(z)^{-\lambda-\mu+\frac{2n}{r}}\frac{\Gamma_r^d(\lambda)\Gamma_r^d(\mu)}{(2\pi\sqrt{-1})^{2n}}\iint_{(a+\sqrt{-1}\fn^+)^2} {\rm e}^{(x+y|z)_{\fn^+}} \\
 \quad{}\times f\bigl(x^\itinv,-y^\itinv\bigr)\det_{\fn^+}(x+y)^k\det_{\fn^+}(x)^{-\lambda-k}\det_{\fn^+}(y)^{-\mu-k}\,{\rm d}x{\rm d}y \\
 =\det_{\fn^+}(z)^{-\lambda-\mu+\frac{2n}{r}}\det_{\fn^+}\left(\frac{\partial}{\partial z}\right)^k\frac{\Gamma_r^d(\lambda)\Gamma_r^d(\mu)}{(2\pi\sqrt{-1})^{2n}}
\iint_{(a+\sqrt{-1}\fn^+)^2} {\rm e}^{(x+y|z)_{\fn^+}} \\
 \quad{}\times f\bigl(x^\itinv,-y^\itinv\bigr)\det_{\fn^+}(x)^{-\lambda-k}\det_{\fn^+}(y)^{-\mu-k}\,{\rm d}x{\rm d}y \\
 =\det_{\fn^+}(z)^{-\lambda-\mu+\frac{2n}{r}}\det_{\fn^+}\left(\frac{\partial}{\partial z}\right)^k\frac{\Gamma_r^d(\lambda)\Gamma_r^d(\mu)}{\Gamma_r^d(\lambda+k)\Gamma_r^d(\mu+k)} \\
 \quad{}\times\det_{\fn^+}(z)^{(\lambda+k)+(\mu+k)-\frac{2n}{r}}\bigl\langle f(x,-y), {\rm e}^{(x+y|\overline{z})_{\fp^+}}\bigr\rangle_{(\lambda+k)\otimes(\mu+k),(x,y)} \\
 =\frac{\det_{\fn^+}(z)^{-\lambda-\mu+\frac{2n}{r}}}{(\lambda)_{\underline{k}_r,d}(\mu)_{\underline{k}_r,d}}\det_{\fn^+}\!\!\left(\frac{\partial}{\partial z}\right)^k
\det_{\fn^+}(z)^{\lambda+\mu+2k-\frac{2n}{r}}\bigl\langle f(x,y), {\rm e}^{(x-y|\overline{z})_{\fp^+}}\bigr\rangle_{(\lambda+k)\otimes(\mu+k),(x,y)} .
\end{gather*}
This proves the 2nd formula. The 1st formula is also proved similarly.
\end{proof}

These theorems give analogues of the Rodrigues formulas for Jacobi polynomials.
Especially, when $\fp^+_2$ is simple, let $r_2:=\rank\fp^+_2$, fix a~Jordan frame $\{e_1,\dots,e_{r_2}\}\subset\fp^+_2$,
and for $\bk\in\BZ_{++}^{r_2}$, let $\Delta^{\fn^+_2}_\bk(x_2)\in\cP_\bk\bigl(\fp^+_2\bigr)$ be as in~(\ref{principal_minor}).
Then by putting $k=k_{r_2}$, $f(x)=\Delta^{\fn^+_2}_{\bk-\underline{k_{r_2}}_{r_2}}(x_2)$ in Theorem~\ref{thm_key_identity}\,(1), we have
\begin{align}
&\bigl\langle \Delta^{\fn^+_2}_\bk(x_2),{\rm e}^{(x|\overline{z})_{\fp^+}}\bigr\rangle_{\lambda,x}
=\Bigl\langle \det_{\fn^+_2}(x_2)^{k_{r_2}}\Delta^{\fn^+_2}_{\bk-\underline{k_{r_2}}_{r_2}}(x_2),{\rm e}^{(x|\overline{z})_{\fp^+}}\Bigr\rangle_{\lambda,x} \label{key_simple}\\
&=\frac{\det_{\fn^+}(z)^{-\lambda+\frac{n}{r}}}{(\lambda)_{\underline{2k_{r_2}/\varepsilon_2}_r,d}}
\det_{\fn^+_2}\left(\frac{1}{\varepsilon_2}\frac{\partial}{\partial z_2}\right)^{k_{r_2}}\det_{\fn^+}(z)^{\lambda+\frac{2k_{r_2}}{\varepsilon_2}-\frac{n}{r}}
\Bigl\langle \Delta^{\fn^+_2}_{\bk-\underline{k_{r_2}}_{r_2}}(x_2),{\rm e}^{(x|\overline{z})_{\fp^+}}\Bigr\rangle_{\lambda+\frac{2k_{r_2}}{\varepsilon_2},x}. \nonumber
\end{align}
Then $\Delta^{\fn^+_2}_{\bk-\underline{k_{r_2}}_{r_2}}(x_2)$ depends only on the smaller algebra $\fp^+(e')_2$, with $e'=e_1+\cdots+e_{r_2-1}$,
and by Proposition~\ref{prop_reduction}, computation of the inner product on $\fp^+$ is reduced to that on $\fp^+(e')_2$.
Similarly, when $\fp^+_2=\fp^+_{11}\oplus\fp^+_{22}$ is not simple, let $r':=\rank\fp^+_{11}$, $r'':=\rank\fp^+_{22}$, fix Jordan frames in~$\fp^+_{11}$,~$\fp^+_{22}$,
and for $\bk\in\BZ_{++}^{r'}$, $\bl\in\BZ_{++}^{r''}$, let $\Delta^{\fn^+_{11}}_\bk(x_{11})\in\cP_\bk\bigl(\fp^+_{11}\bigr)$, $\Delta^{\fn^+_{22}}_\bl(x_{22})\in\cP_\bl\bigl(\fp^+_{22}\bigr)$ be as in
(\ref{principal_minor}). Then by putting $k=k_{r'}$, $f(x)=\Delta^{\fn^+_{11}}_{\bk-\underline{k_{r'}}_{r'}}(x_{11})\Delta^{\fn^+_{22}}_\bl(x_{22})$
or $k=l_{r''}$, $f(x)=\Delta^{\fn^+_{11}}_\bk(x_{11})\Delta^{\fn^+_{22}}_{\bl-\underline{l_{r''}}_{r''}}(x_{22})$ in Theorem~\ref{thm_key_identity}\,(2), we have
\begin{align}
&\Bigl\langle \Delta^{\fn^+_{11}}_\bk(x_{11})\Delta^{\fn^+_{22}}_\bl(x_{22}),{\rm e}^{(x|\overline{z})_{\fp^+}}\Bigr\rangle_{\lambda,x} \notag\\
&\qquad{} =\frac{\det_{\fn^+} (z)^{-\lambda+\frac{n}{r}}}{(\lambda)_{\underline{k_{r'}}_r,d}}
\det_{\fn^+_{22}} \left(\frac{\partial}{\partial z_{22}}\right) ^{k_{r'}}\nonumber\\
& \qquad\quad \ {}\times \det_{\fn^+} (z)^{\lambda+k_{r'}-\frac{n}{r}}
\Bigl\langle \Delta^{\fn^+_{11}}_{\bk-\underline{k_{r'}}_{r'}}\!(x_{11})\Delta^{\fn^+_{22}}_\bl (x_{22}),{\rm e}^{(x|\overline{z})_{\fp^+}}\Bigr\rangle_{\lambda+k_{r'},x}
\label{key_nonsimple1}\\
&\qquad{} =\frac{\det_{\fn^+} (z)^{-\lambda+\frac{n}{r}}}{(\lambda)_{\underline{l_{r''}}_r,d}}
\det_{\fn^+_{11}} \left(\frac{\partial}{\partial z_{11}}\right)^{l_{r''}}\nonumber\\
&\qquad\quad \ {}\times \det_{\fn^+} (z)^{\lambda+l_{r''}-\frac{n}{r}}
\Bigl\langle \Delta^{\fn^+_{11}}_\bk (x_{11})\Delta^{\fn^+_{22}}_{\bl-\underline{l_{r''}}_{r''}} (x_{22}),{\rm e}^{(x|\overline{z})_{\fp^+}}\Bigr\rangle_{\lambda+l_{r''},x},
\label{key_nonsimple2}
\end{align}
and again by Proposition~\ref{prop_reduction}, computation of the inner product on $\fp^+$ is reduced to that on a~smaller algebra of rank $r'+r''-1$.
Similarly, for $\fp^+\oplus\fp^+$ case, for $\bk\in\BZ_{++}^r$, by putting $k=k_r$, $f(x,y)=\Delta^{\fn^+}_{\bk-\underline{k_r}_r}(x-y)$ in Theorem~\ref{thm_key_tensor}, we have
\begin{align}
&\bigl\langle \Delta^{\fn^+}_\bk(x-y), {\rm e}^{(x-y|\overline{z})_{\fp^+}}\bigr\rangle_{\lambda\otimes\mu,(x,y)} \label{key_tensor}\\
&=\frac{\det_{\fn^+}\hspace{-1pt}(z)^{-\lambda-\mu+\frac{2n}{r}}}{(\lambda)_{\underline{k_r}_r,d}(\mu)_{\underline{k_r}_r,d}}
\det_{\fn^+}\!\!\left(\hspace{-1pt}\frac{\partial}{\partial z}\hspace{-1pt}\right)\!\!{\vphantom{\biggr|}}^{k_r}
\!\det_{\fn^+}\hspace{-1pt}(z)^{\lambda+\mu+2k_r-\frac{2n}{r}}
\hspace{-1pt}\Big\langle\hspace{-1pt} \Delta^{\fn^+}_{\bk-\underline{k_r}_r}\!(x\hspace{-1pt}-\hspace{-1pt}y), {\rm e}^{(x-y|\overline{z})_{\fp^+}}
\hspace{-1pt}\Big\rangle_{\substack{(\lambda+k_r)\otimes\hspace{14pt} \\ (\mu+k_r),(x,y)}}\!\!,\nonumber
\end{align}
and again by Proposition~\ref{prop_reduction}, computation of the inner product on $\fp^+$ is reduced to that on a~smaller algebra of rank $r-1$.
For $\bk=\underline{k}_r$ case, see also~\cite{BCK, Cl}.

In addition, by Proposition~\ref{prop_det_factorize}\,(2) and Theorem~\ref{thm_key_identity}, we easily get the following, which is an analogue of~\cite[Corollary 12.8\,(3)]{KS1}.
We note that $\det_{\fn^+_2}(x_2)=\det_{\fn^+_{11}}(x_{11})\det_{\fn^+_{22}}(x_{22})$ holds, when $\fp^+_2$ is not simple.
\begin{Theorem}\label{thm_factorize}
We fix $\tilde{\bk}\in\BZ_{++}^{r_2}$ \big(when $\fp^+_2$ is simple\big) or $\tilde{\bk}\in\BZ_{++}^{r'}\times\BZ_{++}^{r''}$ \big(when $\fp^+_2$ is not simple\big), and $a\in\BZ_{>0}$.
Then $(\lambda)_{\underline{2a}_r,d}\bigl\langle \det_{\fn^+_2}(x_2)^{\varepsilon_2a}f(x),{\rm e}^{(x|\overline{z})_{\fp^+}}\bigr\rangle_{\lambda,x}$ is
holomorphically continued for $\Re\lambda>\frac{n}{r}-2a-1$, and there exist $C_{\tilde{\bk},a}^{\fn^+,\fn^+_2}\in\BC$ such that
\begin{align*}
&(\lambda)_{\underline{2a}_r,d}\bigl\langle \det_{\fn^+_2}(x_2)^{\varepsilon_2a}f(x_2),{\rm e}^{(x|\overline{z})_{\fp^+}}\bigr\rangle_{\lambda,x}\big|_{\lambda=\frac{n}{r}-a}
=C_{\tilde{\bk},a}^{\fn^+,\fn^+_2}\det_{\fn^+}(z)^a\bigl\langle f(x_2),{\rm e}^{(x|\overline{z})_{\fp^+}}\bigr\rangle_{\frac{n}{r}+a,x}
\end{align*}
holds for all $f(x_2)\in\cP_{\tilde{\bk}}\bigl(\fp^+_2\bigr)$.
\end{Theorem}

\begin{proof}
Since $\langle\cdot,\cdot\rangle_\lambda$ is holomorphically continued for $\Re\lambda>\frac{n}{r}-1=\frac{d}{2}(r-1)$, by Theorem~\ref{thm_key_identity},
\begin{align*}
&\det_{\fn^+}(z)^{\lambda-\frac{n}{r}}(\lambda)_{\underline{2a}_r,d}\bigl\langle \det_{\fn^+_2}(x_2)^{\varepsilon_2a}f(x_2),{\rm e}^{(x|\overline{z})_{\fp^+}}\bigr\rangle_{\lambda,x} \\
&\qquad\qquad{}=\det_{\fn^+_2}\left(\frac{1}{\varepsilon_2}\frac{\partial}{\partial z_2}\right)^{\varepsilon_2a}
\det_{\fn^+}(z)^{\lambda+2a-\frac{n}{r}}\bigl\langle f(x_2),{\rm e}^{(x|\overline{z})_{\fp^+}}\bigr\rangle_{\lambda+2a,x}
\end{align*}
is holomorphically continued for $\Re\lambda+2a>\frac{n}{r}-1$, and when $\lambda=\frac{n}{r}-a$,
\begin{align*}
&\det_{\fn^+}(z)^{-a}(\lambda)_{\underline{2a}_r,d}\bigl\langle \det_{\fn^+_2}(x_2)^{\varepsilon_2a}f(x_2),{\rm e}^{(x|\overline{z})_{\fp^+}}\bigr\rangle_{\lambda,x}\big|_{\lambda=\frac{n}{r}-a} \\
&\qquad\qquad=\det_{\fn^+_2}\left(\frac{1}{\varepsilon_2}\frac{\partial}{\partial z_2}\right)^{\varepsilon_2a}
\det_{\fn^+}(z)^{a}\bigl\langle f(x_2),{\rm e}^{(x|\overline{z})_{\fp^+}}\bigr\rangle_{\frac{n}{r}+a,x}
\end{align*}
becomes a~polynomial. Hence the theorem follows from Proposition~\ref{prop_det_factorize}\,(2).
\end{proof}

\subsection[Tensor product of finite-dimensional representations of U(s)]{Tensor product of finite-dimensional representations of $\boldsymbol{{\rm U}(s)}$}

In this subsection we treat some results on finite-dimensional representations of ${\rm U}(s)$ used later, which easily follows from the Littlewood--Richardson rule.
As in Section~\ref{subsection_classification}, for $\bk\in\BZ_{++}^s$, let $V_\bk^{(s)}$ be the irreducible representation of ${\rm U}(s)$ with the highest weight $\bk$,
and we write $\bk^\vee:=(k_s,\dots,k_1)$.
\begin{Lemma}\label{lem_LR}
Let $\bk,\bl,\bn\in\BZ_{++}^s$. If $\Hom_{{\rm U}(s)}\bigl(V_\bn^{(s)},V_\bk^{(s)}\otimes V_\bl^{(s)}\bigr)\ne \{0\}$, then we have
\begin{enumerate}\itemsep=0pt
\item[$(1)$] $n_{s-i-j}\ge k_{s-i}+l_{s-j}$ for $0\le i,j<s$, $i+j<s$.
\item[$(2)$] $n_{i+j-1}\le k_i+l_j$ for $1\le i,j\le s$, $i+j\le s+1$.
\end{enumerate}
\end{Lemma}

\begin{proof}
(1) Suppose $\Hom_{{\rm U}(s)}\bigl(V_\bn^{(s)},V_\bk^{(s)}\otimes V_\bl^{(s)}\bigr)\ne \{0\}$. Then by the Littlewood--Richardson rule,
there exists a~skew semistandard tableau $Y$ of shape $\bn/\bk$, weight $\bl$ with the lattice word condition.
Let $Y(a,b)$ $(1\le a\le s,\, k_a+1\le b\le n_a)$ be the $(a,b)$-entry of $Y$, so that ${\#\{(a,b)\mid Y(a,b)=j\}=l_j}$ holds.
Then by the lattice word condition, we have
\begin{align*}
n_{s-i-j}-k_{s-i}&\ge \#\{(a,b)\mid Y(a,b)=s-i-j,\, a\le s-i\} \\
&\ge \#\{(a,b)\mid Y(a,b)=s-i-j+1,\, a\le s-i+1\} \\
&\ge \#\{(a,b)\mid Y(a,b)=s-i-j+2,\, a\le s-i+2\} \ge\cdots \\
&\ge \#\{(a,b)\mid Y(a,b)=s-j,\, a\le s\}=l_{s-j}.
\end{align*}

(2) Since $\Hom_{{\rm U}(s)}\bigl(V_\bn^{(s)},V_\bk^{(s)}\otimes V_\bl^{(s)}\bigr)\simeq
\Hom_{{\rm U}(s)}\bigl(V_{\underline{k_1+l_1}_s-\bn^\vee}^{(s)},V_{\underline{k_1}_s-\bk^\vee}^{(s)}\otimes V_{\underline{l_1}_s-\bl^\vee}^{(s)}\bigr)$ holds,
by (1) we have $k_1+l_1-n_{i+j+1}\ge (k_1-k_{i+1})+(l_1-l_{j+1})$, that is, $n_{i+j+1}\le k_{i+1}+l_{j+1}$ holds.
\end{proof}

Using this lemma, we give necessary conditions on $\bigl(\tilde{\bk},\bm\bigr)$ for existence of non-zero $K_1$-homomorphisms
from $\cP_{\tilde{\bk}}\bigl(\fp^+_2\bigr)$ to $\cP_\bm(\fp^+)$, when $\bigl(\fp^+,\fp^+_2\bigr)$ are of tube type and $K_1$ is of type $A$.
First we consider the cases $\bigl(\fp^+,\fp^+_2\bigr)=({\rm M}(s,\BC),\Sym(s,\BC))$, $({\rm M}(2s,\BC),\Alt(2s,\BC))$.
Then for $(\bm,\bn)\in\BZ_{++}^s\times\BZ_{++}^s$ or $\BZ_{++}^{2s}\times\BZ_{++}^s$, by~(\ref{explicit_Ktype}), we have
\begin{gather*}
\Hom_{{\rm U}(s)}(\cP_\bn(\Sym(s,\BC)),\cP_\bm({\rm M}(s,\BC)))\simeq\Hom_{{\rm U}(s)}\bigl(V_{2\bn}^{(s)},V_\bm^{(s)}\otimes V_\bm^{(s)}\bigr), \\
\Hom_{{\rm U}(2s)}(\cP_\bn(\Alt(2s,\BC)),\cP_\bm({\rm M}(2s,\BC)))\simeq\Hom_{{\rm U}(2s)}\bigl(V_{\bn^2}^{(2s)},V_\bm^{(2s)}\otimes V_\bm^{(2s)}\bigr).
\end{gather*}
Hence by Lemma~\ref{lem_LR}\,(1),
\begin{gather}
\Hom_{{\rm U}(s)}(\cP_\bn(\Sym(s,\BC)),\cP_\bm({\rm M}(s,\BC)))\ne\{0\} \nonumber \\
\qquad\text{implies} \ m_{s-i}+m_{s-j}\le 2n_{s-i-j}, \label{LR_M-Sym} \\
\Hom_{{\rm U}(2s)}(\cP_\bn(\Alt(2s,\BC)),\cP_\bm({\rm M}(2s,\BC)))\ne\{0\} \nonumber\\
\qquad\text{implies} \  m_{2s-i}+m_{2s-j}\le n_{\lceil (2s-i-j)/2\rceil}. \label{LR_M-Alt}
\end{gather}
Next we consider the cases $\bigl(\fp^+,\fp^+_2\bigr)=(\Sym(2s,\BC),{\rm M}(s,\BC))$, $(\Alt(2s,\BC),{\rm M}(s,\BC))$.
In general, for $\bk,\bl,\bm\in\BZ_{++}^s$, if $s=s'+s''$, $k_{s'+1}=\cdots=k_s=0$ and $l_{s''+1}=\cdots=l_s=0$, then by~\cite[Theorem~9.2.3]{GW}, we have
\[ \dim\Hom_{{\rm U}(s)}\bigl(V_\bm^{(s)},V_\bk^{(s)}\otimes V_\bl^{(s)}\bigr)=\dim\Hom_{{\rm U}(s')\times {\rm U}(s'')}\bigl(V_\bk^{(s')}\boxtimes V_\bl^{(s'')},V_\bm^{(s)}\bigr). \]
Especially, for $(\bm,\bk)\in \BZ_{++}^{2s}\times\BZ_{++}^s$ or $\BZ_{++}^s\times\BZ_{++}^s$, since
\begin{gather*}
\Hom_{{\rm U}(s)\times {\rm U}(s)}(\cP_\bk({\rm M}(s,\BC)),\cP_\bm(\Sym(2s,\BC)))\simeq\Hom_{{\rm U}(s)\times {\rm U}(s)}\bigl(V_\bk^{(s)}\boxtimes V_\bk^{(s)},V_{2\bm}^{(2s)}\bigr), \\
\Hom_{{\rm U}(s)\times {\rm U}(s)}(\cP_\bk({\rm M}(s,\BC)),\cP_\bm(\Alt(2s,\BC)))\simeq\Hom_{{\rm U}(s)\times {\rm U}(s)}\bigl(V_\bk^{(s)}\boxtimes V_\bk^{(s)},V_{\bm^2}^{(2s)}\bigr)
\end{gather*}
hold by~(\ref{explicit_Ktype}), by Lemma~\ref{lem_LR}\,(2),
\begin{gather}
\Hom_{{\rm U}(s)\times {\rm U}(s)}(\cP_\bk({\rm M}(s,\BC)),\cP_\bm(\Sym(2s,\BC)))\ne\{0\}\quad \ \, \text{implies} \  2m_{i+j-1}\le k_i+k_j,  \label{LR_Sym-M}\\
\Hom_{{\rm U}(s)\times {\rm U}(s)}(\cP_\bk({\rm M}(s,\BC)),\cP_\bm(\Alt(2s,\BC)))\ne\{0\}\quad \quad   \text{implies} \  m_{\lfloor(i+j)/2\rfloor}\le k_i+k_j   \label{LR_Alt-M}
\end{gather}
for $i+j\le 2s+1$, where we set $k_{s+1}=\cdots=k_{2s}=0$. Similarly, for the cases
\begin{align*}
\bigl(\fp^+,\fp^+_{11},\fp^+_{22},K_1\bigr)&= \begin{cases}
(\Sym(r,\BC),\Sym(r',\BC),\Sym(r'',\BC),{\rm U}(r')\times {\rm U}(r'')), \\
({\rm M}(r,\BC),{\rm M}(r',\BC),{\rm M}(r'',\BC),{\rm U}(r')\times {\rm U}(r')\times {\rm U}(r'')\times {\rm U}(r'')), \\
(\Alt(2r,\BC),\Alt(2r',\BC),\Alt(2r'',\BC),{\rm U}(2r')\times {\rm U}(2r'')), \end{cases}
\end{align*}
by Lemma~\ref{lem_LR}\,(2), for $\bk\in\BZ_{++}^{r'}$, $\bl\in\BZ_{++}^{r''}$, $\bm\in\BZ_{++}^r$,
\begin{gather}
\Hom_{K_1}\bigl(\cP_\bk\bigl(\fp^+_{11}\bigr)\boxtimes\cP_\bl\bigl(\fp^+_{22}\bigr),\cP_\bm(\fp^+)\bigr) \ne\{0\} \qquad \text{implies} \quad m_{i+j-1}\le k_i+l_j. \label{LR_nonsimple}
\end{gather}

\section{Easy case}\label{section_easycase}

In this section we treat $\fp^+=(\fp^+)^\sigma\oplus(\fp^+)^{-\sigma}=\fp^+_1\oplus\fp^+_2$ such that $\fp^+_2(e)_2=\fp^+(e)_2$ holds for some (or equivalently any)
maximal tripotent $e\in\fp^+_2$, where $\fp^+(e)_2\subset\fp^+$, $\fp^+_2(e)_2\subset\fp^+_2$ are as in~(\ref{Peirce}).
This section contains some overlap with~\cite[Sections 6.1 and~7]{N2}, but for the sake of completeness, we give the results for this case.
Such $\bigl(\fp^+,\fp^+_1,\fp^+_2\bigr)$ are one of
\begin{gather*}
\bigl(\fp^+,\fp^+_1,\fp^+_2\bigr) = \begin{cases}
({\rm M}(q,s;\BC),{\rm M}(q,s';\BC),{\rm M}(q,s'';\BC)) & (\text{Case }1), \\
(\Alt(s,\BC),\Alt(s-1,\BC),{\rm M}(s-1,1;\BC)) & (\text{Case }2), \\
(\Alt(s,\BC),{\rm M}(s-1,1;\BC),\Alt(s-1,\BC)) & (\text{Case }3), \\
\bigl(\BC^{2s},\bigl(\bigl(1,\sqrt{-1}\bigr)\BC\bigr)^s,\bigl(\bigl(1,-\sqrt{-1}\bigr)\BC\bigr)^s\bigr) \\
\quad{}\simeq\bigl(\BC^{2s},{\rm M}(1,s;\BC),{\rm M}(1,s;\BC)\bigr) & (\text{Case }4), \\
\bigl({\rm M}(1,2;\BO)^\BC,\BO^\BC,\BO^\BC\bigr)\simeq\bigl({\rm M}(1,2;\BO)^\BC,\BC^8,\BC^8\bigr) & (\text{Case }5) \end{cases}
\end{gather*}
($s'+s''=s$ for Case 1). Then the corresponding symmetric pairs are
\[ (G,G_1)= \begin{cases}
({\rm SU}(q,s),{\rm S}({\rm U}(q,s')\times {\rm U}(s''))) & (\text{Case 1}), \\
({\rm SO}^*(2s),{\rm SO}^*(2(s-1))\times {\rm SO}(2)) & (\text{Case 2}), \\
({\rm SO}^*(2s),{\rm U}(s-1,1)) & (\text{Case 3}), \\
({\rm SO}_0(2,2s),{\rm U}(1,s)) & (\text{Case 4}), \\
(E_{6(-14)},{\rm U}(1)\times {\rm Spin}_0(2,8)) & (\text{Case 5}). \end{cases} \]
Let $\rank\fp^+=:r$, $\rank\fp^+_2=:r_2$, and let $d$ be the number defined in~(\ref{str_const}), so that
\[ (r,r_2,d)= \begin{cases}
(\min\{q,s\},\min\{q,s''\},2) & (\text{Case }1), \\
(\lfloor s/2\rfloor,1,4) & (\text{Case }2), \\
(\lfloor s/2\rfloor,\lfloor (s-1)/2\rfloor,4) & (\text{Case }3), \\
(2,1,2s-2) & (\text{Case }4), \\
(2,2,6) & (\text{Case }5). \end{cases} \]
For these cases, for $\bk\in\BZ_{++}^{r_2}$, since $\cP_\bk(\fp^+)$ and $\cP_\bk\bigl(\fp^+_2\bigr)$ are generated by $\cP_\bk\bigl(\fp^+\hspace{-1pt}(e)_2\bigr)=\cP_\bk\bigl(\fp^+_2\hspace{-1pt}(e)_2\bigr)$
as $K^\BC$- and $K^\BC_1$-modules respectively, we have $\cP_\bk\bigl(\fp^+_2\bigr)\subset\cP_\bk(\fp^+)$, and hence by Corollary~\ref{cor_FK} and Theorem~\ref{thm_FK}, the following holds.
Here $(\lambda)_{\bk,d}$ is as in~(\ref{Pochhammer}).
\begin{Theorem}
Let $\Re\lambda>p-1$, $\bk\in\BZ_{++}^{r_2}$, and let $f(x_2)\in\cP_\bk\bigl(\fp^+_2\bigr)$. Then we have
\[
\bigl\langle f(x_2),{\rm e}^{(x|\overline{z})_{\fp^+}}\bigr\rangle_{\lambda,x}=\frac{1}{(\lambda)_{\bk,d}}f(z_2), \qquad
\Vert f(x_2)\Vert_{\lambda,x}^2=\frac{1}{(\lambda)_{\bk,d}}\Vert f(z_2)\Vert_{F,z}^2.
\]
\end{Theorem}

Next we consider the decomposition of the holomorphic discrete series representation of scalar type $\cH_\lambda(D)$ under the subgroup $\widetilde{G}_1\subset\widetilde{G}$.
By Theorem~\ref{thm_HKKS}, we have
\[
\cH_\lambda(D)|_{\widetilde{G}_1}\simeq\hsum_{\bk\in\BZ_{++}^{r_2}}\cH_{\lambda}\bigl(D_1,\cP_\bk\bigl(\fp^+_2\bigr)\bigr),
\]
where
\begin{align*}
&\cH_{\lambda}(D_1,\cP_\bk\bigl(\fp^+_2\bigr))
\simeq\begin{cases}
\cH_{\lambda_1+\lambda_2}\bigl(D_{{\rm U}(q,s')},V_\bk^{(q)\vee}\boxtimes\BC\bigr)\boxtimes V_{(\lambda_2,\dots,\lambda_2)+\bk}^{(s'')} & (\text{Case }1), \vspace{1mm}\\
\cH_{\lambda}\bigl(D_{{\rm SO}^*(2(s-1))},V_{(k,0,\dots,0)}^{(s-1)\vee}\bigr)\boxtimes \chi_{{\rm SO}(2)}^{-\lambda-2k} & (\text{Case }2), \vspace{1mm}\\
\cH_{\frac{\lambda}{2}+\frac{\lambda}{2}}\bigl(D_{{\rm U}(s-1,1)},V_{\bk^2}^{(s-1)\vee}\boxtimes \BC\bigr) & (\text{Case }3), \vspace{1mm}\\
\cH_{(\lambda+k)+0}\bigl(D_{{\rm U}(1,s)},\BC\boxtimes V_{(0,\dots,0,-k)}^{(s)}\bigr) & (\text{Case }4), \vspace{1mm}\\
\cH_{\lambda+\frac{|\bk|}{2}}\bigl(D_{{\rm Spin}_0(2,8)},V_{\!\!\left(\frac{k_1+k_2}{2},\frac{k_1-k_2}{2},\frac{k_1-k_2}{2},-\frac{k_1-k_2}{2}\right)}^{[8]\vee}\bigr)
\!\boxtimes\!\chi_{{\rm U}(1)}^{-\lambda-\frac{3}{2}|\bk|}\!\! & (\text{Case }5),
\end{cases}
\end{align*}
where we set $\lambda=\lambda_1+\lambda_2$ for Case 1.
For each $\bk\in\BZ_{++}^{r_2}$, let $V_\bk$ be an abstract $K_1$-module isomorphic to $\cP_\bk\bigl(\fp^+_2\bigr)$,
and take non-zero $K_1$-invariant vector-valued polynomials ${\rK_\bk^\vee(z_2)\in\cP(\fp^-_2,V_\bk)^{K_1}}$ and $\rK_\bk(x_2)\in\cP\bigl(\fp^+_2,\Hom_\BC(V_\bk,\BC)\bigr)^{K_1}$.
Such polynomials are unique up to constant multiple. Then by~\cite[Section 5.1]{N},~\cite[Section 5]{KP2}, the normal derivative and the multiplication operators
\begin{alignat*}{3}
&\cF_{\bk}^\downarrow\colon \ \cO_\lambda(D)|_{\widetilde{G}_1}\longrightarrow \cO_{\lambda}(D_1,V_\bk),
\qquad && (\cF_{\bk}^\downarrow f)(x_1):=\rK_\bk^\vee\left(\frac{\partial}{\partial x_2}\right)f(x)\biggr|_{x_2=0},& \\
&\cF_{\bk}^\uparrow\colon \ \cO_{\lambda}(D_1,V_\bk)\longrightarrow \cO_\lambda(D)|_{\widetilde{G}_1},
\qquad && (\cF_{\bk}^\uparrow g)(x):=\rK_\bk(x_2)g(x_1)&
\end{alignat*}
give the $\widetilde{G}_1$-intertwining operators (symmetry breaking operator, holographic operator) for all~$\lambda\in\BC$.
Let $\Vert\cdot\Vert_{\lambda,\bk}$ be the $\widetilde{G}_1$-invariant norm on $\cH_{\lambda}(D_1,V_\bk)$
normalized such that $\Vert v\Vert_{\lambda,\bk}=|v|_{V_\bk}$ holds for all constant functions $v\in V_\bk$.
Then by Proposition~\ref{prop_Plancherel}\,(3), the following Parseval--Plancherel-type formula holds.
\begin{Corollary}\label{cor_Plancherel_easycase}
Let $\lambda>p-1$. For each $\bk\in\BZ_{++}^{r_2}$, we take a~vector-valued polynomial ${\rK_\bk^\vee(z_2)\in\cP(\fp^-_2,V_\bk)^{K_1}}$ normalized such that
\[
\big| \bigl\langle f(x_2),\overline{\rK_{\bk}^\vee(\overline{x_2})}\bigr\rangle_{F,\fp^+} \big|_{V_{\bk}}^2
=\Vert f\Vert_{F,\fp^+}^2, \qquad f(x_2)\in\cP_{\bk}\bigl(\fp^+_2\bigr).
\]
Then for $f\in\cH_\lambda(D)$, we have
\[
\Vert f\Vert_{\lambda}^2=\sum_{\bk\in\BZ_{++}^{r_2}}\frac{1}{(\lambda)_{\bk,d}}
\biggl\Vert \rK_\bk^\vee\left(\frac{\partial}{\partial x_2}\right)f(x)\biggr|_{x_2=0}\biggr\Vert_{\lambda,\bk,x_1}^2.
\]
\end{Corollary}

Also, by Proposition~\ref{prop_poles}, we have the following.
\begin{Corollary}\label{cor_submodule_easycase}
Let $\bk\in\BZ_{++}^{r_2}$, and set $k_{r_2+1}=\cdots=k_r:=0$. Then for $a=1,2,\dots,r$,
\[
{\rm d}\tau_\lambda(\cU(\fg_1))\cP_\bk\bigl(\fp^+_2\bigr)\subset M_a^\fg(\lambda)
\]
holds if and only if
\[
\lambda\in \frac{d}{2}(a-1)-k_a-\BZ_{\ge 0}.
\]
Especially, for $a=0,1,\dots,r-1$, we have
\[
\cH_{\frac{d}{2}a}(D)|_{\widetilde{G}_1}\simeq\hsum_{\substack{\bk\in\BZ_{++}^{r_2}\\ k_{a+1}=0}}\cH_{\frac{d}{2}a}\bigl(D_1,\cP_\bk\bigl(\fp^+_2\bigr)\bigr).
\]
\end{Corollary}
Here $M_a^\fg(\lambda)\!\subset\!\cO_\lambda(D)_{\widetilde{K}}$ is the $\bigl(\fg,\widetilde{K}\bigr)$-submodule given in~(\ref{submodule}),
so that $\cH_{\frac{d}{2}a} (D)_{\widetilde{K}}\!\simeq\! M_{a+1}^\fg\hspace{-1pt}\bigl(\frac{d}{2}a\bigr)$ holds for $a=0,1,\dots,r-1$.
The decomposition of $\cH_{\frac{d}{2}a}(D)$ for $a=1$ is earlier given in~\cite{MO2},
and we can also prove those for $a=2,\dots,r-1$ for Cases 1--3 by using the seesaw dual pair theory (see, e.g.,~\cite[Section 3]{HTW},~\cite{Ku}) as in~\cite{MO2}.

\section[Case $\fp^+_2$ is non-simple]{Case $\boldsymbol{\fp^+_2}$ is non-simple}\label{section_nonsimple}

In this section we treat $\fp^+=(\fp^+)^\sigma\oplus(\fp^+)^{-\sigma}=\fp^+_1\oplus\fp^+_2$ such that $\fp^+_2$ is non-simple,
and write $\fp^+_2=:\fp^+_{11}\oplus\fp^+_{22}$, $\fp^+_1=:\fp^+_{12}$, so that
\begin{align*}
&\bigl(\fp^+,\fp^+_{11},\fp^+_{12},\fp^+_{22}\bigr) \\
&= \begin{cases}
(\BC^{d+2},\BC,\BC^{d},\BC) & (\text{Case }1), \\
(\Sym(r,\BC),\Sym(r',\BC),{\rm M}(r',r'';\BC),\Sym(r'',\BC)) & (\text{Case }2), \\
({\rm M}(q,s;\BC),{\rm M}(q',s';\BC),{\rm M}(q',s'';\BC)\oplus {\rm M}(q'',s';\BC),{\rm M}(q'',s'';\BC)) & (\text{Case }3), \\
(\Alt(s,\BC),\Alt(s',\BC),{\rm M}(s',s'';\BC),\Alt(s'',\BC)) & (\text{Case }4), \\
(\Herm(3,\BO)^\BC,\BC,{\rm M}(1,2;\BO)^\BC,\Herm(2,\BO)^\BC) & (\text{Case }5), \\
({\rm M}(1,2;\BO)^\BC,\BC,\Alt(5,\BC),{\rm M}(1,5;\BC)) & (\text{Case }6) \end{cases}
\end{align*}
with $r=r'+r''$, $q=q'+q''$, $s=s'+s''$. Then the corresponding symmetric pairs are
\[ (G,G_1)= \begin{cases}
({\rm SO}_0(2,d+2),{\rm SO}_0(2,d)\times {\rm SO}(2)) & (\text{Case }1), \\
({\rm Sp}(r,\BR),{\rm U}(r',r'')) & (\text{Case }2), \\
({\rm SU}(q,s),{\rm S}({\rm U}(q',s'')\times {\rm U}(q'',s'))) & (\text{Case }3), \\
({\rm SO}^*(2s),{\rm U}(s',s'')) & (\text{Case }4), \\
(E_{7(-25)},{\rm U}(1)\times E_{6(-14)}) & (\text{Case }5), \\
(E_{6(-14)},{\rm U}(1)\times {\rm SO}^*(10)) & (\text{Case }6) \end{cases} \]
(up to covering).
Let $\dim\fp^+=:n$, $\dim\fp^+_{11}=:n'$, $\dim\fp^+_{22}=:n''$, $\rank\fp^+=:r$, $\rank\fp^+_{11}=:r'$, $\rank\fp^+_{22}=:r''$,
and let $d$, $d'$, $d''$ be the numbers defined in~(\ref{str_const}) for $\fp^+$, $\fp^+_{11}$, $\fp^+_{22}$, respectively.
Then the numbers $(r,r',r'',d)$ are given by
\[ (r,r',r'',d)=\begin{cases}
(2,1,1,d) & (\text{Case }1), \\
(r,r',r'',1) & (\text{Case }2), \\
(\min\{q,s\},\min\{q',s'\},\min\{q'',s''\},2) & (\text{Case }3), \\
(\lfloor s/2\rfloor, \lfloor s'/2\rfloor, \lfloor s''/2\rfloor,4) & (\text{Case }4), \\
(3,1,2,8) & (\text{Case }5), \\
(2,1,1,6) & (\text{Case }6), \end{cases} \]
and we have $d=d'=d''$ if $r',r''\ne 1$. Even when $r'$ or $r''=1$, since $d'$, $d''$ are not determined uniquely and any numbers are allowed,
we may assume $d=d'=d''$.

\subsection{Results on weighted Bergman inner products}

First, for $f(x_{11})\in\cP_\bk\bigl(\fp^+_{11}\bigr)$, $g(x_{22})\in\cP_\bl\bigl(\fp^+_{22}\bigr)$,
we want to compute the top terms and the poles of $\bigl\langle f(x_{11})g(x_{22}),{\rm e}^{(x|\overline{z})_{\fp^+}}\bigr\rangle_{\lambda,x}$,
where $\langle\cdot,\cdot\rangle_{\lambda}=\langle\cdot,\cdot\rangle_{\lambda,x}$ is the weighted Bergman inner product given in~(\ref{Bergman_inner_prod}),
with the variable of integration $x=x_{11}+x_{12}+x_{22}$.

\begin{Theorem}\label{thm_topterm_nonsimple}
Let $\Re\lambda>p-1$, $\bk\in\BZ_{++}^{r'}$, $\bl\in\BZ_{++}^{r''}$, and let $f(x_{11})\in\cP_\bk\bigl(\fp^+_{11}\bigr)$, $g(x_{22})\in\cP_\bl\bigl(\fp^+_{22}\bigr)$. Then we have
\[
\bigl\langle f(x_{11})g(x_{22}),{\rm e}^{(x|\overline{z})_{\fp^+}}\bigr\rangle_{\lambda,x}\big|_{z_{12}=0}=C_0^d(\lambda,\bk,\bl)f(z_{11})g(z_{22}),
\]
where
\begin{align}
C_0^d(\lambda,\bk,\bl)&=\frac{\prod_{i=1}^{r'}\prod_{j=1}^{r''}\left(\lambda-\frac{d}{2}(i+j-1)\right)_{k_i+l_j}}
{\prod_{i=1}^{r'+1}\prod_{j=1}^{r''+1}\left(\lambda-\frac{d}{2}(i+j-2)\right)_{k_i+l_j}} \notag\\
&=\frac{\prod_{a=2}^r\prod_{i=\max\{1,a-r''\}}^{\min\{a-1,r'\}}\left(\lambda-\frac{d}{2}(a-1)\right)_{k_i+l_{a-i}}}
{\prod_{a=1}^r\prod_{i=\max\{1,a-r''\}}^{\min\{a,r'+1\}}\left(\lambda-\frac{d}{2}(a-1)\right)_{k_i+l_{a+1-i}}}. \label{const0}
\end{align}
Here we put $k_{r'+1}=l_{r''+1}:=0$.
\end{Theorem}

\begin{Theorem}\label{thm_poles_nonsimple}
For $\bk\in\BZ_{++}^{r'}$, $\bl\in\BZ_{++}^{r''}$, we define $\phi_0(\bk,\bl)\in\BZ_{++}^{r'+r''}$ by
\begin{align}
\phi_0(\bk,\bl)_a:=\min\{k_i+l_j\mid 1\le i\le r'+1,\,1\le j\le r''+1,\,i+j=a+1\}, \label{phi0}
\end{align}
where $1\le a\le r'+r''$, $k_{r'+1}=l_{r''+1}:=0$. Then for $f(x_{11})\in\cP_\bk\bigl(\fp^+_{11}\bigr)$, $g(x_{22})\in\cP_\bl\bigl(\fp^+_{22}\bigr)$, as a~function of $\lambda$,
\begin{equation}\label{formula_poles_nonsimple}
(\lambda)_{\phi_0(\bk,\bl),d}\bigl\langle f(x_{11})g(x_{22}),{\rm e}^{(x|\overline{z})_{\fp^+}}\bigr\rangle_{\lambda,x}
\end{equation}
is holomorphically continued for all $\BC$.
\end{Theorem}

\begin{Remark}\label{rem_poles_nonsimple}\quad
\begin{enumerate}\itemsep=0pt
\item By~\cite[Corollary 5.7\,(3)]{N2}, if $f(x_{11})g(x_{22})\ne 0$ and if $\bk\in\BZ_{++}^{r'}$, $\bl\in\BZ_{++}^{r''}$ satisfy ``$\bk=\underline{0}_{r'}$'' or ``$\bl=\underline{0}_{r''}$''
or ``$\bk=(\underline{k}_a,\underline{0}_{r'-a})$, $1\le a\le r'$ and $l_{a+1}=0$'' or ``$\bl=(\underline{l}_a,\underline{0}_{r''-a})$, $1\le a\le r''$ and $k_{a+1}=0$'',
then~(\ref{formula_poles_nonsimple}) gives a~non-zero polynomial in~$z\in\fp^+$ for all $\lambda\in\BC$.
\item The restriction of Theorem~\ref{thm_poles_nonsimple} to $z_{12}=0$, the holomorphy of
\[
(\lambda)_{\phi_0(\bk,\bl),d}\bigl\langle f(x_{11})g(x_{22}),{\rm e}^{(x|\overline{z})_{\fp^+}}\bigr\rangle_{\lambda,x}\big|_{z_{12}=0}
\]
follows from Theorem~\ref{thm_topterm_nonsimple}, since it is easily proved that $(\lambda)_{\phi_0(\bk,\bl),d}C_0^d(\lambda,\bk,\bl)$ is holomorphic for all $\lambda\in\BC$.
\item Under Corollary~\ref{cor_FK}, Theorem~\ref{thm_poles_nonsimple} is equivalent to the inclusion
\begin{equation}\label{Ktype_incl_nonsimple}
\cP_\bk\bigl(\fp^+_{11}\bigr)\cP_\bl\bigl(\fp^+_{22}\bigr)\subset\bigoplus_{\substack{\bm\in\BZ_{++}^{r'+r''} \\ m_a\le \phi_0(\bk,\bl)_a}}\cP_\bm(\fp^+).
\end{equation}
When $\fp^+$ is classical, this follows from~(\ref{LR_nonsimple}), but in the following we prove the theorem by another way, which is available also for the exceptional case.
\end{enumerate}
\end{Remark}
Here $(\lambda)_{\bm,d}$ is as in~(\ref{Pochhammer}), and $\underline{k}_a:=(\underbrace{k,\dots,k}_a)$.

\begin{proof}[Proof of Theorem~\ref{thm_topterm_nonsimple}]
The 2nd equality of~(\ref{const0}) is easy. For the 1st equality, by the last paragraph of Section~\ref{subsection_reduction},
we may assume $\fp^+$, $\fp^+_{11}$, $\fp^+_{22}$ are of tube type, so that $r=r'+r''$ holds,
and by the $K_1$-equivariance, may assume $f(x_{11})=\Delta^{\fn^+_{11}}_\bk(x_{11})$, $g(x_{22})=\Delta^{\fn^+_{22}}_\bl(x_{22})$.
We prove the theorem by induction on $r'$. When $r'=0$, this follows from Corollary~\ref{cor_FK}, where we set $\prod_{i=1}^0({\cdots})=:1$.
Next we assume the theorem for $r'-1$, and prove it for $r'$. Since we have defined $k_{r'+1}=l_{r''+1}=0$, by~(\ref{key_nonsimple1}), we have
\begin{gather*}
 \bigl\langle \Delta^{\fn^+_{11}}_\bk(x_{11})\Delta^{\fn^+_{22}}_\bl(x_{22}),{\rm e}^{(x|\overline{z})_{\fp^+}}\bigr\rangle_{\lambda,x}\big|_{z_{12}=0} \\
 \quad{}=\frac{1}{(\lambda)_{\underline{k_{r'}}_r,d}}\det_{\fn^+}(z_{11}+z_{22})^{-\lambda+\frac{n}{r}}
\det_{\fn^+_{22}}\left(\frac{\partial}{\partial z_{22}}\right)^{k_{r'}}\det_{\fn^+}(z_{11}+z_{22})^{\lambda+k_{r'}-\frac{n}{r}} \\
 \quad{}\eqspace{}\times\bigl\langle \Delta^{\fn^+_{11}}_{\bk-\underline{k_{r'}}_{r'}}(x_{11})\Delta^{\fn^+_{22}}_\bl(x_{22}),{\rm e}^{(x|\overline{z})_{\fp^+}}\bigr\rangle_{\lambda+k_{r'},x}
\big|_{z_{12}=0} \\
 \quad{}=\frac{\det_{\fn^+_{11}}(z_{11})^{k_{r'}}}{(\lambda)_{\underline{k_{r'}}_r,d}}\det_{\fn^+_{22}}(z_{22})^{-\lambda+\frac{d}{2}(r-1)+1}
\det_{\fn^+_{22}}\left(\frac{\partial}{\partial z_{22}}\right)^{k_{r'}}\det_{\fn^+_{22}}(z_{22})^{\lambda+k_{r'}-\frac{d}{2}(r-1)-1} \\
 \quad{}\eqspace{}\times\frac{\prod_{i=1}^{r'-1}\prod_{j=1}^{r''}\left(\lambda+k_{r'}-\frac{d}{2}(i+j-1)\right)_{k_i-k_{r'}+l_j}}
{\prod_{i=1}^{r'}\prod_{j=1}^{r''+1}\left(\lambda+k_{r'}-\frac{d}{2}(i+j-2)\right)_{k_i-k_{r'}+l_j}}
\Delta^{\fn^+_{11}}_{\bk-\underline{k_{r'}}_{r'}}(z_{11})\Delta^{\fn^+_{22}}_\bl(z_{22}) \\
 \quad{}=\frac{1}{(\lambda)_{\underline{k_{r'}}_r,d}}
\frac{\prod_{i=1}^{r'-1}\prod_{j=1}^{r''}\left(\lambda+k_{r'}-\frac{d}{2}(i+j-1)\right)_{k_i-k_{r'}+l_j}}
{\prod_{i=1}^{r'}\prod_{j=1}^{r''+1}\left(\lambda+k_{r'}-\frac{d}{2}(i+j-2)\right)_{k_i-k_{r'}+l_j}}
\prod_{j=1}^{r''}\left(\lambda+l_j-\frac{d}{2}(r'+j-1)\right)_{k_{r'}} \\
 \quad{}\eqspace{}\times\Delta^{\fn^+_{11}}_\bk(z_{11})\Delta^{\fn^+_{22}}_\bl(z_{22}) \\
 \quad{}=\frac{1}{(\lambda)_{\underline{k_{r'}}_r,d}}\frac{\prod_{i=1}^{r'-1}\prod_{j=1}^{r''}\left(\lambda-\frac{d}{2}(i+j-1)\right)_{k_i+l_j}}
{\prod_{i=1}^{r'}\prod_{j=1}^{r''+1}\left(\lambda-\frac{d}{2}(i+j-2)\right)_{k_i+l_j}}
\frac{\prod_{i=1}^{r'}\prod_{j=1}^{r''+1}\left(\lambda-\frac{d}{2}(i+j-2)\right)_{k_{r'}}}
{\prod_{i=1}^{r'-1}\prod_{j=1}^{r''}\left(\lambda-\frac{d}{2}(i+j-1)\right)_{k_{r'}}} \\
 \quad{}\eqspace{}\times \frac{\prod_{j=1}^{r''}\left(\lambda-\frac{d}{2}(r'+j-1)\right)_{k_{r'}+l_j}}
{\prod_{j=1}^{r''}\left(\lambda-\frac{d}{2}(r'+j-1)\right)_{l_j}}\Delta^{\fn^+_{11}}_\bk(z_{11})\Delta^{\fn^+_{22}}_\bl(z_{22})\\
 \quad{}=\frac{1}{(\lambda)_{\underline{k_{r'}}_r,d}}\frac{\prod_{i=1}^{r'-1}\prod_{j=1}^{r''}\left(\lambda-\frac{d}{2}(i+j-1)\right)_{k_i+l_j}}
{\prod_{i=1}^{r'}\prod_{j=1}^{r''+1}\left(\lambda-\frac{d}{2}(i+j-2)\right)_{k_i+l_j}}
\frac{\prod_{i=1}^{r'}\prod_{j=1}^{r''+1}\left(\lambda-\frac{d}{2}(i+j-2)\right)_{k_{r'}}}
{\prod_{i=2}^{r'}\prod_{j=1}^{r''}\left(\lambda-\frac{d}{2}(i+j-2)\right)_{k_{r'}}} \\
 \quad{}\eqspace{}\times \frac{\prod_{j=1}^{r''}\left(\lambda-\frac{d}{2}(r'+j-1)\right)_{k_{r'}+l_j}}
{\prod_{j=1}^{r''+1}\left(\lambda-\frac{d}{2}((r'+1)+j-2)\right)_{k_{r'+1}+l_j}}\Delta^{\fn^+_{11}}_\bk(z_{11})\Delta^{\fn^+_{22}}_\bl(z_{22}) \\
 \quad{}=\frac{1}{\prod_{j=1}^r\left(\lambda-\frac{d}{2}(j-1)\right)_{k_{r'}}}
\frac{\prod_{i=1}^{r'}\prod_{j=1}^{r''}\left(\lambda-\frac{d}{2}(i+j-1)\right)_{k_i+l_j}}
{\prod_{i=1}^{r'+1}\prod_{j=1}^{r''+1}\left(\lambda-\frac{d}{2}(i+j-2)\right)_{k_i+l_j}} \\
 \quad{}\eqspace{}\times\prod_{j=1}^{r''}\left(\lambda-\frac{d}{2}(j-1)\right)_{k_{r'}}
\prod_{i=1}^{r'}\left(\lambda-\frac{d}{2}(i+r''-1)\right)_{k_{r'}}\Delta^{\fn^+_{11}}_\bk(z_{11})\Delta^{\fn^+_{22}}_\bl(z_{22}) \\
 \quad{}=\frac{\prod_{i=1}^{r'}\prod_{j=1}^{r''}\left(\lambda-\frac{d}{2}(i+j-1)\right)_{k_i+l_j}}
{\prod_{i=1}^{r'+1}\prod_{j=1}^{r''+1}\left(\lambda-\frac{d}{2}(i+j-2)\right)_{k_i+l_j}}\Delta^{\fn^+_{11}}_\bk(z_{11})\Delta^{\fn^+_{22}}_\bl(z_{22}),
\end{gather*}
where we have used the induction hypothesis and Proposition~\ref{prop_reduction} at the 2nd equality, and~(\ref{formula_diff}) at the 3rd equality.
Therefore, the theorem holds for all $r'$.
\end{proof}

\begin{proof}[Proof of Theorem~\ref{thm_poles_nonsimple}]
Again, we may assume $\fp^+$, $\fp^+_{11}$, $\fp^+_{22}$ are of tube type, and $f(x_{11})=\Delta^{\fn^+_{11}}_\bk(x_{11})$, $g(x_{22})=\Delta^{\fn^+_{22}}_\bl(x_{22})$.
We prove the theorem by induction on $(r',r'')$. If $r'=0$ or~$r''=0$, then this is true by Corollary~\ref{cor_FK}.
Next we assume that the theorem holds for $(r'-1,r'')$ and $(r',r''-1)$, and prove for $(r',r'')$.
First, by~(\ref{key_nonsimple1}), we have
\begin{gather*}
\det_{\fn^+}(z)^{-\lambda+\frac{n}{r}}\det_{\fn^+_{22}}\left(\frac{\partial}{\partial z_{22}}\right)^{k_{r'}}\det_{\fn^+}(z)^{\lambda+k_{r'}-\frac{n}{r}} \\
\qquad{}\eqspace{}\times(\lambda+k_{r'})_{\phi_0\bigl(\bk-\underline{k_{r'}}_{r'},\bl\bigr),d}\,
\bigl\langle \Delta^{\fn^+_{11}}_{\bk-\underline{k_{r'}}_{r'}}(x_{11})\Delta^{\fn^+_{22}}_\bl(x_{22}),{\rm e}^{(x|\overline{z})_{\fp^+}}\bigr\rangle_{\lambda+k_{r'},x} \\
\qquad{}=(\lambda)_{\underline{k_{r'}}_r,d}(\lambda+k_{r'})_{\phi_0\bigl(\bk-\underline{k_{r'}}_{r'},\bl\bigr),d}\,
\bigl\langle \Delta^{\fn^+_{11}}_\bk(x_{11})\Delta^{\fn^+_{22}}_\bl(x_{22}),{\rm e}^{(x|\overline{z})_{\fp^+}}\bigr\rangle_{\lambda,x} \\
\qquad{}=(\lambda)_{\phi_0\bigl(\bk-\underline{k_{r'}}_{r'},\bl\bigr)+\underline{k_{r'}}_r,d}\,
\bigl\langle \Delta^{\fn^+_{11}}_\bk(x_{11})\Delta^{\fn^+_{22}}_\bl(x_{22}),{\rm e}^{(x|\overline{z})_{\fp^+}}\bigr\rangle_{\lambda,x},
\end{gather*}
and by Proposition~\ref{prop_reduction} and the induction hypothesis for $(r'-1,r'')$,
this is holomorphically continued for all $\lambda\in\BC$. Similarly, by~(\ref{key_nonsimple2}), we have
\begin{gather*}
\det_{\fn^+}(z)^{-\lambda+\frac{n}{r}}\det_{\fn^+_{11}}\left(\frac{\partial}{\partial z_{11}}\right)^{l_{r''}}\det_{\fn^+}(z)^{\lambda+l_{r''}-\frac{n}{r}} \\
\qquad{}\eqspace{}\times(\lambda+l_{r''})_{\phi_0\bigl(\bk,\bl-\underline{l_{r''}}_{r''}\bigr),d}\,
\bigl\langle \Delta^{\fn^+_{11}}_\bk(x_{11})\Delta^{\fn^+_{22}}_{\bl-\underline{l_{r''}}_{r''}}(x_{22}),{\rm e}^{(x|\overline{z})_{\fp^+}}\bigr\rangle_{\lambda+l_{r''},x} \\
\qquad{}=(\lambda)_{\underline{l_{r''}}_r,d}(\lambda+l_{r''})_{\phi_0\bigl(\bk,\bl-\underline{l_{r''}}_{r''}\bigr),d}\,
\bigl\langle \Delta^{\fn^+_{11}}_\bk(x_{11})\Delta^{\fn^+_{22}}_\bl(x_{22}),{\rm e}^{(x|\overline{z})_{\fp^+}}\bigr\rangle_{\lambda,x} \\
\qquad{}=(\lambda)_{\phi_0\bigl(\bk,\bl-\underline{l_{r''}}_{r''}\bigr)+\underline{l_{r''}}_r,d}\,
\bigl\langle \Delta^{\fn^+_{11}}_\bk(x_{11})\Delta^{\fn^+_{22}}_\bl(x_{22}),{\rm e}^{(x|\overline{z})_{\fp^+}}\bigr\rangle_{\lambda,x},
\end{gather*}
and again by Proposition~\ref{prop_reduction} and the induction hypothesis for $(r',r''-1)$,
this is holomorphically continued for all $\lambda\in\BC$. Now, the greatest common divisor of
\begin{align*}
(\lambda)_{\phi_0\bigl(\bk-\underline{k_{r'}}_{r'},\bl\bigr)+\underline{k_{r'}}_r,d}
&=\prod_{a=1}^r\left(\lambda-\frac{d}{2}(a-1)\right)_{\min\{k_i+l_j\mid\, 1\le i\le r',\,1\le j\le r''+1,\,i+j=a+1\}}, \\
(\lambda)_{\phi_0\bigl(\bk,\bl-\underline{l_{r''}}_{r''}\bigr)+\underline{l_{r''}}_r,d}
&=\prod_{a=1}^r\left(\lambda-\frac{d}{2}(a-1)\right)_{\min\{k_i+l_j\mid\, 1\le i\le r'+1,\,1\le j\le r'',\,i+j=a+1\}}
\end{align*}
is
\[
(\lambda)_{\phi_0(\bk,\bl),d}=\prod_{a=1}^r\left(\lambda-\frac{d}{2}(a-1)\right)_{\min\{k_i+l_j\mid\, 1\le i\le r'+1,\,1\le j\le r''+1,\,i+j=a+1\}},
\]
and hence $(\lambda)_{\phi_0(\bk,\bl),d}\,\bigl\langle \Delta^{\fn^+_{11}}_\bk(x_{11})\Delta^{\fn^+_{22}}_\bl(x_{22}),{\rm e}^{(x|\overline{z})_{\fp^+}}\bigr\rangle_{\lambda,x}$ is
holomorphically continued for all $\lambda\in\BC$. This completes the proof of Theorem~\ref{thm_poles_nonsimple}.
\end{proof}

By Theorem~\ref{thm_topterm_nonsimple} and Proposition~\ref{prop_Plancherel}\,(1), we get the following.
\begin{Corollary}
Let $\Re\lambda>p-1$, $\bk\in\BZ_{++}^{r'}$, $\bl\in\BZ_{++}^{r''}$, and let $f(x_{11})\in\cP_\bk\bigl(\fp^+_{11}\bigr)$, $g(x_{22})\in\cP_\bl\bigl(\fp^+_{22}\bigr)$. Then we have
\[ \Vert f(x_{11})g(x_{22})\Vert_{\lambda,x}^2=C_0^d(\lambda,\bk,\bl)\Vert f(z_{11})g(z_{22})\Vert_{F,z}^2, \]
where $C_0^d(\lambda,\bk,\bl)$ is as in~\eqref{const0}.
\end{Corollary}

Next we give a~rough estimate of zeroes of $(\lambda)_{\phi_0(\bk,\bl),d}C_0^d(\lambda,\bk,\bl)$.

\begin{Proposition}\label{prop_zeroes_nonsimple}
For $\bk\in\BZ_{++}^{r'}$, $\bl\in\BZ_{++}^{r''}$, let $C_0^d(\lambda,\bk,\bl)$, $\phi_0(\bk,\bl)$ be as in~\eqref{const0},~\eqref{phi0}.
For ${1\le m_1'\le m_2'\le r'}$, $1\le m_1''\le m_2''\le r''$, if $k_1=\cdots=k_{m_1'}$, $l_1=\cdots=l_{m_1''}$ and ${k_{m_2'+1}\!=l_{m_2''+1}\!=0}$, then we have
\begin{align*}
&\big\{\lambda\in\BC\mid (\lambda)_{\phi_0(\bk,\bl),d}C_0^d(\lambda,\bk,\bl)=0\big\} \\ &\subset \begin{cases}
\big[ \frac{d}{2}\max\{m_1',m_1''\}-(k_1+l_1)+1,\, \frac{d}{2}(m_2'+m_2''-1)-\max\{k_{m_2'},l_{m_2''}\}\big], & k_1l_1\ge 1, \\
\varnothing, & k_1l_1=0. \end{cases}
\end{align*}
Especially, for $m=1,2,\dots,r'+r''-1$, if $\phi_0(\bk,\bl)_{m+1}=0$ and $\phi_0(\bk,\bl)_m\ne 0$, then $C_0^d(\lambda,\bk,\bl)$ is non-zero holomorphic at $\lambda=\frac{d}{2}m$,
and has a~pole at $\lambda=\frac{d}{2}(m-1)$.
\end{Proposition}

\begin{proof}
By direct computation, we have
\begin{gather*}
\phi_0(\bk,\bl)_a=\begin{cases} k_1+l_a, & 1\le a\le \min\{m_1',m_2''\}, \\ k_1, & m_2''<a\le m_1' \end{cases}
=\begin{cases} k_a+l_1, & 1\le a\le \min\{m_1'',m_2'\}, \\ l_1, & m_2'<a\le m_1'', \end{cases} \\
\phi_0(\bk,\bl)_{m_2'+m_2''+1}=0,
\end{gather*}
and
\begin{align*}
C_0^d(\lambda,\bk,\bl)&=C_0^d\bigl(\lambda,\bigl(k_1,\dots,k_{m_2'}\bigr),\bigl(l_1,\dots,l_{m_2''}\bigr)\bigr) \\
&=\frac{1}{\prod_{a=1}^{\max\{m_1',m_1''\}}\left(\lambda-\frac{d}{2}(a-1)\right)_{\phi_0(\bk,\bl)_a}} \\
&\eqspace{}\times\prod_{a=\max\{m_1',m_1''\}+1}^{m_2'+m_2''}\frac{\prod_{i=\max\{1,a-m_2''\}}^{\min\{a-1,m_2'\}}\left(\lambda-\frac{d}{2}(a-1)\right)_{k_i+l_{a-i}}}
{\prod_{i=\max\{1,a-m_2''\}}^{\min\{a,m_2'+1\}}\left(\lambda-\frac{d}{2}(a-1)\right)_{k_i+l_{a+1-i}}}.
\end{align*}
For $2\le a\le m_2'+m_2''$, let
\begin{align*}
\phi(\bk,\bl)_a&:={\min}_2\{k_i+l_j\mid 1\le i\le m_2'+1,\ 1\le j\le m_2''+1,\ i+j=a+1\}, \\
\psi(\bk,\bl)_a&:=\max\{k_i+l_j\mid 1\le i\le m_2',\ 1\le j\le m_2'',\ i+j=a\},
\end{align*}
where ${\min}_2$ denotes the second smallest element, so that
\[
\max\big\{k_{m_2'},l_{m_2''}\big\}\le \phi(\bk,\bl)_a\le \psi(\bk,\bl)_a\le k_1+l_1
\]
holds for every $a$. Then we have
\begin{align*}
&\big\{ \lambda\in\BC\mid (\lambda)_{\phi_0(\bk,\bl),d}C_0^d(\lambda,\bk,\bl)=0\big\} \\
&\qquad{}=\bigcup_{a=\max\{m_1',m_1''\}+1}^{m_2'+m_2''}\left\{\lambda\in\BC\ \middle|\ \begin{array}{l} \left(\lambda-\frac{d}{2}(a-1)\right)_{\phi_0(\bk,\bl)_a} \\
\ds {}\times\frac{\prod_{i=\max\{1,a-m_2''\}}^{\min\{a-1,m_2'\}}\left(\lambda-\frac{d}{2}(a-1)\right)_{k_i+l_{a-i}}}
{\prod_{i=\max\{1,a-m_2''\}}^{\min\{a,m_2'+1\}}\left(\lambda-\frac{d}{2}(a-1)\right)_{k_i+l_{a+1-i}}}=0 \end{array} \right\} \\
&\qquad{}\subset\bigcup_{a=\max\{m_1',m_1''\}+1}^{m_2'+m_2''}\left\{\lambda\in\BC\ \middle|\
\frac{\left(\lambda-\frac{d}{2}(a-1)\right)_{\psi(\bk,\bl)_a}}{\left(\lambda-\frac{d}{2}(a-1)\right)_{\phi(\bk,\bl)_a}}=0 \right\} \\
&\qquad{}=\bigcup_{a=\max\{m_1',m_1''\}+1}^{m_2'+m_2''}\left\{\frac{d}{2}(a-1)-j\ \middle|\ j\in\BZ,\, \phi(\bk,\bl)_a\le j\le \psi(\bk,\bl)_a-1 \right\} \\
&\qquad{}\subset\left[ \frac{d}{2}\max\{m_1',m_1''\}-(k_1+l_1)+1,\, \frac{d}{2}(m_2'+m_2''-1)-\max\{k_{m_2'},l_{m_2''}\}\right].
\end{align*}
When $k_1=0$, i.e., $\bk=(0,\dots,0)$, by direct computation, we have
\[ (\lambda)_{\phi_0(\bk,\bl),d}C_0^d(\lambda,\bk,\bl)=\frac{(\lambda)_{\bl,d}}{(\lambda)_{\bl,d}}=1, \]
and this is non-zero everywhere. The $l_1=0$, i.e., $\bl=(0,\dots,0)$ case is similar.

If $\phi_0(\bk,\bl)_{m+1}=0$ and $\phi_0(\bk,\bl)_m\ne 0$, then $k_{m_2'+1}=l_{m_2''+1}=0$ and $\max\big\{k_{m'_2},l_{m''_2}\big\}\ne 0$ hold for some $m_2'$, $m_2''$ with $m_2'+m_2''=m$,
and by the above formula $(\lambda)_{\phi_0(\bk,\bl),d}C_0^d(\lambda,\bk,\bl)$ is non-zero at $\lambda=\frac{d}{2}m$, $\frac{d}{2}(m-1)$.
Hence, we get the last claim.
\end{proof}

Especially, if $\phi_0(\bk,\bl)_{a+1}=0$ and $\phi_0(\bk,\bl)_a\ne 0$, then for non-zero polynomials $f(x_{11})\in\cP_\bk\bigl(\fp^+_{11}\bigr)$, $g(x_{22})\in\cP_\bl\bigl(\fp^+_{22}\bigr)$,
$\bigl\langle f(x_{11})g(x_{22}),{\rm e}^{(x|\overline{z})_{\fp^+}}\bigr\rangle_{\lambda,x}$ has a~pole at $\lambda=\frac{d}{2}(a-1)$,
and combining with Corollary~\ref{cor_FK} and~(\ref{Ktype_incl_nonsimple}) we have
\[ \cP_\bk\bigl(\fp^+_{11}\bigr)\cP_\bl\bigl(\fp^+_{22}\bigr)\subset\bigoplus_{\substack{\bm\in\BZ_{++}^{r'+r''} \\ m_{a+1}=0}}\cP_\bm(\fp^+), \qquad
\cP_\bk\bigl(\fp^+_{11}\bigr)\cP_\bl\bigl(\fp^+_{22}\bigr)\not\subset\bigoplus_{\substack{\bm\in\BZ_{++}^{r'+r''} \\ m_a=0}}\cP_\bm(\fp^+), \]
that is, for $a=0,1,\dots,r-1$, we have
\begin{align}
\cP\bigl(\fp^+_{11}\oplus\fp^+_{22}\bigr)\cap \bigoplus_{\substack{\bm\in\BZ_{++}^r \\ m_{a+1}=0}}\cP_\bm(\fp^+)
&=\bigoplus_{\substack{(\bk,\bl)\in\BZ_{++}^{r'}\times\BZ_{++}^{r''} \\ \phi_0(\bk,\bl)_{a+1}=0}}\cP_\bk\bigl(\fp^+_{11}\bigr)\cP_\bl\bigl(\fp^+_{22}\bigr) \label{formula_discWallach_nonsimple}\\
&=\bigoplus_{\substack{0\le a'\le r' \\ 0\le a''\le r'' \\ a'+a''\le a}}\bigoplus_{\substack{(\bk,\bl)\in\BZ_{++}^{a'}\times\BZ_{++}^{a''} \\ k_{a'},l_{a''}\ge 1}}
\cP_{(\bk,\underline{0}_{r'-a'})}\bigl(\fp^+_{11}\bigr)\cP_{(\bl,\underline{0}_{r''-a''})}\bigl(\fp^+_{22}\bigr),\nonumber
\end{align}
where we set $\phi_0(\bk,\bl)_{r'+r''+1}=\cdots=\phi_0(\bk,\bl)_r:=0$ and $k_0=l_0:=+\infty$. This equality is earlier given in~\cite[Section 6]{Fr}.

Next we consider general $(\bk,\bl)\in\BZ_{++}^{r'}\times\BZ_{++}^{r''}$. Then again by the above proposition, \linebreak 
$\frac{1}{C_0^d(\lambda,\bk,\bl)}\bigl\langle f(x_{11})g(x_{22}), {\rm e}^{(x|\overline{z})_{\fp^+}}\bigr\rangle_{\lambda,x}$ is holomorphic
for $\Re\lambda>\frac{d}{2}(r'+r''-1)-\max\{k_{r'},l_{r''}\}$.
Then by Theorem~\ref{thm_factorize}, comparing the top terms, we get the following.
\begin{Theorem}\label{thm_factorize_nonsimple}
Suppose $\fp^+$, $\fp^+_{11}$, $\fp^+_{22}$ are of tube type, and consider the meromorphic continuation of $\langle\cdot,\cdot\rangle_\lambda$.
Then for $\bk\in\BZ_{++}^{r'}$, $\bl\in\BZ_{++}^{r''}$, $a=1,2,\dots,\min\{k_{r'},l_{r''}\}$ and for $f(x_{11})\in\cP_\bk\bigl(\fp^+_{11}\bigr)$, $g(x_{22})\in\cP_\bl\bigl(\fp^+_{22}\bigr)$, we have
\begin{align*}
&\frac{1}{C_0^d(\lambda,\bk,\bl)}\bigl\langle f(x_{11})g(x_{22}),{\rm e}^{(x|\overline{z})_{\fp^+}}\bigr\rangle_{\lambda,x}\big|_{\lambda=\frac{n}{r}-a} \\
&=\frac{\det_{\fn^+}(z)^a}{C_0^d\left(\frac{n}{r}+a,\bk-\underline{a}_{r'},\bl-\underline{a}_{r''}\right)}
\bigl\langle \det_{\fn^+_{11}}(x_{11})^{-a}\det_{\fn^+_{22}}(x_{22})^{-a}f(x_{11})g(x_{22}),{\rm e}^{(x|\overline{z})_{\fp^+}}\bigr\rangle_{\frac{n}{r}+a,x}.
\end{align*}
\end{Theorem}

\subsection[Results on restriction of $\cH_\lambda(D)$ to subgroups]{Results on restriction of $\boldsymbol{\cH_\lambda(D)}$ to subgroups}

Next we consider the decomposition of the holomorphic discrete series representation of scalar type $\cH_\lambda(D)$ under the subgroup $\widetilde{G}_1\subset\widetilde{G}$.
By Theorem~\ref{thm_HKKS}, we have
\[ \cH_\lambda(D)|_{\widetilde{G}_1}\simeq\hsum_{\bk\in\BZ_{++}^{r'}}\,\hsum_{\bl\in\BZ_{++}^{r''}}\cH_{\varepsilon_1\lambda}\bigl(D_1,\cP_\bk\bigl(\fp^+_{11}\bigr)\boxtimes\cP_\bl\bigl(\fp^+_{22}\bigr)\bigr), \]
where
\begin{align*}
&\cH_{\varepsilon_1\lambda}\bigl(D_1,\cP_\bk\bigl(\fp^+_{11}\bigr)\boxtimes\cP_\bl\bigl(\fp^+_{22}\bigr)\bigr) \\
&\simeq\begin{cases}
\cH_{\lambda+k+l}\bigl(D_{{\rm SO}_0(2,d)}\bigr)\boxtimes\chi_{{\rm SO}(2)}^{-k+l} & (\text{Case }1), \\
\cH_{\lambda+\lambda}\bigl(D_{{\rm U}(r',r'')},V_{2\bk}^{(r')\vee}\boxtimes V_{2\bl}^{(r'')}\bigr) & (\text{Case }2), \\
\cH_{\lambda_1+\lambda_2}\bigl(D_{{\rm U}(q',s'')},V_{\bk}^{(q')\vee}\boxtimes V_{\bl}^{(s'')}\bigr)\hboxtimes\cH_{\lambda_1+\lambda_2}\bigl(D_{{\rm U}(q'',s')},V_{\bl}^{(q'')\vee}\boxtimes V_{\bk}^{(s')}\bigr) & (\text{Case }3), \\
\cH_{\frac{\lambda}{2}+\frac{\lambda}{2}}\bigl(D_{{\rm U}(s',s'')},V_{\bk^2}^{(s')\vee}\boxtimes V_{\bl^2}^{(s'')}\bigr) & (\text{Case }4), \\
\cH_{\lambda+k+\frac{|\bl|}{2}}\bigl(D_{E_{6(-14)}},V_{(l_1-l_2,0,0,0,0)}^{[10]\vee}\bigr)\boxtimes\chi_{{\rm U}(1)}^{-\lambda+2k-2|\bl|} & (\text{Case }5), \\
\cH_{\lambda+k}\bigl(D_{{\rm SO}^*(10)},V_{\left(\frac{l}{2},\frac{l}{2},\frac{l}{2},\frac{l}{2},-\frac{l}{2}\right)}^{(5)\vee}\bigr)\boxtimes\chi_{{\rm U}(1)}^{-\lambda+3k-3l} & (\text{Case }6),
\end{cases}
\end{align*}
where we set $\lambda=\lambda_1+\lambda_2$ for Case 3. For each $\bk\in\BZ_{++}^{r'}$, $\bl\in\BZ_{++}^{r''}$,
let $V_\bk\boxtimes W_\bl$ be an abstract $K_1$-module isomorphic to $\cP_\bk\bigl(\fp^+_{11}\bigr)\boxtimes \cP_\bl\bigl(\fp^+_{22}\bigr)$,
let $\Vert\cdot\Vert_{\varepsilon_1\lambda,\bk,\bl}$ be the $\widetilde{G}_1$-invariant norm on ${\cH_{\varepsilon_1\lambda}(D_1,\allowbreak V_\bk\boxtimes W_\bl)}$
normalized such that $\Vert v\Vert_{\varepsilon_1\lambda,\bk,\bl}=|v|_{V_\bk\boxtimes W_\bl}$ holds for all constant functions $v\in V_\bk\boxtimes W_\bl$,
and for $\lambda>p-1$ let
\[
\cF_{\lambda,\bk,\bl}^\downarrow\colon \ \cH_\lambda(D)|_{\widetilde{G}_1}\longrightarrow \cH_{\varepsilon_1\lambda}(D_1,V_\bk\boxtimes W_\bl)
\]
be the symmetry breaking operator given in~(\ref{SBO1}) and~(\ref{SBO2}), using a~vector-valued polynomial
$\rK_{\bk,\bl}(x_{11},x_{22})\in \bigl(\bigl(\cP_\bk\bigl(\fp^+_{11}\bigr)\boxtimes\cP_\bl\bigl(\fp^+_{22}\bigr)\bigr)\otimes \bigl(\overline{V_\bk}\boxtimes\overline{W_\bl}\bigr)\bigr)^{K_1}$ satisfying
\begin{gather*}
\big|\langle f(x_{11})g(x_{22}), \rK_{\bk,\bl}(x_{11},x_{22})\rangle_{F,x}\big|_{V_\bk\boxtimes W_\bl}=\Vert f(z_{11})g(z_{22})\Vert_{F,z}
\end{gather*}
for $f(x_{11})\in\cP_\bk\bigl(\fp^+_{11}\bigr)$, $g(x_{22})\in\cP_\bl\bigl(\fp^+_{22}\bigr)$,
so that
\[
\big\Vert \cF_{\lambda,\bk,\bl}^\downarrow (f(x_{11})g(x_{22}))\big\Vert_{\varepsilon_1\lambda,\bk,\bl}=\Vert f(z_{11})g(z_{22})\Vert_{F,z}, \qquad
f(x_{11})\in\cP_\bk\bigl(\fp^+_{11}\bigr),\ g(x_{22})\in\cP_\bl\bigl(\fp^+_{22}\bigr)
\]
holds. Also, when $\fp^+_{11}$, $\fp^+_{22}$ are of tube type, we fix $[K_1,K_1]$-isomorphisms
$V_{\bk+\underline{a}_{r'}}\boxtimes W_{\bl+\underline{a}_{r''}}\simeq V_\bk\boxtimes W_\bl$ for each $a\in\BZ_{>0}$.
Then by Proposition~\ref{prop_Plancherel}\,(3) and Theorem~\ref{thm_topterm_nonsimple}, the following Parseval--Plancherel-type formula holds.
\begin{Corollary}\label{cor_Plancherel_nonsimple}
For $\lambda>p-1$ and for $f\in\cH_\lambda(D)$ we have
\[ \Vert f\Vert_{\lambda}^2=\sum_{\bk\in\BZ_{++}^{r'}}\sum_{\bl\in\BZ_{++}^{r''}}C_0^d(\lambda,\bk,\bl)
\big\Vert \cF_{\lambda,\bk,\bl}^\downarrow f\big\Vert_{\varepsilon_1\lambda,\bk,\bl}^2, \]
where $C_0^d(\lambda,\bk,\bl)$ is as in~\eqref{const0}.
\end{Corollary}

Next we consider the meromorphic continuation for smaller $\lambda$.
Then by Propositions~\ref{prop_poles},~\ref{prop_zeroes_nonsimple}, Theorems~\ref{thm_poles_nonsimple},~\ref{thm_factorize_nonsimple} and the formula~(\ref{formula_discWallach_nonsimple}),
we have the following.
\begin{Corollary}\label{cor_submodule_nonsimple}
For $\bk\hspace{-0.5pt}\in\hspace{-0.5pt}\BZ_{++}^{r'}$, $\bl\hspace{-0.5pt}\in\hspace{-0.5pt}\BZ_{++}^{r''}$, let $\phi_0(\bk,\bl)\hspace{-0.5pt}\in\hspace{-0.5pt}\BZ_{++}^{r'+r''}$ be as in~\eqref{phi0},
and set $\phi_0(\bk,\bl)_{r'+r''+1}\allowbreak=\cdots=\phi_0(\bk,\bl)_r:=0$.
\begin{enumerate}\itemsep=0pt
\item[$(1)$] For $a=1,2,\dots,r$,
\[ {
\rm d}\tau_\lambda(\cU(\fg_1))\cP_\bk\bigl(\fp^+_{11}\bigr)\cP_\bl\bigl(\fp^+_{22}\bigr)\subset M_a^\fg(\lambda)
\]
holds if
\[
\lambda\in \frac{d}{2}(a-1)-\phi_0(\bk,\bl)_a-\BZ_{\ge 0},
\]
where $M_a^\fg(\lambda)\subset\cO_\lambda(D)_{\widetilde{K}}$ is the $\bigl(\fg,\widetilde{K}\bigr)$-submodule given in~\eqref{submodule}.
\item[$(2)$] For $a=0,1,\dots,r-1$ we have
\begin{align*}
\cH_{\frac{d}{2}a}(D)|_{\widetilde{G}_1}&\simeq\hsum_{\substack{(\bk,\bl)\in\BZ_{++}^{r'}\times\BZ_{++}^{r''} \\ \phi_0(\bk,\bl)_{a+1}=0}}
\cH_{\varepsilon_1\frac{d}{2}a}\bigl(D_1,\cP_\bk\bigl(\fp^+_{11}\bigr)\boxtimes\cP_\bl\bigl(\fp^+_{22}\bigr)\bigr) \\
&=\hsum_{\substack{0\le a'\le r' \\ 0\le a''\le r'' \\ a'+a''\le a}}\hspace{6pt}\hsum_{\substack{(\bk,\bl)\in\BZ_{++}^{a'}\times\BZ_{++}^{a''} \\ k_{a'},l_{a''}\ge 1}}
\cH_{\varepsilon_1\frac{d}{2}a}\bigl(D_1,\cP_{(\bk,\underline{0}_{r'-a'})}\bigl(\fp^+_{11}\bigr)\boxtimes\cP_{(\bl,\underline{0}_{r''-a''})}\bigl(\fp^+_{22}\bigr)\bigr),
\end{align*}
where we set $k_0=l_0:=+\infty$.
\item[$(3)$] For $a=0,1,\dots,r-1$, if $\phi_0(\bk,\bl)_{a+1}=0$, then $\cF_{\lambda,\bk,\bl}^\downarrow$ is holomorphic at $\lambda=\frac{d}{2}a$,
and its restriction gives the symmetry breaking operator
\[
\cF_{\frac{d}{2}a,\bk,\bl}^\downarrow
\colon\ \cH_{\frac{d}{2}a}(D)|_{\widetilde{G}_1}\longrightarrow \cH_{\varepsilon_1\frac{d}{2}a}(D_1,V_\bk\boxtimes W_\bl).
\]
\item[$(4)$] Suppose $\fp^+$, $\fp^+_{11}$, $\fp^+_{22}$ are of tube type. For $a=1,2,\dots,\min\{k_{r'},l_{r''}\}$, if
\[
\rK_{\bk,\bl}(x_{11},x_{22})=c\det_{\fn^+_{11}}(x_{11})^a\det_{\fn^+_{22}}(x_{22})^a\rK_{\bk-\underline{a}_{r'},\bl-\underline{a}_{r''}}(x_{11},x_{22}),
\]
then we have
\begin{gather*}
\cF_{\frac{n}{r}-a,\bk,\bl}^\downarrow=c\cF_{\frac{n}{r}+a,\bk-\underline{a}_{r'},\bl-\underline{a}_{r''}}^\downarrow\circ\det_{\fn^-}\left(\frac{\partial}{\partial x}\right)^a\colon \\
\cO_{\frac{n}{r}-a}(D)|_{\widetilde{G}_1}\longrightarrow \cO_{\varepsilon_1\left(\frac{n}{r}-a\right)}(D_1,V_\bk\boxtimes W_\bl)
\simeq \cO_{\varepsilon_1\left(\frac{n}{r}+a\right)}(D_1,V_{\bk-\underline{a}_{r'}}\boxtimes W_{\bl-\underline{a}_{r''}}).
\end{gather*}
\end{enumerate}
\end{Corollary}
If $\bk$, $\bl$ satisfy the condition in Remark~\ref{rem_poles_nonsimple}\,(1), then ``only if'' in Corollary~\ref{cor_submodule_nonsimple}\,(1) also holds.
Also, Corollary~\ref{cor_submodule_nonsimple}\,(2) for $\lambda=\frac{d}{2}$ ($a=1$ case) is earlier given in~\cite{MO2}.
The parameter set in Corollary~\ref{cor_submodule_nonsimple}\,(2) also appears in Howe's correspondence for the dual pair $({\rm U}(r',r''),{\rm U}(a))$ (see, e.g.,~\cite{Ad, KV, Prz}),
and especially, we can prove (2) for Cases 2--4 by using the seesaw dual pair theory (see, e.g.,~\cite[Section 3]{HTW},~\cite{Ku}) as in~\cite{MO2}.

\section[Case $\fp^+_2$ is simple]{Case $\boldsymbol{\fp^+_2}$ is simple}\label{section_simple}

In this section we treat $\fp^+=(\fp^+)^\sigma\oplus(\fp^+)^{-\sigma}=\fp^+_1\oplus\fp^+_2$ such that $\fp^+_2$ is simple,
and $\fp^+_2(e)_2\subsetneq \fp^+(e)_2$ holds for some (or equivalently any) maximal tripotent $e\in\fp^+_2$, so that
\[ (\fp^+,\fp^+_1,\fp^+_2)= \begin{cases}
(\BC^n,\BC^{n'},\BC^{n''}) & (\text{Case }1), \\
(\Sym(r,\BC),\Sym(r',\BC)\oplus\Sym(r'',\BC),{\rm M}(r',r'';\BC)) & (\text{Case }2), \\
(\Alt(s,\BC),\Alt(s',\BC)\oplus\Alt(s'',\BC),{\rm M}(s',s'';\BC)) & (\text{Case }3), \\
({\rm M}(r,\BC),\Alt(r,\BC),\Sym(r,\BC)) & (\text{Case }4), \\
({\rm M}(r,\BC),\Sym(r,\BC),\Alt(r,\BC)) & (\text{Case }5), \\
\bigl(\Herm(3,\BO)^\BC,{\rm M}(2,6;\BC),\Alt(6,\BC)\bigr) & (\text{Case }6), \\
\bigl(\Herm(3,\BO)^\BC,\Alt(6,\BC),{\rm M}(2,6;\BC)\bigr) & (\text{Case }7), \\
\bigl(\Herm(3,\BO)^\BC,\BC\oplus\Herm(2,\BO)^\BC,{\rm M}(1,2;\BO)^\BC\bigr) & (\text{Case }8), \\
\bigl({\rm M}(1,2;\BO)^\BC,{\rm M}(2,4;\BC),{\rm M}(4,2;\BC)\bigr) & (\text{Case }9), \\
\bigl({\rm M}(1,2;\BO)^\BC,\BC\oplus {\rm M}(1,5;\BC),\Alt(5,\BC)\bigr) & (\text{Case }10) \end{cases} \]
($n=n'+n''$, $n\ge 3$, $n''\ne 2$ for Case 1, $r=r'+r''$ for Case 2, and $s=s'+s''$, $s',s''\ge 2$ for Case 3).
Then the corresponding symmetric pairs are
\[ (G,G_1)= \begin{cases}
({\rm SO}_0(2,n),{\rm SO}_0(2,n')\times {\rm SO}(n'')) & (\text{Case }1), \\
({\rm Sp}(r,\BR),{\rm Sp}(r',\BR)\times {\rm Sp}(r'',\BR)) & (\text{Case }2), \\
({\rm SO}^*(2s),{\rm SO}^*(2s')\times {\rm SO}^*(2s'')) & (\text{Case }3), \\
({\rm SU}(r,r),{\rm SO}^*(2r)) & (\text{Case }4), \\
({\rm SU}(r,r),{\rm Sp}(r,\BR)) & (\text{Case }5), \\
(E_{7(-25)},{\rm SU}(2,6)) & (\text{Case }6), \\
(E_{7(-25)},{\rm SU}(2)\times {\rm SO}^*(12)) & (\text{Case }7), \\
(E_{7(-25)},{\rm SL}(2,\BR)\times {\rm Spin}_0(2,10)) & (\text{Case }8), \\
(E_{6(-14)},{\rm SU}(2,4)\times {\rm SU}(2)) & (\text{Case }9), \\
(E_{6(-14)},{\rm SL}(2,\BR)\times {\rm SU}(1,5)) & (\text{Case }10) \end{cases} \]
(up to covering). Let $\dim\fp^+=:n$, $\dim\fp^+_2=:n_2$, $\rank\fp^+=:r$, $\rank\fp^+_2=:r_2$,
let $d$, $d_2$ be the numbers defined in~(\ref{str_const}) for $\fp^+$, $\fp^+_2$ respectively, and let $\varepsilon_2\in\{1,2\}$ be as in~(\ref{epsilon_j}).
Then the numbers $(r,r_2,d,d_2,\varepsilon_2)$ are given by
\[ (r,r_2,d,d_2,\varepsilon_2)= \begin{cases}
(2,2,n-2,n''-2,1) & (\text{Case }1,\, n''\ge 3), \\
(2,1,n-2,-,2) & (\text{Case }1,\, n''=1), \\
(r,\min\{r',r''\},1,2,2) & (\text{Case }2), \\
(\lfloor s/2\rfloor, \min\{s',s''\}, 4,2,1) & (\text{Case }3), \\
(r,r,2,1,1) & (\text{Case }4), \\
(r,\lfloor r/2\rfloor, 2,4,2) & (\text{Case }5), \\
(3,3,8,4,1) & (\text{Case }6), \\
(3,2,8,2,1) & (\text{Case }7), \\
(3,2,8,6,1) & (\text{Case }8), \\
(2,2,6,2,1) & (\text{Case }9), \\
(2,2,6,4,1) & (\text{Case }10). \end{cases} \]
When $\varepsilon_2=1$, we have $d_2=d/2$ or $r_2=2$. Similarly, when $\varepsilon_2=2$, we have $d_2=2d$ or $r_2=1$.
For $r_2=1$ case, since $d_2$ is not determined uniquely and any number is allowed, we may assume $d_2=2d$.

\subsection{Results on weighted Bergman inner products}

First, for $f(x_2)\in\cP_\bk\bigl(\fp^+_2\bigr)$, we want to compute the top terms and the poles of $\bigl\langle f(x_2), {\rm e}^{(x|\overline{z})_{\fp^+}}\bigr\rangle_{\lambda,x}$,
where $\langle\cdot,\cdot\rangle_{\lambda}=\langle\cdot,\cdot\rangle_{\lambda,x}$ is the weighted Bergman inner product given in~(\ref{Bergman_inner_prod}),
with the variable of integration $x=x_1+x_2$.

\begin{Theorem}\label{thm_topterm_simple}
Let $\Re\lambda>p-1$, $\bk\in\BZ_{++}^{r_2}$, and let $f(x_2)\in\cP_\bk\bigl(\fp^+_2\bigr)$. Then we have
\[
\bigl\langle f(x_2),{\rm e}^{(x|\overline{z})_{\fp^+}}\bigr\rangle_{\lambda,x}\big|_{z_1=0}=C_{\varepsilon_2}^{d,d_2}(\lambda,\bk)f(z_2),
\]
where, for $\varepsilon_2=1$,
\begin{align}
C_1^{d,d_2}(\lambda,\bk)&=\frac{\bigl(\lambda+k_1-\frac{d-d_2}{2}\bigr)_{k_2}}{(\lambda)_{k_1+k_2}\left(\lambda-\frac{d}{2}\right)_{k_2}}, && r_2=2, \nonumber \\
C_1^{d,d_2}(\lambda,\bk)&=\frac{\prod_{1\le i< j\le r_2}\left(\lambda-\frac{d}{4}(i+j-2)\right)_{k_i+k_j}}
{\prod_{1\le i< j\le r_2+1}\left(\lambda-\frac{d}{4}(i+j-3)\right)_{k_i+k_j}} \notag \\
&=\frac{\prod_{a=2}^{2r_2-2}\prod_{i=\max\{1,a+1-r_2\}}^{\lfloor a/2\rfloor}\left(\lambda-\frac{d}{4}(a-1)\right)_{k_i+k_{a+1-i}}}
{\prod_{a=1}^{2r_2-1}\prod_{i=\max\{1,a+1-r_2\}}^{\lceil a/2\rceil}\left(\lambda-\frac{d}{4}(a-1)\right)_{k_i+k_{a+2-i}}}, && d_2=d/2, \label{const1}
\end{align}
and for $\varepsilon_2=2$,
\begin{align}
&C_2^{d,d_2}(\lambda,\bk)
=\frac{2^{|\bk|}\prod_{1\le i< j\le r_2}(2\lambda-1-d(i+j-1))_{k_i+k_j}}{\prod_{1\le i<j\le r_2+1}(2\lambda-1-d(i+j-2))_{k_i+k_j}}
\frac{\prod_{i=1}^{r_2}\left(\lambda-\frac{1}{2}-\frac{d}{2}(2i-1)\right)_{k_i}}{\prod_{i=1}^{r_2}(\lambda-d(i-1))_{k_i}} \notag \\
&\qquad{}=\frac{\prod_{1\le i< j\le r_2}\left(\lambda-\frac{d}{2}(i+j-1)\right)_{\left\lfloor\frac{k_i+k_j}{2}\right\rfloor}
\prod_{1\le i\le j\le r_2}\left(\lambda-\frac{1}{2}-\frac{d}{2}(i+j-1)\right)_{\left\lceil\frac{k_i+k_j}{2}\right\rceil}}
{\prod_{1\le i\le j\le r_2+1}\!\left(\lambda-\frac{d}{2}(i+j-2)\right)_{\left\lfloor\frac{k_i+k_j}{2}\right\rfloor}
\prod_{1\le i< j\le r_2+1}\!\left(\lambda-\frac{1}{2}-\frac{d}{2}(i+j-2)\right)_{\left\lceil\frac{k_i+k_j}{2}\right\rceil}} \notag \\
&\qquad{}=\frac{\prod_{a=3}^{2r_2-1}\prod_{i=\max\{1,a-r_2\}}^{\lceil a/2\rceil-1}
\left(\lambda-\frac{d}{2}(a-1)\right)_{\left\lfloor\frac{k_i+k_{a-i}}{2}\right\rfloor}}
{\prod_{a=1}^{2r_2}\prod_{i=\max\{1,a-r_2\}}^{\lceil a/2\rceil}
\left(\lambda-\frac{d}{2}(a-1)\right)_{\left\lfloor\frac{k_i+k_{a+1-i}}{2}\right\rfloor}} \notag \\
&\qquad\quad \ {}\times\frac{\prod_{a=2}^{2r_2}\prod_{i=\max\{1,a-r_2\}}^{\lfloor a/2\rfloor}
\left(\lambda-\frac{1}{2}-\frac{d}{2}(a-1)\right)_{\left\lceil\frac{k_i+k_{a-i}}{2}\right\rceil}}
{\prod_{a=2}^{2r_2}\prod_{i=\max\{1,a-r_2\}}^{\lfloor a/2\rfloor}
\left(\lambda-\frac{1}{2}-\frac{d}{2}(a-1)\right)_{\left\lceil\frac{k_i+k_{a+1-i}}{2}\right\rceil}}. \label{const2}
\end{align}
Here we put $k_{r_2+1}:=0$.
\end{Theorem}

\begin{Theorem}\label{thm_poles_simple}
For $\bk\in\BZ_{++}^{r_2}$, $\varepsilon_2=1,2$, we define $\phi_{\varepsilon_2}(\bk)\in\BZ_{++}^{\varepsilon_2 r_2}$ by
\begin{align}
&\phi_1(\bk)_a:=\min\{ k_i+k_j\mid 1\le i<j\le r_2+1,\, i+j=2a+1\}, && 1\le a\le r_2, \label{phi1}\\
&\phi_2(\bk)_a:=\min\left\{ \left\lfloor\frac{k_i+k_j}{2}\right\rfloor\ \middle|\ 1\le i\le j\le r_2+1,\ i+j=a+1\right\}, && 1\le a\le 2r_2, \label{phi2}
\end{align}
where $k_{r_2+1}:=0$. Then for $f(x_2)\in\cP_\bk\bigl(\fp^+_2\bigr)$, as a~function of $\lambda$,
\begin{equation}\label{formula_poles_simple}
(\lambda)_{\phi_{\varepsilon_2}(\bk),d}\bigl\langle f(x_2),{\rm e}^{(x|\overline{z})_{\fp^+}}\bigr\rangle_{\lambda,x}
\end{equation}
is holomorphically continued for all $\BC$.
\end{Theorem}

\begin{Remark}\label{rem_poles_simple}\samepage\quad
\begin{enumerate}
\item By~\cite[Corollary 6.6]{N2}, if $f(x_2)\ne 0$ and if $\bk\in\BZ_{++}^{r_2}$ satisfies ``$\varepsilon_2=1$, $\bk=(\underline{k}_a,\underline{0}_{r_2-a})$, $1\le a\le r_2$''
or ``$\varepsilon_2=1$, $\bk=(k_1,\underline{k_2}_{a-1},\underline{0}_{r_2-a})$, $2\le a\le r_2,\ a\colon\text{even}$''
or ``$\varepsilon_2=2$, $\bk=(\underline{k+1}_l,\underline{k}_{a-l},\underline{0}_{r_2-a})$, $0\le l<a\le r_2$'',
then~(\ref{formula_poles_simple}) gives a~non-zero polynomial in $z\in\fp^+$ for all $\lambda\in\BC$.
\item The restriction of Theorem~\ref{thm_poles_simple} to $z_1=0$, the holomorphy of
\[ (\lambda)_{\phi_{\varepsilon_2}(\bk),d}\bigl\langle f(x_2),{\rm e}^{(x|\overline{z})_{\fp^+}}\bigr\rangle_{\lambda,x}\big|_{z_1=0} \]
follows from Theorem~\ref{thm_topterm_simple}, since it is easily proved that $(\lambda)_{\phi_{\varepsilon_2}(\bk),d}C_{\varepsilon_2}^{d,d_2}(\lambda,\bk)$ is holomorphic for all $\lambda\in\BC$.
\item Under Corollary~\ref{cor_FK}, Theorem~\ref{thm_poles_simple} is equivalent to the inclusion
\begin{equation}\label{Ktype_incl_simple}
\cP_\bk\bigl(\fp^+_2\bigr)\subset\bigoplus_{\substack{\bm\in\BZ_{++}^{\varepsilon_2 r_2} \\ m_a\le \phi_{\varepsilon_2}(\bk)_a}}\cP_\bm(\fp^+).
\end{equation}
\end{enumerate}
Here $(\lambda)_{\bm,d}$ is as in~(\ref{Pochhammer}), and $\underline{k}_a:=(\underbrace{k,\dots,k}_a)$.
\end{Remark}

\begin{proof}[Proof of Theorem~\ref{thm_topterm_simple}]
By the last paragraph of Section~\ref{subsection_reduction}, we may assume $\fp^+$, $\fp^+_2$ are of tube type, so that $r=\varepsilon_2r_2$ holds,
and by the $K_1$-equivariance, may assume $f(x_2)=\Delta_\bk^{\fn^+_2}(x_2)$.

First suppose $\varepsilon_2=1$, $r_2=2$. Then the theorem follows from~\cite[Theorems 6.3\,(1), Corollary~6.5\,(2)]{N2},
\begin{gather}
 \bigl\langle \Delta_\bk^{\fn^+_2}(x_2),{\rm e}^{(x|\overline{z})_{\fp^+}}\bigr\rangle_{\lambda,x}
=\frac{\det_{\fn^+}(z)^{-\lambda+\frac{d}{2}+1}}{(\lambda)_{(k_1+k_2,2k_2),d}}\det_{\fn^+_2}\left(\frac{\partial}{\partial z_2}\right)^{k_2}
\det_{\fn^+}(z)^{\lambda+2k_2-\frac{d}{2}-1}\Delta_{(k_1-k_2,0)}^{\fn^+_2}(z_2) \nonumber \\
\qquad{} =\frac{1}{(\lambda)_{(k_1+k_2,2k_2),d}}\det_{\fn^+}(z)^{-\lambda+\frac{d}{2}+1}\det_{\fn^+_2}\left(\frac{\partial}{\partial z_2}\right)^{k_2}\det_{\fn^+_2}(z_2)^{\lambda+k_2-\frac{d}{2}-1} \notag \\
\qquad{} \eqspace{}\times{}_1F_0\,\biggl(-\lambda-2k_2+\frac{d}{2}+1 ;-\frac{\det_{\fn^+}(z_1)}{\det_{\fn^+_2}(z_2)}\biggr)\,\Delta_\bk^{\fn^+_2}(z_2) \nonumber\\
\qquad{}=\frac{\bigl(\lambda+k_1-\frac{d-d_2}{2}\bigr)_{k_2}}{(\lambda)_{(k_1+k_2,k_2),d}}\det_{\fn^+}(z)^{-\lambda+\frac{d}{2}+1}\det_{\fn^+_2}(z_2)^{\lambda-\frac{d}{2}-1} \nonumber \\
\qquad{}\eqspace{}\times{}_2F_1\,\biggl(\begin{matrix} -\lambda-k_2+\frac{d}{2}+1,\, -\lambda-k_1+\frac{d-d_2}{2}+1 \\ -\lambda-k_1-k_2+\frac{d-d_2}{2}+1 \end{matrix};
-\frac{\det_{\fn^+}(z_1)}{\det_{\fn^+_2}(z_2)}\biggr)\,\Delta_\bk^{\fn^+_2}(z_2) \nonumber \\
\qquad{} =\frac{\bigl(\lambda+k_1-\frac{d-d_2}{2}\bigr)_{k_2}}{(\lambda)_{(k_1+k_2,k_2),d}}
{}_2F_1\,\biggl(\begin{matrix} -k_2,\, -k_1-\frac{d_2}{2} \\ -\lambda-k_1-k_2+\frac{d-d_2}{2}+1 \end{matrix}; -\frac{\det_{\fn^+}(z_1)}{\det_{\fn^+_2}(z_2)}\biggr)\,
\Delta_\bk^{\fn^+_2}(z_2), \label{rank2_2F1}
\end{gather}
where ${}_1F_0(\alpha;t):=\sum_{m=0}^\infty\frac{(\alpha)_m}{m!}t^m$,
${}_2F_1\left(\begin{smallmatrix}\alpha,\beta\\ \gamma\end{smallmatrix};t\right):=\sum_{m=0}^\infty\frac{(\alpha)_m(\beta)_m}{(\gamma)_mm!}t^m$.

Next suppose $\varepsilon_2=1$, $d_2=d/2$. The 2nd equality of~(\ref{const1}) is easy. We prove the 1st equality by induction on $r_2$.
If $r_2=1$, then we have $\fp^+=\fp^+_2$ under the tube type assumption, and this follows from Corollary~\ref{cor_FK}.
Next we assume the theorem for $r_2-1$, and prove it for $r_2$. Then since $r=r_2$ holds, by~(\ref{key_simple}) we have
\begin{align*}
&\bigl\langle \Delta_\bk^{\fn^+_2}(x_2),{\rm e}^{(x|\overline{z})_{\fp^+}}\bigr\rangle_{\lambda,x}\big|_{z_1=0} \\
&=\frac{\det_{\fn^+}(z_2)^{-\lambda+\frac{n}{r}}}{(\lambda)_{\underline{2k_{r_2}}_{r_2},d}}\det_{\fn^+_2}\!\left(\frac{\partial}{\partial z_2}\right)^{k_{r_2}}
\det_{\fn^+}(z_2)^{\lambda+2k_{r_2}-\frac{n}{r}}
\Bigl\langle \Delta_{\bk-\underline{k_{r_2}}_{r_2}}^{\fn^+_2}(x_2),{\rm e}^{(x|\overline{z})_{\fp^+}}\Bigr\rangle_{\lambda+2k_{r_2},x}\Bigr|_{z_1=0} \\ 
&=\frac{1}{(\lambda)_{\underline{2k_{r_2}}_{r_2},d}}\det_{\fn^+_2}(z_2)^{-\lambda+\frac{d}{2}(r_2-1)+1}\det_{\fn^+_2}\left(\frac{\partial}{\partial z_2}\right)^{k_{r_2}}
\det_{\fn^+_2}(z_2)^{\lambda+2k_{r_2}-\frac{d}{2}(r_2-1)-1} \\
&\eqspace{}\times \frac{\prod_{1\le i< j\le r_2-1}\left(\lambda+2k_{r_2}-\frac{d}{4}(i+j-2)\right)_{(k_i-k_{r_2})+(k_j-k_{r_2})}}
{\prod_{1\le i< j\le r_2}\left(\lambda+2k_{r_2}-\frac{d}{4}(i+j-3)\right)_{(k_i-k_{r_2})+(k_j-k_{r_2})}}\Delta_{\bk-\underline{k_{r_2}}_{r_2}}^{\fn^+_2}(z_2) \\
&=\frac{1}{(\lambda)_{\underline{2k_{r_2}}_{r_2},d}}\prod_{i=1}^{r_2}\left(\lambda+k_i-\frac{d}{2}(r_2-1)+\frac{d}{4}(r_2-i)\right)_{k_{r_2}} \\
&\eqspace{}\times \frac{\prod_{1\le i< j\le r_2-1}\left(\lambda+2k_{r_2}-\frac{d}{4}(i+j-2)\right)_{k_i+k_j-2k_{r_2}}}
{\prod_{1\le i< j\le r_2}\left(\lambda+2k_{r_2}-\frac{d}{4}(i+j-3)\right)_{k_i+k_j-2k_{r_2}}}\det_{\fn^+_2}(z_2)^{k_{r_2}}\Delta_{\bk-\underline{k_{r_2}}_{r_2}}^{\fn^+_2}(z_2) \\
&=\frac{1}{(\lambda)_{\underline{2k_{r_2}}_{r_2},d}}
\frac{\prod_{i=1}^{r_2}\left(\lambda-\frac{d}{4}(i+r_2-2)\right)_{k_i+k_{r_2}}}{\prod_{i=1}^{r_2}\left(\lambda-\frac{d}{4}(i+(r_2+1)-3)\right)_{k_i}} \\
&\eqspace{}\times \frac{\prod_{1\le i< j\le r_2-1}\left(\lambda-\frac{d}{4}(i+j-2)\right)_{k_i+k_j}}{\prod_{1\le i< j\le r_2}\left(\lambda-\frac{d}{4}(i+j-3)\right)_{k_i+k_j}}
\frac{\prod_{1\le i< j\le r_2}\left(\lambda-\frac{d}{4}(i+j-3)\right)_{2k_{r_2}}}{\prod_{1\le i< j\le r_2-1}\left(\lambda-\frac{d}{4}(i+j-2)\right)_{2k_{r_2}}}\Delta_{\bk}^{\fn^+_2}(z_2) \\
&=\frac{\left(\lambda-\frac{d}{4}(2r_2-2)\right)_{2k_{r_2}}}{\prod_{i=1}^{r_2}\left(\lambda-\frac{d}{2}(i-1)\right)_{2k_{r_2}}} \\
&\eqspace{}\times \frac{\prod_{1\le i< j\le r_2}\left(\lambda-\frac{d}{4}(i+j-2)\right)_{k_i+k_j}}{\prod_{1\le i< j\le r_2+1}\left(\lambda-\frac{d}{4}(i+j-3)\right)_{k_i+k_j}}
\frac{\prod_{1\le i\le j\le r_2-1}\left(\lambda-\frac{d}{4}(i+j-2)\right)_{2k_{r_2}}}{\prod_{1\le i< j\le r_2-1}\left(\lambda-\frac{d}{4}(i+j-2)\right)_{2k_{r_2}}}\Delta_{\bk}^{\fn^+_2}(z_2) \\
&=\frac{\prod_{1\le i< j\le r_2}\left(\lambda-\frac{d}{4}(i+j-2)\right)_{k_i+k_j}}
{\prod_{1\le i< j\le r_2+1}\left(\lambda-\frac{d}{4}(i+j-3)\right)_{k_i+k_j}}\Delta_\bk^{\fn^+_2}(z_2),
\end{align*}
where we have used the induction hypothesis and Proposition~\ref{prop_reduction} at the 2rd equality, and~(\ref{formula_diff}) at the 3rd equality. Hence the theorem follows.

Next suppose $\varepsilon_2=2$. The 2nd equality of~(\ref{const2}) follows from the formula $(2\mu-1)_k=2^k(\mu)_{\lfloor k/2\rfloor}\left(\mu-\frac{1}{2}\right)_{\lceil k/2\rceil}$,
and the 3rd equality is easy.
Next we prove the 1st equality by induction on $r_2$. When $r_2=1$, the theorem follows from~\cite[Theorem 6.3\,(2), Corollary 6.5\,(3)]{N2},
\begin{align}
&\bigl\langle \Delta_{k_1}^{\fn^+_2}(x_2),{\rm e}^{(x|\overline{z})_{\fp^+}}\bigr\rangle_{\lambda,x}
=\frac{\det_{\fn^+}(z)^{-\lambda+\frac{d}{2}+1}}{(\lambda)_{(k_1,k_1),d}}\det_{\fn^+_2}\left(\frac{1}{2}\frac{\partial}{\partial z_2}\right)^{k_1} \det_{\fn^+}(z)^{\lambda+k_1-\frac{d}{2}-1} \notag \\
&\qquad{}=\frac{1}{(\lambda)_{(k_1,k_1),d}}\det_{\fn^+}(z)^{-\lambda+\frac{d}{2}+1}\det_{\fn^+_2}\left(\frac{1}{2}\frac{\partial}{\partial z_2}\right)^{k_1}
\det_{\fn^+_2}(z_2)^{2\left(\lambda+k_1-\frac{d}{2}-1\right)} \notag \\
&\qquad{}\eqspace{}\times{}_1F_0\,\biggl( -\lambda-k_1+\frac{d}{2}+1; -\frac{\det_{\fn^+}(z_1)}{\det_{\fn^+_2}(z_2)^2}\biggr) \notag \\
&\qquad{}=\frac{\left(\lambda+\left\lceil\frac{k_1}{2}\right\rceil-\frac{d+1}{2}\right)_{\lfloor k_1/2\rfloor}}{(\lambda)_{(k_1,\lfloor k_1/2\rfloor),d}}
\det_{\fn^+}(z)^{-\lambda+\frac{d}{2}+1}\det_{\fn^+_2}(z_2)^{2\left(\lambda-\frac{d}{2}-1\right)+k_1} \notag \\
&\qquad{}\eqspace{}\times{}_2F_1\,\biggl(\begin{matrix} -\lambda-\frac{k_1}{2}+\frac{d}{2}+1,\, -\lambda-\frac{k_1}{2}+\frac{d}{2}+\frac{3}{2} \\[1pt] -\lambda-k_1+\frac{d}{2}+\frac{3}{2} \end{matrix};
-\frac{\det_{\fn^+}(z_1)}{\det_{\fn^+_2}(z_2)^2}\biggr) \notag \\
&\qquad{}=\frac{\left(\lambda+\left\lceil\frac{k_1}{2}\right\rceil-\frac{d+1}{2}\right)_{\lfloor k_1/2\rfloor}}{(\lambda)_{(k_1,\lfloor k_1/2\rfloor),d}}
{}_2F_1\,\biggl(\begin{matrix} -\frac{k_1}{2},\, -\frac{k_1}{2}+\frac{1}{2} \\ -\lambda-k_1+\frac{d}{2}+\frac{3}{2} \end{matrix};
-\frac{\det_{\fn^+}(z_1)}{\det_{\fn^+_2}(z_2)^2}\biggr){}\det_{\fn^+_2}(z_2)^{k_1}. \label{rank2'_2F1}
\end{align}
Next we assume the theorem for $r_2-1$, and prove it for $r_2$. Then since $r=2r_2$, $d_2=2d$ hold, by~(\ref{key_simple}) we have
\begin{align*}
&\bigl\langle \Delta_\bk^{\fn^+_2}(x_2),{\rm e}^{(x|\overline{z})_{\fp^+}}\bigr\rangle_{\lambda,x}\big|_{z_1=0} \\
&=\frac{\det_{\fn^+}(z_2)^{-\lambda+\frac{n}{r}}}{(\lambda)_{\underline{k_{r_2}}_{2r_2},d}}\det_{\fn^+_2}\!\left(\frac{1}{2}\frac{\partial}{\partial z_2}\right)^{k_{r_2}}
\det_{\fn^+}(z_2)^{\lambda+k_{r_2}-\frac{n}{r}} \Bigl\langle \Delta_{\bk-\underline{k_{r_2}}_{r_2}}^{\fn^+_2}(x_2),{\rm e}^{(x|\overline{z})_{\fp^+}}\Bigr\rangle_{\lambda+k_{r_2},x}\Bigr|_{z_1=0} \\
&=\frac{1}{(\lambda)_{\underline{k_{r_2}}_{2r_2},d}}\det_{\fn^+_2}(z_2)^{2\left(-\lambda+\frac{d}{2}(2r_2-1)+1\right)}
\det_{\fn^+_2}\left(\frac{1}{2}\frac{\partial}{\partial z_2}\right)^{k_{r_2}}\det_{\fn^+_2}(z_2)^{2\left(\lambda+k_{r_2}-\frac{d}{2}(2r_2-1)-1\right)} \\
&\eqspace{}\times \frac{2^{|\bk|-k_{r_2}r_2}\prod_{1\le i< j\le r_2-1}(2(\lambda+k_{r_2})-1-d(i+j-1))_{(k_i-k_{r_2})+(k_j-k_{r_2})}}
{\prod_{1\le i<j\le r_2}(2(\lambda+k_{r_2})-1-d(i+j-2))_{(k_i-k_{r_2})+(k_j-k_{r_2})}} \\
&\eqspace{}\times \frac{\prod_{i=1}^{r_2-1}\left(\lambda+k_{r_2}-\frac{1}{2}-\frac{d}{2}(2i-1)\right)_{k_i-k_{r_2}}}
{\prod_{i=1}^{r_2-1}\left(\lambda+k_{r_2}-d(i-1)\right)_{k_i-k_{r_2}}}\Delta_{\bk-\underline{k_{r_2}}_{r_2}}^{\fn^+_2}(z_2) \\
&=\frac{2^{|\bk|-2k_{r_2}r_2}}{(\lambda)_{\underline{k_{r_2}}_{2r_2},d}}\prod_{i=1}^{r_2}(2\lambda+k_i-d(2r_2-1)-2+d(r_2-i)+1)_{k_{r_2}} \\
&\eqspace{}\times \frac{\prod_{1\le i< j\le r_2-1}(2\lambda+2k_{r_2}-1-d(i+j-1))_{k_i+k_j-2k_{r_2}}}{\prod_{1\le i<j\le r_2}(2\lambda+2k_{r_2}-1-d(i+j-2))_{k_i+k_j-2k_{r_2}}} \\
&\eqspace{}\times \frac{\prod_{i=1}^{r_2-1}\left(\lambda+k_{r_2}-\frac{1}{2}-\frac{d}{2}(2i-1)\right)_{k_i-k_{r_2}}}{\prod_{i=1}^{r_2-1}(\lambda+k_{r_2}-d(i-1))_{k_i-k_{r_2}}}
\det_{\fn^+_2}(z_2)^{k_{r_2}}\Delta_{\bk-\underline{k_{r_2}}_{r_2}}^{\fn^+_2}(z_2) \\
&=\frac{2^{|\bk|-2k_{r_2}r_2}}{(\lambda)_{\underline{k_{r_2}}_{2r_2},d}}\frac{\prod_{i=1}^{r_2}(2\lambda-1-d(i+r_2-1))_{k_i+k_{r_2}}}{\prod_{i=1}^{r_2}(2\lambda-1-d(i+(r_2+1)-2))_{k_i}} \\
&\eqspace{}\times \frac{\prod_{1\le i< j\le r_2-1}(2\lambda-1-d(i+j-1))_{k_i+k_j}}{\prod_{1\le i<j\le r_2}(2\lambda-1-d(i+j-2))_{k_i+k_j}}
\frac{\prod_{1\le i<j\le r_2}(2\lambda-1-d(i+j-2))_{2k_{r_2}}}{\prod_{1\le i< j\le r_2-1}(2\lambda-1-d(i+j-1))_{2k_{r_2}}} \\
&\eqspace{}\times \frac{\prod_{i=1}^{r_2}\left(\lambda-\frac{1}{2}-\frac{d}{2}(2i-1)\right)_{k_i}}{\prod_{i=1}^{r_2}(\lambda-d(i-1))_{k_i}}
\frac{\prod_{i=1}^{r_2}(\lambda-d(i-1))_{k_{r_2}}}{\prod_{i=1}^{r_2}\left(\lambda-\frac{1}{2}-\frac{d}{2}(2i-1)\right)_{k_{r_2}}}\Delta_{\bk}^{\fn^+_2}(z_2) \\
&=\frac{2^{|\bk|-2k_{r_2}r_2}(2\lambda-1-d(2r_2-1))_{2k_{r_2}}}{\prod_{i=1}^{r_2}\left(\lambda-\frac{d}{2}(2i-2)\right)_{k_{r_2}}\prod_{i=1}^{r_2}\left(\lambda-\frac{d}{2}(2i-1)\right)_{k_{r_2}}} \\
&\eqspace{}\times \frac{\prod_{1\le i< j\le r_2}(2\lambda-1-d(i+j-1))_{k_i+k_j}}{\prod_{1\le i<j\le r_2+1}(2\lambda-1-d(i+j-2))_{k_i+k_j}}
\frac{\prod_{1\le i\le j\le r_2-1}(2\lambda-1-d(i+j-1))_{2k_{r_2}}}{\prod_{1\le i< j\le r_2-1}(2\lambda-1-d(i+j-1))_{2k_{r_2}}} \\
&\eqspace{}\times \frac{\prod_{i=1}^{r_2}\left(\lambda-\frac{1}{2}-\frac{d}{2}(2i-1)\right)_{k_i}}{\prod_{i=1}^{r_2}(\lambda-d(i-1))_{k_i}}
\frac{\prod_{i=1}^{r_2}(\lambda-d(i-1))_{k_{r_2}}}{\prod_{i=1}^{r_2}\left(\lambda-\frac{1}{2}-\frac{d}{2}(2i-1)\right)_{k_{r_2}}}\Delta_{\bk}^{\fn^+_2}(z_2) \\
&=\frac{2^{|\bk|-2k_{r_2}r_2}\prod_{i=1}^{r_2}(2\lambda-1-d(2i-1))_{2k_{r_2}}}
{\prod_{i=1}^{r_2}\left(\lambda-\frac{1}{2}-\frac{d}{2}(2i-1)\right)_{k_{r_2}}\prod_{i=1}^{r_2}\left(\lambda-\frac{d}{2}(2i-1)\right)_{k_{r_2}}} \\
&\eqspace{}\times \frac{\prod_{1\le i< j\le r_2}(2\lambda-1-d(i+j-1))_{k_i+k_j}}{\prod_{1\le i<j\le r_2+1}(2\lambda-1-d(i+j-2))_{k_i+k_j}}
\frac{\prod_{i=1}^{r_2}\left(\lambda-\frac{1}{2}-\frac{d}{2}(2i-1)\right)_{k_i}}{\prod_{i=1}^{r_2}(\lambda-d(i-1))_{k_i}}\Delta_{\bk}^{\fn^+_2}(z_2) \\
&=\frac{2^{|\bk|}\prod_{1\le i< j\le r_2}(2\lambda-1-d(i+j-1))_{k_i+k_j}}{\prod_{1\le i<j\le r_2+1}(2\lambda-1-d(i+j-2))_{k_i+k_j}}
\frac{\prod_{i=1}^{r_2}\left(\lambda-\frac{1}{2}-\frac{d}{2}(2i-1)\right)_{k_i}}{\prod_{i=1}^{r_2}(\lambda-d(i-1))_{k_i}}\Delta_\bk^{\fn^+_2}(z_2),
\end{align*}
where we have used the induction hypothesis and Proposition~\ref{prop_reduction} at the 2nd equality, and~(\ref{formula_diff}) at the 3rd equality. Hence the theorem follows.
\end{proof}

\begin{proof}[Proof of Theorem~\ref{thm_poles_simple}]
Again it is enough to prove when both $\fp^+$ and $\fp^+_2$ are of tube type
(i.e., $\bigl(\fp^+,\fp^+_2\bigr)=\bigl(\BC^n,\BC^{n''}\bigr)$, $(\Sym(2r_2,\BC),{\rm M}(r_2,\BC))$, $(\Alt(2r_2,\BC),{\rm M}(r_2,\BC))$, $({\rm M}(r_2,\BC),\Sym(r_2,\BC))$, $({\rm M}(2r_2,\BC),\Alt(2r_2,\BC))$,
$(\Herm(3,\BO)^\BC,\Alt(6,\BC))$), and may assume $f(x_2)=\Delta_\bk^{\fn^+_2}(x_2)$.
When $\bigl(\fp^+,\fp^+_2\bigr)=\bigl(\BC^n,\BC^{n''}\bigr)$, the theorem follows from~(\ref{rank2_2F1}) and~(\ref{rank2'_2F1}).
When $\bigl(\fp^+,\fp^+_2\bigr)=(\Sym(2r_2,\BC),\allowbreak {\rm M}(r_2,\BC))$, $(\Alt(2r_2,\BC),{\rm M}(r_2,\BC))$,~(\ref{Ktype_incl_simple}) follows from~(\ref{LR_Sym-M}) and~(\ref{LR_Alt-M}), and hence the theorem holds.

Next we consider the cases $\bigl(\fp^+,\fp^+_2\bigr)=({\rm M}(r_2,\BC),\Sym(r_2,\BC))$, $d=2$, $\varepsilon_2=1$, and $\bigl(\fp^+,\fp^+_2\bigr)=({\rm M}(2r_2,\BC),\Alt(2r_2,\BC))$, $d=2$, $\varepsilon_2=2$.
We prove the theorem by induction on $r_2$. When $r_2=1$, if $\varepsilon_2=1$, then since $\fp^+=\fp^+_2$ holds, the theorem follows from Corollary~\ref{cor_FK}.
Similarly, if $\varepsilon_2=2$, then the theorem follows from~(\ref{rank2'_2F1}). Next we assume the theorem for $r_2-1$, and prove it for~$r_2$.
By~(\ref{key_simple}), we have
\begin{align*}
&\det_{\fn^+}(z)^{-\lambda+\frac{n}{r}}\det_{\fn^+_2}\left(\frac{1}{\varepsilon_2}\frac{\partial}{\partial z_2}\right)^{k_{r_2}}
\det_{\fn^+}(z)^{\lambda+\frac{2k_{r_2}}{\varepsilon_2}-\frac{n}{r}} \\
&\qquad{}\eqspace{}\times\left(\lambda+\frac{2k_{r_2}}{\varepsilon_2}\right)_{\phi_{\varepsilon_2}\bigl(\bk-\underline{k_{r_2}}_{r_2}\bigr),d}
\Bigl\langle \Delta_{\bk-\underline{k_{r_2}}_{r_2}}^{\fn^+_2}(x_2),{\rm e}^{(x|\overline{z})_{\fp^+}}\Bigr\rangle_{\lambda+\frac{2k_{r_2}}{\varepsilon_2},x} \\
&\qquad{}=(\lambda)_{\underline{2k_{r_2}/\varepsilon_2}_{\varepsilon_2 r_2},d}\left(\lambda+\frac{2k_{r_2}}{\varepsilon_2}\right)_{\phi_{\varepsilon_2}\bigl(\bk-\underline{k_{r_2}}_{r_2}\bigr),d}
\bigl\langle \Delta_\bk^{\fn^+_2}(x_2),{\rm e}^{(x|\overline{z})_{\fp^+}}\bigr\rangle_{\lambda,x} \\
&\qquad{}=(\lambda)_{\phi_{\varepsilon_2}\bigl(\bk-\underline{k_{r_2}}_{r_2}\bigr)+\underline{\frac{2k_{r_2}}{\varepsilon_2}}_{\varepsilon_2 r_2},d}\,
\bigl\langle \Delta_\bk^{\fn^+_2}(x_2),{\rm e}^{(x|\overline{z})_{\fp^+}}\bigr\rangle_{\lambda,x},
\end{align*}
and by Proposition~\ref{prop_reduction} and the induction hypothesis, this is holomorphic for all $\lambda\in\BC$. Now we have
\begin{alignat*}{3}
&\phi_1\bigl(\bk-\underline{k_{r_2}}_{r_2}\bigr)_a+2k_{r_2}=\min\{k_i+k_j\mid 1\le i<j\le r_2,\, i+j=2a+1\}, && 1\le a\le r_2,& \\
&\phi_2\bigl(\bk-\underline{k_{r_2}}_{r_2}\bigr)_a+k_{r_2}=\min\left\{\left\lfloor\frac{k_i+k_j}{2}\right\rfloor\ \middle|\ 1\le i\le j\le r_2,\, i+j=a+1\right\}, \quad&& 1\le a\le 2r_2,&
\end{alignat*}
and therefore by Corollary~\ref{cor_FK}, we have
\begin{alignat*}{3}
&\cP_\bk\bigl(\fp^+_2\bigr)\subset\bigoplus_{\substack{\bm\in\BZ_{++}^{r_2} \\ m_a\le k_i+k_j,\, 1\le i<j\le r_2,\, i+j=2a+1}}\cP_\bm(\fp^+), \qquad&& \varepsilon_2=1,& \\
&\cP_\bk\bigl(\fp^+_2\bigr)\subset\bigoplus_{\substack{\bm\in\BZ_{++}^{2r_2} \\ m_a\le \big\lfloor\frac{k_i+k_j}{2}\big\rfloor, \, 1\le i\le j\le r_2,\, i+j=a+1}}\cP_\bm(\fp^+), \qquad&& \varepsilon_2=2.&
\end{alignat*}
On the other hand, by~(\ref{LR_M-Sym}) and~(\ref{LR_M-Alt}),
\begin{alignat*}{4}
&\Hom_{{\rm U}(r_2)}(\cP_\bk(\Sym(r_2,\BC)),\cP_\bm({\rm M}(r_2,\BC)))\ne\{0\} \qquad&&\text{implies} \quad&& m_{r_2-i}\le k_{r_2-2i},& \\
&\Hom_{{\rm U}(2r_2)}(\cP_\bk(\Alt(2r_2,\BC)),\cP_\bm({\rm M}(2r_2,\BC)))\ne\{0\} \qquad&&\text{implies} \quad&& m_{2r_2-i}\le \left\lfloor\frac{k_{r_2-i}}{2}\right\rfloor,&
\end{alignat*}
and combining with the above formulas, we get
\begin{align*}
&\cP_\bk(\Sym(r_2,\BC))\subset\bigoplus_{\substack{\bm\in\BZ_{++}^{r_2} \\ m_a\le k_i+k_j,\, 1\le i<j\le r_2+1,\, i+j=2a+1}}\cP_\bm({\rm M}(r_2,\BC)), \\
&\cP_\bk(\Alt(2r_2,\BC))\subset\bigoplus_{\substack{\bm\in\BZ_{++}^{2r_2} \\ m_a\le \big\lfloor\frac{k_i+k_j}{2}\big\rfloor, \, 1\le i\le j\le r_2+1,\, i+j=a+1}}\cP_\bm({\rm M}(2r_2,\BC)),
\end{align*}
with $k_{r_2+1}=0$. Hence by Corollary~\ref{cor_FK}, $(\lambda)_{\phi_{\varepsilon_2}(\bk),d}\,\bigl\langle \Delta_\bk^{\fn^+_2}(x_2),{\rm e}^{(x|\overline{z})_{\fp^+}}\bigr\rangle_{\lambda,x}$
is holomorphic for all~$\lambda\in\BC$.

Now the case $\bigl(\fp^+,\fp^+_2\bigr)=(\Herm(3,\BO)^\BC,\Alt(6,\BC))$ is remaining, but we postpone the proof for this case until Section~\ref{section_rank3}.
\end{proof}

By Theorem~\ref{thm_topterm_simple} and Proposition~\ref{prop_Plancherel}\,(1), we get the following.
\begin{Corollary}
Let $\Re\lambda>p-1$, $\bk\in\BZ_{++}^{r_2}$, and let $f(x_2)\in\cP_\bk\bigl(\fp^+_2\bigr)$. Then we have
\[ \Vert f(x_2)\Vert_{\lambda,x}^2=C_{\varepsilon_2}^{d,d_2}(\lambda,\bk)\Vert f(z_2)\Vert_{F,z}^2, \]
where $C_{\varepsilon_2}^{d,d_2}(\lambda,\bk)$ is as in~\eqref{const1} and~\eqref{const2}.
\end{Corollary}

Next we give a~rough estimate of zeroes of $(\lambda)_{\phi_{\varepsilon_2}(\bk),d}C_{\varepsilon_2}^{d,d_2}(\lambda,\bk)$.

\begin{Proposition}\label{prop_zeroes_simple}
For $\bk\in\BZ_{++}^{r_2}$, let $C_{\varepsilon_2}^{d,d_2}(\lambda,\bk)$, $\phi_{\varepsilon_2}(\bk)$ be as in~{\rm \eqref{const1}--\eqref{phi2}}.
\begin{enumerate}\itemsep=0pt
\item[$(1)$] Let $\varepsilon_2=1$, $r_2=2$. Then we have
\[
\big\{\lambda\in\BC\mid (\lambda)_{\phi_1(\bk),d}C_1^{d,d_2}(\lambda,\bk)=0 \big\}
\subset\begin{cases} \big[ \frac{d-d_2}{2}-(k_1+k_2)+1, \frac{d-d_2}{2}-k_1\big], & k_2>0, \\ \varnothing, & k_2=0. \end{cases}
\]
\item[$(2)$] Let $\varepsilon_2=1$, $d_2=d/2$. For $2\le m_1\le m_2\le r_2$, if $k_2=\cdots=k_{m_1}$ and $k_{m_2+1}=0$, then we have
\begin{align*}
&\big\{\lambda\in\BC\mid (\lambda)_{\phi_1(\bk),d}C_1^{d,d_2}(\lambda,\bk)=0 \big\} \\ &\quad\qquad{}\subset \begin{cases}
\left[ \frac{d}{4}(m_1-1)-(k_1+k_2)+1,\, \frac{d}{4}(2m_2-3)-k_{m_2-1} \right], & k_1>k_2>0, \\
\left[ \frac{d}{4}\left(2\left\lceil \frac{m_1}{2}\right\rceil-1\right)-(k_1+k_2)+1,\, \frac{d}{4}(2m_2-3)-k_{m_2-1} \right], & k_1=k_2>0, \\
\varnothing, & k_2=0. \end{cases}
\end{align*}
\item[$(3)$] Let $\varepsilon_2=2$. For $1\le m_1,m_2,m_3\le r_2$ with $m_1,m_2\le m_3$, if $k_1-k_{m_1}\le 1$, $k_a=1$ for $m_2+1\le a\le m_3$ and $k_{m_3+1}=0$, then we have
\begin{align*}
&\big\{\lambda\in\BC\mid (\lambda)_{\phi_2(\bk),d}C_2^{d,d_2}(\lambda,\bk)=0 \big\} \\
&\quad\qquad{}\subset\begin{cases} \left[ \frac{d}{2}m_1-k_1+\frac{3}{2},\, \frac{d}{2}(m_2+m_3-1)-\frac{k_{m_2}}{2}+\frac{1}{2}\right], & k_1\ge 2,\, m_2<m_3, \\[4pt]
\left[ \frac{d}{2}m_1-k_1+\frac{3}{2},\, \frac{d}{2}(2m_3-1)-\left\lceil\frac{k_{m_3}}{2}\right\rceil+\frac{1}{2}\right], & k_1\ge 2,\, m_2=m_3, \\
\varnothing, & k_1\le 1. \end{cases}
\end{align*}
\end{enumerate}
Especially, for $m=1,2,\dots,\varepsilon_2r_2-1$, if $\phi_{\varepsilon_2}(\bk)_{m+1}=0$ and $\phi_{\varepsilon_2}(\bk)_m\ne 0$,
then $C_{\varepsilon_2}^{d,d_2}(\lambda,\bk)$ is non-zero holomorphic at $\lambda=\frac{d}{2}m$, and has a~pole at $\lambda=\frac{d}{2}(m-1)$.
\end{Proposition}

\begin{proof}
(1) Clear from $(\lambda)_{\phi_1(\bk),d}C_1^{d,d_2}(\lambda,\bk)=\bigl(\lambda+k_1-\frac{d-d_2}{2}\bigr)_{k_2}$.

(2) By direct computation, we have
\[
\phi_1(\bk)_a=\begin{cases} k_1+k_2, & a=1, \\ 2k_2, & 2\le a\le \lfloor m_1/2\rfloor, \\ \min\{k_1+k_{m_1+1},2k_2\}, & a=\frac{m_1+1}{2}, \\ 0, & a\ge m_2+1, \end{cases}
\]
and
\begin{align*}
&C_1^{d,d/2}(\lambda,\bk)=C_1^{d,d/2}(\lambda,(k_1,\dots,k_{m_2})) \\
&\quad{}=\frac{1}{\prod_{a=1}^{\lfloor m_1/2\rfloor}\left(\lambda-\frac{d}{2}(a-1)\right)_{\phi_1(\bk)_a}}
\frac{\prod_{a=2\lceil m_1/2\rceil}^{2m_2-2}\prod_{i=\max\{1,a+1-m_2\}}^{\lfloor a/2\rfloor}\left(\lambda-\frac{d}{4}(a-1)\right)_{k_i+k_{a+1-i}}}
{\prod_{a=2\lceil m_1/2\rceil}^{2m_2-1}\prod_{i=\max\{1,a+1-m_2\}}^{\lceil a/2\rceil}\left(\lambda-\frac{d}{4}(a-1)\right)_{k_i+k_{a+2-i}}} \\
&\quad{}\eqspace{}\times \begin{cases}
\ds \frac{\left(\lambda-\frac{d}{4}(m_1-1)\right)_{k_1+k_2}}{\left(\lambda-\frac{d}{4}(m_1-1)\right)_{k_1+k_{m_1+1}}\left(\lambda-\frac{d}{4}(m_1-1)\right)_{2k_2}}, & m_1\colon\text{odd}, \\
1, & m_1\colon\text{even}. \end{cases}
\end{align*}
For $2\le a\le 2m_2-2$, let
\begin{align*}
\begin{split}
&\phi'(\bk)_a:=\begin{cases} \min\{ k_i+k_j\mid 1\le i<j\le m_2+1,\, i+j=a+2\}, & a\colon\text{even}, \\ {\min}_2\{ k_i+k_j\mid 1\le i<j\le m_2+1,\, i+j=a+2\}, & a\colon\text{odd}, \end{cases} \\
&\psi'(\bk)_a:=\max\{ k_i+k_j\mid 1\le i<j\le m_2,\, i+j=a+1\},
\end{split}
\end{align*}
where ${\min}_2$ denotes the second smallest element. Then we have
\begin{align*}
&\big\{ \lambda\in\BC\mid (\lambda)_{\phi_1(\bk),d}C_1^{d,d/2}(\lambda,\bk)=0\big\} \\
&\quad{}=\bigcup_{a=2\lceil m_1/2\rceil+1,\text{odd}}^{2m_2-3}\left\{\lambda\in\BC \,\middle|\, \begin{array}{@{}l@{}} \left(\lambda-\frac{d}{4}(a-1)\right)_{\phi_1(\bk)_{\lceil a/2\rceil}} \\
\ds {}\times \frac{\prod_{i=\max\{1,a+1-r_2\}}^{\lfloor a/2\rfloor}\left(\lambda-\frac{d}{4}(a-1)\right)_{k_i+k_{a+1-i}}}
{\prod_{i=\max\{1,a+1-r_2\}}^{\lceil a/2\rceil}\left(\lambda-\frac{d}{4}(a-1)\right)_{k_i+k_{a+2-i}}}=0 \end{array} \right\} \\
&\quad{}\eqspace{}\cup \bigcup_{a=2\lceil m_1/2\rceil,\text{even}}^{2m_2-2}\left\{\lambda\in\BC\, \middle|\,
\frac{\prod_{i=\max\{1,a+1-r_2\}}^{\lfloor a/2\rfloor}\left(\lambda-\frac{d}{4}(a-1)\right)_{k_i+k_{a+1-i}}}
{\prod_{i=\max\{1,a+1-r_2\}}^{\lceil a/2\rceil}\left(\lambda-\frac{d}{4}(a-1)\right)_{k_i+k_{a+2-i}}}=0 \right\} \\
&\quad{}\eqspace{}\cup \begin{cases} \ds \left\{\lambda\in\BC\, \middle|\,
\frac{\left(\lambda-\frac{d}{4}(m_1-1)\right)_{k_1+k_2}}{\left(\lambda-\frac{d}{4}(m_1-1)\right)_{\max\{k_1+k_{m_1+1},2k_2\}}}=0 \right\}, & m_1\colon\text{odd}, \\
\varnothing, & m_1\colon\text{even} \end{cases} \\
&\quad{}\subset\bigcup_{a=2\lceil m_1/2\rceil}^{2m_2-2}\left\{\lambda\in\BC\, \middle|\,
\frac{\left(\lambda-\frac{d}{4}(a-1)\right)_{\psi'(\bk)_a}}{\left(\lambda-\frac{d}{4}(a-1)\right)_{\phi'(\bk)_a}}=0 \right\} \\
&\quad{}\eqspace{}\cup \begin{cases} \ds \left\{\lambda\in\BC\, \middle|\,
\frac{\left(\lambda-\frac{d}{4}(m_1-1)\right)_{k_1+k_2}}{\left(\lambda-\frac{d}{4}(m_1-1)\right)_{\max\{k_1+k_{m_1+1},2k_2\}}}=0 \right\}, &
\begin{pmatrix}m_1\colon\text{odd}, \\ k_1>k_2\end{pmatrix}, \\
\varnothing, & \text{otherwise} \end{cases} \\
&\quad{}\subset\bigcup_{\substack{a=m_1 \,  (k_1>k_2) \\ a=2\lceil m_1/2\rceil\, (k_1=k_2)}}^{2m_2-2}\left\{\frac{d}{4}(a-1)-j\, \middle|\, j\in\BZ,\, \phi'(\bk)_a\le j\le \psi'(\bk)_a-1\right\} \\
&\quad{}\subset\begin{cases} \left[ \frac{d}{4}(m_1-1)-(k_1+k_2)+1,  \frac{d}{4}(2m_2-3)-k_{m_2-1} \right], & k_1>k_2, \\
\left[ \frac{d}{4}\left(2\left\lceil \frac{m_1}{2}\right\rceil-1\right)-(k_1+k_2)+1,  \frac{d}{4}(2m_2-3)-k_{m_2-1} \right], & k_1=k_2. \end{cases}
\end{align*}
When $k_2=0$, by direct computation we have
\[ (\lambda)_{\phi_1(\bk),d}C_1^{d,d/2}(\lambda,\bk)=\frac{(\lambda)_{k_1}}{(\lambda)_{k_1}}=1, \]
and this is non-zero everywhere.

If $\phi_1(\bk)_{m+1}=0$ and $\phi_1(\bk)_m\ne 0$, then $k_{m+1}=0$ and $k_m\ne 0$ hold, and by the above formula~$(\lambda)_{\phi_1(\bk),d}C_1^{d,d_2}(\lambda,\bk)$ is non-zero
at $\lambda=\frac{d}{2}m$, $\frac{d}{2}(m-1)$. Hence we get the last claim.

(3) It is enough to prove for the following three cases:
{\def\labelenumi{(\roman{enumi})}
\begin{enumerate}\itemsep=0pt
\item $1\le m_1\le m_2\le m_3\le r_2$, $k_{m_2}\ge 2$.
\item $1\le m_2< m_1=m_3\le r_2$, $k_{m_2}\ge 2$, i.e., $\bk=(\underbrace{2,\dots,2}_{m_2},\underbrace{1,\dots,1}_{m_3-m_2},\underbrace{0,\dots,0}_{r_2-m_3})$.
\item $k_1\le 1$, i.e., $\bk=(1,\dots,1,0,\dots,0)$.
\end{enumerate}}
First we consider Case (i). We take $1\le m_0\le m_1$ such that $k_a=k_1$ for $1\le a\le m_0$ and $k_a=k_1-1$ for $m_0+1\le a\le m_1$.
By direct computation we have
\[
\phi_2(\bk)_a=\begin{cases} k_1, & 1\le a\le m_0, \\ k_1-1, & m_0+1\le a\le m_1, \\ \left\lfloor \frac{k_1+k_{m_1+1}}{2} \right\rfloor, & a=m_1+1, \\
1, & 2m_2+1\le a\le m_2+m_3, \\ 0, & m_2+m_3+1\le a, \end{cases}
\]
and
\begin{gather*}
C_2^{d,2d}(\lambda,\bk)=C_2^{d,2d}(\lambda,(k_1,\dots,k_{m_3})) \\
\hphantom{C_2^{d,2d}(\lambda,\bk)}{}=\frac{1}{\prod_{a=1}^{m_1+1}\left(\lambda-\frac{d}{2}(a-1)\right)_{\phi_0(\bk)_a}}
\frac{\left(\lambda-\frac{1}{2}-\frac{d}{2}m_1\right)_{k_1}}{\left(\lambda-\frac{1}{2}-\frac{d}{2}m_1\right)_{\left\lceil\frac{k_1+k_{m_1+1}}{2}\right\rceil}} \\
\ \hphantom{C_2^{d,2d}(\lambda,\bk)=}{}\times\prod_{a=m_1+2}^{\min\{2m_2,2m_3-1\}}\left(\frac{\prod_{i=\max\{1,a-m_3\}}^{\lceil a/2\rceil-1}
\left(\lambda-\frac{d}{2}(a-1)\right)_{\left\lfloor\frac{k_i+k_{a-i}}{2}\right\rfloor}}
{\prod_{i=\max\{1,a-m_3\}}^{\lceil a/2\rceil}\left(\lambda-\frac{d}{2}(a-1)\right)_{\left\lfloor\frac{k_i+k_{a+1-i}}{2}\right\rfloor}}\right. \\
\hphantom{C_2^{d,2d}(\lambda,\bk)=\times\prod_{a=m_1+2}^{\min\{2m_2,2m_3-1\}}\Bigg(}{}\times\left.\frac{\prod_{i=\max\{1,a-m_3\}}^{\lfloor a/2\rfloor}\left(\lambda-\frac{1}{2}-\frac{d}{2}(a-1)\right)_{\left\lceil\frac{k_i+k_{a-i}}{2}\right\rceil}}
{\prod_{i=\max\{1,a-m_3\}}^{\lfloor a/2\rfloor}\left(\lambda-\frac{1}{2}-\frac{d}{2}(a-1)\right)_{\left\lceil\frac{k_i+k_{a+1-i}}{2}\right\rceil}}\right) \\
\ \hphantom{C_2^{d,2d}(\lambda,\bk)=}{}\times\begin{cases} \ds \frac{\prod_{a=\max\{2m_2,m_3\}+1}^{m_2+m_3}\left(\lambda-\frac{d}{2}(a-1)+\frac{k_{a-m_3}-1}{2}\right)}
{\prod_{a=2m_2+1}^{m_2+m_3}\left(\lambda-\frac{d}{2}(a-1)\right)}, & m_2<m_3, \\
\ds \frac{1}{\left(\lambda-\frac{d}{2}(2m_3-1)\right)_{\lfloor k_{m_3}/2\rfloor}}
\frac{\left(\lambda-\frac{1}{2}-\frac{d}{2}(2m_3-1)\right)_{k_{m_3}}}{\left(\lambda-\frac{1}{2}-\frac{d}{2}(2m_3-1)\right)_{\lceil k_{m_3}/2\rceil}}, & m_2=m_3. \end{cases}
\end{gather*}
Let
\begin{alignat*}{3}
&\phi(\bk)_a:=\min\{k_i+k_j\mid 1\le i\le j\le m_3+1,\, i+j=a+1\}, && 1\le a\le 2m_3,& \\
&\phi'(\bk)_a:={\min}_2\{k_i+k_j\mid 1\le i\le j\le m_3+1,\, i+j=a+1\},\qquad && 3\le a\le 2m_3-1,& \\
&\phi''(\bk)_a:=\min\{k_i+k_j\mid 1\le i<j\le m_3+1,\, i+j=a+1\}, && 2\le a\le 2m_3,& \\
&\psi(\bk)_a:=\max\{k_i+k_j\mid 1\le i\le j\le m_3,\, i+j=a\}, && 2\le a\le 2m_3,& \\
&\psi''(\bk)_a:=\max\{k_i+k_j\mid 1\le i<j\le m_3,\, i+j=a\}, && 3\le a\le 2m_3-1,&
\end{alignat*}
where ${\min}_2$ denotes the second smallest element, so that $\phi_2(\bk)_a=\lfloor \phi(\bk)_a/2\rfloor$ holds for $1\le a\le 2m_3$. Then since the zeroes of
\begin{align*}
&\frac{\prod_{i=\max\{1,a-m_3\}}^{\lceil a/2\rceil-1}\left(\lambda-\frac{d}{2}(a-1)\right)_{\left\lfloor\frac{k_i+k_{a-i}}{2}\right\rfloor}}
{\prod_{i=\max\{1,a-m_3\}}^{\lceil a/2\rceil}\left(\lambda-\frac{d}{2}(a-1)\right)_{\left\lfloor\frac{k_i+k_{a+1-i}}{2}\right\rfloor}}\\
&\qquad{}\times
\frac{\prod_{i=\max\{1,a-m_3\}}^{\lfloor a/2\rfloor}\left(\lambda-\frac{1}{2}-\frac{d}{2}(a-1)\right)_{\left\lceil\frac{k_i+k_{a-i}}{2}\right\rceil}}
{\prod_{i=\max\{1,a-m_3\}}^{\lfloor a/2\rfloor}\left(\lambda-\frac{1}{2}-\frac{d}{2}(a-1)\right)_{\left\lceil\frac{k_i+k_{a+1-i}}{2}\right\rceil}}
\times \left(\lambda-\frac{d}{2}(a-1)\right)_{\phi_2(\bk)_a}
\end{align*}
are contained in
\begin{gather*}
\left\{\lambda\in\BC\, \middle|\,
\frac{\left(\lambda-\frac{d}{2}(a-1)\right)_{\lfloor \psi''(\bk)_a/ 2\rfloor}}
{\left(\lambda-\frac{d}{2}(a-1)\right)_{\lfloor \phi'(\bk)_a/2 \rfloor}}
\frac{\left(\lambda-\frac{1}{2}-\frac{d}{2}(a-1)\right)_{\lceil \psi(\bk)_a/2 \rceil}}
{\left(\lambda-\frac{1}{2}-\frac{d}{2}(a-1)\right)_{\lceil \phi''(\bk)_a/2 \rceil}}=0\right\} \\
\qquad\qquad{}\subset \left\{\lambda\in\BC\ \middle|\ \frac{(2\lambda-1-d(a-1))_{\psi(\bk)_a}}{(2\lambda-1-d(a-1))_{\phi(\bk)_a}}=0\right\} \\
\qquad\qquad{}=\left\{\frac{d}{2}(a-1)-\frac{j}{2}+\frac{1}{2}\ \middle|\ j\in\BZ,\, \phi(\bk)_a\le j\le \psi(\bk)_a-1 \right\}
\end{gather*}
for $m_1+2\le a\le \min\{2m_2,2m_3-1\}$, we have
\begin{align*}
&\big\{\lambda\in\BC\mid (\lambda)_{\phi_2(\bk),d}C_2^{d,2d}(\lambda,\bk)=0\big\} \\
&\subset \left\{\frac{d}{2}m_1-j+\frac{1}{2}\ \middle|\ j\in\BZ,\, \left\lceil\frac{k_1+k_{m_1+1}}{2}\right\rceil\le j\le k_1-1 \right\} \\
&\eqspace{}\cup\bigcup_{a=m_1+2}^{\min\{2m_2,2m_3-1\}}
\left\{\frac{d}{2}(a-1)-\frac{j}{2}+\frac{1}{2}\ \middle|\ j\in\BZ,\, \phi(\bk)_a\le j\le \psi(\bk)_a-1 \right\} \\
&\eqspace{}\cup\begin{cases} \left\{\frac{d}{2}(a-1)-\frac{k_{a-m_3}}{2}+\frac{1}{2}\ \middle|\ a\in\BZ,\, \max\{2m_2,m_3\}+1\le a\le m_2+m_3\right\}, & m_2<m_3, \\[5pt]
\left\{\frac{d}{2}(2m_3-1)-j+\frac{1}{2}\ \middle|\ j\in\BZ,\, \left\lceil\frac{k_{m_3}}{2}\right\rceil\le j\le k_{m_3}-1 \right\}, & m_2=m_3
\end{cases} \\
&\subset\begin{cases} \left[ \frac{d}{2}m_1-k_1+\frac{3}{2},\, \frac{d}{2}(m_2+m_3-1)-\frac{k_{m_2}}{2}+\frac{1}{2}\right], & m_2<m_3, \\[5pt]
\left[ \frac{d}{2}m_1-k_1+\frac{3}{2},\, \frac{d}{2}(2m_3-1)-\left\lceil\frac{k_{m_3}}{2}\right\rceil+\frac{1}{2}\right], & m_2=m_3. \end{cases}
\end{align*}
Next we consider Case (ii). By direct computation we have
\[
\phi_2(\bk)=(\underbrace{2,\dots,2}_{m_2},\underbrace{1,\dots,1}_{m_3},\underbrace{0,\dots,0}_{2r_2-m_2-m_3}), \qquad
C_1^{d,2d}(\lambda,\bk)=\frac{\prod_{a=m_3+1}^{m_2+m_3}\left(\lambda+\frac{1}{2}-\frac{d}{2}(a-1)\right)}{(\lambda)_{\phi_2(\bk),d}},
\]
and since $m_1=m_3$, we get
\begin{align*}
\big\{\lambda\in\BC\mid (\lambda)_{\phi_2(\bk),d}C_2^{d,2d}(\lambda,\bk)=0\big\}
&=\left\{\frac{d}{2}(a-1)-\frac{1}{2}\, \middle|\, a\in\BZ,\, m_3+1\le a\le m_2+m_3\right\} \\
&\subset\left[ \frac{d}{2}m_1-2+\frac{3}{2},\, \frac{d}{2}(m_2+m_3-1)-\frac{2}{2}+\frac{1}{2}\right].
\end{align*}
Similarly, for Case~(iii), when $k_1=\cdots=k_{m_3}=1$ and $k_{m_3+1}=0$, $0\le m_3\le r_2$, by direct computation we have
\[ (\lambda)_{\phi_2(\bk),d}C_1^{d,2d}(\lambda,\bk)=\frac{\prod_{a=1}^{m_3}\left(\lambda-\frac{d}{2}(a-1)\right)}{\prod_{a=1}^{m_3}\left(\lambda-\frac{d}{2}(a-1)\right)}=1, \]
and this is non-zero everywhere.

If $\phi_2(\bk)_{m+1}=0$ and $\phi_2(\bk)_m\ne 0$, then $k_{m_2}\ge 2$, $k_{m_2+1}=\cdots=k_{m_3}=1$ and $k_{m_3+1}=0$ hold for some $m_2\le m_3$ with $m_2+m_3=m$, and by the above formula
$(\lambda)_{\phi_2(\bk),d}C_2^{d,2d}(\lambda,\bk)$ is non-zero at $\lambda=\frac{d}{2}m$, $\frac{d}{2}(m-1)$. Hence we get the last claim.
\end{proof}

Especially, if $\phi_{\varepsilon_2}(\bk)_{a+1}=0$ and $\phi_{\varepsilon_2}(\bk)_a\ne 0$, then for a~non-zero polynomial $f(x_2)\in\cP_\bk\bigl(\fp^+_2\bigr)$,
$\bigl\langle f(x_2),{\rm e}^{(x|\overline{z})_{\fp^+}}\bigr\rangle_{\lambda,x}$ has a~pole at $\lambda=\frac{d}{2}(a-1)$,
and combining with Corollary~\ref{cor_FK} and~(\ref{Ktype_incl_simple}), we have
\[ \cP_\bk\bigl(\fp^+_2\bigr)\subset\bigoplus_{\substack{\bm\in\BZ_{++}^{\varepsilon_2 r_2} \\ m_{a+1}=0}}\cP_\bm(\fp^+), \qquad
\cP_\bk\bigl(\fp^+_2\bigr)\not\subset\bigoplus_{\substack{\bm\in\BZ_{++}^{\varepsilon_2 r_2} \\ m_a=0}}\cP_\bm(\fp^+), \]
that is, for $a=0,1,\dots,r-1$, we have{\samepage
\begin{align}
\cP\bigl(\fp^+_2\bigr)\cap \bigoplus_{\substack{\bm\in\BZ_{++}^r \\ m_{a+1}=0}}\cP_\bm(\fp^+)
&=\bigoplus_{\substack{\bk\in\BZ_{++}^{r_2} \\ \phi_{\varepsilon_2}(\bk)_{a+1}=0}}\cP_\bk\bigl(\fp^+_2\bigr) \notag\\
&=\begin{cases}
\ds \bigoplus_{\substack{\bk\in\BZ_{++}^{r_2} \\ k_{a+1}=0}}\cP_\bk\bigl(\fp^+_2\bigr), & \varepsilon_2=1, \\
\ds \bigoplus_{\substack{0\le a_1\le a_2\le r_2 \\ a_1+a_2\le a}}\, \bigoplus_{\substack{\bk\in\BZ_{++}^{a_1} \\ k_{a_1}\ge 2}}
\cP_{(\bk,\underline{1}_{a_2-a_1},\underline{0}_{r_2-a_2})}\bigl(\fp^+_2\bigr), & \varepsilon_2=2, \end{cases} \label{formula_discWallach_simple}
\end{align}
where we set $\phi_{\varepsilon_2}(\bk)_{\varepsilon_2r_2+1}=\cdots=\phi_{\varepsilon_2}(\bk)_r:=0$ and $k_0:=+\infty$.}

Next we consider general $\bk\in\BZ_{++}^{r_2}$. Then again by the above proposition,
$\frac{1}{C_{\varepsilon_2}^{d,d_2}(\lambda,\bk)}\bigl\langle f(x_2), \allowbreak {\rm e}^{(x|\overline{z})_{\fp^+}}\bigr\rangle_{\lambda,x}$ is holomorphic
for $\Re\lambda\ge \frac{d}{2}(\varepsilon_2r_2-1)+1-\big\lfloor \frac{k_{r_2}}{\varepsilon_2}\big\rfloor$.
Then by Theorem~\ref{thm_factorize}, comparing the top terms, we get the following.
\begin{Theorem}\label{thm_factorize_simple}
Suppose $\fp^+$, $\fp^+_2$ are of tube type, and consider the meromorphic continuation of~$\langle\cdot,\cdot\rangle_\lambda$.
Then for $\bk\in\BZ_{++}^{r_2}$, $a=1,2,\dots,\lfloor k_{r_2}/\varepsilon_2\rfloor$ and for $f(x_2)\in\cP_\bk\bigl(\fp^+_2\bigr)$, we have
\begin{gather*}
\frac{1}{C_{\varepsilon_2}^{d,d_2}(\lambda,\bk)}\bigl\langle f(x_2),{\rm e}^{(x|\overline{z})_{\fp^+}}\bigr\rangle_{\lambda,x}\biggr|_{\lambda=\frac{n}{r}-a} \\
\qquad\qquad\quad{}=\frac{\det_{\fn^+}(z)^a}{C_{\varepsilon_2}^{d,d_2}\bigl(\frac{n}{r}+a,\bk-\underline{\varepsilon_2a}_{r_2}\bigr)}
\bigl\langle \det_{\fn^+_2}(x_2)^{-\varepsilon_2a}f(x_2),{\rm e}^{(x|\overline{z})_{\fp^+}}\bigr\rangle_{\frac{n}{r}+a,x}.
\end{gather*}
\end{Theorem}

\subsection[Results on restriction of $\cH_\lambda(D)$ to subgroups]{Results on restriction of $\boldsymbol{\cH_\lambda(D)}$ to subgroups}

Next we consider the decomposition of the holomorphic discrete series representation of scalar type $\cH_\lambda(D)$ under the subgroup $\widetilde{G}_1\subset\widetilde{G}$.
By Theorem~\ref{thm_HKKS}, we have
\[ \cH_\lambda(D)|_{\widetilde{G}_1}\simeq\hsum_{\bk\in\BZ_{++}^{r_2}}\cH_{\varepsilon_1\lambda}\bigl(D_1,\cP_\bk\bigl(\fp^+_2\bigr)\bigr), \]
where
\begin{align*}
&\cH_{\varepsilon_1\lambda}\bigl(D_1,\cP_\bk\bigl(\fp^+_2\bigr)\bigr) \\
&\simeq\begin{cases}
\cH_{\lambda+|\bk|}\bigl(D_{{\rm SO}_0(2,n')}\bigr)\boxtimes V_{(k_1-k_2,0,\dots,0)}^{[n'']\vee} & (\text{Case }1), \\
\cH_{\lambda}\bigl(D_{{\rm Sp}(r',\BR)},V_\bk^{(r')\vee}\bigr)\hboxtimes \cH_{\lambda}\bigl(D_{{\rm Sp}(r'',\BR)},V_\bk^{(r'')\vee}\bigr) & (\text{Case }2), \\
\cH_{\lambda}\bigl(D_{{\rm SO}^*(2s')},V_\bk^{(s')\vee}\bigr)\hboxtimes \cH_{\lambda}\bigl(D_{{\rm SO}^*(2s'')},V_\bk^{(s'')\vee}\bigr) & (\text{Case }3), \\
\cH_{2\lambda}\bigl(D_{{\rm SO}^*(2r)},V_{2\bk}^{(r)\vee}\bigr) & (\text{Case }4), \\
\cH_{\lambda}\bigl(D_{{\rm Sp}(r,\BR)},V_{\bk^2}^{(r)\vee}\bigr) & (\text{Case }5), \\
\cH_{\lambda}\bigl(D_{{\rm SU}(2,6)},\BC\boxtimes V_{(k_1+k_2,k_1+k_3,k_2+k_3)^2}^{(6)}\bigr) & (\text{Case }6), \\
\cH_{\lambda}\Bigl(D_{{\rm SO}^*(12)},V_{\left(\frac{k_1+k_2}{2},\frac{k_1+k_2}{2},\frac{k_1+k_2}{2},\frac{k_1+k_2}{2},\frac{k_1-k_2}{2},\frac{-k_1+k_2}{2}\right)}^{(6)\vee}\Bigr)
\boxtimes V_{(k_1,k_2)}^{(2)\vee} & (\text{Case }7), \\
\cH_{\lambda+|\bk|}(D_{{\rm SL}(2,\BR)})\hspace{-1pt}\hboxtimes\hspace{-1pt}
\cH_{\lambda+\frac{|\bk|}{2}}\!\Bigl(D_{{\rm Spin}_0(2,10)},V_{\left(\frac{k_1+k_2}{2},\frac{k_1-k_2}{2},\frac{k_1-k_2}{2},\frac{k_1-k_2}{2},\frac{k_1-k_2}{2}\right)}^{[10]\vee}\Bigr)\! & (\text{Case }8), \\
\cH_{\lambda}\bigl(D_{{\rm SU}(2,4)},\BC\boxtimes V_{(k_1+k_2,k_1+k_2,k_1,k_2)}^{(4)}\bigr)\boxtimes V_{(k_1,k_2)}^{(2)\vee} & (\text{Case }9), \\
\cH_{\lambda+\frac{|\bk|}{2}}(D_{{\rm SL}(2,\BR)})\hboxtimes \cH_{\lambda}\bigl(D_{{\rm SU}(1,5)},\BC\boxtimes V_{(k_1+k_2,k_1,k_1,k_2,k_2)}^{(5)}\bigr) & (\text{Case }10).
\end{cases}
\end{align*}
For each $\bk\in\BZ_{++}^{r_2}$, let $V_\bk$ be an abstract $K_1$-module isomorphic to $\cP_\bk\bigl(\fp^+_2\bigr)$,
let $\Vert\cdot\Vert_{\varepsilon_1\lambda,\bk}$ be the $\widetilde{G}_1$-invariant norm on $\cH_{\varepsilon_1\lambda}(D_1,V_\bk)$
normalized such that $\Vert v\Vert_{\varepsilon_1\lambda,\bk}=|v|_{V_\bk}$ holds for all constant functions $v\in V_\bk$,
and for $\lambda>p-1$ let
\[
\cF_{\lambda,\bk}^\downarrow\colon \ \cH_\lambda(D)|_{\widetilde{G}_1}\longrightarrow \cH_{\varepsilon_1\lambda}(D_1,V_\bk)
\]
be the symmetry breaking operator given in~(\ref{SBO1}) and~(\ref{SBO2}), using a~vector-valued polynomial
$\rK_{\bk}(x_2)\in \bigl(\cP_\bk\bigl(\fp^+_2\bigr)\otimes \overline{V_\bk}\bigr)^{K_1}$ satisfying
\[
\big|\langle f(x_2), \rK_{\bk}(x_2)\rangle_{F,x}\big|_{V_\bk}=\Vert f(z_2)\Vert_{F,z}, \qquad f(x_2)\in\cP_\bk\bigl(\fp^+_2\bigr),
\]
so that
\[
\big\Vert \cF_{\lambda,\bk}^\downarrow (f(x_2))\big\Vert_{\varepsilon_1\lambda,\bk}=\Vert f(z_2)\Vert_{F,z}, \qquad f(x_2)\in\cP_\bk\bigl(\fp^+_2\bigr)
\]
holds. Also, when $\fp^+_2$ is of tube type, we fix $[K_1,K_1]$-isomorphisms
$V_{\bk+\underline{a}_{r_2}}\simeq V_\bk$ for each $a\in\BZ_{>0}$.
Then by Proposition~\ref{prop_Plancherel}\,(3) and Theorem~\ref{thm_topterm_simple}, the following Parseval--Plancherel-type formula holds.
\begin{Corollary}\label{cor_Plancherel_simple}
For $\lambda>p-1$ and for $f\in\cH_\lambda(D)$, we have
\[ \Vert f\Vert_{\lambda}^2=\sum_{\bk\in\BZ_{++}^{r_2}}C_{\varepsilon_2}^{d,d_2}(\lambda,\bk)
\big\Vert \cF_{\lambda,\bk}^\downarrow f\big\Vert_{\varepsilon_1\lambda,\bk}^2, \]
where $C_{\varepsilon_2}^{d,d_2}(\lambda,\bk)$ is as in~\eqref{const1} and~\eqref{const2}.
\end{Corollary}

Next we consider the meromorphic continuation for smaller $\lambda$.
Then by Propositions~\ref{prop_poles},~\ref{prop_zeroes_simple}, Theorems~\ref{thm_poles_simple},~\ref{thm_factorize_simple} and the formula~(\ref{formula_discWallach_simple}), we have the following.
\begin{Corollary}\label{cor_submodule_simple}
For $\bk\in\BZ_{++}^{r_2}$, let $\phi_{\varepsilon_2}(\bk)\in\BZ_{++}^{\varepsilon_2 r_2}$ be as in~\eqref{phi1} and~\eqref{phi2},
and set $\phi_{\varepsilon_2}(\bk)_{\varepsilon_2 r_2+1}\allowbreak=\cdots=\phi_{\varepsilon_2}(\bk)_r:=0$.
\begin{enumerate}\itemsep=0pt
\item[$(1)$] For $a=1,2,\dots,r$,
\[
{\rm d}\tau_\lambda(\cU(\fg_1))\cP_\bk\bigl(\fp^+_2\bigr)\subset M_a^\fg(\lambda)
\]
holds if
\[ \lambda\in \frac{d}{2}(a-1)-\phi_{\varepsilon_2}(\bk)_a-\BZ_{\ge 0}, \]
where $M_a^\fg(\lambda)\subset\cO_\lambda(D)_{\widetilde{K}}$ is the $\bigl(\fg,\widetilde{K}\bigr)$-submodule given in~\eqref{submodule}.
\item[$(2)$] For $a=0,1,\dots,r-1$, we have
\begin{align*}
\cH_{\frac{d}{2}a}(D)|_{\widetilde{G}_1}&\simeq\hsum_{\substack{\bk\in\BZ_{++}^{r_2} \\ \phi_{\varepsilon_2}(\bk)_{a+1}=0}}
\cH_{\varepsilon_1\frac{d}{2}a}\bigl(D_1,\cP_\bk\bigl(\fp^+_2\bigr)\bigr) \\
&=\begin{cases}
\ds \hspace{5pt}\hsum_{\substack{\bk\in\BZ_{++}^{r_2} \\ k_{a+1}=0}}\cH_{\varepsilon_1\frac{d}{2}a}\bigl(D_1,\cP_\bk\bigl(\fp^+_2\bigr)\bigr), & \varepsilon_2=1, \\
\ds \hspace{5pt}\hsum_{\substack{0\le a_1\le a_2\le r_2 \\ a_1+a_2\le a}}\hspace{6pt}\hsum_{\substack{\bk\in\BZ_{++}^{a_1} \\ k_{a_1}\ge 2}}
\cH_{\varepsilon_1\frac{d}{2}a}\bigl(D_1,\cP_{(\bk,\underline{1}_{a_2-a_1},\underline{0}_{r_2-a_2})}\bigl(\fp^+_2\bigr)\bigr), & \varepsilon_2=2, \end{cases}
\end{align*}
where we set $k_0:=+\infty$.
\item[$(3)$] For $a=0,1,\dots,r-1$, if $\phi_{\varepsilon_2}(\bk)_{a+1}=0$, then $\cF_{\lambda,\bk}^\downarrow$ is holomorphic at $\lambda=\frac{d}{2}a$,
and its restriction gives the symmetry breaking operator
\[
\cF_{\frac{d}{2}a,\bk}^\downarrow
\colon\ \cH_{\frac{d}{2}a}(D)|_{\widetilde{G}_1}\longrightarrow \cH_{\varepsilon_1\frac{d}{2}a}(D_1,V_\bk).
\]
\item[$(4)$] Suppose $\fp^+$, $\fp^+_2$ are of tube type. For $a=1,2,\dots,\lfloor k_{r_2}/\varepsilon_2\rfloor$, if
\[
\rK_{\bk}(x_2)=c\det_{\fn^+_2}(x_2)^{\varepsilon_2a}\rK_{\bk-\underline{\varepsilon_2a}_{r_2}}(x_2),
\]
then we have
\begin{gather*}
\cF_{\frac{n}{r}-a,\bk}^\downarrow=c\cF_{\frac{n}{r}+a,\bk-\underline{\varepsilon_2a}_{r_2}}^\downarrow\circ\det_{\fn^-}\left(\frac{\partial}{\partial x}\right)^a\colon \\
\cO_{\frac{n}{r}-a}(D)|_{\widetilde{G}_1}\longrightarrow \cO_{\varepsilon_1\left(\frac{n}{r}-a\right)}(D_1,V_\bk)
\simeq \cO_{\varepsilon_1\left(\frac{n}{r}+a\right)}\bigl(D_1,V_{\bk-\underline{\varepsilon_2a}_{r_2}}\bigr).
\end{gather*}
\end{enumerate}
\end{Corollary}
If $\bk$ satisfies the condition in Remark~\ref{rem_poles_simple}\,(1), then ``only if'' in Corollary~\ref{cor_submodule_simple}\,(1) also holds.
Also, Corollary~\ref{cor_submodule_simple}\,(2) for $\lambda=\frac{d}{2}$ ($a=1$ case) is earlier given in~\cite{MO2},
and that for $(G,G_1)=({\rm SU}(r,r),{\rm SO}^*(2r))$ is earlier given in~\cite{Sek}. See also~\cite{KO} for $(G,G_1)=({\rm O}(p,q),{\rm O}(p,q')\times {\rm O}(q''))$ case.
The parameter sets in Corollary~\ref{cor_submodule_simple}\,(2) also appear in Howe's correspondence for the dual pairs $({\rm SO}^*(2r_2),{\rm Sp}(a))$, $({\rm Sp}(r_2,\BR),{\rm O}(a))$ (see, e.g.,~\cite{Ad, KV, Prz}),
and especially, we can prove~(2) for Cases 2--5 by using the seesaw dual pair theory (see, e.g.,~\cite[Section 3]{HTW},~\cite{Ku}) as in~\cite{MO2}.

\section[Case $\fp^+_2$ is simple of rank 3]{Case $\boldsymbol{\fp^+_2}$ is simple of rank 3}\label{section_rank3}

In the previous section, we skipped the proof of Theorem~\ref{thm_poles_simple} for $\bigl(\fp^+,\fp^+_2\bigr)=(\Herm(3,\BO)^\BC,\allowbreak\Alt(6,\BC))$.
In this section we assume that $\fp^+_2$ is simple, $\fp^+$ and $\fp^+_2$ are both of tube type, and $\rank\fp^+=\rank\fp^+_2=3$, that is, we treat the cases
\begin{gather*}
\bigl(\fp^+,\fp^+_1,\fp^+_2\bigr)=\bigl(\fp^+,(\fp^+)^\sigma,(\fp^+)^{-\sigma}\bigr) \\
\qquad\qquad{}=\bigl(\Herm(3,\BF)^\BC,\Alt(3,\BF')^\BC,\Herm(3,\BF')^\BC\bigr), \qquad (\BF,\BF')=(\BC,\BR),(\BH,\BC),(\BO,\BH) \\
\qquad\qquad{}\simeq \begin{cases} ({\rm M}(3,\BC),\Alt(3,\BC),\Sym(3,\BC)) & (\text{Case }1,\, (\BF,\BF')=(\BC,\BR)), \\
(\Alt(6,\BC),\Alt(3,\BC)\oplus\Alt(3,\BC),{\rm M}(3,\BC)) & (\text{Case }2,\, (\BF,\BF')=(\BH,\BC)), \\
\bigl(\Herm(3,\BO)^\BC,{\rm M}(2,6;\BC),\Alt(6,\BC)\bigr) & (\text{Case }3,\, (\BF,\BF')=(\BO,\BH)). \end{cases}
\end{gather*}
Then the corresponding symmetric pairs are
\[
(G,G_1)= \begin{cases} ({\rm SU}(3,3),{\rm SO}^*(6)) & (\text{Case }1), \\
({\rm SO}^*(12),{\rm SO}^*(6)\times {\rm SO}^*(6)) & (\text{Case }2), \\
\bigl(E_{7(-25)},{\rm SU}(2,6)\bigr) & (\text{Case }3), \end{cases}
\]
and for these cases we have $d=\dim_\BR\BF$, $d_2=\dim_\BR\BF'=\frac{d}{2}$, $\varepsilon_2=1$.
The purpose of this section is to prove Theorem~\ref{thm_poles_simple} for these cases, that is,
for $\bk\in\BZ_{++}^{3}$, $f(x_2)\in\cP_\bk\bigl(\fp^+_2\bigr)$, we prove that
\[
(\lambda)_{\phi_1(\bk),d}\bigl\langle f(x_2),{\rm e}^{(x|\overline{z})_{\fp^+}}\bigr\rangle_{\lambda,x}
\]
is holomorphically continued for all $\lambda\in\BC$, where
\[
\phi_1(\bk):=(k_1+k_2,\min\{k_1,k_2+k_3\},k_3)\in\BZ_{++}^3.
\]

\subsection{Preliminaries}

First we prepare some notations. We fix a~maximal tripotent $e\in\fp^+_2\subset\fp^+$, and regard~$\fp^+$,~$\fp^+_2$ as Jordan algebras of rank 3.
Let $\fn^+_2\subset\fn^+$ be the Euclidean real forms, let $\Omega\subset\fn^+$ be the symmetric cone,
and consider the $\BC$-bilinear form $(\cdot|\cdot)_{\fn^+}:=(\cdot|Q(\overline{e})\cdot)_{\fp^+}\colon\fp^+\times\fp^+\to\BC$,
and the determinant polynomials $\det_{\fn^+}(x)$, $\det_{\fn^+_2}(x_2)$ on $\fp^+$, $\fp^+_2$. Then $\det_{\fn^+_2}$ coincides with the restriction of $\det_{\fn^+}$ on $\fn^+_2$.
Let $K\subset G$, $K_1\subset G_1$ be the maximal compact subgroups, with the complexification~$K^\BC$,~$K_1^\BC$,
and let $L_2\subset K_1^\BC$ be the subgroup with the Lie algebra $\fl_2=[\fn^+_2,\fn^-_2]\simeq\BR\oplus\mathfrak{sl}(3,\BF')$, defined as in Section~\ref{subsection_KKT}.

Next we fix a~Jordan frame $\{e_1,e_2,e_3\}\subset\fp^+_2\subset\fp^+$ with $\sum_{j=1}^3 e_j=e$,
set $e^k:=\sum_{j=1}^k e_j$, $\check{e}^k:=\sum_{j=4-k}^3 e_j$, $k=1,2,3$,
and let $\Delta_k(x):=\det_{\fn^+(e^k)}(x)$, $\check{\Delta}_k(x):=\det_{\fn^+(\check{e}^k)}(x)$ be the determinant polynomials on
$\fp^+\bigl(e^k\bigr)_2,\fp^+\bigl(\check{e}^k\bigr)_2\subset\fp^+$, regarded as polynomials on $\fp^+$, as in Section~\ref{subsection_polynomials}.
Especially we have $\Delta_3(x)=\check{\Delta}_3(x)=\det_{\fn^+}(x)$.
Also, for $\bm\in\BZ_{+}^3$, $x_2\in\fp^+_2$, as in~(\ref{principal_minor}) let
\begin{align*}
\begin{split}
&\Delta^{\fn^+_2}_\bm(x_2)=\Delta_1(x_2)^{m_1-m_2}\Delta_2(x_2)^{m_2-m_3}\Delta_3(x_2)^{m_3}, \\
&\check{\Delta}^{\fn^+_2}_\bm(x_2)=\check{\Delta}_1(x_2)^{m_1-m_2}\check{\Delta}_2(x_2)^{m_2-m_3}\check{\Delta}_3(x_2)^{m_3}.
\end{split}
\end{align*}
Let $M_{L_2}A_{L_2}N_{L_2}^\top\subset L_2$ be the minimal parabolic subgroup defined as in Section~\ref{subsection_KKT}, so that
$\Delta^{\fn^+_2}_\bm(x_2)$ is relatively invariant under the action of $M_{L_2}A_{L_2}N_{L_2}^\top$.

Next, for $x\in\fp^+$, let $x^\sharp\in\fp^+$ be the adjugate element, which is characterized by
\[
x^\sharp:=\det_{\fn^+}(x)x^\itinv
\]
for invertible $x$, so that
\[
\Delta_2(x)=\check{\Delta}_1\bigl(x^\sharp\bigr), \qquad \check{\Delta}_2(x)=\Delta_1\bigl(x^\sharp\bigr)
\]
hold. Then the following holds.

\begin{Lemma}\label{lem_rank3}
Let $\fp^+=\fp^+_1\oplus\fp^+_2$ be as above.
\begin{enumerate}\itemsep=0pt
\item[$(1)$] For $x,y\in\fp^+$, we have $\det_{\fn^+}(x+y)=\det_{\fn^+}(x)+\bigl(x^\sharp|y\bigr)_{\fn^+}+\bigl(x|y^\sharp\bigr)_{\fn^+}+\det_{\fn^+}(y)$.
\item[$(2)$] For $x_1\in\fp^+_1$, $x_2\in\fp^+_2$, we have $(x_j)^\sharp\in\fp^+_2$, and $(x_1)^\sharp$ is at most of rank~$1$.
\item[$(3)$] For $x_1\in\fp^+_1$, $x_2\in\fp^+_2$, we have $\det_{\fn^+}(x_1+x_2)=\det_{\fn^+_2}(x_2)+\bigl(x_2|(x_1)^\sharp\bigr)_{\fn^+}$.
\end{enumerate}
\end{Lemma}

\begin{proof}
(1) Since $\det_{\fn^+}$ is homogeneous of degree 3, $\det_{\fn^+}(sx+ty)$, $s,t\in\BC$ is of the form
\[ \det_{\fn^+}(sx+ty)=f_3(x,y)s^3+f_2(x,y)s^2t+f_1(x,y)st^2+f_0(x,y)t^3, \]
with $f_j(x,y)\in\cP\bigl(\fp^+\oplus\fp^+\bigr)$. Then clearly we have
\begin{align*}
f_3(x,y)&=\det_{\fn^+}(sx+ty)\bigr|_{s=1,t=0}=\det_{\fn^+}(x), \\ f_0(x,y)&=\det_{\fn^+}(sx+ty)\bigr|_{s=0,t=1}=\det_{\fn^+}(y),
\end{align*}
and by~\cite[Proposition III.4.2\,(ii)]{FK}, we have
\begin{align*}
f_2(x,y)&=\frac{\partial}{\partial t}\det_{\fn^+}(sx+ty)\biggr|_{s=1,t=0}=\det_{\fn^+}(x)\bigl(x^\itinv|y\bigr)_{\fn^+}=\bigl(x^\sharp|y\bigr)_{\fn^+}, \\
f_1(x,y)&=\frac{\partial}{\partial s}\det_{\fn^+}(sx+ty)\biggr|_{s=0,t=1}=\det_{\fn^+}(y)\bigl(x|y^\itinv\bigr)_{\fn^+}=\bigl(x|y^\sharp\bigr)_{\fn^+}.
\end{align*}
Then by putting $s=t=1$, we get the desired formula.

(2) Since $\sigma_-:=-\sigma$ acts on $\fp^+$ as a~Jordan algebra automorphism, for $x_j\in\fp^+_j$ we have
\[ \sigma_-\bigl((x_j)^\sharp\bigr)=(\sigma_-(x_j))^\sharp=(\mp x_j)^\sharp=(\mp 1)^2(x_j)^\sharp=(x_j)^\sharp, \]
and hence $(x_j)^\sharp\in\fp^+_2$ holds. Also, for $x_1\in\fp^+_1$ we have
\begin{align*}
\det_{\fn^+}(x_1)&=\frac{1}{2}\left(\det_{\fn^+}(x_1)+\det_{\fn^+}(\sigma_-(x_1))\right)=\frac{1}{2}\left(\det_{\fn^+}(x_1)+\det_{\fn^+}(-x_1)\right)=0, \\
\bigl((x_1)^\sharp\bigr)^\sharp&=\det_{\fn^+}(x_1)x_1=0,
\end{align*}
and hence for any $l\in K_1^\BC$ we have
\[ \Delta_2\bigl(l(x_1)^\sharp\bigr)=\check{\Delta}_1\bigl(\bigl(l(x_1)^\sharp\bigr)^\sharp\bigr)=\check{\Delta}_1\bigl(l^\sharp\bigl((x_1)^\sharp\bigr)^\sharp\bigr)=0, \]
where $l^\sharp:=\chi(l)^2l^{\top -1}$. Thus $(x_1)^\sharp$ is at most of rank 1.

(3) For $x=x_2\in\fp^+_2$, $y=x_1\in\fp^+_1$, we have $\bigl((x_2)^\sharp|x_1\bigr)_{\fn^+}=0$ since $(x_2)^\sharp\in\fp^+_2$ is orthogonal to~$x_1\in\fp^+_1$,
and we have $\det_{\fn^+}(x_1)=0$ by the above argument. Hence the desired formula follows from (1).
\end{proof}

\subsection{Proof of theorem on poles}

Now we give the proof of Theorem~\ref{thm_poles_simple} when $\fp^+$, $\fp^+_2$ are of tube type and of rank 3.
\begin{proof}[Proof of Theorem~\ref{thm_poles_simple} for rank 3 cases]
By the $K_1$-equivariance, it is enough to prove the theorem when $f(x_2)=\Delta^{\fn^+_2}_\bk(x_2)\in\cP_\bk\bigl(\fp^+_2\bigr)$, $\bk\in\BZ_{++}^3$.
By~(\ref{key_simple}), for $z=z_1+z_2\in\Omega\subset\fn^+\subset\fp^+$, we have
\begin{align*}
&\bigl\langle \Delta_\bk^{\fn^+_2}(x_2),{\rm e}^{(x|\overline{z})_{\fp^+}}\bigr\rangle_{\lambda,x} \\
&=\frac{\det_{\fn^+}(z)^{-\lambda+d+1}}{(\lambda)_{\underline{2k_3}_3,d}}\det_{\fn^+_2}\!\left(\frac{\partial}{\partial z_2}\right)\!\!{\vphantom{\biggr)}}^{k_3} \det_{\fn^+}(z)^{\lambda+2k_3-d-1}
\bigl\langle \Delta_{(k_1-k_3,k_2-k_3,0)}^{\fn^+_2}(x_2),{\rm e}^{(x|\overline{z})_{\fp^+}}\bigr\rangle_{\lambda+2k_3,x} \\
&=\frac{1}{(\lambda)_{\underline{2k_3}_3,d}}\det_{\fn^+}(z)^{-\lambda+d+1}\det_{\fn^+_2}\left(\frac{\partial}{\partial z_2}\right)^{k_3}
\bigl(\det_{\fn^+_2}(z_2)+\bigl(z_2|z_1^\sharp\bigr)_{\fn^+}\bigr)^{\lambda+2k_3-d-1}\\
&\eqspace{}\times \frac{\left(\lambda+k_1+k_3-\frac{d}{4}\right)_{k_2-k_3}}{(\lambda+2k_3)_{(k_1+k_2-2k_3,k_2-k_3),d}}
{}_2F_1\left(\begin{matrix}-k_2+k_3,-k_1+k_3-\frac{d}{4}\\ -\lambda-k_1-k_2+\frac{d}{4}+1\end{matrix};-\frac{\Delta_2(z_1)}{\Delta_2(z_2)}\right)
\Delta_{\substack{(k_1-k_3,\hspace{7pt}\\ \ k_2-k_3,0)}}^{\fn^+_2}\!(z_2) \\
&=\frac{(-1)^{k_2-k_3}\left(-\lambda-k_1-k_2+\frac{d}{4}+1\right)_{k_2-k_3}}{(\lambda)_{(k_1+k_2,k_2+k_3,2k_3),d}}
\det_{\fn^+}(z)^{-\lambda+d+1}\det_{\fn^+_2}\left(\frac{\partial}{\partial z_2}\right)^{k_3} \\
&\eqspace{}\times \det_{\fn^+_2}(z_2)^{\lambda+k_3-d-1}\sum_{m=0}^\infty \frac{(-\lambda-2k_3+d+1)_m}{m!}{}\,\biggl(-\frac{\bigl(z_2|z_1^\sharp\bigr)_{\fn^+}}{\det_{\fn^+_2}(z_2)}\biggr)^m \\
&\eqspace{}\times \sum_{n=0}^{k_2-k_3} \frac{(-k_2+k_3)_n\left(-k_1+k_3-\frac{d}{4}\right)_n}{\left(-\lambda-k_1-k_2+\frac{d}{4}+1\right)_nn!}
\left(-\frac{\Delta_2(z_1)}{\Delta_2(z_2)}\right)^n \Delta_{(k_1,k_2,k_3)}^{\fn^+_2}(z_2),
\end{align*}
where we have used Proposition~\ref{prop_reduction},~(\ref{rank2_2F1}) and Lemma~\ref{lem_rank3}\,(3) at the 2nd equality,
and used the binomial formula and the definition of ${}_2F_1$ at the 3rd equality. Now we have
\begin{gather*}
\biggl(-\frac{\bigl(z_2|z_1^\sharp\bigr)_{\fn^+}}{\det_{\fn^+_2}(z_2)}\biggr)^m \in \cP_{2m}\bigl(\fp^+_1\bigr)\otimes\cP_{(0,-m,-m)}\bigl(\fp^+_2\bigr), \\
\left(-\frac{\Delta_2(z_1)}{\Delta_2(z_2)}\right)^n \Delta_{(k_1,k_2,k_3)}^{\fn^+_2}(z_2) \in \cP_{2n}\bigl(\fp^+_1\bigr)\otimes\cP_{(k_1-n,k_2-n,k_3)}\bigl(\fp^+_2\bigr),
\end{gather*}
where $\cP_m\bigl(\fp^+_1\bigr)$ denotes the space of all homogeneous polynomials on $\fp^+_1$ of degree $m$, since $z_1^\sharp$ is at most of rank 1.
For $\bl\in\BZ^3$ with $l_1\ge 0$, $0\le l_2\le k_1-k_2$, $n\le l_3\le k_2-k_3$, $|\bl|=m+n$, we define the polynomials
\begin{align*}
F^n_{\bk,\bl}(z_1,z_2)\in{}&\cP_{2(m+n)}\bigl(\fp^+_1\bigr)\otimes\cP_{(k_1+l_1-m-n,k_2+l_2-m-n,k_3+l_3-m-n)}\bigl(\fp^+_2\bigr) \\
&=\cP_{2|\bl|}\bigl(\fp^+_1\bigr)\otimes\cP_{(k_1-l_2-l_3,k_2-l_1-l_3,k_3-l_1-l_2)}\bigl(\fp^+_2\bigr)
\end{align*}
such that
\begin{equation}\label{def_Fkl1}
\sum_{\substack{\bl\in\BZ^3,\, |\bl|=m+n \\ 0\le l_1 \\ 0\le l_2\le k_1-k_2 \\ n\le l_3\le k_2-k_3}}F^n_{\bk,\bl}(z_1,z_2)
=\frac{1}{m!n!}\,\biggl(-\frac{\bigl(z_2|z_1^\sharp\bigr)_{\fn^+}}{\det_{\fn^+_2}(z_2)}\biggr)^m\left(-\frac{\Delta_2(z_1)}{\Delta_2(z_2)}\right)^n\Delta_{(k_1,k_2,k_3)}^{\fn^+_2}(z_2)
\end{equation}
holds. Then we have
\begin{align*}
&\bigl\langle \Delta_\bk^{\fn^+_2}(x_2),{\rm e}^{(x|\overline{z})_{\fp^+}}\bigr\rangle_{\lambda,x} \\
&\qquad{}=\frac{(-1)^{k_2-k_3}}{(\lambda)_{(k_1+k_2,k_2+k_3,2k_3),d}}\det_{\fn^+}(z)^{-\lambda+d+1}\det_{\fn^+_2}\left(\frac{\partial}{\partial z_2}\right)^{k_3}\det_{\fn^+_2}(z_2)^{\lambda+k_3-d-1} \\
&\qquad{}\eqspace{}\times \smash{\sum_{\substack{\bl\in\BZ^3,\, 0\le l_1 \\ 0\le l_2\le k_1-k_2 \\ 0\le l_3\le k_2-k_3}}}\sum_{n=0}^{l_3}
(-\lambda-2k_3+d+1)_{|\bl|-n} \left(-\lambda-k_1-k_2+n+\frac{d}{4}+1\right)_{k_2-k_3-n} \\
&\qquad{}\eqspace{}\hspace{155pt}\times (-k_2+k_3)_n\left(-k_1+k_3-\frac{d}{4}\right)_n F_{\bk,\bl}^n(z_1,z_2) \\
&\qquad{}=\frac{(-1)^{k_2-k_3}}{(\lambda)_{(k_1+k_2,k_2+k_3,2k_3),d}}\det_{\fn^+}(z)^{-\lambda+d+1}\det_{\fn^+_2}\left(\frac{\partial}{\partial z_2}\right)^{k_3}\det_{\fn^+_2}(z_2)^{\lambda+k_3-d-1} \\
&\qquad{}\eqspace{}\times \sum_{\substack{\bl\in\BZ^3,\, 0\le l_1 \\ 0\le l_2\le k_1-k_2 \\ 0\le l_3\le k_2-k_3}}
(-\lambda-2k_3+d+1)_{l_1+l_2}\left(-\lambda-k_1-k_2+l_3+\frac{d}{4}+1\right)_{k_2-k_3-l_3} \\
&\qquad{}\eqspace{}\times \sum_{n=0}^{l_3}(-k_2+k_3)_n\left(-k_1+k_3-\frac{d}{4}\right)_n\left(-\lambda-k_1-k_2+n+\frac{d}{4}+1\right)_{l_3-n} \\
&\qquad{}\eqspace{}\times (-\lambda-2k_3+l_1+l_2+d+1)_{l_3-n}F_{\bk,\bl}^n(z_1,z_2).
\end{align*}
Now we set
\begin{align}
\tilde{F}_{\bk,\bl}(\lambda;z_1,z_2): ={}& \sum_{n=0}^{l_3}(-k_2+k_3)_n\left(-k_1+k_3-\frac{d}{4}\right)_n\left(-\lambda-k_1-k_2+n+\frac{d}{4}+1\right)_{l_3-n} \notag\\
&\times (-\lambda-2k_3+l_1+l_2+d+1)_{l_3-n}F_{\bk,\bl}^n(z_1,z_2) \notag\\
&\in \BC[\lambda]_{\le 2l_3}\otimes \cP_{2|\bl|}\bigl(\fp^+_1\bigr)\otimes\cP_{(k_1-l_2-l_3,k_2-l_1-l_3,k_3-l_1-l_2)}\bigl(\fp^+_2\bigr), \label{def_Fkl2}
\end{align}
where $\BC[\lambda]_{\le l}$ denotes the space of polynomials in $\lambda$ of degree at most $l$. Then since
\begin{gather*}
\det_{\fn^+_2}\left(\frac{\partial}{\partial z_2}\right)^{k_3}\det_{\fn^+_2}(z_2)^{\lambda+k_3-d-1}\tilde{F}_{\bk,\bl}(\lambda;z_1,z_2) \\
\qquad{}=(-1)^{k_3}\left(-\lambda-k_1-k_3+l_2+l_3+\frac{d}{2}+1\right)_{k_3}\left(-\lambda-k_2-k_3+l_1+l_3+\frac{3}{4}d+1\right)_{k_3} \\
\qquad\eqspace{}\times (-\lambda-2k_3+l_1+l_2+d+1)_{k_3}\det_{\fn^+_2}(z_2)^{\lambda-d-1}\tilde{F}_{\bk,\bl}(\lambda;z_1,z_2)
\end{gather*}
holds by~(\ref{formula_diff}), we have
\begin{align}
&\bigl\langle \Delta_\bk^{\fn^+_2}(x_2),{\rm e}^{(x|\overline{z})_{\fp^+}}\bigr\rangle_{\lambda,x} \notag \\
&=\frac{(-1)^{k_2}}{(\lambda)_{(k_1+k_2,k_2+k_3,2k_3),d}}\det_{\fn^+}(z)^{-\lambda+d+1}\det_{\fn^+_2}(z_2)^{\lambda-d-1} \notag \\
&\ \hphantom{=}{}\times \sum_{\substack{\bl\in\BZ^3,\, 0\le l_1 \\ 0\le l_2\le k_1-k_2 \\ 0\le l_3\le k_2-k_3}}
(-\lambda-2k_3+d+1)_{l_1+l_2}\left(-\lambda-k_1-k_2+l_3+\frac{d}{4}+1\right)_{k_2-k_3-l_3} \notag \\
&\ \hphantom{=}{}\times \left(-\lambda-k_1-k_3+l_2+l_3+\frac{d}{2}+1\right)_{k_3}\left(-\lambda-k_2-k_3+l_1+l_3+\frac{3}{4}d+1\right)_{k_3} \notag \\
&\ \hphantom{=}{}\times  (-\lambda-2k_3+l_1+l_2+d+1)_{k_3} \tilde{F}_{\bk,\bl}(\lambda;z_1,z_2) \notag \\
&=\frac{(-1)^{k_2}(-\lambda-2k_3+d+1)_{k_3}}{(\lambda)_{(k_1+k_2,k_2+k_3,2k_3),d}}\,\biggl(1+\frac{\bigl(z_2|z_1^\sharp\bigr)_{\fn^+}}{\det_{\fn^+_2}(z_2)}\biggr)^{-\lambda+d+1} \notag \\
&\ \hphantom{=}{}\times \sum_{\substack{\bl\in\BZ^3,\, 0\le l_1 \\ 0\le l_2\le k_1-k_2 \\ 0\le l_3\le k_2-k_3}}
(-\lambda-k_3+d+1)_{l_1+l_2}\left(-\lambda-k_1-k_2+l_3+\frac{d}{4}+1\right)_{k_2-k_3-l_3} \notag \\
&\ \hphantom{=}{}\times \left(-\lambda-k_1-k_3+l_2+l_3+\frac{d}{2}+1\right)\!\!{\vphantom{\biggr)}}_{k_3}\!\!\left(-\lambda-k_2-k_3+l_1+l_3+\frac{3}{4}d+1\right)\!\!{\vphantom{\biggr)}}_{k_3}\!\!
\tilde{F}_{\bk,\bl}(\lambda;z_1,z_2) \notag \\ 
&=\frac{(-1)^{k_2-k_3}}{(\lambda)_{(k_1+k_2,k_2+k_3,k_3),d}}\,\biggl(1+\frac{\bigl(z_2|z_1^\sharp\bigr)_{\fn^+}}{\det_{\fn^+_2}(z_2)}\biggr)^{-\lambda+d+1} \notag \\
&\ \hphantom{=}{}\times \sum_{\substack{\bl\in\BZ^3,\,  0\le l_1 \\ 0\le l_2\le k_1-k_2 \\ 0\le l_3\le k_2-k_3}}
(-\lambda-k_3+d+1)_{l_1+l_2}\left(-\lambda-k_1-k_2+l_3+\frac{d}{4}+1\right)_{k_2-k_3-l_3} \label{rank3_proof1} \\
&\ \hphantom{=} {}\times \!\left(\!-\lambda-k_1-k_3+l_2+l_3+\frac{d}{2}+1 \right)\!\!{\vphantom{\biggr)}}_{k_3} \!\! \left(\!-\lambda-k_2-k_3+l_1+l_3+\frac{3}{4}d+1\right)\!\!{\vphantom{\biggr)}}_{k_3}\!
\tilde{F}_{\bk,\bl}(\lambda;z_1,z_2), \notag
\end{align}
where we have used Lemma~\ref{lem_rank3}\,(3) at the 2nd equality. Therefore,
\[ (\lambda)_{(k_1+k_2,k_2+k_3,k_3),d}\,\bigl\langle \Delta_\bk^{\fn^+_2}(x_2),{\rm e}^{(x|\overline{z})_{\fp^+}}\bigr\rangle_{\lambda,x} \]
is holomorphic for all $\lambda\in\BC$, and hence by Corollary~\ref{cor_FK}, we have
\begin{equation}\label{rank3_proof2}
\Delta_\bk^{\fn^+_2}(x_2)\in\bigoplus_{\substack{\bm\in\BZ_{++}^3,\,  m_1\le k_1+k_2 \\ m_2\le k_2+k_3, \,  m_3\le k_3}}\cP_\bm(\fp^+).
\end{equation}
On the other hand, if $k_1=k_2=k_3$, since
\[ \tilde{F}_{\underline{k_1}_3,(l_1,0,0)}(\lambda;z_1,z_2)=F_{\underline{k_1}_3,(l_1,0,0)}^0(z_1,z_2)
=\frac{1}{l_1!}\,\biggl(-\frac{\bigl(z_2|z_1^\sharp\bigr)_{\fn^+}}{\det_{\fn^+_2}(z_2)}\biggr)^{l_1}\det_{\fn^+_2}(z_2)^{k_1} \]
holds, we have
\begin{align*}
&\bigl\langle \det_{\fn^+_2}(x_2)^{k_1},{\rm e}^{(x|\overline{z})_{\fp^+}}\bigr\rangle_{\lambda,x} \\
&=\frac{\left(-\lambda-2k_1+\frac{d}{2}+1\right)_{k_1}}{(\lambda)_{(2k_1,2k_1,k_1),d}}\,\biggl(1+\frac{\bigl(z_2|z_1^\sharp\bigr)_{\fn^+}}{\det_{\fn^+_2}(z_2)}\biggr)^{-\lambda+d+1} \\
&\eqspace{}\times \sum_{l_1=0}^\infty (-\lambda-k_1+d+1)_{l_1} \left(-\lambda-2k_1+l_1+\frac{3}{4}d+1\right)_{k_1}
\frac{1}{l_1!}\,\biggl(-\frac{\bigl(z_2|z_1^\sharp\bigr)_{\fn^+}}{\det_{\fn^+_2}(z_2)}\biggr)^{l_1}\!\det_{\fn^+_2}(z_2)^{k_1} \\
&=\frac{(-1)^{k_1}\left(-\lambda-2k_1+\frac{3}{4}d+1\right)_{k_1}}{(\lambda)_{(2k_1,k_1,k_1),d}}\,\biggl(1+\frac{\bigl(z_2|z_1^\sharp\bigr)_{\fn^+}}{\det_{\fn^+_2}(z_2)}\biggr)^{-\lambda+d+1} \\
&\eqspace{}\times {}_2F_1\,\biggl(\begin{matrix} -\lambda-k_1+d+1,-\lambda-k_1+\frac{3}{4}d+1 \\ -\lambda-2k_1+\frac{3}{4}d+1 \end{matrix};
-\frac{\bigl(z_2|z_1^\sharp\bigr)_{\fn^+}}{\det_{\fn^+_2}(z_2)}\biggr){}\det_{\fn^+_2}(z_2)^{k_1} \\
&=\frac{\left(\lambda+k_1-\frac{3}{4}d\right)_{k_1}}{(\lambda)_{(2k_1,k_1,k_1),d}}{}_2F_1\,\biggl(\begin{matrix} -k_1,-k_1-\frac{d}{4} \\ -\lambda-2k_1+\frac{3}{4}d+1 \end{matrix};
-\frac{\bigl(z_2|z_1^\sharp\bigr)_{\fn^+}}{\det_{\fn^+_2}(z_2)}\biggr){}\det_{\fn^+_2}(z_2)^{k_1}
\end{align*}
as in~\cite[Corollary 6.5]{N2}, and hence
\[
(\lambda)_{(2k_1,k_1,k_1),d}\bigl\langle \det_{\fn^+_2}(x_2)^{k_1},{\rm e}^{(x|\overline{z})_{\fp^+}}\bigr\rangle_{\lambda,x}
\]
is holomorphic for all $\lambda\in\BC$, that is, we have
\[
\det_{\fn^+_2}(x_2)^{k_1}\in\bigoplus_{\substack{\bm\in\BZ_{++}^3, \, m_1\le 2k_1 \\ m_2\le k_1, \, m_3\le k_1}}\cP_\bm(\fp^+).
\]
Then for general $\bk\in\BZ_{++}^3$, by Lemma~\ref{lem_diff}, we also have
\begin{equation}\label{rank3_proof3}
\Delta_\bk^{\fn^+_2}(x_2)\in \BC\check{\Delta}_{(k_1-k_3,k_1-k_2,0)}^{\fn^+_2}\left(\frac{\partial}{\partial x_2}\right)\det_{\fn^+_2}(x_2)^{k_1}\subset
\bigoplus_{\substack{\bm\in\BZ_{++}^3, \, m_1\le 2k_1 \\ m_2\le k_1, \, m_3\le k_1}}\cP_\bm(\fp^+),
\end{equation}
and combining~(\ref{rank3_proof2}) and~(\ref{rank3_proof3}), we get
\[
\Delta_\bk^{\fn^+_2}(x_2)\in\bigoplus_{\substack{\bm\in\BZ_{++}^3, \, m_1\le k_1+k_2 \\ m_2\le \min\{k_1,k_2+k_3\}, \, m_3\le k_3}}\cP_\bm(\fp^+).
\]
Therefore, by Corollary~\ref{cor_FK},
\[
(\lambda)_{(k_1+k_2,\min\{k_1,k_2+k_3\},k_3),d} \bigl\langle \Delta_\bk^{\fn^+_2}(x_2),{\rm e}^{(x|\overline{z})_{\fp^+}}\bigr\rangle_{\lambda,x}
\]
is holomorphic for all $\lambda\in\BC$.
\end{proof}

\subsection{Conjecture on weighted Bergman inner products}

In the previous subsection we roughly computed the inner product $\bigl\langle \Delta_\bk^{\fn^+_2}(x_2),{\rm e}^{(x|\overline{z})_{\fp^+}}\bigr\rangle_{\lambda,x}$
for $\bk\in\BZ_{++}^{3}$. By~(\ref{rank3_proof1}), we have
\begin{gather*}
\bigl\langle \Delta_\bk^{\fn^+_2}(x_2),{\rm e}^{(x|\overline{z})_{\fp^+}}\bigr\rangle_{\lambda,x} \\
\qquad{}=\frac{\left(-\lambda-k_1-k_2+\frac{d}{4}+1\right)_{k_2-k_3}\left(-\lambda-k_1-k_3+\frac{d}{2}+1\right)_{k_3}\left(-\lambda-k_2-k_3+\frac{3}{4}d+1\right)_{k_3}}
{(\lambda)_{(k_1+k_2,k_2+k_3,k_3),d}} \\
\qquad{}\eqspace{}\times (-1)^{k_2-k_3}\,\biggl(1+\frac{\bigl(z_2|z_1^\sharp\bigr)_{\fn^+}}{\det_{\fn^+_2}(z_2)}\biggr)^{-\lambda+d+1}\sum_{\substack{\bl\in\BZ^3,\, 0\le l_1 \\ 0\le l_2\le k_1-k_2 \\ 0\le l_3\le k_2-k_3}}
\frac{(-\lambda-k_3+d+1)_{l_1+l_2}}{\left(-\lambda-k_1-k_2+\frac{d}{4}+1\right)_{l_3}} \\
\qquad{}\eqspace{}\times \frac{\left(-\lambda-k_1+\frac{d}{2}+1\right)_{l_2+l_3}}{\left(-\lambda-k_1-k_3+\frac{d}{2}+1\right)_{l_2+l_3}}
\frac{\left(-\lambda-k_2+\frac{3}{4}d+1\right)_{l_1+l_3}}{\left(-\lambda-k_2-k_3+\frac{3}{4}d+1\right)_{l_1+l_3}}\tilde{F}_{\bk,\bl}(\lambda;z_1,z_2) \\
\qquad{}=C_1^{d,d/2}(\lambda,\bk)\,\biggl(1+\frac{\bigl(z_2|z_1^\sharp\bigr)_{\fn^+}}{\det_{\fn^+_2}(z_2)}\biggr)^{-\lambda+d+1}
\sum_{\substack{\bl\in\BZ^3,\, 0\le l_1 \\ 0\le l_2\le k_1-k_2 \\ 0\le l_3\le k_2-k_3}}\frac{(-\lambda-k_3+d+1)_{l_1+l_2}}{\left(-\lambda-k_1-k_2+\frac{d}{4}+1\right)_{l_3}} \\
\qquad{}\eqspace{}\times \frac{\left(-\lambda-k_1+\frac{d}{2}+1\right)_{l_2+l_3}}{\left(-\lambda-k_1-k_3+\frac{d}{2}+1\right)_{l_2+l_3}}
\frac{\left(-\lambda-k_2+\frac{3}{4}d+1\right)_{l_1+l_3}}{\left(-\lambda-k_2-k_3+\frac{3}{4}d+1\right)_{l_1+l_3}}\tilde{F}_{\bk,\bl}(\lambda;z_1,z_2),
\end{gather*}
where
\begin{align}
&C_1^{d,d/2}(\lambda,\bk)=\frac{\prod_{1\le i< j\le 3}\left(\lambda-\frac{d}{4}(i+j-2)\right)_{k_i+k_j}}
{\prod_{1\le i< j\le 4}\left(\lambda-\frac{d}{4}(i+j-3)\right)_{k_i+k_j}} \notag \\
&\qquad{}=\frac{\left(\lambda+k_1+k_3-\frac{d}{4}\right)_{k_2-k_3}\left(\lambda+k_1-\frac{d}{2}\right)_{k_3}\left(\lambda+k_2-\frac{3}{4}d\right)_{k_3}}{(\lambda)_{(k_1+k_2,k_2+k_3,k_3),d}} \label{const_rank3} \\
&\qquad{}=\frac{\left(\lambda+k_1+k_3-\frac{d}{4}\right)_{k_2-k_3}\left(\lambda+\max\{k_1,k_2+k_3\}-\frac{d}{2}\right)_{\min\{k_3,k_1-k_2\}}\left(\lambda+k_2-\frac{3}{4}d\right)_{k_3}}
{(\lambda)_{(k_1+k_2,\min\{k_1,k_2+k_3\},k_3),d}} \notag
\end{align}
with $k_4:=0$. Now by the definition of $\tilde{F}_{\bk,\bl}(\lambda;z_1,z_2)$~(\ref{def_Fkl1}) and~(\ref{def_Fkl2}), we can show that
\[ \tilde{F}_{\bk,\bl}(\lambda;z_1,z_2)\in \BC[\lambda]_{\le 2l_3}\otimes \bigl(\cP_{2|\bl|}\bigl(\fp^+_1\bigr)\otimes\cP_{(k_1-l_2-l_3,k_2-l_1-l_3,k_3-l_1-l_2)}\bigl(\fp^+_2\bigr)\bigr)_\bk^{M_{L_2}N_{L_2}^\top}, \]
holds, where for $m\in\BZ_{\ge 0}$, $\bk,\bn\in\BZ_+^3$, $\cP_m\bigl(\fp^+_1\bigr)$ denotes the space of all homogeneous polynomials on $\fp^+_1$ of degree $m$, and
\begin{align*}
&\bigl(\cP_m\bigl(\fp^+_1\bigr)\otimes\cP_\bn\bigl(\fp^+_2\bigr)\bigr)_\bk^{M_{L_2}N_{L_2}^\top} \\
&:=\left\{f(x_1,x_2)\in\cP_m\bigl(\fp^+_1\bigr)\otimes\cP_\bn\bigl(\fp^+_2\bigr)\ \middle|\
\begin{matrix} f(man.x_1,man.x_2)={\rm e}^{2(t_1k_1+t_2k_2+t_3k_3)}f(x_1,x_2) \\ \bigl(m\in M_{L_2},\, a={\rm e}^{t_1h_1+t_2h_2+t_3h_3}\in A_{L_2},\, n\in N_{L_2}^\top\bigr) \end{matrix}\right\}.
\end{align*}
For Cases 1 and 2, $\cP_m\bigl(\fp^+_1\bigr)\otimes\cP_\bn\bigl(\fp^+_2\bigr)$ are given by
\begin{align*}
\cP_m\bigl(\fp^+_1\bigr)\otimes\cP_\bn\bigl(\fp^+_2\bigr)&=\cP_m(\Alt(3,\BC))\otimes\cP_\bn(\Sym(3,\BC)) \\
&\simeq V_{(m,m,0)}^{(3)\vee}\otimes V_{(2n_1,2n_2,2n_3)}^{(3)\vee} && (\text{Case }1), \\
\cP_m\bigl(\fp^+_1\bigr)\otimes\cP_\bn\bigl(\fp^+_2\bigr)&= \bigoplus_{l=0}^m \bigl(\cP_l(\Alt(3,\BC))\boxtimes\cP_{m-l}(\Alt(3,\BC))\bigr)\otimes\cP_\bn({\rm M}(3,\BC)) \\
&\simeq \bigoplus_{l=0}^m \left(V_{(l,l,0)}^{(3)\vee}\boxtimes V_{(m-l,m-l,0)}^{(3)}\right)\otimes \left(V_{(n_1,n_2,n_3)}^{(3)\vee}\boxtimes V_{(n_1,n_2,n_3)}^{(3)}\right) && (\text{Case }2),
\end{align*}
and these decompose multiplicity-freely under $K_1^\BC$ by the Pieri rule.
Hence the space $\bigl(\cP_m\bigl(\fp^+_1\bigr)\allowbreak\otimes\cP_\bn\bigl(\fp^+_2\bigr)\bigr)_\bk^{M_{L_2}N_{L_2}^\top}$ is at most 1-dimensional,
and thus there exist polynomials
\[
f_{\bk,\bl}(\lambda)\in\BC[\lambda]_{\le 2l_3}, \qquad
F_{\bk,\bl}(z_1,z_2)\in\left(\cP_{2|\bl|}\bigl(\fp^+_1\bigr)\otimes \cP_{(k_1-l_2-l_3,k_2-l_1-l_3,k_3-l_1-l_2)}\bigl(\fp^+_2\bigr)\right)^{M_{L_2}N_{L_2}^\top}_{\bk}
\]
such that
\[ \tilde{F}_{\bk,\bl}(\lambda;z_1,z_2)=f_{\bk,\bl}(\lambda)F_{\bk,\bl}(z_1,z_2) \]
holds. We expect that this also holds for Case 3. In addition, we conjecture the following.
\begin{Conjecture}
For each $\bk\in\BZ_{++}^3$, $\bl\in\BZ^3$ with $l_1\ge 0$, $0\le l_2\le k_1-k_2$, $0\le l_3\le k_2-k_3$, there exists a~polynomial
\[ F_{\bk,\bl}(z_1,z_2)\in\left(\cP_{2|\bl|}\bigl(\fp^+_1\bigr)\otimes \cP_{(k_1-l_2-l_3,k_2-l_1-l_3,k_3-l_1-l_2)}\bigl(\fp^+_2\bigr)\right)^{M_{L_2}N_{L_2}^\top}_{\bk} \]
such that for $\Re\lambda>p-1$ and for $z=z_1+z_2\in\Omega\subset\fn^+\subset\fp^+$, we have
\begin{align}
&\bigl\langle \Delta_\bk^{\fn^+_2}(x_2),{\rm e}^{(x|\overline{z})_{\fp^+}}\bigr\rangle_{\lambda,x} \notag \\
&\qquad{}=C_1^{d,d/2}(\lambda,\bk)\,\biggl(1+\frac{\bigl(z_2|z_1^\sharp\bigr)_{\fn^+}}{\det_{\fn^+_2}(z_2)}\biggr)^{-\lambda+d+1}
\sum_{\substack{\bl\in\BZ^3,\, 0\le l_1 \\ 0\le l_2\le k_1-k_2 \\ 0\le l_3\le k_2-k_3}}\frac{(-\lambda-k_3+d+1)_{l_1+l_2}}{\left(-\lambda-k_1-k_2+\frac{d}{4}+1\right)_{l_3}} \notag \\
&\qquad{}\eqspace{}\times \frac{\left(-\lambda-k_2+\frac{3}{4}d+1\right)_{l_1+l_3}}{\left(-\lambda-k_1-k_3+\frac{d}{2}+1\right)_{l_2}}
\frac{\left(-\lambda-k_1+\frac{d}{2}+1\right)_{l_2+l_3}}{\left(-\lambda-k_2-k_3+\frac{3}{4}d+1\right)_{l_1}}F_{\bk,\bl}(z_1,z_2) \label{conjecture1} \\
&\qquad{}=C_1^{d,d/2}(\lambda,\bk)\sum_{\substack{\bl\in\BZ^3,\, 0\le l_1 \\ 0\le l_2\le k_1-k_2 \\ 0\le l_3\le k_2-k_3}}\frac{(-k_3)_{l_1+l_2}}{\left(-\lambda-k_1-k_2+\frac{d}{4}+1\right)_{l_3}} \notag \\
&\qquad{}\eqspace{}\times \frac{\left(-k_2-\frac{d}{4}\right)_{l_1+l_3}}{\left(-\lambda-k_1-k_3+\frac{d}{2}+1\right)_{l_2}}
\frac{\left(-k_1-\frac{d}{2}\right)_{l_2+l_3}}{\left(-\lambda-k_2-k_3+\frac{3}{4}d+1\right)_{l_1}}F_{\bk,\bl}(z_1,z_2), \label{conjecture2}
\end{align}
where $C_1^{d,d/2}(\lambda,\bk)$ is as in~\eqref{const_rank3}.
\end{Conjecture}

By~\cite[Theorem 6.3\,(1)]{N2}, both~(\ref{conjecture1}) and~(\ref{conjecture2}) hold if $k_2=k_3$. In general,~(\ref{conjecture1}) is equivalent to
\[
f_{\bk,\bl}(\lambda)=\left(-\lambda-k_1-k_3+l_2+\frac{d}{2}+1\right)_{l_3}\left(-\lambda-k_2-k_3+l_1+\frac{3}{4}d+1\right)_{l_3}.
\]
Also, if~(\ref{conjecture2}) is true, then the symmetry breaking operator
\[
\cF_{\lambda,\bk}^\downarrow\colon \ \cH_\lambda(D)|_{\widetilde{G}_1}\longrightarrow \cH_{\varepsilon_1\lambda}(D_1,V_\bk)
\]
with $V_\bk:=\cP_\bk\bigl(\fp^+_2\bigr)$ given in~(\ref{SBO1}) and~(\ref{SBO2}), where
\begin{align*}
\cH_{\varepsilon_1\lambda}(D_1,V_\bk)
&\simeq\begin{cases}
\cH_{2\lambda}\bigl(D_{{\rm SO}^*(6)},V_{2\bk}^{(3)\vee}\bigr) & (\text{Case }1), \\
\cH_{\lambda}\bigl(D_{{\rm SO}^*(6)},V_\bk^{(3)\vee}\bigr)\hboxtimes \cH_{\lambda}\bigl(D_{{\rm SO}^*(6)},V_\bk^{(3)\vee}\bigr) & (\text{Case }2), \\
\cH_{\lambda}\bigl(D_{{\rm SU}(2,6)},\BC\boxtimes V_{(k_1+k_2,k_1+k_3,k_2+k_3)^2}^{(6)}\bigr) & (\text{Case }3), \end{cases}
\end{align*}
is given by the differential operator of the form
\begin{align*}
\bigl(\cF_{\lambda,\bk}^\downarrow f\bigr)(x_1)
&=\sum_{\substack{\bl\in\BZ^3,\, 0\le l_1 \\ 0\le l_2\le k_1-k_2 \\ 0\le l_3\le k_2-k_3}}\frac{(-k_3)_{l_1+l_2}}{\left(-\lambda-k_1-k_2+\frac{d}{4}+1\right)_{l_3}}
\frac{\left(-k_2-\frac{d}{4}\right)_{l_1+l_3}}{\left(-\lambda-k_1-k_3+\frac{d}{2}+1\right)_{l_2}} \\
&\eqspace{}\times\frac{\left(-k_1-\frac{d}{2}\right)_{l_2+l_3}}{\left(-\lambda-k_2-k_3+\frac{3}{4}d+1\right)_{l_1}}
\rK_{\bk,\bl}\left(\frac{\partial}{\partial x}\right)f(x)\biggr|_{x_2=0},
\end{align*}
where $\rK_{\bk,\bl}(z)=\rK_{\bk,\bl}(z_1,z_2)\in \cP(\fp^-,V_\bk)^{K_1}$ satisfies
\[ \bigl(v,\rK_{\bk,\bl}(\overline{z_1},\overline{z_2})\bigr)_{V_\bk}=F_{\bk,\bl}(z_1,z_2) \]
for a~suitable fixed lowest weight vector $v\in V_\bk$.
By the analytic continuation, this gives the symmetry breaking operator between the spaces of all holomorphic functions even for $\Re\lambda\le p-1$ except for the poles.
Moreover, the equality~(\ref{conjecture1}) $=$~(\ref{conjecture2}) at $\lambda=\frac{n}{r}-a$, $a=1,2,\dots,k_3$ corresponds to Theorem~\ref{thm_factorize_simple}
or Corollary~\ref{cor_submodule_simple}\,(4),
\begin{gather*}
\cF_{\frac{n}{r}-a,\bk}^\downarrow=c\cF_{\frac{n}{r}+a,\bk-\underline{a}_3}^\downarrow\circ\det_{\fn^-}\left(\frac{\partial}{\partial x}\right)^a\colon \\
\cO_{\frac{n}{r}-a}(D)|_{\widetilde{G}_1}\longrightarrow \cO_{\varepsilon_1\left(\frac{n}{r}-a\right)}(D_1,V_\bk)
\simeq \cO_{\varepsilon_1\left(\frac{n}{r}+a\right)}(D_1,V_{\bk-\underline{a}_3})
\end{gather*}
for some $c\in\BC$.

\section{Tensor product case}\label{section_tensor}

In this section we consider the direct sum $\fp^+\oplus\fp^+$ with $\fp^+$ simple, and the involution $\sigma\colon (x,y)\mapsto (y,x)$,
so that $\fp^+_1=\{(x,x)\mid x\in\fp^+\}$, $\fp^+_2=\{(x,-x)\mid x\in\fp^+\}$. Then the corresponding symmetric pair is of the form $(G\times G,\Delta(G))$.
Let $\dim\fp^+=:n$, $\rank\fp^+=:r$, and let $d$ be the number defined in~(\ref{str_const}).
This section contains some overlap with~\cite[Section 5.2]{N}. See also~\cite{PZ}.

\subsection{Results on weighted Bergman inner products}

We consider the outer tensor product $\cH_\lambda(D)\hboxtimes\cH_\mu(D)$ of holomorphic discrete series representations of scalar type,
and let $\langle\cdot,\cdot\rangle_{\lambda\otimes\mu}$ denote its inner product. For $f\in\cP_\bk(\fp^+)$, we want to compute the inner product
$\bigl\langle f(x-y),{\rm e}^{(x|\overline{z})_{\fp^+}+(y|\overline{w})_{\fp^+}}\bigr\rangle_{\lambda\otimes\mu,(x,y)}$, with the variable of inte\-gra\-tion~$(x,y)$.
To do this, for $f\in\cP(\fp^+)$, $\bm,\bn\in\BZ_{++}^r$, define $\tilde{f}_{\bm,\bn}(x,y)\in\cP_\bm(\fp^+)\otimes\cP_\bn(\fp^+)$ by
\[ f(x+y)=\sum_{\bm\in\BZ_{++}^r}\sum_{\bn\in\BZ_{++}^r}\tilde{f}_{\bm,\bn}(x,y)\in\bigoplus_{\bm\in\BZ_{++}^r}\bigoplus_{\bn\in\BZ_{++}^r}\cP_\bm(\fp^+)\otimes\cP_\bn(\fp^+). \]
If $f\in\cP_\bk(\fp^+)$, then $\tilde{f}_{\bm,\bn}(x,y)\ne 0$ holds only if $m_j\le k_j$, $n_j\le k_j$ hold for all $j=1,\dots,r$,
since the map $\cP_\bk(\fp^+)\to\cP_\bm(\fp^+)\otimes\cP_\bn(\fp^+)$, $f\mapsto \tilde{f}_{\bm,\bn}$ is $K$-equivariant,
and the restricted weights of $\cP_\bm(\fp^+)$, $\cP_\bn(\fp^+)$ sit in $\left(\frac{1}{2}\BZ_{\le 0}\right)^r$
under a~suitable identification $(\fa_\fl)^\vee\simeq \BR^r$, where $\fa_\fl\subset\fk^\BC$ is as in Section~\ref{subsection_simul_Peirce}.
Moreover, we have $\tilde{f}_{\bk,\underline{0}_r}(x,y)=f(x)$, $\tilde{f}_{\underline{0}_r,\bk}(x,y)=f(y)$.
Therefore, by Corollary~\ref{cor_FK}, we easily get the following. Here, for $\lambda\in\BC$, $\bm,\bn\in(\BZ_{\ge 0})^r$,
let $(\lambda+\bm)_{\bn,d}:=\prod_{j=1}^r\left(\lambda+m_j-\frac{d}{2}(j-1)\right)_{n_j}$.

\begin{Theorem}\label{thm_poles_tensor}
Let $\Re\lambda,\Re\mu>p-1$, $\bk\in\BZ_{++}^r$, and let $f\in\cP_\bk(\fp^+)$. Then we have
\[
\bigl\langle f(x-y),{\rm e}^{(x|\overline{z})_{\fp^+}+(y|\overline{w})_{\fp^+}}\bigr\rangle_{\lambda\otimes\mu,(x,y)}
=\sum_{\substack{\bm\in\BZ_{++}^r\\ m_j\le k_j}}\sum_{\substack{\bn\in\BZ_{++}^r\\ n_j\le k_j}}\frac{1}{(\lambda)_{\bm,d}(\mu)_{\bn,d}}\tilde{f}_{\bm,\bn}(z,-w).
\]
Especially, as a~function of $(\lambda,\mu)$,
\begin{gather}
(\lambda)_{\bk,d}(\mu)_{\bk,d}\bigl\langle f(x-y),{\rm e}^{(x|\overline{z})_{\fp^+}+(y|\overline{w})_{\fp^+}}\bigr\rangle_{\lambda\otimes\mu,(x,y)} \notag \\
\qquad{}=\sum_{\substack{\bm\in\BZ_{++}^r\\ m_j\le k_j}}\sum_{\substack{\bn\in\BZ_{++}^r\\ n_j\le k_j}}(\lambda+\bm)_{\bk-\bm,d}(\mu+\bn)_{\bk-\bn,d}\tilde{f}_{\bm,\bn}(z,-w) \label{inner_prod_tensor}
\end{gather}
is holomorphically continued for all $\BC^2$.
\end{Theorem}

If $\fp^+$ is of tube type, then for particular $\lambda,\mu\in\BC$, the analytic continuation of~(\ref{inner_prod_tensor}) is factorized as follows.
\begin{Theorem}\label{thm_factorize_tensor}
Suppose $\fp^+$ is of tube type, and consider the meromorphic continuation of $\langle\cdot,\cdot\rangle_{\lambda\otimes\mu}$ for $\lambda,\mu\in\BC$.
Then for $\bk\in\BZ_{++}^r$, $a=1,2,\dots,k_r$ and for $f\in\cP_\bk(\fp^+)$, we have
\begin{gather*}
(\lambda)_{\bk,d}\bigl\langle f(x-y),{\rm e}^{(x|\overline{z})_{\fp^+}+(y|\overline{w})_{\fp^+}}\bigr\rangle_{\lambda\otimes\mu,(x,y)}\big|_{\lambda=\frac{n}{r}-a} \\
\qquad{}=\left(\frac{n}{r}+a\right)_{\bk-\underline{a}_r,d}\det_{\fn^+}(z)^a
\bigl\langle \det_{\fn^+}(x-y)^{-a}f(x-y),{\rm e}^{(x|\overline{z})_{\fp^+}+(y|\overline{w})_{\fp^+}}\bigr\rangle_{\left(\frac{n}{r}+a\right)\otimes\mu,(x,y)}, \\
(\mu)_{\bk,d}\bigl\langle f(x-y),{\rm e}^{(x|\overline{z})_{\fp^+}+(y|\overline{w})_{\fp^+}}\bigr\rangle_{\lambda\otimes\mu,(x,y)}\big|_{\mu=\frac{n}{r}-a} \\
\qquad{}=\left(\frac{n}{r}+a\right)_{\bk-\underline{a}_r,d}\det_{\fn^+}(-w)^a
\bigl\langle \det_{\fn^+}(x-y)^{-a}f(x-y),{\rm e}^{(x|\overline{z})_{\fp^+}+(y|\overline{w})_{\fp^+}}\bigr\rangle_{\lambda\otimes\left(\frac{n}{r}+a\right),(x,y)}.
\end{gather*}
\end{Theorem}

\begin{proof}
Put $g(x):=\det_{\fn^+}(x)^{-a}f(x)\in\cP_{\bk-\underline{a}_r}(\fp^+)=\cP_{\bk-(\smash{\overbrace{\scriptstyle a,\dots,a}^r})}(\fp^+)$. Then by~\cite[Proposition~5.5]{N2}, we have
\begin{align*}
\left(\frac{n}{r}-a+\bk\right)_{\underline{a}_r,d}\tilde{g}_{\bm,\bn}(x,y)&=\left(\frac{n}{r}+\bm\right)_{\underline{a}_r,d}\det_{\fn^+}(x)^{-a}\tilde{f}_{\bm+\underline{a}_r,\bn}(x,y) \\
&=\left(\frac{n}{r}+\bn\right)_{\underline{a}_r,d}\det_{\fn^+}(y)^{-a}\tilde{f}_{\bm,\bn+\underline{a}_r}(x,y).
\end{align*}
Using this, we get
\begin{gather*}
 (\lambda)_{\bk,d}\bigl\langle f(x-y),{\rm e}^{(x|\overline{z})_{\fp^+}+(y|\overline{w})_{\fp^+}}\bigr\rangle_{\lambda\otimes\mu,(x,y)}\big|_{\lambda=\frac{n}{r}-a} \\
\qquad{} =\sum_{\substack{\bm\in\BZ_{++}^r\\ m_j\le k_j}}\sum_{\substack{\bn\in\BZ_{++}^r\\ n_j\le k_j}}\frac{(\lambda+\bm)_{\bk-\bm,d}}{(\mu)_{\bn,d}}\tilde{f}_{\bm,\bn}(z,-w)\biggr|_{\lambda=\frac{n}{r}-a} \\
\qquad{} =\sum_{\substack{\bm\in\BZ_{++}^r\\ a\le m_j\le k_j}}\sum_{\substack{\bn\in\BZ_{++}^r\\ n_j\le k_j}}\frac{\left(\frac{n}{r}-a+\bm\right)_{\bk-\bm,d}}{(\mu)_{\bn,d}}\tilde{f}_{\bm,\bn}(z,-w) \\
\qquad{} =\sum_{\substack{\bm\in\BZ_{++}^r\\ m_j\le k_j-a}}\sum_{\substack{\bn\in\BZ_{++}^r\\ n_j\le k_j}}\frac{\left(\frac{n}{r}+\bm\right)_{\bk-\underline{a}_r-\bm,d}}{(\mu)_{\bn,d}}
\tilde{f}_{\bm+\underline{a}_r,\bn}(z,-w) \\
\qquad{} =\sum_{\substack{\bm\in\BZ_{++}^r\\ m_j\le k_j-a}}\sum_{\substack{\bn\in\BZ_{++}^r\\ n_j\le k_j-a}}\frac{\left(\frac{n}{r}+\bm\right)_{\bk-\underline{a}_r-\bm,d}}{(\mu)_{\bn,d}}
\frac{\left(\frac{n}{r}-a+\bk\right)_{\underline{a}_r,d}}{\left(\frac{n}{r}+\bm\right)_{\underline{a}_r,d}}\det_{\fn^+}(z)^a\tilde{g}_{\bm,\bn}(z,-w) \\
\qquad{} =\sum_{\substack{\bm\in\BZ_{++}^r\\ m_j\le k_j-a}}\sum_{\substack{\bn\in\BZ_{++}^r\\ n_j\le k_j-a}}\frac{\left(\frac{n}{r}+a+\bm\right)_{\bk-\underline{a}_r-\bm,d}}{(\mu)_{\bn,d}}
\det_{\fn^+}(z)^a\tilde{g}_{\bm,\bn}(z,-w) \\
\qquad{} =\left(\frac{n}{r}+a\right)_{\bk-\underline{a}_r,d}\det_{\fn^+}(z)^a
\bigl\langle \det_{\fn^+}(x-y)^{-a}f(x-y),{\rm e}^{(x|\overline{z})_{\fp^+}+(y|\overline{w})_{\fp^+}}\bigr\rangle_{\left(\frac{n}{r}+a\right)\otimes\mu,(x,y)}.
\end{gather*}
The 2nd formula is also proved similarly. We can also prove this by an argument similar to Proposition~\ref{prop_det_factorize}\,(2) and Theorem~\ref{thm_factorize}
by using Theorem~\ref{thm_key_tensor}.
\end{proof}

We note that~(\ref{inner_prod_tensor}) may vanish for some $(\lambda,\mu)\in\BC^2$ in general. For such $(\lambda,\mu)\in\BC^2$ the following holds. For rank 1 case see also~\cite[Sections 8, 9]{KP2}.
\begin{Proposition}\label{prop_atzero_tensor}
Suppose $f\in\cP_\bk(\fp^+)$ is non-zero, and we consider the analytic continuation of~\eqref{inner_prod_tensor}.
If~\eqref{inner_prod_tensor} vanishes at $(\lambda,\mu)=(\lambda_0,\mu_0)\in\BC^2$,
then there exist $l_1,l_2\in\BZ_{>0}$ such that
\begin{align*}
&\lim_{\lambda\to\lambda_0}\frac{1}{(\lambda-\lambda_0)^{l_1}}
\Big((\lambda)_{\bk,d}(\mu)_{\bk,d}\bigl\langle f(x-y),{\rm e}^{(x|\overline{z})_{\fp^+}+(y|\overline{w})_{\fp^+}}\bigr\rangle_{\lambda\otimes\mu,(x,y)}\big|_{\mu=\mu_0}\Big), \\
&\lim_{\mu\to\mu_0}\frac{1}{(\mu-\mu_0)^{l_2}}
\Big((\lambda)_{\bk,d}(\mu)_{\bk,d}\bigl\langle f(x-y),{\rm e}^{(x|\overline{z})_{\fp^+}+(y|\overline{w})_{\fp^+}}\bigr\rangle_{\lambda\otimes\mu,(x,y)}\big|_{\lambda=\lambda_0}\Big)
\end{align*}
converge to non-zero polynomials. These are linearly independent in $\cP\bigl(\fp^+\oplus\fp^+\bigr)$.
\end{Proposition}
\begin{proof}
Let $\Lambda(\bk):=\big\{(\bm,\bn)\in\BZ_{++}^r\times\BZ_{++}^r\mid \tilde{f}_{\bm,\bn}(x,y)\ne 0\big\}$. Then since
\begin{gather*}
(\lambda)_{\bk,d}(\mu)_{\bk,d}\bigl\langle f(x-y),{\rm e}^{(x|\overline{z})_{\fp^+}+(y|\overline{w})_{\fp^+}}\bigr\rangle_{\lambda\otimes\mu,(x,y)}\big|_{\mu=\mu_0} \\
\qquad{}=\sum_{\substack{(\bm,\bn)\in\Lambda(\bk)\\ (\mu_0+\bn)_{\bk-\bn,d}\ne 0}}(\lambda+\bm)_{\bk-\bm,d}(\mu_0+\bn)_{\bk-\bn,d}\tilde{f}_{\bm,\bn}(x,-y)
\end{gather*}
vanishes at $\lambda=\lambda_0$, this is divisible by $(\lambda-\lambda_0)^{l_1}$ for some $l_1\in\BZ_{>0}$. Let $l_1$ be the maximum integer satisfying this. Then
\begin{gather*}
\lim_{\lambda\to\lambda_0}\frac{1}{(\lambda-\lambda_0)^{l_1}}
\Big((\lambda)_{\bk,d}(\mu)_{\bk,d}\bigl\langle f(x-y),{\rm e}^{(x|\overline{z})_{\fp^+}+(y|\overline{w})_{\fp^+}}\bigr\rangle_{\lambda\otimes\mu,(x,y)}\big|_{\mu=\mu_0}\Big) \\
\qquad{}=\lim_{\lambda\to\lambda_0}\sum_{\substack{(\bm,\bn)\in\Lambda(\bk)\\ (\mu_0+\bn)_{\bk-\bn,d}\ne 0}}\frac{(\lambda+\bm)_{\bk-\bm,d}}{(\lambda-\lambda_0)^{l_1}}(\mu_0+\bn)_{\bk-\bn,d}
\tilde{f}_{\bm,\bn}(x,-y)
\end{gather*}
is non-zero. Similarly,
\begin{gather*}
\lim_{\mu\to\mu_0}\frac{1}{(\mu-\mu_0)^{l_2}}
\Big((\lambda)_{\bk,d}(\mu)_{\bk,d}\bigl\langle f(x-y),{\rm e}^{(x|\overline{z})_{\fp^+}+(y|\overline{w})_{\fp^+}}\bigr\rangle_{\lambda\otimes\mu,(x,y)}\big|_{\lambda=\lambda_0}\Big) \\
\qquad{}=\lim_{\mu\to\mu_0}\sum_{\substack{(\bm,\bn)\in\Lambda(\bk)\\ (\lambda_0+\bm)_{\bk-\bm,d}\ne 0}}(\lambda_0+\bm)_{\bk-\bm,d}\frac{(\mu+\bn)_{\bk-\bn,d}}{(\mu-\mu_0)^{l_2}}
\tilde{f}_{\bm,\bn}(x,-y)
\end{gather*}
is non-zero for some $l_2\in\BZ_{>0}$. Now since $\tilde{f}_{\bm,\bn}(x,-y)$ are linearly independent for all $(\bm,\bn)\in\Lambda(\bk)$, and since
\begin{gather*}
\{(\bm,\bn)\in\Lambda(\bk)\mid (\mu_0+\bn)_{\bk-\bn,d}\ne 0\}\cap\{(\bm,\bn)\in\Lambda(\bk)\mid (\lambda_0+\bm)_{\bk-\bm,d}\ne 0\} \\
\qquad{}=\{(\bm,\bn)\in\Lambda(\bk)\mid (\lambda_0+\bm)_{\bk-\bm,d}(\mu_0+\bn)_{\bk-\bn,d}\ne 0\}=\varnothing
\end{gather*}
holds, the above two limits are linearly independent.
\end{proof}

Next we compute the value of the inner product at $w=-z$.
\begin{Theorem}\label{thm_topterm_tensor}
Let $\Re\lambda,\Re\mu>p-1$, $\bk\in\BZ_{++}^r$, and let $f\in\cP_\bk(\fp^+)$. Then we have
\[
\bigl\langle f(x-y), {\rm e}^{(x-y|\overline{z})_{\fp^+}}\bigr\rangle_{\lambda\otimes\mu,(x,y)}=\frac{\tilde{C}^d(\lambda,\mu,\bk)}{(\lambda)_{\bk,d}(\mu)_{\bk,d}}f(z),
\]
where
\begin{align}
\tilde{C}^d(\lambda,\mu,\bk)&=\frac{\prod_{1\le i\le j\le r}\left(\lambda+\mu-1-\frac{d}{2}(i+j-2)\right)_{k_i+k_j}}
{\prod_{1\le i<j\le r+1}\left(\lambda+\mu-1-\frac{d}{2}(i+j-3)\right)_{k_i+k_j}} \notag \\
&=\frac{\prod_{a=1}^{2r-1}\prod_{i=\max\{1,a+1-r\}}^{\lceil a/2\rceil}\left(\lambda+\mu-1-\frac{d}{2}(a-1)\right)_{k_i+k_{a+1-i}}}
{\prod_{a=1}^{2r-1}\prod_{i=\max\{1,a+1-r\}}^{\lceil a/2\rceil}\left(\lambda+\mu-1-\frac{d}{2}(a-1)\right)_{k_i+k_{a+2-i}}}. \label{const_tensor}
\end{align}
Here we put $k_{r+1}:=0$.
\end{Theorem}
\begin{proof}
The 2nd equality of~(\ref{const_tensor}) is easy. For the 1st equality,
by the last paragraph of Section~\ref{subsection_reduction}, we may assume $\fp^+$ is of tube type, and by the $K$-equivariance, may assume $f(x)=\Delta^{\fn^+}_\bk(x)$.
We prove the theorem by induction on $r$. When $\bk=(0,\dots,0)$ (``$r=0$ case''), this is clear. Next we assume the theorem for $r-1$, and prove for $r$.
Since we have defined $k_{r+1}=0$, by~(\ref{key_tensor}), we have
\begin{align*}
&\bigl\langle \Delta^{\fn^+}_\bk(x-y), {\rm e}^{(x-y|\overline{z})_{\fp^+}}\bigr\rangle_{\lambda\otimes\mu,(x,y)} \\
&=\frac{\det_{\fn^+}\hspace{-1pt}(z)^{-\lambda-\mu+\frac{2n}{r}}}{(\lambda)_{\underline{k_r}_r,d}(\mu)_{\underline{k_r}_r,d}}
\det_{\fn^+}\!\!\left(\hspace{-1pt}\frac{\partial}{\partial z}\hspace{-1pt}\right)\!\!{\vphantom{\biggr|}}^{k_r}
\!\det_{\fn^+}\hspace{-1pt}(z)^{\lambda+\mu+2k_r-\frac{2n}{r}}
\hspace{-1pt}\left\langle\hspace{-1pt} \Delta^{\fn^+}_{\bk-\underline{k_r}_r}\!(x\hspace{-1pt}-\hspace{-1pt}y), {\rm e}^{(x-y|\overline{z})_{\fp^+}}
\hspace{-1pt}\right\rangle_{\substack{(\lambda+k_r)\otimes\hspace{14pt} \\ (\mu+k_r),(x,y)}} \\
&=\frac{\det_{\fn^+}(z)^{-\lambda-\mu+d(r-1)+2}}{(\lambda)_{\underline{k_r}_r,d}(\mu)_{\underline{k_r}_r,d}(\lambda+k_r)_{\bk-\underline{k_r}_r,d}(\mu+k_r)_{\bk-\underline{k_r},d}}
\det_{\fn^+}\left(\frac{\partial}{\partial z}\right)^{k_r}\det_{\fn^+}(z)^{\lambda+\mu+2k_r-d(r-1)-2} \\
&\eqspace{}\times\frac{\prod_{1\le i\le j\le r-1}\left((\lambda+k_r)+(\mu+k_r)-1-\frac{d}{2}(i+j-2)\right)_{(k_i-k_r)+(k_j-k_r)}}
{\prod_{1\le i<j\le r}\left((\lambda+k_r)+(\mu+k_r)-1-\frac{d}{2}(i+j-3)\right)_{(k_i-k_r)+(k_j-k_r)}}\Delta^{\fn^+}_{\bk-\underline{k_r}_r}(z) \\
&=\frac{1}{(\lambda)_{\bk,d}(\mu)_{\bl,d}}\frac{\prod_{1\le i\le j\le r-1}\left(\lambda+\mu+2k_r-1-\frac{d}{2}(i+j-2)\right)_{k_i+k_j-2k_r}}
{\prod_{1\le i<j\le r}\left(\lambda+\mu+2k_r-1-\frac{d}{2}(i+j-3)\right)_{k_i+k_j-2k_r}} \\
&\eqspace{}\times\prod_{i=1}^r\left(\lambda+\mu+k_i-1-d(r-1)+\frac{d}{2}(r-i)\right)_{k_r}\det_{\fn^+}(z)^{k_r}\Delta^{\fn^+}_{\bk-\underline{k_r}_r}(z) \\
&=\frac{1}{(\lambda)_{\bk,d}(\mu)_{\bl,d}}
\frac{\prod_{1\le i\le j\le r-1}\left(\lambda+\mu-1-\frac{d}{2}(i+j-2)\right)_{k_i+k_j}}{\prod_{1\le i<j\le r}\left(\lambda+\mu-1-\frac{d}{2}(i+j-3)\right)_{k_i+k_j}}\Delta^{\fn^+}_{\bk}(z) \\
&\eqspace{}\times\frac{\prod_{1\le i<j\le r}\left(\lambda+\mu-1-\frac{d}{2}(i+j-3)\right)_{2k_r}}{\prod_{1\le i\le j\le r-1}\left(\lambda+\mu-1-\frac{d}{2}(i+j-2)\right)_{2k_r}}
\frac{\prod_{i=1}^r\left(\lambda+\mu-1-\frac{d}{2}(i+r-2)\right)_{k_i+k_r}}{\prod_{i=1}^r\left(\lambda+\mu-1-\frac{d}{2}(i+(r+1)-3)\right)_{k_i}} \\
&=\frac{1}{(\lambda)_{\bk,d}(\mu)_{\bl,d}}
\frac{\prod_{1\le i\le j\le r}\left(\lambda+\mu-1-\frac{d}{2}(i+j-2)\right)_{k_i+k_j}}{\prod_{1\le i<j\le r+1}\left(\lambda+\mu-1-\frac{d}{2}(i+j-3)\right)_{k_i+k_j}}\Delta^{\fn^+}_{\bk}(z),
\end{align*}
where we have used the induction hypothesis and Proposition~\ref{prop_reduction} at the 2nd equality, and~(\ref{formula_diff}) at the 3rd equality.
Therefore, the theorem holds for all $r$.
\end{proof}

By Theorem~\ref{thm_topterm_tensor}, we get the following.
\begin{Corollary}\label{cor_topterm_tensor}
Let $\Re\lambda,\Re\mu>p-1$, $\bk\in\BZ_{++}^r$, and let $f\in\cP_\bk(\fp^+)$. Then we have
\[
\Vert f(x-y)\Vert_{\lambda\otimes\mu,(x,y)}^2=\frac{\tilde{C}^d(\lambda,\mu,\bk)}{(\lambda)_{\bk,d}(\mu)_{\bk,d}}\Vert f\Vert_{F,\fp^+}^2,
\]
where $\tilde{C}^d(\lambda,\mu,\bk)$ is as in~\eqref{const_tensor}.
\end{Corollary}
\begin{proof}
As in the proof of Proposition~\ref{prop_Plancherel}\,(1), we have
\begin{align*}
\Vert f(x-y)\Vert_{\lambda\otimes\mu,(x,y)}^2
&=\Big\langle\bigl\langle f(x-y),{\rm e}^{\frac{1}{2}(x+y|\overline{z+w})_{\fp^+}+\frac{1}{2}(x-y|\overline{z-w})_{\fp^+}}\bigr\rangle_{\lambda\otimes\mu,(x,y)},f(z-w)\Big\rangle_{F,(z,w)} \\
&=\Big\langle\bigl\langle f(x-y),{\rm e}^{\frac{1}{2}(x-y|\overline{z-w})_{\fp^+}}\bigr\rangle_{\lambda\otimes\mu,(x,y)},f(z-w)\Big\rangle_{F,(z,w)} \\
&=\frac{\tilde{C}^d(\lambda,\mu,\bk)}{(\lambda)_{\bk,d}(\mu)_{\bk,d}}\left\langle f\left(\frac{z-w}{2}\right),f(z-w)\right\rangle_{F,(z,w)} \\
&=\frac{\tilde{C}^d(\lambda,\mu,\bk)}{(\lambda)_{\bk,d}(\mu)_{\bk,d}}\Vert f\Vert_{F,\fp^+}.
\tag*{\qed}
\end{align*}
\renewcommand{\qed}{}
\end{proof}

Next we give a~rough estimate of zeroes of $\tilde{C}^d(\lambda,\mu,\bk)$.
\begin{Proposition}\label{prop_zeroes_tensor}
For $\bk\in\BZ_{++}^r$, let $\tilde{C}^d(\lambda,\mu,\bk)$ be as in~\eqref{const_tensor}.
For $1\le m_1\le m_2\le r$, if $k_1=\cdots=k_{m_1}$ and $k_{m_2+1}=0$, then
\begin{align*}
&\big\{(\lambda,\mu)\in\BC^2\mid \tilde{C}^d(\lambda,\mu,\bk)=0\big\} \\
&\qquad{}\subset\begin{cases} \left\{(\lambda,\mu)\in\BC^2\ \middle|\ \frac{d}{2}(m_1-1)-2k_1+2\le \lambda+\mu\le d(m_2-1)-k_{m_2}+1\right\}, & k_1\ge 1, \\
\varnothing, & k_1=0. \end{cases}
\end{align*}
Especially, for $m=0,1,\dots,r-1$, if $k_{m+1}=0$, then $\tilde{C}^d(\lambda,\mu,\bk)\ne 0$ for $\lambda,\mu\ge \frac{d}{2}m$.
\end{Proposition}

\begin{proof}
By direct computation, we have
\begin{align*}
\tilde{C}^d(\lambda,\mu,\bk)&=\tilde{C}^d(\lambda,\mu,(k_1,\dots,k_{m_2})) \\
&=\prod_{a=m_1}^{2m_2-1}\prod_{i=\max\{1,a+1-m_2\}}^{\lceil a/2\rceil}
\frac{\left(\lambda+\mu-1-\frac{d}{2}(a-1)\right)_{k_i+k_{a+1-i}}}{\left(\lambda+\mu-1-\frac{d}{2}(a-1)\right)_{k_i+k_{a+2-i}}}.
\end{align*}
For $1\le a\le 2m_2-1$, let
\begin{align*}
&\phi(\bk)_a:=\min\{ k_i+k_j\mid 1\le i<j\le m_2+1,\, i+j=a+2\}, \\
&\psi(\bk)_a:=\max\{ k_i+k_j\mid 1\le i\le j\le m_2,\, i+j=a+1\}.
\end{align*}
Then we have
\begin{gather*}
\big\{(\lambda,\mu)\in\BC^2\mid \tilde{C}^d(\lambda,\mu,\bk)=0\big\} \\
\qquad{}=\bigcup_{a=m_1}^{2m_2-1}\left\{(\lambda,\mu)\in\BC^2\, \middle|\, \prod_{i=\max\{1,a+1-m_2\}}^{\lceil a/2\rceil}
\frac{\left(\lambda+\mu-1-\frac{d}{2}(a-1)\right)_{k_i+k_{a+1-i}}}{\left(\lambda+\mu-1-\frac{d}{2}(a-1)\right)_{k_i+k_{a+2-i}}}=0 \right\} \\
\qquad{}\subset \bigcup_{a=m_1}^{2m_2-1}\left\{(\lambda,\mu)\in\BC^2\, \middle|\,
\frac{\left(\lambda+\mu-1-\frac{d}{2}(a-1)\right)_{\psi(\bk)_a}}{\left(\lambda+\mu-1-\frac{d}{2}(a-1)\right)_{\phi(\bk)_a}}=0 \right\} \\
\qquad{}= \bigcup_{a=m_1}^{2m_2-1}\left\{(\lambda,\mu)\in\BC^2\, \middle|\,
\lambda+\mu=\frac{d}{2}(a-1)-j+1,\, j\in\BZ,\, \phi(\bk)_a\le j\le \psi(\bk)_a-1\right\} \\
\qquad{}\subset \left\{(\lambda,\mu)\in\BC^2\, \middle|\, \frac{d}{2}(m_1-1)-2k_1+2\le \lambda+\mu\le \frac{d}{2}(2m_2-2)-k_{m_2}+1\right\}.
\end{gather*}
When $k_1=0$, we have $\tilde{C}^d(\lambda,\mu,\bk)=1$, and this is non-zero everywhere.
The last statement is clear.
\end{proof}

\subsection[Results on restriction of $\cH_\lambda(D)\hboxtimes\cH_\mu(D)$ to subgroups]{Results on restriction of $\boldsymbol{\cH_\lambda(D)\hboxtimes\cH_\mu(D)}$ to subgroups}

Next we consider the decomposition of the tensor product representation $\cH_\lambda(D)\hboxtimes\cH_\mu(D)|_{\Delta(\widetilde{G})}\allowbreak=\cH_\lambda(D)\hotimes\cH_\mu(D)$
under the diagonal subgroup $\Delta(\widetilde{G})\subset\widetilde{G}\times\widetilde{G}$.
By Theorem~\ref{thm_HKKS}, we have
\[
\cH_\lambda(D)\hotimes\cH_\mu(D)\simeq\hsum_{\bk\in\BZ_{++}^r}\cH_{\lambda+\mu}\bigl(D,\cP_\bk(\fp^+)\bigr).
\]
For each $\bk\in\BZ_{++}^r$, let $V_\bk$ be the abstract $K$-module isomorphic to $\cP_\bk(\fp^+)$,
let $\Vert\cdot\Vert_{\nu,\bk}$ be the $\widetilde{G}$-invariant norm on $\cH_\nu(D,V_\bk)$ normalized such that $\Vert v\Vert_{\nu,\bk}=|v|_{V_\bk}$ holds for all constant functions
$v\in V_\bk$, and for $\lambda,\mu>p-1$, we consider the symmetry breaking operator
\begin{gather*}
\hat{\cF}_{\lambda,\mu,\bk}^\downarrow\colon \ \cH_\lambda(D)\hotimes\cH_\mu(D)\longrightarrow \cH_{\lambda+\mu}(D,V_\bk), \\
\bigl(\hat{\cF}_{\lambda,\mu,\bk}^\downarrow f\bigr)(x):=\hat{F}_{\lambda,\mu,\bk}\left(\frac{\partial}{\partial x},\frac{\partial}{\partial y}\right)f(x,y)\biggr|_{y=x}, \\
\hat{F}_{\lambda,\mu,\bk}(z,w):=(\lambda)_{\bk,d}(\mu)_{\bk,d}\bigl\langle {\rm e}^{(x|z)_{\fp^+}+(y|w)_{\fp^+}},\rK_\bk(x-y)\bigr\rangle_{\lambda\otimes\mu,(x,y)}\in\cP(\fp^-\oplus\fp^-,V_\bk),
\end{gather*}
where $\rK_\bk(x)\in\cP(\fp^+,\overline{V_\bk})^K$ is normalized such that
\[
\big|\langle f(x),\rK_\bk(x)\rangle_{F,\fp^+}\big|_{V_\bk}^2=\Vert f\Vert_{F,\fp^+}^2, \qquad f(x)\in\cP_\bk(\fp^+).
\]
This operator coincides with the one given in~\cite{PZ} up to constant multiple.
In our normalization, $\hat{\cF}_{\lambda,\mu,\bk}^\downarrow\colon \cO_\lambda(D)\hotimes\cO_\mu(D)\rightarrow \cO_{\lambda+\mu}(D,V_\bk)$ is holomorphically continued for all $(\lambda,\mu)\in\BC^2$.
We note that the normalization of $\hat{\cF}_{\lambda,\mu,\bk}^\downarrow$ is different from~(\ref{SBO1}) and~(\ref{SBO2}), that is, we have
\[
\hat{\cF}_{\lambda,\mu,\bk}^\downarrow (f(x-y))=\tilde{C}^d(\lambda,\mu,\bk)\left\langle f(x),\rK_\bk(x)\right\rangle_{F,\fp^+}
\in V_\bk, \qquad f(x)\in\cP_\bk(\fp^+).
\]
Also, when $\fp^+$ is of tube type, we fix $[K,K]$-isomorphisms $V_{\bk+\underline{a}_r}\simeq V_\bk$ for each $a\in\BZ_{>0}$.
Then as in Proposition~\ref{prop_Plancherel}, by Theorems~\ref{thm_poles_tensor},~\ref{thm_factorize_tensor}, Propositions~\ref{prop_atzero_tensor},~\ref{prop_zeroes_tensor}
and Corollary~\ref{cor_topterm_tensor}, the following hold.

\begin{Corollary}\label{cor_Plancherel_tensor}
For $\lambda,\mu>p-1$ and for $f\in\cH_\lambda(D)\hotimes\cH_\mu(D)$, we have
\[ \Vert f\Vert_{\lambda\otimes\mu}^2=\sum_{\bk\in\BZ_{++}^r}\frac{1}{(\lambda)_{\bk,d}(\mu)_{\bk,d}\tilde{C}^d(\lambda,\mu,\bk)}\Vert \hat{\cF}_{\lambda,\mu,\bk}^\downarrow f\Vert_{\lambda+\mu,\bk}^2, \]
where $\tilde{C}^d(\lambda,\mu,\bk)$ is as in~\eqref{const_tensor}.
\end{Corollary}

\begin{Corollary}\label{cor_submodule_tensor}
Let $\bk\in\BZ_{++}^r$.
\begin{enumerate}\itemsep=0pt
\item[$(1)$] For $a_1,a_2=1,2,\dots,r$,
\[
{\rm d}(\tau_\lambda\hotimes\tau_\mu)(\cU(\fg))\big\{f(x-y)\mid f\in\cP_\bk(\fp^+)\big\}\subset M_{a_1}^\fg(\lambda)\otimes M_{a_2}^\fg(\mu)
\]
holds if and only if
\[
(\lambda,\mu)\in \left(\frac{d}{2}(a_1-1)-k_{a_1}-\BZ_{\ge 0}\right)\times\left(\frac{d}{2}(a_2-1)-k_{a_2}-\BZ_{\ge 0}\right),
\]
where $M_a^\fg(\lambda)\subset\cO_\lambda(D)_{\widetilde{K}}$ is the $\bigl(\fg,\widetilde{K}\bigr)$-submodule given in~\eqref{submodule}.
\item[$(2)$] Let $\lambda,\mu\in\left\{0,\frac{d}{2},\dots,\frac{d}{2}(r-1)\right\}\cup\left(\frac{d}{2}(r-1),\infty\right)$. If $\min\{\lambda,\mu\}=\frac{d}{2}a$, $a\in\{0,1,\dots,r-1\}$,
then we have
\[
\cH_\lambda(D)\hotimes\cH_\mu(D)\simeq\hsum_{\substack{\bk\in\BZ_{++}^r \\ k_{a+1}=0}}\cH_{\lambda+\mu}\bigl(D,\cP_\bk(\fp^+)\bigr).
\]
\item[$(3)$] For $a=0,1,\dots,r-1$, if $k_{a+1}=0$, then for $\lambda,\mu\in\left\{\frac{d}{2}a,\frac{d}{2}(a+1),\dots,\frac{d}{2}(r-1)\right\}\cup\left(\frac{d}{2}(r-1),\infty\right)$,
the restriction of $\hat{\cF}_{\lambda,\mu,\bk}^\downarrow$ gives the symmetry breaking operator
\[
\hat{\cF}_{\lambda,\mu,\bk}^\downarrow\colon \ \cH_\lambda(D)\hotimes\cH_\mu(D)\longrightarrow \cH_{\lambda+\mu}(D,V_\bk).
\]
\item[$(4)$] Suppose $\fp^+$ is of tube type. For $a=1,2,\dots,k_r$, if
\[
\rK_\bk(x)=c\det_{\fn^+}(x)^a\rK_{\bk-\underline{a}_r}(x),
\]
then we have
\begin{gather*}
\hat{\cF}_{\frac{n}{r}-a,\mu,\bk}^\downarrow=c\frac{(\mu)_{\bk,d}}{(\mu)_{\bk-\underline{a}_r,d}}\hat{\cF}_{\frac{n}{r}+a,\mu,\bk-\underline{a}_r}^\downarrow
\circ\biggl(\det_{\fn^-}\left(\frac{\partial}{\partial x}\right)^a\otimes 1\biggr)\colon \\
 \cO_{\frac{n}{r}-a}(D)\hotimes \cO_\mu(D)\longrightarrow \cO_{\frac{n}{r}-a+\mu}(D,V_\bk)\simeq\cO_{\frac{n}{r}+a+\mu}(D,V_{\bk-\underline{a}_r}), \\
\hat{\cF}_{\lambda,\frac{n}{r}-a,\bk}^\downarrow=c\frac{(\lambda)_{\bk,d}}{(\lambda)_{\bk-\underline{a}_r,d}}\hat{\cF}_{\lambda,\frac{n}{r}+a,\bk-\underline{a}_r}^\downarrow
\circ\biggl(1\otimes\det_{\fn^-}\left(-\frac{\partial}{\partial y}\right)^a\biggr)\colon \\
 \cO_\lambda(D)\hotimes \cO_{\frac{n}{r}-a}(D)\longrightarrow \cO_{\lambda+\frac{n}{r}-a}(D,V_\bk)\simeq\cO_{\lambda+\frac{n}{r}+a}(D,V_{\bk-\underline{a}_r}).
\end{gather*}
\item[$(5)$] Suppose $\hat{\cF}_{\lambda_0,\mu_0,\bk}^\downarrow=0$. Then there exist $l_1,l_2\in\BZ_{>0}$ such that
\[ \lim_{\lambda\to\lambda_0}\frac{\hat{\cF}_{\lambda,\mu_0,\bk}^\downarrow}{(\lambda-\lambda_0)^{l_1}},\,
\lim_{\mu\to\mu_0}\frac{\hat{\cF}_{\lambda_0,\mu,\bk}^\downarrow}{(\mu-\mu_0)^{l_2}}\colon \ \cO_{\lambda_0}(D)\hotimes\cO_{\mu_0}(D)\longrightarrow \cO_{\lambda_0+\mu_0}(D,V_\bk) \]
converge to non-zero operators. These are linearly independent in $\Hom_{\widetilde{G}}(\cO_{\lambda_0}(D)\hotimes\cO_{\mu_0}(D),\allowbreak\cO_{\lambda_0+\mu_0}(D,V_\bk))$, and especially,
this space is at least $2$-dimensional.
\end{enumerate}
\end{Corollary}

Corollary~\ref{cor_submodule_tensor}\,(2) and (3) are earlier given in~\cite[Theorems 3.3 and 4.4]{PZ}.
As an example of~(5), by~(4) we easily get the following.
\begin{Corollary}\label{cor_linear_indep_tensor}
Suppose $\fp^+$ is of tube type, and let $\bk\in\BZ_{++}^r$. For $a_1,a_2=1,2,\dots,k_r$ with $a_1+a_2\ge k_r+1$, we have
\begin{gather*}
\lim_{\mu\to\frac{n}{r}-a_2}\frac{\hat{\cF}_{\frac{n}{r}-a_1,\mu,\bk}^\downarrow}{(\mu-a_1+\bk)_{\underline{a_1}_r,d}}
=c\hat{\cF}_{\frac{n}{r}+a_1,\frac{n}{r}-a_2,\bk-\underline{a_1}_r}^\downarrow
\circ\biggl(\det_{\fn^-}\left(\frac{\partial}{\partial x}\right)^{a_1}\otimes 1\biggr), \\
\lim_{\lambda\to\frac{n}{r}-a_1}\frac{\hat{\cF}_{\lambda,\frac{n}{r}-a_2,\bk}^\downarrow}{(\lambda-a_2+\bk)_{\underline{a_2}_r,d}}
=c\hat{\cF}_{\frac{n}{r}-a_1,\frac{n}{r}+a_2,\bk-\underline{a_2}_r}^\downarrow
\circ\biggl(1\otimes\det_{\fn^-}\left(-\frac{\partial}{\partial y}\right)^{a_2}\biggr)\colon \\
 \cO_{\frac{n}{r}-a_1}(D)\hotimes \cO_{\frac{n}{r}-a_2}(D)\longrightarrow \cO_{2\frac{n}{r}-a_1-a_2}(D,V_\bk),
\end{gather*}
and these are linearly independent.
\end{Corollary}

\subsection[Example: $\bk=(k,\dots,k)$ case]{Example: $\boldsymbol{\bk=(k,\dots,k)}$ case}

In this subsection we assume $\fp^+=\fn^{+\BC}$ is of tube type, and consider the case $\bk=\underline{k}_r=(\underbrace{k,\dots,k}_r)$,
so that $\cP_{\underline{k}_r}(\fp^+)=\BC\det_{\fn^+}(x)^k$ holds. For $\lambda,\mu\in\BC$, $k\in\BZ_{\ge 0}$, we define the polynomial $RC_{\lambda,\mu,k}^{\fn^+}(z,w)\in\cP(\fp^+\oplus\fp^+)$ by
\[
RC_{\lambda,\mu,k}^{\fn^+}(z,w)
:=(\lambda)_{\underline{k}_r,d}(\mu)_{\underline{k}_r,d}\bigl\langle \det_{\fn^+}(x-y)^k,{\rm e}^{(x|\overline{z})_{\fp^+}+(y|\overline{w})_{\fp^+}}\bigr\rangle_{\lambda\otimes\mu,(x,y)}.
\]
This is originally defined for $\Re\lambda,\Re\mu>p-1=\frac{2n}{r}-1$, and holomorphically continued for all~$\lambda,\mu\in\BC$.
To describe this explicitly, for $\bm\in\BZ_{++}^r$ let
\[
\Phi_\bm^{\fn^+}(x):=\int_{K_L}\Delta^{\fn^+}_\bm(kx)\,{\rm d}k\in\cP_\bm(\fp^+)^{K_L}, \qquad d_\bm^{\fp^+}:=\dim\cP_\bm(\fp^+),
\]
where $K_L\subset K\subset G$ is the subgroup given in Section~\ref{subsection_KKT}, which acts on $\fn^+$ as Jordan algebra automorphisms. Also, for $\alpha,\beta,\gamma\in\BC$ let
\[
{}_2F_1^{\fn^+}\left(\begin{matrix}\alpha,\beta \\ \gamma\end{matrix};z\right):=\sum_{\bm\in\BZ_{++}^r}\frac{(\alpha)_{\bm,d}(\beta)_{\bm,d}}{(\gamma)_{\bm,d}}
\frac{d_\bm^{\fp^+}}{\left(\frac{n}{r}\right)_{\bm,d}}\Phi_\bm^{\fn^+}(z).
\]
Then the following holds.

\begin{Theorem}
For $\lambda,\mu\in\BC$, $k\in\BZ_{\ge 0}$, $z,w\in\fp^+$, we have
\begin{align*}
RC_{\lambda,\mu,k}^{\fn^+}(z,w)
&=(\mu)_{\underline{k}_r,d}\det_{\fn^+}(z)^k {}_2F_1^{\fn^+}\left(\begin{matrix} -k, -\lambda-k+\frac{n}{r} \\ \mu \end{matrix};-P\bigl(z^{\mathit{-1/2}}\bigr)w\right) \\
&=(\lambda)_{\underline{k}_r,d}\det_{\fn^+}(-w)^k {}_2F_1^{\fn^+}\left(\begin{matrix} -k, -\mu-k+\frac{n}{r} \\ \lambda \end{matrix};-P\bigl(w^{\mathit{-1/2}}\bigr)z\right).
\end{align*}
\end{Theorem}

\begin{proof}
By~\cite[Proposition XII.1.3\,(ii)]{FK}, we have
\begin{align}
\det_{\fn^+}(x-y)^k&=\det_{\fn^+}(x)^k\sum_{\substack{\bm\in\BZ_{++}^r\\ m_1\le k}}
\frac{(-k)_{\bm,d}d_\bm^{\fp^+}}{\left(\frac{n}{r}\right)_{\bm,d}}\Phi_\bm^{\fn^+}\bigl(P\bigl(x^{\mathit{-1/2}}\bigr)y\bigr) \label{binom1} \\
&=\det_{\fn^+}(-y)^k\sum_{\substack{\bm\in\BZ_{++}^r\\ m_1\le k}}
\frac{(-k)_{\bm,d}d_\bm^{\fp^+}}{\left(\frac{n}{r}\right)_{\bm,d}}\Phi_\bm^{\fn^+}\bigl(P\bigl(y^{\mathit{-1/2}}\bigr)x\bigr), \label{binom2}
\end{align}
and by~\cite[Lemma XIV.1.2]{FK}, we have
\[
\det_{\fn^+}(x)^k\Phi_\bm^{\fn^+}\bigl(P\bigl(x^{\mathit{-1/2}}\bigr)y\bigr)=\det_{\fn^+}(y)^k\Phi_{\underline{k}_r-\bm^\vee}^{\fn^+}\bigl(P\bigl(y^{\mathit{-1/2}}\bigr)x\bigr)
\in\cP_{\underline{k}_r-\bm^\vee}(\fp^+)\otimes\cP_\bm(\fp^+),
\]
where $\bm^\vee:=(m_r,\dots,m_1)$. Then by Corollary~\ref{cor_FK} and~(\ref{binom1}), we get
\begin{align*}
&RC_{\lambda,\mu,k}^{\fn^+}(z,w) \\
&=(\lambda)_{\underline{k}_r,d}(\mu)_{\underline{k}_r,d}\sum_{\substack{\bm\in\BZ_{++}^r\\ m_1\le k}}\frac{(-k)_{\bm,d}d_\bm^{\fp^+}}{\left(\frac{n}{r}\right)_{\bm,d}}
\bigl\langle \det_{\fn^+}(x)^k\Phi_\bm^{\fn^+}\bigl(P\bigl(x^{\mathit{-1/2}}\bigr)y\bigr),{\rm e}^{(x|\overline{z})_{\fp^+}+(y|\overline{w})_{\fp^+}}\bigr\rangle_{\lambda\otimes\mu,(x,y)} \\
&=(\lambda)_{\underline{k}_r,d}(\mu)_{\underline{k}_r,d}\sum_{\substack{\bm\in\BZ_{++}^r\\ m_1\le k}}\frac{(-k)_{\bm,d}}{(\lambda)_{\underline{k}_r-\bm^\vee,d}(\mu)_{\bm,d}}
\frac{d_\bm^{\fp^+}}{\left(\frac{n}{r}\right)_{\bm,d}}\det_{\fn^+}(z)^k\Phi_\bm^{\fn^+}\bigl(P\bigl(z^{\mathit{-1/2}}\bigr)w\bigr) \\
&=(\mu)_{\underline{k}_r,d}\det_{\fn^+}(z)^k\sum_{\substack{\bm\in\BZ_{++}^r\\ m_1\le k}}\frac{(-k)_{\bm,d}\left(-\lambda-k+\frac{n}{r}\right)_{\bm,d}}{(\mu)_{\bm,d}}
\frac{d_\bm^{\fp^+}}{\left(\frac{n}{r}\right)_{\bm,d}}(-1)^{|\bm|}\Phi_\bm^{\fn^+}(P(z^{\mathit{-1/2}})w) \\
&=(\mu)_{\underline{k}_r,d}\det_{\fn^+}(z)^k {}_2F_1^{\fn^+}\left(\begin{matrix} -k, -\lambda-k+\frac{n}{r} \\ \mu \end{matrix};-P\bigl(z^{\mathit{-1/2}}\bigr)w\right).
\end{align*}
The 2nd equality is also proved similarly by using~(\ref{binom2}).
\end{proof}

For special $\lambda,\mu\in\BC$, $RC_{\lambda,\mu,k}^{\fn^+}$ is factorized as follows.
\begin{Theorem}\label{thm_factorize_RC}
Let $a=1,2,\dots,k$.
\begin{enumerate}\itemsep=0pt
\item[$(1)$] $RC_{\frac{n}{r}-a,\mu,k}^{\fn^+}(z,w)=(\mu+k-a)_{\underline{a}_r,d}\det_{\fn^+}(z)^a RC_{\frac{n}{r}+a,\mu,k-a}^{\fn^+}(z,w)$.
\item[$(2)$] $RC_{\lambda,\frac{n}{r}-a,k}^{\fn^+}(z,w)=(\lambda+k-a)_{\underline{a}_r,d}\det_{\fn^+}(-w)^a RC_{\lambda,\frac{n}{r}+a,k-a}^{\fn^+}(z,w)$.
\item[$(3)$] Suppose $\lambda+\mu=\frac{n}{r}-2k+a$. Then
\begin{align*}
RC_{\lambda,\mu,k}^{\fn^+}(z,w)&=(\mu+k-a)_{\underline{a}_r,d}\det_{\fn^+}(z+w)^a RC_{\lambda,\mu,k-a}^{\fn^+}(z,w) \\
&=(-1)^{ar}(\lambda+k-a)_{\underline{a}_r,d}\det_{\fn^+}(z+w)^a RC_{\lambda,\mu,k-a}^{\fn^+}(z,w).
\end{align*}
\end{enumerate}
\end{Theorem}

\begin{proof}
(1) and (2) follow from Theorem~\ref{thm_factorize_tensor}.

For (3), by~\cite[Proposition XV.3.4\,(ii)]{FK}, we have
\begin{gather*}
RC_{\lambda,\mu,k}^{\fn^+}(z,w)
=(\mu)_{\underline{k}_r,d}\det_{\fn^+}(z)^k {}_2F_1^{\fn^+}\left(\begin{matrix} -k, -\lambda-k+\frac{n}{r} \\ \mu \end{matrix};-P\bigl(z^{\mathit{-1/2}}\bigr)w\right) \\
\quad\qquad{}=(\mu)_{\underline{k}_r,d}\det_{\fn^+}(z)^k\det_{\fn^+}\bigl(e+P\bigl(z^{\mathit{-1/2}}\bigr)w\bigr)^{\lambda+\mu+2k-\frac{n}{r}} \\
\quad\qquad{}\eqspace{}\times{}_2F_1^{\fn^+}\left(\begin{matrix} \lambda+\mu+k-\frac{n}{r},\mu+k \\ \mu \end{matrix};-P\bigl(z^{\mathit{-1/2}}\bigr)w\right) \\
\quad\qquad{}=(\mu)_{\underline{k}_r,d}\det_{\fn^+}(z)^k\det_{\fn^+}\bigl(e+P\bigl(z^{\mathit{-1/2}}\bigr)w\bigr)^{a}
{}_2F_1^{\fn^+}\left(\begin{matrix} -k+a,\mu+k \\ \mu \end{matrix};-P\bigl(z^{\mathit{-1/2}}\bigr)w\right) \\
\quad\qquad{}=(\mu)_{\underline{k-a}_r,d}(\mu+k-a)_{\underline{a}_r,d}\det_{\fn^+}(z)^{k-a}\det_{\fn^+}(z+w)^{a} \\
\quad\qquad{}\eqspace{}\times{}_2F_1^{\fn^+}\left(\begin{matrix} -k+a,-\lambda-k+a+\frac{n}{r} \\ \mu \end{matrix};-P\bigl(z^{\mathit{-1/2}}\bigr)w\right) \\
\quad\qquad{}=(\mu+k-a)_{\underline{a}_r,d}\det_{\fn^+}(z+w)^{a} RC_{\lambda,\mu,k-a}^{\fn^+}(z,w),
\end{gather*}
and we have $(\mu+k-a)_{\underline{a}_r,d}=\left(-\lambda-k+\frac{n}{r}\right)_{\underline{a}_r,d}=(-1)^{ar}(\lambda+k-a)_{\underline{a}_r,d}$.
\end{proof}

The polynomial $RC_{\lambda,\mu,k}^{\fn^+}(z,w)$ gives the symmetry breaking operator
\begin{gather*}
\mathcal{RC}_{\lambda,\mu,k}^\downarrow\colon \ \cO_\lambda(D)\hotimes\cO_\mu(D)\longrightarrow \cO_{\lambda+\mu+2k}(D), \\
\bigl(\mathcal{RC}_{\lambda,\mu,k}^\downarrow f\bigr)(x):=RC_{\lambda,\mu,k}^{\fn^+}\left(\frac{\partial}{\partial x},\frac{\partial}{\partial y}\right)f(x,y)\biggr|_{y=x},
\end{gather*}
where we normalize $\frac{\partial}{\partial x}$ with respect to the bilinear form $(\cdot|\cdot)_{\fn^+}=(\cdot|Q(\overline{e})\cdot)_{\fp^+}$ on $\fp^+=\fn^{+\BC}$.
We note that $\cO_{\lambda+\mu}(D,\cP_{\underline{k}_r}(\fp^+))\simeq \cO_{\lambda+\mu+2k}(D)$ holds if $\fp^+$ is of tube type.
When $\fp^+=\BC$, ${G={\rm SL}(2,\BR)}$, this is proportional to the Rankin--Cohen bidifferential operator
(see, e.g.,~\cite{C,R} and~\cite[Sections 8 and 9]{KP2}). For general $\fp^+$ see also, e.g.,~\cite{BCK, Cl, OR, P}.
In our normalization we have the following.
\begin{Proposition}
$\ds \mathcal{RC}_{\lambda,\mu,k}^\downarrow\bigl(\det_{\fn^+}(x-y)^k\bigr)=\left(\lambda+\mu+k-\frac{n}{r}\right)_{\underline{k}_r,d}\left(\frac{n}{r}\right)_{\underline{k}_r,d}$.
\end{Proposition}
\begin{proof}
As in Proposition~\ref{prop_Plancherel}\,(2), by Corollary~\ref{cor_topterm_tensor} we have
\begin{align*}
\mathcal{RC}_{\lambda,\mu,k}^\downarrow\bigl(\det_{\fn^+}(x-y)^k\bigr)&=(\lambda)_{\underline{k}_r,d}(\mu)_{\underline{k}_r,d}\big\Vert \det_{\fn^+}(x-y)^k\big\Vert_{\lambda\otimes\mu,(x,y)}^2 \\
&=\tilde{C}^d(\lambda,\mu,\underline{k}_r)\big\Vert \det_{\fn^+}(x)^k\big\Vert_{F,\fp^+}^2,
\end{align*}
and we have
\[
\tilde{C}^d(\lambda,\mu,\underline{k}_r)=\left(\lambda+\mu+k-\frac{n}{r}\right)_{\underline{k}_r,d}, \qquad
\big\Vert \det_{\fn^+}(x)^k\big\Vert_{F,\fp^+}^2=\left(\frac{n}{r}\right)_{\underline{k}_r,d}
\]
(for the 2nd equality see~\cite[Proposition XI.4.1\,(ii)]{FK}). Hence we get the proposition.
\end{proof}

Next we consider the zeroes of $RC_{\lambda,\mu,k}^{\fn^+}$. If this vanishes at $(\lambda,\mu)$, then by Corollary~\ref{cor_submodule_tensor}\,(4),
$\Hom_{\widetilde{G}}(\cO_\lambda(D)\hotimes\cO_\mu(D),\cO_{\lambda+\mu+2k}(D))$ is at least 2-dimensional. More precisely, the following holds.
\begin{Theorem}
For $k\in\BZ_{\ge 1}$, $1\le i,j\le r$, let
\[ Z_{i,j}^{k,d}:=\left\{\left(\frac{d}{2}(i-1)+1-a_1,\frac{d}{2}(j-1)+1-a_2\right)\ \middle|\
\begin{matrix} a_1,a_2\in\BZ,\, 1\le a_1,a_2\le k,\\ a_1+a_2\ge k+1 \end{matrix}\right\}. \]
\begin{enumerate}\itemsep=0pt
\item[$(1)$] $RC_{\lambda,\mu,k}^{\fn^+}=0$ holds if
\[ (\lambda,\mu)\in\bigcup_{\substack{1\le i,j\le r \\ i+j\ge r+1}}Z_{i,j}^{k,d}. \]
\item[$(2)$] For $(\lambda,\mu)$ in the above set, let
\begin{gather*}
l:=\min\big\{i+j\, \big|\, 1\le i,j\le r,\, i+j\ge r+1,\, (\lambda,\mu)\in Z_{i,j}^{k,d}\big\}, \\
\alpha:=\#\big\{(i,j)\in\{1,\dots,r\}^2\, \big|\, i+j=l,\, (\lambda,\mu)\in Z_{i,j}^{k,d}\big\}.
\end{gather*}
Then we have
\[ \dim\Hom_{\widetilde{G}}(\cO_\lambda(D)\hotimes\cO_\mu(D),\cO_{\lambda+\mu+2k}(D))\ge \alpha+1. \]
\end{enumerate}
\end{Theorem}

\begin{proof}
For $1\le i,j\le r$, let $Z_{i,j}':=Z_{r-i+1,j}^{k,d}$, so that we have
\[ \bigcup_{\substack{1\le i,j\le r \\ i+j\ge r+1}}Z_{i,j}^{k,d}=\bigcup_{1\le i\le j\le r}Z_{i,j}'. \]

(1) For $\bm\in\BZ_{++}^r$, $m_1\le k$, the coefficient of $\det_{\fn^+}(z)^k\Phi_\bm^{\fn^+}\bigl(P\bigl(z^{\mathit{-1/2}}\bigr)w\bigr)$ in $RC_{\lambda,\mu,k}^{\fn^+}(z,w)$ is given by
\begin{align*}
&F_\bm(\lambda,\mu):=
(-1)^{kr}(-k)_{\bm,d}\left(-\lambda-k+\frac{n}{r}\right)_{\bm,d}\left(-\mu-k+\frac{n}{r}\right)_{\underline{k}_r-\bm^\vee,d}\frac{d_\bm^{\fp^+}}{\left(\frac{n}{r}\right)_{\bm,d}} \\
&\qquad{}=\frac{(-k)_{\bm,d}\,d_\bm^{\fp^+}}{(-1)^{kr}\left(\frac{n}{r}\right)_{\bm,d}}
\prod_{i=1}^r\left(-\lambda-k+1+\frac{d}{2}(r-i)\right)_{m_i}\prod_{j=1}^r\left(-\mu-k+1+\frac{d}{2}(j-1)\right)_{k-m_j}.
\end{align*}
We fix $(\lambda_0,\mu_0)\in\BC^2$. For $1\le i,j\le r$, let $\frac{d}{2}(r-i)+1-\lambda_0=:a_i^{(1)}$, $\frac{d}{2}(j-1)+1-\mu_0=:a_j^{(2)}$. If~$a_i^{(1)}\in\{1,\dots,k\}$, then
\[ \left(-\lambda_0-k+1+\frac{d}{2}(r-i)\right)_{m_i}=\bigl(-k+a_i^{(1)}\bigr)_{m_i}=0 \quad \text{holds if} \quad k-a_i^{(1)}+1\le m_i\le k. \]
Similarly, if $a_j^{(2)}\in\{1,\dots,k\}$, then
\[ \left(-\mu_0-k+1+\frac{d}{2}(j-1)\right)_{k-m_j}=\bigl(-k+a_j^{(2)}\bigr)_{k-m_j}=0 \quad \text{holds if} \quad 0\le m_j <a_j^{(2)}. \]
If $(\lambda_0,\mu_0)\in Z_{i,j}'$ holds for some $1\le i\le j\le r$, then we have $k-a_i^{(1)}+1\le a_j^{(2)}$,
and at least one of the two above formulas vanishes for $k\ge m_i\ge m_j\ge 0$.
Hence $F_\bm(\lambda_0,\mu_0)$ vanishes for all $\bm$, and $RC_{\lambda,\mu,k}^{\fn^+}=0$ holds.

(2) Again we fix $(\lambda_0,\mu_0)$, define $a_i^{(1)}$, $a_j^{(2)}$ as above, and let
\begin{alignat*}{3}
&\{i\in\{1,\dots,r\}\mid a^{(1)}_i\in\{1,\dots,k\}\}=:\{i(1),\dots,i(r')\}, && i(1)<\cdots< i(r'),& \\
&\{j\in\{1,\dots,r\}\mid a^{(2)}_j\in\{1,\dots,k\}\}=:\{j(1),\dots,j(r'')\}, \qquad && j(1)<\cdots< j(r'').&
\end{alignat*}
Then for $0\le u\le r'$, $0\le v\le r''$, $F_{\bm}(\lambda,\mu)$ has a~zero of order $u$ at $\lambda=\lambda_0$ and of order $r''-v$ at $\mu=\mu_0$ if
\[
\bm\in W(u,v):=\left\{ \bm\in\BZ_{++}^r\, \middle|\, \begin{array}{@{}c@{}} 0\le m_j\le k, \ j=1,\dots,r, \\
\begin{array}{@{}ll@{}}m_{i(u)}\ge k-a_{i(u)}^{(1)}+1,& m_{i(u+1)}\le k-a_{i(u+1)}^{(1)}, \\ m_{j(v)}\ge a_{j(v)}^{(2)},& m_{j(v+1)}\le a_{j(v+1)}^{(2)}-1 \end{array}\end{array} \right\},
\]
where we ignore the conditions for $m_{i(0)}$, $m_{i(r'+1)}$, $m_{j(0)}$, $m_{j(r''+1)}$.
For $\bm\in W(u,v)$, let
\[
\hat{F}_\bm(\lambda,\mu):=(\lambda-\lambda_0)^{-u}(\mu-\mu_0)^{-r''+v}F_\bm(\lambda,\mu),
\]
so that $\hat{F}_\bm(\lambda_0,\mu_0)\ne 0$ holds. Then we have
\begin{gather*}
RC_{\lambda,\mu,k}^{\fn^+}(z,w)=\sum_{\substack{\bm\in\BZ_{++}^r \\ m_1\le k}}F_\bm(\lambda,\mu)
\det_{\fn^+}(z)^k\Phi_\bm^{\fn^+}\bigl(P\bigl(z^{\mathit{-1/2}}\bigr)w\bigr) \\
\hphantom{RC_{\lambda,\mu,k}^{\fn^+}(z,w)}{}=\!\sum_{\substack{0\le u\le r' \\ 0\le v\le r''}}\sum_{\bm\in W(u,v)}\!\!(\lambda-\lambda_0)^u(\mu-\mu_0)^{r''-v}\hat{F}_\bm(\lambda,\mu)
\det_{\fn^+}(z)^k\Phi_\bm^{\fn^+}\bigl(P\bigl(z^{\mathit{-1/2}}\bigr)w\bigr).
\end{gather*}
Now let
\[
l':=\max\left\{ u-v\, \middle|\, \begin{matrix} 1\le u\le r',\, 1\le v\le r'', \\ i(u)\le j(v),\, a_{i(u)}^{(1)}+a_{j(v)}^{(2)}\ge k+1 \end{matrix}\right\}.
\]
If $u-v\le l'-1$, then since $i(u+1)\le j(v)$ and $k-a_{i(u+1)}^{(1)}+1\le a_{j(v)}^{(2)}$ we have $W(u,v)=\varnothing$.
Then for $(s,t)\in\BC^2$, we can take the limit
\begin{gather*}
\lim_{\nu\to 0}\frac{1}{\nu^{l'+r''}}RC_{\lambda_0+\nu s,\mu_0+\nu t,k}^{\fn^+}(z,w) \\
\qquad{}=\lim_{\nu\to 0}\sum_{u-v\ge l'}\sum_{\bm\in W(u,v)}\frac{(\nu s)^u(\nu t)^{r''-v}}{\nu^{l'+r''}}\hat{F}_\bm(\lambda_0+\nu s,\mu_0+\nu t)
\det_{\fn^+}(z)^k\Phi_\bm^{\fn^+}(P(z^{\mathit{-1/2}})w) \\
\qquad{}=\sum_{u-v=l'}s^u t^{r''-v}\sum_{\bm\in W(u,v)}\hat{F}_\bm(\lambda_0,\mu_0)\det_{\fn^+}(z)^k\Phi_\bm^{\fn^+}(P(z^{\mathit{-1/2}})w).
\end{gather*}
By varying $(s,t)$, we get $\alpha+1$ linearly independent polynomials. By substituting $(z,w)=\left(\frac{\partial}{\partial x},\frac{\partial}{\partial y}\right)$ and restricting to $y=x$,
these polynomials give symmetry breaking operators in $\Hom_{\widetilde{G}}(\cO_\lambda(D)\hotimes\cO_\mu(D),\cO_{\lambda+\mu+2k}(D))$, and hence this space is at least $(\alpha+1)$-dimensional.
\end{proof}

\begin{Example}\quad
\begin{enumerate}\itemsep=0pt
\item Suppose $d\in 2\BZ$, $k\ge \frac{n}{r}=\frac{d}{2}(r-1)+1$. If
\[ (\lambda,\mu)\in\bigcap_{j=1}^r Z_{r-j+1,j}^{k,d}=\left\{(\lambda,\mu)\in\BZ^2\ \middle|\ \begin{matrix}\frac{n}{r}-k\le \lambda,\mu\le 0,\\ \lambda+\mu\le \frac{n}{r}-k\end{matrix}\right\}, \]
then we have
\[ \dim\Hom_{\widetilde{G}}(\cO_\lambda(D)\hotimes\cO_\mu(D),\cO_{\lambda+\mu+2k}(D))\ge r+1. \]
\item Suppose $d=1$, i.e., $\fp^+=\Sym(r,\BC)$, and let $\delta\in\{0,1\}$. For $k\ge \left\lfloor \frac{r+\delta}{2}\right\rfloor$, if
\begin{align*}
(\lambda,\mu)\in&\bigcap_{\substack{1\le j\le r\\ j\in 2\BZ+\delta}} Z_{r-j+1,j}^{k,1} \\
&=\left\{(\lambda,\mu)\in\left(\BZ+\frac{r+\delta}{2}\right)\times \left(\BZ+\frac{1-\delta}{2}\right)\ \middle|\
\begin{matrix} \frac{1}{2}r-k\le \lambda,\mu\le \frac{1}{2},\\ \lambda+\mu\le \frac{1}{2}(r+1)-k \end{matrix} \right\},
\end{align*}
then we have
\[ \dim\Hom_{\widetilde{G}}(\cO_\lambda(D)\hotimes\cO_\mu(D),\cO_{\lambda+\mu+2k}(D))\ge \left\lfloor \frac{r+\delta}{2}\right\rfloor+1. \]
Similarly, for $k\ge \left\lfloor \frac{r-\delta}{2}\right\rfloor$, if
\begin{align*}
(\lambda,\mu)\in&\bigcap_{\substack{2\le j\le r\\ j\in 2\BZ+\delta}} Z_{r-j+2,j}^{k,1} \\
&=\left\{(\lambda,\mu)\in\left(\BZ+\frac{r+1+\delta}{2}\right)\times \left(\BZ+\frac{1-\delta}{2}\right)\ \middle|\
\begin{matrix} \frac{1}{2}r-k\le \lambda,\mu\le 1,\\ \lambda+\mu\le \frac{1}{2}r+1-k \end{matrix} \right\},
\end{align*}
then we have
\[ \dim\Hom_{\widetilde{G}}(\cO_\lambda(D)\hotimes\cO_\mu(D),\cO_{\lambda+\mu+2k}(D))\ge \left\lfloor \frac{r-\delta}{2}\right\rfloor+1. \]
\item Suppose $d=2$, i.e., $\fp^+={\rm M}(r,\BC)$. For $\alpha=1,2,\dots,\min\{k,r\}$, if
\[ (\lambda,\mu)\in\left\{(\lambda,\mu)\in\BZ^2\, \middle|\, \begin{matrix}\alpha-k\le \lambda,\mu\le r-\alpha,\\ r+\alpha-2k\le\lambda+\mu\le 2r-\alpha-k\end{matrix}\right\}, \]
then we have
\[ \dim\Hom_{\widetilde{G}}(\cO_\lambda(D)\hotimes\cO_\mu(D),\cO_{\lambda+\mu+2k}(D))\ge \alpha+1. \]
\end{enumerate}
\end{Example}

When $\fp^+=\BC$, it is proved in~\cite[Theorem 9.1]{KP2} by the F-method that $\Hom_{\widetilde{G}}(\cO_\lambda(D)\hotimes\cO_\mu(D),\allowbreak\cO_{\lambda+\mu+2k}(D))$ is precisely 2-dimensional
if $(\lambda,\mu)\in Z_{1,1}^{k,d}$, and precisely 1-dimensional otherwise. For general $\fp^+$, we need further study to determine the precise dimension of this space.
Also, for $(\lambda,\mu)\in Z_{r,r}^{k,d}\cup Z_{r,1}^{k,d}\cup Z_{1,r}^{k,d}$, we can consider an analogue of~\cite[Theorem 9.2]{KP2} as follows.
We note that the result on linear independence for $\frac{n}{r}\in\BZ_{\ge 2}$ case differs from $\fp^+=\BC$ case.

\begin{Theorem}
\quad
\begin{enumerate}\itemsep=0pt
\item[$(1)$] If $\frac{n}{r}-\lambda_0=:a_1\in\{1,2,\dots,k\}$, then we have
\begin{equation}\label{atzero_RC1}
\lim_{\mu\to\mu_0}\frac{\mathcal{RC}_{\frac{n}{r}-a_1,\mu,k}^\downarrow}{(\mu+k-a_1)_{\underline{a_1}_r,d}}=\mathcal{RC}_{\frac{n}{r}+a_1,\mu_0,k-a_1}^\downarrow
\circ\biggl(\det_{\fn^+}\left(\frac{\partial}{\partial x}\right)^{a_1}\otimes 1\biggr).
\end{equation}
\item[$(2)$] If $\frac{n}{r}-\mu_0=:a_2\in\{1,2,\dots,k\}$, then we have
\begin{equation}\label{atzero_RC2}
\lim_{\lambda\to\lambda_0}\frac{\mathcal{RC}_{\lambda,\frac{n}{r}-a_2,k}^\downarrow}{(\lambda+k-a_2)_{\underline{a_2}_r,d}}=\mathcal{RC}_{\lambda_0,\frac{n}{r}+a_2,k-a_2}^\downarrow
\circ\biggl(1\otimes\det_{\fn^+}\left(-\frac{\partial}{\partial y}\right)^{a_2}\biggr).
\end{equation}
\item[$(3)$] If $\lambda_0+\mu_0-\frac{n}{r}+2k=:a_3\in\{1,2,\dots,k\}$, then we have
\begin{equation}\label{atzero_RC3}
\lim_{\nu\to 0}\frac{\mathcal{RC}_{\lambda_0-\nu,\mu_0+\nu,k}^\downarrow}{(\mu_0+\nu+k-a_3)_{\underline{a_3}_r,d}}
=\det_{\fn^+}\left(\frac{\partial}{\partial x}\right)^{a_3}\circ \mathcal{RC}_{\lambda_0,\mu_0,k-a_3}^\downarrow.
\end{equation}
\item[$(4)$] If $(\lambda_0,\mu_0)\in Z_{r,r}^{k,d}$, then~\eqref{atzero_RC1} and~\eqref{atzero_RC2} are linearly independent.
\item[$(5)$] If $(\lambda_0,\mu_0)\in Z_{r,1}^{k,d}$, then~\eqref{atzero_RC1} and~\eqref{atzero_RC3} are linearly independent.
\item[$(6)$] If $(\lambda_0,\mu_0)\in Z_{1,r}^{k,d}$, then~\eqref{atzero_RC2} and~\eqref{atzero_RC3} are linearly independent.
\item[$(7)$] Suppose $\frac{n}{r}\in\BZ_{\ge 2}$. If $(\lambda_0,\mu_0)\in Z_{r,r}^{k,d}\cap\bigl(Z_{r,1}^{k,d}\cup Z_{1,r}^{k,d}\bigr)$,
then~\eqref{atzero_RC1},~\eqref{atzero_RC2} and~\eqref{atzero_RC3} are linearly independent.
\end{enumerate}
\end{Theorem}

\begin{proof}
(1), (2), (3) follow from Theorem~\ref{thm_factorize_RC}.
(4) follows from Corollary~\ref{cor_linear_indep_tensor}.

For (5), we have
\begin{align*}
&RC_{\lambda,\mu,k}^{\fn^+}(z,w) \\
&=\det_{\fn^+}(z)^k\sum_{\substack{\bm\in\BZ_{++}^r\\ m_1\le k}}(-k)_{\bm,d}\left(-\lambda-k+\frac{n}{r}\right)_{\bm,d}(\mu+\bm)_{\underline{k}_r-\bm,d}
\frac{d_\bm^{\fp^+}}{\left(\frac{n}{r}\right)_{\bm,d}}\Phi_\bm^{\fn^+}\bigl(-P\bigl(z^{\mathit{-1/2}}\bigr)w\bigr).
\end{align*}
Put $\frac{n}{r}-\lambda_0=:a_1\in\{1,\dots,k\}$, $1-\mu_0=:a_2'\in\{1,\dots,k\}$, so that $a_3=2k-a_1-a_2'+1\in\{1,\dots,k\}$. Then we have
\begin{align*}
\begin{split}
&\det_{\fn^+}(z)^{-k}\lim_{\mu\to 1-a_2'}\frac{RC_{\frac{n}{r}-a_1,\mu,k}^{\fn^+}(z,w)}{(\mu+k-a_1)_{\underline{a_1}_r,d}} \\
&=\lim_{\mu\to 1-a_2'}\sum_{\substack{\bm\in\BZ_{++}^r\\ m_1\le k-a_1}}(-k)_{\bm,d}(a_1-k)_{\bm,d}(\mu+\bm)_{\underline{k-a_1}_r-\bm,d}
\frac{d_\bm^{\fp^+}}{\left(\frac{n}{r}\right)_{\bm,d}}\Phi_\bm^{\fn^+}\bigl(-P\bigl(z^{\mathit{-1/2}}\bigr)w\bigr) \\
&=\sum_{\substack{\bm\in\BZ_{++}^r\\ m_1\le k-a_1}}(-k)_{\bm,d}(a_1-k)_{\bm,d}(1-a_2'+\bm)_{\underline{k-a_1}_r-\bm,d}
\frac{d_\bm^{\fp^+}}{\left(\frac{n}{r}\right)_{\bm,d}}\Phi_\bm^{\fn^+}\bigl(-P\bigl(z^{\mathit{-1/2}}\bigr)w\bigr),
\end{split}\\
\begin{split}
&\det_{\fn^+}(z)^{-k}\lim_{\nu\to 0}\frac{RC_{\frac{n}{r}-a_1-\nu,1-a_2'+\nu,k}^{\fn^+}(z,w)}{(\nu-k+a_1)_{\underline{2k-a_1-a_2'+1}_r,d}} \\
&=\lim_{\nu\to 0}\sum_{\substack{\bm\in\BZ_{++}^r\\ m_1\le k}}(-k)_{\bm,d}\frac{(\nu+a_1-k)_{\bm,d}(\nu+1-a_2'+\bm)_{\underline{k}_r-\bm,d}}{(\nu-k+a_1)_{\underline{2k-a_1-a_2'+1}_r,d}}
\frac{d_\bm^{\fp^+}}{\left(\frac{n}{r}\right)_{\bm,d}}\Phi_\bm^{\fn^+}\bigl(-P\bigl(z^{\mathit{-1/2}}\bigr)w\bigr) \\
&=\lim_{\nu\to 0}\sum_{\substack{\bm\in\BZ_{++}^r\\ m_1\le k}}(-k)_{\bm,d}(\nu+1-a_2'+\bm)_{\underline{a_1+a_2'-k-1}_r,d}
\frac{d_\bm^{\fp^+}}{\left(\frac{n}{r}\right)_{\bm,d}}\Phi_\bm^{\fn^+}\bigl(-P\bigl(z^{\mathit{-1/2}}\bigr)w\bigr) \\
&=\sum_{\substack{\bm\in\BZ_{++}^r\\ m_1\le k}}(-k)_{\bm,d}(1-a_2'+\bm)_{\underline{a_1+a_2'-k-1}_r,d}
\frac{d_\bm^{\fp^+}}{\left(\frac{n}{r}\right)_{\bm,d}}\Phi_\bm^{\fn^+}\bigl(-P\bigl(z^{\mathit{-1/2}}\bigr)w\bigr).
\end{split}
\end{align*}
Then the coefficient for $\bm=(k,\underline{0}_{r-1})$ is zero for the former formula, and is non-zero for the latter one, and hence these are linearly independent.

(6) Proved similarly by comparing the coefficients for $\bm=(\underline{k}_{r-1},0)$.

(7) By the assumption $\frac{n}{r}\ge 2$, we have $r\ge 2$, and thus $\underline{k}_r$, $(k,\underline{0}_{r-1})$, $\underline{0}_r$
(for $(\lambda_0,\mu_0)\in Z_{r,r}^{k,d}\cap Z_{r,1}^{k,d}$ case)
or $\underline{k}_r$, $(\underline{k}_{r-1},0)$, $\underline{0}_r$ (for $(\lambda_0,\mu_0)\in Z_{r,r}^{k,d}\cap Z_{1,r}^{k,d}$ case) are distinct.
Hence, comparing the coefficients for these $\bm$, we get the linear independence of three formulas.
\end{proof}

\subsection*{Acknowledgements}

The author would like to thank Professor T.~Kobayashi for much helpful advice on this research.
This work was supported by Grant-in-Aid for JSPS Fellows Grant Number JP20J00114.

\LastPageEnding


\addcontentsline{toc}{section}{References}
\begin{thebibliography}{100}
\footnotesize\itemsep=0pt

\bibitem{Ad}
Adams J., The theta correspondence over {$\mathbb{R}$}, in Harmonic Analysis,
 Group Representations, Automorphic Forms and Invariant Theory, \textit{Lect.
 Notes Ser. Inst. Math. Sci. Natl. Univ. Singap.}, Vol.~12, \href{https://doi.org/10.1142/9789812770790_0001}{World Sci. Publ.},
 Hackensack, NJ, 2007, 1--39.

\bibitem{A}
Arazy J., A survey of invariant {H}ilbert spaces of analytic functions on
 bounded symmetric domains, in Multivariable Operator Theory ({S}eattle, {WA},
 1993), \textit{Contemp. Math.}, Vol. 185, \href{https://doi.org/10.1090/conm/185/02147}{Amer. Math. Soc.}, Providence, RI,
 1995, 7--65.

\bibitem{BS1}
Ben~Sa\"{\i}d S., Espaces de {B}ergman pond\'er\'es et s\'erie discr\`ete
 holomorphe de {$\widetilde{{\rm U}(p,q)}$}, \href{https://doi.org/10.1006/jfan.1999.3553}{\textit{J.~Funct. Anal.}}
 \textbf{173} (2000), 154--181.

\bibitem{BS2}
Ben~Sa\"{\i}d S., Connection between the {P}lancherel formula and the {H}owe
 dual pair {$({\rm Sp}(2n,\mathbb R),{\rm O}(k))$}, \href{https://doi.org/10.1007/s00209-002-0443-5}{\textit{Math.~Z.}} \textbf{241}
 (2002), 743--760.

\bibitem{BS3}
Ben~Sa\"{\i}d S., Weighted {B}ergman spaces on bounded symmetric domains,
 \href{https://doi.org/10.2140/pjm.2002.206.39}{\textit{Pacific~J.~Math.}} \textbf{206} (2002), 39--68.

\bibitem{BCK}
Ben~Sa\"{\i}d S., Clerc J.-L., Koufany K., Conformally covariant bi-differential
 operators on a~simple real {J}ordan algebra, \href{https://doi.org/10.1093/imrn/rny082}{\textit{Int. Math. Res. Not.}}
 \textbf{2020} (2020), 2287--2351, \href{https://arxiv.org/abs/1809.06290}{arXiv:1809.06290}.

\bibitem{Cl}
Clerc J.-L., Symmetry breaking differential operators, the source operator and
 {R}odrigues formulae, \href{https://doi.org/10.2140/pjm.2020.307.79}{\textit{Pacific~J.~Math.}} \textbf{307} (2020), 79--107,
 \href{https://arxiv.org/abs/1902.06073}{arXiv:1902.06073}.

\bibitem{C}
Cohen H., Sums involving the values at negative integers of {$L$}-functions of
 quadratic characters, \href{https://doi.org/10.1007/BF01436180}{\textit{Math. Ann.}} \textbf{217} (1975), 271--285.

\bibitem{D}
Dib H., Fonctions de {B}essel sur une alg\`ebre de {J}ordan, \textit{J.~Math.
 Pures Appl.} \textbf{69} (1990), 403--448.

\bibitem{EHW}
Enright T., Howe R., Wallach N., A classification of unitary highest weight
 modules, in Representation Theory of Reductive Groups ({P}ark {C}ity, {U}tah,
 1982), \textit{Progr. Math.}, Vol.~40, \href{https://doi.org/10.1007/978-1-4684-6730-7_7}{Birkh\"auser}, Boston, MA, 1983,
 97--143.

\bibitem{FKKLR}
Faraut J., Kaneyuki S., Kor\'anyi A., Lu Q.-K., Roos G., Analysis and geometry
 on complex homogeneous domains, \textit{Progr. Math.}, Vol. 185,
 \href{https://doi.org/10.1007/978-1-4612-1366-6}{Birkh\"auser}, Boston, MA, 2000.

\bibitem{FK0}
Faraut J., Kor\'anyi A., Function spaces and reproducing kernels on bounded
 symmetric domains, \href{https://doi.org/10.1016/0022-1236(90)90119-6}{\textit{J.~Funct. Anal.}} \textbf{88} (1990), 64--89.

\bibitem{FK}
Faraut J., Kor\'anyi A., Analysis on symmetric cones, \textit{Oxford Math. Monogr.}, The
 Clarendon Press, Oxford University Press, New York, 1994.

\bibitem{Fr}
Frajria P.M., Derived functors of unitary highest weight modules at reduction
 points, \href{https://doi.org/10.2307/2001820}{\textit{Trans. Amer. Math. Soc.}} \textbf{327} (1991), 703--738.

\bibitem{GW}
Goodman R., Wallach N.R., Symmetry, representations, and invariants,
 \textit{Grad. Texts in Math.}, Vol. 255, \href{https://doi.org/10.1007/978-0-387-79852-3}{Springer}, Dordrecht, 2009.

\bibitem{HK1}
Hilgert J., Kr\"otz B., Weighted {B}ergman spaces associated with causal
 symmetric spaces, \href{https://doi.org/10.1007/s002290050167}{\textit{Manuscripta Math.}} \textbf{99} (1999), 151--180.

\bibitem{HK2}
Hilgert J., Kr\"otz B., The {P}lancherel theorem for invariant {H}ilbert
 spaces, \href{https://doi.org/10.1007/PL00004862}{\textit{Math.~Z.}} \textbf{237} (2001), 61--83.

\bibitem{HTW}
Howe R., Tan E.-C., Willenbring J.F., Stable branching rules for classical
 symmetric pairs, \href{https://doi.org/10.1090/S0002-9947-04-03722-5}{\textit{Trans. Amer. Math. Soc.}} \textbf{357} (2005),
 1601--1626, \href{https://arxiv.org/abs/math.RT/0311159}{arXiv:math.RT/0311159}.

\bibitem{IKO}
Ibukiyama T., Kuzumaki T., Ochiai H., Holonomic systems of {G}egenbauer type
 polynomials of matrix arguments related with {S}iegel modular forms,
 \href{https://doi.org/10.2969/jmsj/06410273}{\textit{J.~Math. Soc. Japan}} \textbf{64} (2012), 273--316.

\bibitem{J0}
Jakobsen H.P., Intertwining differential operators for {${\rm Mp}(n,{\bf R})$}
 and {${\rm SU}(n,n)$}, \href{https://doi.org/10.2307/1997976}{\textit{Trans. Amer. Math. Soc.}} \textbf{246} (1978),
 311--337.

\bibitem{J}
Jakobsen H.P., Hermitian symmetric spaces and their unitary highest weight
 modules, \href{https://doi.org/10.1016/0022-1236(83)90076-9}{\textit{J.~Funct. Anal.}} \textbf{52} (1983), 385--412.

\bibitem{JV}
Jakobsen H.P., Vergne M., Restrictions and expansions of holomorphic
 representations, \href{https://doi.org/10.1016/0022-1236(79)90023-5}{\textit{J.~Funct. Anal.}} \textbf{34} (1979), 29--53.

\bibitem{Ju}
Juhl A., Families of conformally covariant differential operators,
 {$Q$}-curvature and holography, \textit{Progr. Math.}, Vol. 275,
 \href{https://doi.org/10.1007/978-3-7643-9900-9}{Birkh\"auser}, Basel, 2009.

\bibitem{KV}
Kashiwara M., Vergne M., On the {S}egal--{S}hale--{W}eil representations and
 harmonic polynomials, \href{https://doi.org/10.1007/BF01389900}{\textit{Invent. Math.}} \textbf{44} (1978), 1--47.

\bibitem{Kmf0-1}
Kobayashi T., Multiplicity-free representations and visible actions on complex
 manifolds, \href{https://doi.org/10.34508/repsympo.2005.0_33}{\textit{Publ. Res. Inst. Math. Sci.}} \textbf{41} (2005), 497--549.

\bibitem{Kmf1}
Kobayashi T., Multiplicity-free theorems of the restrictions of unitary highest
 weight modules with respect to reductive symmetric pairs, in Representation
 Theory and Automorphic Forms, \textit{Progr. Math.}, Vol. 255, \href{https://doi.org/10.1007/978-0-8176-4646-2_3}{Birkh\"auser},
 Boston, MA, 2008, 45--109, \href{https://arxiv.org/abs/math.RT/0607002}{arXiv:math.RT/0607002}.

\bibitem{Kfmeth}
Kobayashi T., {$F$}-method for constructing equivariant differential operators,
 in Geometric Analysis and Integral Geometry, \textit{Contemp. Math.}, Vol.
 598, \href{https://doi.org/10.1090/conm/598/11998}{Amer. Math. Soc.}, Providence, RI, 2013, 139--146, \href{https://arxiv.org/abs/1212.6862}{arXiv:1212.6862}.

\bibitem{K1}
Kobayashi T., A program for branching problems in the representation theory of
 real reductive groups, in Representations of Reductive Groups, \textit{Progr.
 Math.}, Vol. 312, \href{https://doi.org/10.1007/978-3-319-23443-4_10}{Birkh\"auser}, Cham, 2015, 277--322, \href{https://arxiv.org/abs/1509.08861}{arXiv:1509.08861}.

\bibitem{KKP}
Kobayashi T., Kubo T., Pevzner M., Conformal symmetry breaking operators for
 differential forms on spheres, \textit{Lecture Notes in Math.}, Vol. 2170,
 \href{https://doi.org/10.1007/978-981-10-2657-7_1}{Springer}, Singapore, 2016.

\bibitem{KO}
Kobayashi T., {\O}rsted B., Analysis on the minimal representation of {${\rm
 O}(p,q)$}.~{II}. {B}ranching laws, \href{https://doi.org/10.1016/S0001-8708(03)00013-6}{\textit{Adv. Math.}} \textbf{180} (2003),
 513--550, \href{https://arxiv.org/abs/math.RT/0111085}{arXiv:math.RT/0111085}.

\bibitem{KOSS}
Kobayashi T., {\O}rsted B., Somberg P., Sou\v{c}ek V., Branching laws for
 {V}erma modules and applications in parabolic geometry.~{I}, \href{https://doi.org/10.1016/j.aim.2015.08.020}{\textit{Adv.
 Math.}} \textbf{285} (2015), 1796--1852, \href{https://arxiv.org/abs/1305.6040}{arXiv:1305.6040}.

\bibitem{KP1}
Kobayashi T., Pevzner M., Differential symmetry breaking operators:~{I}.
 {G}eneral theory and {F}-method, \href{https://doi.org/10.1007/s00029-015-0207-9}{\textit{Selecta Math. (N.S.)}} \textbf{22}
 (2016), 801--845, \href{https://arxiv.org/abs/1301.2111}{arXiv:1301.2111}.

\bibitem{KP2}
Kobayashi T., Pevzner M., Differential symmetry breaking operators:~{II}.
 {R}ankin--{C}ohen operators for symmetric pairs, \href{https://doi.org/10.1007/s00029-015-0208-8}{\textit{Selecta Math.
 (N.S.)}} \textbf{22} (2016), 847--911, \href{https://arxiv.org/abs/1301.2111}{arXiv:1301.2111}.

\bibitem{KP3}
Kobayashi T., Pevzner M., Inversion of {R}ankin--{C}ohen operators via
 holographic transform, \href{https://doi.org/10.5802/aif.3386}{\textit{Ann. Inst. Fourier (Grenoble)}} \textbf{70}
 (2020), 2131--2190, \href{https://arxiv.org/abs/1812.09733}{arXiv:1812.09733}.

\bibitem{KS1}
Kobayashi T., Speh B., Symmetry breaking for representations of rank one
 orthogonal groups, \href{https://doi.org/10.1090/memo/1126}{\textit{Mem. Amer. Math. Soc.}} \textbf{238} (2015),
 v+110~pages, \href{https://arxiv.org/abs/1310.3213}{arXiv:1310.3213}.

\bibitem{KS2}
Kobayashi T., Speh B., Symmetry breaking for representations of rank one
 orthogonal groups~{II}, \textit{Lecture Notes in Math.}, Vol. 2234, \href{https://doi.org/10.1007/978-981-13-2901-2}{Springer},
 Singapore, 2018, \href{https://arxiv.org/abs/1801.00158}{arXiv:1801.00158}.

\bibitem{Ku}
Kudla S.S., Seesaw dual reductive pairs, in Automorphic Forms of Several
 Variables ({K}atata, 1983), \textit{Progr. Math.}, Vol.~46, Birkh\"auser,
 Boston, MA, 1984, 244--268.

\bibitem{L}
Loos O., Bounded symmetric domains and {J}ordan pairs, \textit{Math. Lectures},
 Department of Mathematics, University of California, Irvine, 1977.

\bibitem{Ma}
Martens S., The characters of the holomorphic discrete series, \href{https://doi.org/10.1073/pnas.72.9.3275}{\textit{Proc.
 Nat. Acad. Sci. USA}} \textbf{72} (1975), 3275--3276.

\bibitem{MS}
Merigon S., Sepp\"anen H., Branching laws for discrete {W}allach points,
 \href{https://doi.org/10.1016/j.jfa.2010.01.025}{\textit{J.~Funct. Anal.}} \textbf{258} (2010), 3241--3265, \href{https://arxiv.org/abs/0906.5580}{arXiv:0906.5580}.

\bibitem{MO2}
M\"ollers J., Oshima Y., Discrete branching laws for minimal holomorphic
 representations, \textit{J.~Lie Theory} \textbf{25} (2015), 949--983,
 \href{https://arxiv.org/abs/1402.3351}{arXiv:1402.3351}.

\bibitem{MO}
M\"ollers J., Oshima Y., Restriction of most degenerate representations of
 {$O(1,N)$} with respect to symmetric pairs, \textit{J.~Math. Sci. Univ.
 Tokyo} \textbf{22} (2015), 279--338.

\bibitem{N}
Nakahama R., Construction of intertwining operators between holomorphic
 discrete series representations, \href{https://doi.org/10.3842/SIGMA.2019.036}{\textit{SIGMA}} \textbf{15} (2019), 036,
 101~pages, \href{https://arxiv.org/abs/1804.07100}{arXiv:1804.07100}.

\bibitem{N2}
Nakahama R., Computation of weighted {B}ergman inner products on bounded
 symmetric domains and restriction to subgroups, \href{https://doi.org/10.3842/SIGMA.2022.033}{\textit{SIGMA}} \textbf{18}
 (2022), 033, 105~pages, \href{https://arxiv.org/abs/2105.13976}{arXiv:2105.13976}.

\bibitem{Ne1}
Neretin Yu.A., Matrix analogues of the {$B$}-function, and the {P}lancherel
 formula for {B}erezin kernel representations, \href{https://doi.org/10.1070/SM2000v191n05ABEH000477}{\textit{Sb. Math.}} \textbf{191}
 (2000), 67--100.

\bibitem{Ne2}
Neretin Yu.A., On the separation of spectra in the analysis of {B}erezin
 kernels, \href{https://doi.org/10.1007/BF02482409}{\textit{Funct. Anal. Appl.}} \textbf{34} (2000), 197--207,
 \href{https://arxiv.org/abs/math.RT/9906075}{arXiv:math.RT/9906075}.

\bibitem{Ne3}
Neretin Yu.A., Plancherel formula for {B}erezin deformation of {$L^2$} on
 {R}iemannian symmetric space, \href{https://doi.org/10.1006/jfan.2000.3691}{\textit{J.~Funct. Anal.}} \textbf{189} (2002),
 336--408, \href{https://arxiv.org/abs/math.RT/9911020}{arXiv:math.RT/9911020}.

\bibitem{OZ}
{\O}rsted B., Zhang G., Tensor products of analytic continuations of
 holomorphic discrete series, \href{https://doi.org/10.4153/CJM-1997-060-5}{\textit{Canad. J.~Math.}} \textbf{49} (1997),
 1224--1241.

\bibitem{OR}
Ovsienko V., Redou P., Generalized transvectants-{R}ankin--{C}ohen brackets,
 \href{https://doi.org/10.1023/A:1022956710255}{\textit{Lett. Math. Phys.}} \textbf{63} (2003), 19--28.

\bibitem{P}
Peetre J., Hankel forms of arbitrary weight over a~symmetric domain via the
 transvectant, \href{https://doi.org/10.1216/rmjm/1181072389}{\textit{Rocky Mountain~J.~Math.}} \textbf{24} (1994),
 1065--1085.

\bibitem{PZ}
Peng L., Zhang G., Tensor products of holomorphic representations and bilinear
 differential operators, \href{https://doi.org/10.1016/j.jfa.2003.09.006}{\textit{J.~Funct. Anal.}} \textbf{210} (2004),
 171--192.

\bibitem{Prz}
Przebinda T., The duality correspondence of infinitesimal characters,
 \href{https://doi.org/10.4064/cm-70-1-93-102}{\textit{Colloq. Math.}} \textbf{70} (1996), 93--102.

\bibitem{R}
Rankin R.A., The construction of automorphic forms from the derivatives of
 a~given form, \href{https://doi.org/10.18311/jims/1956/16987}{\textit{J.~Indian Math. Soc.~(N.S.)}} \textbf{20} (1956),
 103--116.

\bibitem{Sat}
Satake I., Algebraic structures of symmetric domains, \textit{Princet. Leg. Libr.},
 \href{https://doi.org/10.1515/9781400856800}{Princeton University Press}, 1981.

\bibitem{Sek}
Sekiguchi H., Branching rules of singular unitary representations with respect
 to symmetric pairs {$(A_{2n-1},D_n)$}, \href{https://doi.org/10.1142/S0129167X13500110}{\textit{Internat.~J.~Math.}}
 \textbf{24} (2013), 1350011, 25~pages.

\bibitem{Sep1}
Sepp\"anen H., Branching laws for minimal holomorphic representations,
 \href{https://doi.org/10.1016/j.jfa.2007.04.004}{\textit{J.~Funct. Anal.}} \textbf{251} (2007), 174--209,
 \href{https://arxiv.org/abs/math.RT/0703795}{arXiv:math.RT/0703795}.

\bibitem{Sep2}
Sepp\"anen H., Branching of some holomorphic representations of {${\rm
 SO}(2,n)$}, \textit{J.~Lie Theory} \textbf{17} (2007), 191--227,
 \href{https://arxiv.org/abs/0907.0128}{arXiv:0907.0128}.

\bibitem{Sep3}
Sepp\"anen H., Tube domains and restrictions of minimal representations,
 \href{https://doi.org/10.1142/S0129167X08005114}{\textit{Internat.~J.~Math.}} \textbf{19} (2008), 1247--1268,
 \href{https://arxiv.org/abs/math.RT/0703796}{arXiv:math.RT/0703796}.

\bibitem{Sh}
Shimura G., Arithmetic of differential operators on symmetric domains,
 \href{https://doi.org/10.1215/S0012-7094-81-04845-6}{\textit{Duke Math.~J.}} \textbf{48} (1981), 813--843.

\bibitem{vDP1}
van Dijk G., Pevzner M., Berezin kernels of tube domains, \href{https://doi.org/10.1006/jfan.2000.3706}{\textit{J.~Funct.
 Anal.}} \textbf{181} (2001), 189--208.

\bibitem{vDP2}
van Dijk G., Pevzner M., Matrix-valued {B}erezin kernels, in Geometry and
 Analysis on Finite- and Infinite-dimensional {L}ie groups ({B}\c{e}dlewo,
 2000), \textit{Banach Center Publ.}, Vol.~55, \href{https://doi.org/10.4064/bc55-0-13}{Polish Acad. Sci. Inst. Math.},
 Warsaw, 2002, 269--288.

\bibitem{Z1}
Zhang G., Berezin transform on line bundles over bounded symmetric domains,
 \textit{J.~Lie Theory} \textbf{10} (2000), 111--126.

\bibitem{Z2}
Zhang G., Berezin transform on real bounded symmetric domains, \href{https://doi.org/10.1090/S0002-9947-01-02832-X}{\textit{Trans.
 Amer. Math. Soc.}} \textbf{353} (2001), 3769--3787.

\end{thebibliography}
\end{document}